\pdfoutput=1
\documentclass[12pt]{amsart}
\usepackage{fullpage}
\usepackage{microtype}
\usepackage{stmaryrd}

\usepackage[utf8]{inputenc}
\usepackage{amsfonts}
\usepackage{amsmath}
\usepackage{mathtools}
\usepackage{fancyhdr}
\usepackage{graphicx}
\usepackage{enumitem}
\usepackage{stackrel,amssymb}
\usepackage{stmaryrd}
\usepackage{bm}
\usepackage{tikz}
\usetikzlibrary{calc}
\usetikzlibrary{cd}
\usepackage{setspace}
\usepackage{multicol}
\usepackage{romannum}
\usepackage{amsthm} %
\usepackage[bookmarksdepth=3]{hyperref}
\usepackage{cleveref}
\usepackage{cite}
\usepackage{mathtools}
\usepackage{mathrsfs}
\usepackage{booktabs,rotating}
\input{format_basic.tex}

\newcommand{\Z}{\mathbb{Z}}
\newcommand{\Q}{\mathbb{Q}}
\newcommand{\Qp}{{\mathbb{Q}_p}}
\newcommand{\R}{\mathbb{R}}
\newcommand{\C}{\mathbb{C}}
\renewcommand{\O}{\mathcal{O}}
\let\Im\undefined
\DeclareMathOperator{\Im}{Im}

\newcommand{\Gal}{\text{Gal}}
\newcommand{\Hom}{\text{Hom}}
\newcommand{\End}{\text{End}}
\newcommand{\Ext}{\text{Ext}}
\newcommand{\Tor}{\text{Tor}}
\DeclareMathOperator{\rank}{rank}

\DeclareMathOperator{\Spec}{Spec}
\DeclareMathOperator{\Sp}{Sp}

\DeclareMathOperator{\supp}{supp}

\newcommand{\id}{\text{id}}

\renewcommand{\to}{\longrightarrow}
\newcommand{\from}{\longleftarrow}
\renewcommand{\mapsto}{\longmapsto}
\newcommand{\onto}{\twoheadrightarrow}
\newcommand{\into}{\xhookrightarrow{}}

\newcommand{\simto}{\stackrel{\sim}{\to}}

\newcommand{\p}{\mathfrak{p}}

\newcommand{\m}{\mathfrak{m}}
\newcommand{\g}{\mathfrak{g}}

\newcommand{\GL}{\text{GL}}

\newcommand{\GSp}{\text{GSp}}
\renewcommand{\Sp}{\text{Sp}}
\newcommand{\blank}{{-}}

\renewcommand{\hat}{\widehat}
\renewcommand{\bar}{\overline}
\renewcommand{\tilde}{\widetilde}
\renewcommand{\setminus}{\smallsetminus}

\usepackage{todonotes}

\newcommand*\fixitem {\item[]%
  \refstepcounter{enumi}\hskip-\leftmargin\labelenumi\hskip\labelsep}

\crefname{equation}{}{}
\Crefname{equation}{}{}
\crefname{enumi}{}{}
\Crefname{enumi}{}{}

\DeclareMathAlphabet{\mathpzc}{OT1}{pzc}{m}{it}

\newcommand{\A}{\mathbb{A}}

\renewcommand{\O}{\mathcal{O}}
\DeclareMathOperator{\cotimes}{\hat{\otimes}}

\DeclareMathOperator{\tot}{tot}

\DeclareMathOperator{\Res}{Res}
\DeclareMathOperator{\Ind}{Ind}
\DeclareMathOperator{\SigmaInd}{\Sigma \hspace{0.4mm}\text{-Ind}}

\DeclareMathOperator{\cInd}{c-Ind}

\newcommand{\otimesL}{\otimes^\mathbb{L}}

\newcommand{\sing}{\text{sing}}
\newcommand{\ad}{\text{ad}}

\newcommand{\Acal}{\mathcal{A}}

\newcommand{\Ccal}{\mathcal{C}}
\newcommand{\Dcal}{\mathcal{D}}

\newcommand{\Ucal}{\mathcal{U}}
\newcommand{\Vcal}{\mathcal{V}}

\newcommand{\Wcal}{\mathcal{W}}
\newcommand{\Hcal}{\mathcal{H}}

\newcommand{\Kcal}{\mathcal{K}}
\newcommand{\Tbb}{\mathbb{T}}
\newcommand{\boldLambda}{\tilde{\Lambda}}

\newcommand{\ram}{\text{ram}}
\newcommand{\ord}{\text{ord}}

\newcommand{\la}{\text{la}}

\newcommand{\an}{\text{an}}
\newcommand{\alg}{\text{alg}}
\newcommand{\sm}{\text{sm}}

\newcommand{\der}{\text{der}}

\newcommand{\adm}{\text{adm}}

\newcommand{\red}{\text{red}}
\newcommand{\ndiv}{\text{ndiv}}
\newcommand{\un}{\text{un}}

\newcommand{\cusp}{{\text{cusp}}}
\newcommand{\Eis}{{\text{Eis}}}
\newcommand{\Sigmala}{{\Sigma-\text{la}}}
\newcommand{\Sigmaan}{{\Sigma-\text{an}}}

\newcommand{\Fcal}{\mathcal F}
\newcommand{\Mcal}{\mathcal M}
\newcommand{\Ncal}{\mathcal N}

\newcommand{\Tcal}{\mathcal T}

\newcommand{\Ascr}{\mathscr{A}}
\newcommand{\Bscr}{\mathscr{B}}
\newcommand{\Cscr}{\mathscr{C}}
\newcommand{\Fscr}{\mathscr{F}}
\newcommand{\Tscr}{\mathscr{T}}
\newcommand{\Xscr}{\mathscr{X}}
\newcommand{\Yscr}{\mathscr{Y}}
\newcommand{\Zscr}{\mathscr{Z}}
\newcommand{\Scalp}{\mathcal{S}_p}
\newcommand{\Gbf}{\mathbf{G}}
\newcommand{\Hbf}{\mathbf{H}}
\newcommand{\Mbf}{\mathbf{M}}
\newcommand{\Nbf}{\mathbf{N}}
\newcommand{\Pbf}{\mathbf{P}}
\newcommand{\Sbf}{\mathbf{S}}
\newcommand{\Tbf}{\mathbf{T}}
\newcommand{\Abf}{\mathbf{A}}
\newcommand{\Zbf}{\mathbf{Z}}
\newcommand{\Ubf}{\mathbf{U}}
\newcommand{\Bbf}{\mathbf{B}}
\newcommand{\Jbf}{\mathbf{J}}
\newcommand{\hbf}{\mathbf{h}}
\newcommand{\Nfrak}{\mathfrak{N}}
\newcommand{\Zfrak}{\mathfrak{Z}}
\newcommand{\Gm}{\mathbb{G}_\text{m}}

\newcommand{\dfrak}{\mathfrak{d}}
\newcommand{\loccit}{\textit{loc. cit.}}

\newcommand{\HM}[1][]{%
\ifthenelse{\equal{#1}{}}{\mathpzc{M}}{\mathpzc{M}_{#1}}%
}
\makeatletter
\newcommand\newsubcommand[3]{\newcommand#1[1][]{\ifthenelse{\equal{##1}{}}{#2\sc@sub{#3}}{#2\scc@sub{#3}{##1}}}}
\def\scc@sub#1#2{\def\scc@thesub{#2,#1}\@ifnextchar_{\scc@mergesubs}{_{\scc@thesub}}}
\def\scc@mergesubs_#1{_{\scc@thesub,#1}}
\def\sc@sub#1{\def\sc@thesub{#1}\@ifnextchar_{\sc@mergesubs}{_{\sc@thesub}}}
\def\sc@mergesubs_#1{_{\sc@thesub,#1}}
\makeatother
\makeatletter
\newcommand\newsubsubcommand[2]{\newcommand#1[1][]{#2\sd@sub{##1}}}
\def\sd@sub#1{\def\sd@thesub{#1}\@ifnextchar_{\sd@mergesubs}{_{\sd@thesub}}}
\def\sd@mergesubs_#1{_{\sd@thesub,#1}}
\makeatother

\newsubcommand{\JM}{M}{0}
\newsubcommand{\JMs}{M}{s}
\newsubsubcommand{\JMsub}{M}
\newsubcommand{\JMsbf}{\mathbf{M}}{s}
\newcommand{\NJM}{\mathfrak{M}}
\renewcommand{\boxtimes}{\otimes}
\DeclareMathOperator*{\bigboxtimes}{\bigotimes}
\usepackage{scalerel}

\begin{document}
\pagenumbering{arabic}

\title{$S$-arithmetic (co)homology and $p$-adic automorphic forms}
\author{Guillem Tarrach}

\begin{abstract}
We study the $S$-arithmetic (co)homology of reductive groups over number fields with coefficients in (duals of) certain locally algebraic and locally analytic representations for finite sets of primes $S$.
We use our results to construct eigenvarieties associated to parabolic subgroups at places in $S$ and certain classes of supercuspidal and algebraic representations of their Levi factors.
We show that these agree with eigenvarieties constructed using overconvergent homology and that for definite unitary groups they are closely related to the Bernstein eigenvarieties constructed by Breuil--Ding in \cite{breuil_ding_bernstein_eigenvarieties}.
\end{abstract}

\maketitle

\setcounter{tocdepth}{1}
\tableofcontents

\section{Introduction}

Let $p$ be a prime number.
One notable feature of the theory of $p$-adic automorphic forms is that it has more than one incarnation, depending on what space of $p$-adic automorphic forms one is working with.
The most widely applicable constructions of such spaces (for example, completed cohomology and overconvergent cohomology) are defined in different ways using the cohomology of arithmetic groups (such as $\GL_n(\Z)$).
These groups contain important arithmetic information: if $\Gbf$ is a reductive group over a number field $F$, the cohomology of an arithmetic subgroup of $\Gbf(F)$ with coefficients in a finite-dimensional representation of $\Gbf(F \otimes_\Q \C)$ can be computed in terms of automorphic representations of $\Gbf(\A_F)$.

The goal of this article is to study a new construction involving instead the homology of $S$-arithmetic groups (such as $\GL_n(\Z[1/N])$ for a positive integer $N$) for finite sets of primes $S$ of $F$.
More precisely, we will study the homology of such groups with coefficients in certain (usually infinite-dimensional) representations that play an important role in the representation theory of the local groups at primes in $S$,
and show that these also contain important arithmetic information.
For example, we prove that one may compute the $S$-arithmetic homology of some smooth or locally algebraic complex representations of these local groups in terms of automorphic representations as in the arithmetic case.
Similarly, the homology $p$-adic locally analytic representations is related to $p$-adic automorphic forms.
One can therefore use $S$-arithmetic homology to give a construction of $p$-adic automorphic forms that works in a more general setting than previous constructions in the literature. 
This approach is better suited to study $p$-adic automorphic forms by using the representation theory of $p$-adic groups than, say, overconvergent cohomology.

\subsection{Automorphic representations and \texorpdfstring{$S$}{S}-arithmetic cohomology of locally algebraic representations}

Let us first explain the relation to classical automorphic representations.
As usual, it is more convenient to work with cohomology from the adelic point of view (as opposed to working directly with $S$-arithmetic groups), so let us fix a level structure $K^S$ away from $S$, i.e. an open compact subgroup of $\Gbf(\A_F^{S, \infty})$, and assume it is neat.
We will write $H^*(K^S, \blank)$ and $H_*(K^S, \blank)$ for $S$-arithmetic cohomology and homology of level $K^S$ -- see \Cref{subsection: arithmetic homology} for a precise definition.
They are isomorphic to a direct sum of the (co)homologies of $S$-arithmetic subgroups of $\Gbf(F)$.
It is well-known that the cohomology of such groups with coefficients in finite-dimensional representations of $\Gbf(F \otimes_\Q \C)$ is finite-dimensional and can also be computed in terms of automorphic representations (cf. \cite[Chapter XIV]{borel_wallach}), although in this case the only representations involved are those which are Steinberg at places in $S$.
The following theorem shows that when the coefficients are certain locally algebraic representations, a similar result holds with different conditions on the automorphic representations involved.

\begin{theorem}\label{theorem: introduction 1}
For each $v \in S$, let $\Pbf_v$ be a parabolic subgroup of $\Gbf_v := F_v \times_F \Gbf$ with Levi subgroup $\Mbf_v$.
For all $v \in S$, let $\omega_v$ be a supercuspidal representation of $\Mbf_v(F_v)$ that is compactly induced from a finite-dimensional representation of an open compact-mod-center subgroup,
and let $W$ be an irreducible finite-dimensional algebraic representation of $\Gbf(F \otimes_\Q \C)$ over $\C$.
Consider the $S$-arithmetic cohomology
$$
    H^* \left( K^S, \left( W \boxtimes_\C \bigboxtimes_{v \in S} \Ind_{\Pbf_v(F_v)}^{\Gbf(F_v)} (\omega_v)^\sm \right)' \right),
$$
where $\Ind_{\Pbf_v(F_v)}^{\Gbf(F_v)} (\blank)^\sm$ denotes smooth parabolic induction and $(\blank)'$ denotes the abstract dual.
This cohomology space is finite-dimensional and admits a decomposition into \emph{cuspidal} and \emph{Eisenstein} direct summands, the first of which can be described in terms of cuspidal automorphic representations $\pi = \bigotimes_v' \pi_v$ for $\Gbf(\A_F)$ satisfying the following conditions with respect to the data of $K^S$, $(\Pbf_v, \Mbf_v, \omega_v)_{v \in S}$ and $W$:
\begin{enumerate}
    \item $\pi^{S \infty} := \bigotimes_{\substack{v \not\in S \\ v \nmid \infty}}' \pi_v$ has non-zero $K^S$-invariant vectors,
    \item $\pi_\infty := \bigotimes_{v \mid \infty} \pi_v$ is cohomological of weight $W$,
    \item for all $v \in S$, $\omega_v(\delta_{\Pbf_v(F_v)}^{-1})$ is a subrepresentation of the Jacquet module of $\pi_v$ with respect to the parabolic opposite to $\Pbf_v$, where $\delta_{\Pbf_v(F_v)}$ is the modulus character of $\Pbf_v(F_v)$.
\end{enumerate}
\end{theorem}

See \Cref{theorem: properties of classical cuspidals in cohomology} for a more precise statement.
The condition on the supercuspidal representations $\omega_v$ in \Cref{theorem: introduction 1} is expected to be satisfied for all supercuspidal representations of reductive groups in this setting, and this is known for example when $\Mbf_v$ is a product of terms of the following forms:
\begin{itemize}
    \item A general linear group (cf. \cite{bushnell_kutzko}) or an inner form of it (cf. \cite{secherre_stevens_supercuspidals}),
    \item A unitary, symplectic or special orthogonal group, if $v$ does not lie above 2 (cf. \cite{stevens_supercuspidals_classical}),
    \item A tamely ramified group, if $v$ does not divide the order of its Weyl group (cf. \cite{fintzen_types}).
\end{itemize}
In fact, an analogue of \Cref{theorem: introduction 1} can be proven even if the $\omega_v$ are not supercuspidal, but we will only be interested in this case.
Thus, for the rest of the introduction we will assume that the hypotheses of \Cref{theorem: introduction 1} are satisfied.
By using well-known results of Zelevinsky \cite{zelevinsky_GLn}, the following theorem can also be deduced from \Cref{theorem: introduction 1}.

\begin{theorem}\label{theorem: introduction 1.5}
If for all $v \in S$, $\Gbf_v$ is isomorphic to $\GL_{n/F_v}$, then for all finite-length smooth (or locally algebraic) representations $\pi$ of $\prod_{v \in S} \Gbf_v(F_v)$, the cohomology $H^*(K^S, (W \otimes_\C \pi)')$ is finite-dimensional.
\end{theorem}

Both \Cref{theorem: introduction 1} and \Cref{theorem: introduction 1.5} have analogues for homology instead of cohomology, where one does not take duals in the coefficients.
Our methods work more naturally for homology, and this is the reason why the duals appear in these statements.

\subsection{\texorpdfstring{$S$}{S}-arithmetic cohomology of locally analytic representations}

In order to study $p$-adic automorphic forms, we are led to consider similar cohomology groups as we vary some of the supercuspidal data $p$-adically. In order to make sense of this, we have to make a few changes. Fixing an algebraic closure $\bar \Q_p$ of $\Q_p$ and an isomorphism $\bar \Q_p \simto \C$, we may change the field of coefficients from $\C$ to $\bar \Q_p$, and in fact to some finite Galois extension $E$ of $\Q_p$ in $\bar \Q_p$ which we may assume is large enough to split $\Mbf_v$ for all $v \mid p$.
Let $\Nbf_v$ be the unipotent radical of $\Pbf_v$.
If $\Sigma \subseteq S$ is a subset consisting only of places dividing $p$, then we will replace the smooth induction by locally ($\Q_p$-)analytic induction at the places in $\Sigma$.
We may also view $W$ as an algebraic representation of $\Gbf(F \otimes_\Q \Q_p) \subseteq \Gbf(F \otimes_\Q \bar \Q_p)$.
The representation $W$ is then a tensor product of irreducible algebraic representations $W_v$ for $v \mid p$, and we will write $U$ for the tensor product of those corresponding to places outside $\Sigma$.
If $\delta_v \colon \Mbf_v(F_v) \to E^\times$ is a continuous character for all $v \in \Sigma$, then we may consider the $K^S$-cohomology of the \emph{continuous} dual of
\begin{align}\label{eqn: sigmaind in introduction}
    U \boxtimes_E \left(
    \hat{\bigboxtimes}_{v \in \Sigma} \Ind_{\Pbf_v(F_v)}^{\Gbf(F_v)} (W_v^{\Nbf_v} \otimes \omega_v \otimes \delta_v)^\la
    \right)
    \hat{\boxtimes}_E
    \left(
    \bigboxtimes_{v \in S \setminus \Sigma} \Ind_{\Pbf_v(F_v)}^{\Gbf(F_v)} (\omega_v)^\sm
    \right),
\end{align}
where now $\Ind_{\Pbf_v(F_v)}^{\Gbf(F_v)} (\blank)^\la$ denotes locally ($\Q_p)$-analytic induction.
In order to shorten notation, let us set $V_\delta = (\bigboxtimes_{v \in \Sigma} W_v^{\Nbf_v} \otimes \omega_v \otimes \delta_v) \boxtimes ( \bigboxtimes_{v \in S \setminus \Sigma} \omega_v )$, and $\SigmaInd_P^G(U, V_\delta)$ for \Cref{eqn: sigmaind in introduction}.
If $U$ is trivial, the groups $\Gbf(F_v)$ act smoothly on this representation for $v \in S \setminus \Sigma$ and locally analytically when $v \in \Sigma$.
For brevity we will refer to it as a locally $\Sigma$-analytic representation of $G := \prod_{v \in S} \Gbf(F_v)$.
By duality, studying the cohomology of the dual of $\SigmaInd_P^G(U, V_\delta)$ is equivalent to  studying the homology of $\SigmaInd_P^G(U, V_\delta)$. The same methods that prove the finite dimensionality in \Cref{theorem: introduction 1} show that this is finite-dimensional, and can be used to give a criterion for its vanishing.
It is also interesting to study the $S$-arithmetic (co)homology of more general locally $\Sigma$-analytic representations of $G$.
A good source of such representations are the Orlik--Strauch functors $\Fcal_P^G$.
We show that the $S$-arithmetic homology of many of the representations in the image of these functors is also finite-dimensional.
Under the assumptions of \Cref{theorem: introduction 1.5}, we can also show finite-dimensionality for all such representations when they are of finite length.
We can relate (co)homology with locally algebraic and locally analytic coefficients via the following theorem.

\begin{theorem}\label{theorem: introduction 2}
Assume that for all $v \in \Sigma$, the maximal split torus in the center of $\Mbf_v$ acts numerically non-critically on $U \otimes V_1$ in the sense of \Cref{definition: small slope}. Then, the inclusion of locally algebraic induction into locally $\Sigma$-analytic induction induces an isomorphism
$$
    H_*( K^S, W \otimes \bigboxtimes_{v \in S} \Ind_{\Pbf_v(F_v)}^{\Gbf(F_v)} (\omega_v)^\sm )
    \simto
    H_*( K^S, \SigmaInd_P^G(U, V_1) )
$$
(and similarly for the cohomology of their duals).
\end{theorem}

In terms of $p$-adic automorphic forms, this theorem gives a numerical criterion for a $p$-adic automorphic form to be classical.
\Cref{theorem: introduction 2} can therefore be seen as a classicality theorem.
Finally, we show that the homology groups behave well as we allow the characters $\delta$ to vary in rigid-analytic families,
and use this to construct eigenvarieties.
More precisely, we prove the following result.

\begin{theorem}\label{theorem: introduction 3}
Assume that $E$ is sufficiently large. Fix $W$ and $\omega = (\omega_v)_{v \in S}$ as above. There exists a rigid-analytic space $\Wcal_\omega$ defined over $E$ parametrising locally analytic twists of $(\omega_v)_{v \in \Sigma}$ and a rigid-analytic space $\Yscr_{K^s, W, \omega}^\Sigma$ over $\Wcal_\omega$ whose fibre along a twist $(\omega_v \otimes \delta_v)_{v \in \Sigma}$ is canonically in bijection with the set of systems of eigenvalues for the action of a certain Hecke algebra on $H_*(K^S, \SigmaInd_P^G(U, V_\delta))$.
\end{theorem}

\subsection{Relation to previous work}

We will compare our constructions of eigenvarieties to previous constructions in the literature.
Let us assume that $S$ is the set of all primes of $F$ dividing $p$ since this is the usual setting in which these constructions are carried out.
Let us distinguish between two types of eigenvarieties: ``trianguline" or ``finite slope" eigenvarieties, i.e. those that correspond to the case when all $\Pbf_v$ are Borel subgroups, and ``paraboline" or ``parabolic" eigenvarieties, which are the rest.
The former correspond to $p$-adic automorphic forms which are finite-slope for all Hecke operators at places in $S$, whereas in the latter you allow infinite slope at some operators.

First, we will compare our construction with that of overconvergent (co)homology.
For the trianguline case, a very general source is \cite{hansen_universal_eigenvarieties}, which encompasses previous constructions such as \cite{chenevier_gln} or \cite{urban_eigenvarieties}.
We prove that our eigenvarieties in these cases are isomorphic to these (and that they do not depend on the supercuspidal and algebraic data).
For the paraboline case, these were carried out by Loeffler \cite{loeffler_overconvergent_algebraic_automorphic_forms} when $\Gbf$ is compact at infinity and by Barrera Salazar--Williams \cite{parabolic_overconvergent} when $\Gbf_v$ is unramified for all $v \in S$.
We prove that the former are isomorphic to ours, and that the latter are isomorphic to a variant of it, where the supercuspidal data is replaced by a different smooth representation $\cInd_J^G ( 1 )$ (here $J$ is a certain subgroup of $G$ and $1$ denotes the trivial representation of $J$), and extend their construction to the case where the group is not unramified.
We remark that \Cref{theorem: introduction 1} suggests that that one should indeed include the supercuspidal data.

Secondly, we will also compare our construction with those based on completed cohomology.
In the trianguline case, these were carried out by Emerton \cite{emerton_interpolation}.
In the paraboline case, these have been constructed by Hill--Loeffler \cite{hill_loeffler} and more recently by Breuil--Ding in \cite{breuil_ding_bernstein_eigenvarieties}, who have called them \emph{Bernstein eigenvarieties}.
We will focus on the comparison with Breuil--Ding, who work in the case where $\Gbf_v \simeq \GL_{n/F_v}$ for all $v \in S$ and $\Gbf(F \otimes_\Q \R)$ is compact.
We prove in this setting that the reduced varieties underlying Bernstein eigenvarieties and our eigenvarieties are (essentially) isomorphic.
In \cite{breuil_ding_bernstein_eigenvarieties} Breuil--Ding prove several important results about their eigenvarieties, their local geometry, and their connection to Galois representations.
As a consequence of the comparison, all of these results may also be applied to our eigenvarieties.

The comparison of our eigenvarieties with overconvergent cohomology follows naturally from our methods along with a comparison of two different finite slope functors. However, the comparison with completed cohomology requires more work, and is a consequence of an interpolation argument using \cite{hansen_universal_eigenvarieties} and Bergdall--Chojecki's adjunction theorem for Emerton's locally analytic Jacquet functor \cite{bergdall_adjunction} applied to numerically non-critical points, which are dense in the considered eigenvarieties.

To the author's knowledge, the first to study the $S$-arithmetic cohomology of locally analytic representations of reductive groups were Kohlhaase--Schraen in their article \cite{vanishing}.
The reader that is familiar with their work will notice that our work is heavily inspired by theirs, and indeed, our strategy for proving the above theorems is also inspired by theirs.
The ideas of Kohlhaase--Schraen have also inspired work by Gehrmann, Rosaria Pati and Rosso \cite{gehrmann_L_invariant} \cite{gehrmann_rosso_L_invariants} \cite{gehrmann_rosaria_L_invariants}, who have explored the applications of $p$-arithmetic cohomology of locally analytic representation to the study of automorphic $\mathcal L$-invariants.

\subsection{Methods}
\label{subsection: introduction methods}

As mentioned above, our strategy for proving \Cref{theorem: introduction 1} and its analogue for locally $\Sigma$-analytic representations are based on the work of Kohlhaase--Schraen \cite{vanishing}, who proved finite-dimensionality and a vanishing criterion for locally analytic induction in a special case, namely when $F = \Q$, $\Gbf(\R)$ is compact, $S = \{ p \}$, $\Gbf_p$ is split, semisimple and adjoint and $\Pbf_p$ is a Borel subgroup.
Let us give an outline of this strategy adapted to our setting.
By our compact induction assumption on $\omega_v$, $W_v^{\Nbf_v} \otimes \omega_v \otimes \delta_v$ is isomorphic to a compact induction $\cInd_{\HM[v]}^{\Mbf_v(F_v)} (\sigma_v)$ for each $v \in S$, where $\HM[v]$ is a (not too small) compact-mod-center subgroup of $\Mbf_v(F_v)$ and $\sigma_v$ is finite-dimensional.
The main idea is that one can relate the $S$-arithmetic homology of locally $\Sigma$-analytic induction to (a twisted version of) overconvergent homology $H_*(K^S J, \Acal^\Sigma_{\sigma, U})$.
Here, $J$ is again a certain compact open subgroup of $G$ and $\Acal^\Sigma_{\sigma, U}$ is a locally $\Sigma$-analytic representation of $J$ (and in fact, of a larger monoid).
Shapiro's lemma shows that there is an isomorphism between $H_*(K^S J, \Acal^\Sigma_{\sigma, U})$ and the $S$-arithmetic homology $H_*(K^S, \cInd_J^G(\Acal^\Sigma_{\sigma, U}))$.
The representation $\cInd_J^G(\Acal^\Sigma_{\sigma, U})$ of $G$ is also naturally a module for a Hecke algebra $\Hcal_{S,E}^-$, which can also be made to act trivially (in some sense) on $\SigmaInd_P^G(U^\Sigma, V_\delta)$ and on the coefficient field $E$. We write $E(1)$ for the field $E$ equipped with this trivial action of $\Hcal_{S,E}^-$.
The key technical result is the following.

\begin{theorem}\label{theorem: introduction 4}
There is an isomorphism in the derived category of $\Hcal_{S,E}^-[G]$-modules
$$
    E(1) \otimesL_{\Hcal_{S,E}^-} \cInd_J^G(\Acal^\Sigma_{\sigma, U}) \simto \SigmaInd_P^G(U, V_\delta).
$$
\end{theorem}

One can formally deduce from this that a similar statement holds after replacing these representations of $G$ by their $S$-arithmetic homology. A rough way of stating this is to say that the $S$-arithmetic homology of $\SigmaInd_P^G(U, V_\delta)$ can be identified with (derived) eigenquotients of 1 for the action of a Hecke algebra at $S$ on this twisted version of overconvergent homology.
In most cases it is possible to ``untwist" this action, and the identification is with systems of eigenvalues in the usual overconvergent homology spaces that can be read off from the action of the compact-mod-center groups $\HM[v]$ from which we are compactly inducing the $\sigma_v$.
For example, if $\Gbf_v$ splits over $F_v$ and $\Pbf_v$ is a Borel subgroup for all $v$, then $\Hcal_{S,E}^-$ has a basis consisting of Hecke operators associated to antidominant cocharacters of some fixed maximal split tori of the $\Gbf(F_v)$; in this case one can take $\HM[v]$ to be the group of $F_v$-points of these tori, and if $\sigma_v$ is irreducible (i.e. a character), a choice of uniformiser $\varpi_v$ of $F_v$ gives one such untwisting, and the eigenvalue of the Hecke operator corresponding to a cocharacter $\beta$ is $\sigma_v(\beta(\varpi_v))$.

We use this identification to prove \Cref{theorem: introduction 1}, as well as finite-dimensionality in the locally $\Sigma$-analytic case, the vanishing criterion we alluded to above, and the properties for representations in the image of Orlik--Strauch functors.
\Cref{theorem: introduction 2} can be deduced without much difficulty from these properties and the properties of the BGG category $\O$.
This is an example of one of the advantages that the $S$-arithmetic approach has over overconvergent (co)homology.
In the latter setting, the analogue of \Cref{theorem: introduction 2} is proven in \cite{urban_eigenvarieties} and in \cite{parabolic_overconvergent} by constructing an \textit{ad hoc} version for locally analytic representations of $J$ of BGG resolutions in the category $\O$.
In the $S$-arithmetic setting, however, where we can work with representations of the whole group $G$, the existence of such a resolution of locally analytic representations can be deduced \textit{directly} from the BGG resolution in $\O$ by applying the Orlik--Strauch functors to it.

Regarding $S$-arithmetic homology in families, the obvious reformulation of \Cref{theorem: introduction 4} when the character $\delta$ takes values in affinoid algebras remains true (and is also true for more general families of compactly induced representations than the $\sigma_v \otimes \delta_v$).
This again allows us to relate our $S$-arithmetic homology to overconvergent homology, but it is not as straightforward to deduce finiteness theorems from this relation, at least without assuming extra hypotheses (such as the existence of slope decompositions).
However, we prove a spectral theorem for compact operators on orthonormalisable Banach algebras to show these statements, and also to show that it is possible to glue these $S$-arithmetic homology groups for different affinoid algebras.
This allows us to construct eigenvarieties directly over the character space of the group $\prod_{v \in \Sigma} \Mbf_v(F_v)$ and not via Fredholm hypersurfaces over weight space, in a similar way to the constructions that use completed cohomology and Emerton's locally analytic Jacquet functor (cf. \cite{emerton_jacquet_one}, \cite{emerton_interpolation}).
In fact, this spectral theorem can be seen as a variant of the spectral theorem underlying the construction of locally analytic Jacquet modules.
It is often the case when studying eigenvarieties that one needs to rely on Fredholm hypersurfaces in order to prove meaningful properties, and this situation is no different.
To do this, we also generalise the usual construction of eigenvarieties using Fredholm hypersurfaces and overconvergent homology to our more general setting and relate it to our other construction.
It is important to note that in order to make this work we have to work with the untwisted (i.e. the usual) version of overconvergent homology mentioned above.

\subsection{Organisation of the article}

In the first section, \Cref{section: structure}, we will collect most of the results on the structure of $p$-adic reductive groups that we will need in the rest of the article (especially in \Cref{section: locally analytic}) putting special emphasis on rigid-analytic aspects.
Next, in \Cref{section: vanishing for Tor} we will write down some of the properties of the Tor groups that will become ubiquitous in the rest of the article.
Especially, we will prove a vanishing criterion that underlies part of the proof of \Cref{theorem: introduction 4}.
\Cref{section: locally analytic} contains the rest of the local part of the argument: we will define the relevant locally ($\Sigma$-)analytically induced representations and prove \Cref{theorem: introduction 4}.
In \Cref{section: homology} we will move on to study $S$-arithmetic homology and its relation to arithmetic homology.
In it, we prove \Cref{theorem: introduction 1}, \Cref{theorem: introduction 1.5} and \Cref{theorem: introduction 2} among other results on these $S$-arithmetic (co)homology groups.
Finally, in \Cref{section: eigenvarieties} we construct eigenvarieties, study some of their properties (in particular, proving \Cref{theorem: introduction 3}), and compare them to previous constructions.

\subsection{Acknowledgements}
I am especially grateful to Jack Thorne for suggesting these topics and for his continuous and invaluable help during the development of this article.
I am also thankful to Tony Scholl and Christian Johansson for helpful comments on this manuscript and to David Schwein for his comments on an earlier version of it.
This research received funding from the European Research Council (ERC) under the European Union's Horizon 2020 research and innovation programme (grant agreement No 714405).

\subsection{Notation}

If $C_\bullet$ is a chain complex, we write $H_*(C_\bullet) = \bigoplus_{i \in \Z} H_i(C_\bullet)$ and similarly for cohomology.
By ``module over a non-commutative ring", we always mean left module unless explicitly specified.
In any case, the only rings in this article that are non-commutative are group rings of non-abelian groups (in particular, all Hecke algebras are commutative), for which we may consider right modules as left modules and vice versa in the usual way.
If $V$ is a locally convex vector space over a finite extension of $\Q_p$, we will write $V'$ for its continuous dual, which we always equip with its strong topology.
All topological tensor products are projective tensor products.
By locally analytic, we always mean locally $\Q_p$-analytic.
We will usually use the same symbols to denote representations of a group and the underlying vector spaces or modules.

All rigid-analytic spaces appearing in this article are Tate rigid-analytic spaces.
If $A$ is an affinoid algebra over a finite extension $E$ of $\Q_p$, we will write $\Sp(A)$ for the corresponding affinoid rigid $E$-analytic space. If $X$ is a rigid $E$-analytic space, we will write $\O(X)$ to denote the algebra of global $E$-analytic functions on $X$.

If $\Hbf$ is an algebraic group over $E$, we will write $H := \Hbf(E)$ for its groups of $E$-points, and similarly for all other letters.
We will write $X^*(\Hbf)$ (resp. $X_*(\Hbf)$) for the group of $E$-rational characters (resp. cocharacters) of $\Hbf$, or $X^*_E(\Hbf)$ (resp. $X_*^E(\Hbf)$) if we want to emphasise the field $E$.
If $F$ is a field extension of $E$, we will write $\Hbf_F$ or $F \times_E \Hbf$ for the base change of $\Hbf$ from $E$ to $F$.
Similarly, if $F$ is a subfield of $E$ we will write $\Res_{E/F} \Hbf$ for the restriction of scalars.

\section{The structure of \texorpdfstring{$p$}{p}-adic reductive groups}
\label{section: structure}

In this section we will gather some facts about reductive groups over $p$-adic local fields that will be used throughout the article.
Throughout this section, except in \Cref{subsection: products of groups}, we let $L$ be a finite extension of $\Q_p$ with normalised valuation $v \colon L^\times \onto \Z$.
We normalise the absolute value $| \blank |$ of $L$ so that it sends a uniformiser to $q^{-1}$, where $q$ is the cardinality of the residue field of $L$.

\subsection{Basic notation}
\label{subsection: local setting}

Let $\Gbf$ be a connected reductive group over $L$, $\Sbf$ a maximal $L$-split torus of $\Gbf$, $\boldsymbol{\Zfrak}$ (resp. $\boldsymbol{\Nfrak}$) its centraliser (resp. normaliser) in $\Gbf$, and $\Zfrak = \boldsymbol{\Zfrak}(L)$ (resp. $\Nfrak = \boldsymbol{\Nfrak}(L)$).
Let $\Phi = \Phi(\Gbf, \Sbf)$ be the relative root system of $\Gbf$ with respect to $\Sbf$ and $W$ be its Weyl group, which is naturally isomorphic to $\Nfrak / \Zfrak$.
If $\alpha \in \Phi$ is a root, we will write $s_\alpha \in W$ for the corresponding simple reflection and $\Ubf_\alpha$ for the corresponding root subgroup of $\Gbf$.
We fix an ordering of the roots of $\Phi$ and let $\Phi^+$ (resp. $\Phi^-$) be the subset of positive (resp. negative) roots and $\Delta \subseteq \Phi^+$ the corresponding root basis.
Write $\Phi^\ndiv$ for the set of non-divisible roots in $\Phi$, i.e. the set of $\alpha \in \Phi$ such that $\frac12 \alpha \not\in \Phi$, and $\Phi^{\ndiv,+} = \Phi^\ndiv \cap \Phi^+$ (resp. $\Phi^{\ndiv,-} = \Phi^\ndiv \cap \Phi^-$).

Fix a subset $\Delta_M \subseteq \Delta$.
The subgroup $\Pbf$ of $\Gbf$ generated by $\boldsymbol{\Zfrak}$ and the root subgroups $\Ubf_\alpha$, where $\alpha \in \Phi$ runs through the positive roots and linear combinations of elements of $\Delta_M$, is a parabolic subgroup containing $\Sbf$, and all parabolic subgroups of $\Gbf$ arise this way (for possibly different choices of $\Sbf$ and orderings of the root system).
Let $\Abf$ be the identity component of $\bigcap_{\alpha \in \Delta_M} \ker(\alpha) \subseteq \Sbf$ and $\Mbf$ its centraliser in $\Gbf$. Then, $\Abf$ is the maximal split torus in the center of $\Mbf$, which is a Levi subgroup of $\Pbf$ containing $\Sbf$.
If $\Nbf$ denotes the unipotent radical of $\Pbf$, then $\Pbf$ admits a Levi decomposition $\Pbf = \Mbf \ltimes \Nbf$.
We will write $\bar \Pbf$ for the opposite parabolic subgroup to $\Pbf$ and $\bar \Nbf$ for its unipotent radical.

The root subsystem $\Phi_M$ of $\Phi$ consisting of linear combinations of elements of $\Delta_M$ is isomorphic to the root system of $\Mbf$ with respect to $\Sbf$, and $\Delta_M$ forms a root basis for the induced ordering of $\Phi_M$.
Set $\Phi^+_M = \Phi^+ \cap \Phi_M$, and define $\Phi^-_M, \Phi^{\ndiv,+}_M$ and $\Phi^{\ndiv, -}_M$ similarly.
The Weyl group of $\Phi_M$ is isomorphic to the subgroup $W_M$ of $W$ generated by the simple reflections $s_\alpha$ with $\alpha \in \Delta_M$.

The inclusion $\Sbf \subseteq \boldsymbol{\Zfrak}$ induces an inclusion $X^*(\boldsymbol{\Zfrak}) \subseteq X^*(\Sbf)$ of finite index.
There is a homomorphism $\ord \colon \Zfrak \to X_*(\Sbf) \otimes_\Z \Q$ sending $t$ to the cocharacter $\ord(t)$ defined by $\langle \alpha, \ord(t) \rangle = v(\alpha(t))$ for $\alpha \in X^*(\boldsymbol{\Zfrak})$. The kernel $\Zfrak_0$ of this map is the maximal compact open subgroup of $\Zfrak$, and $S_0 = S \cap \Zfrak_0$ and $A_0 = A \cap \Zfrak_0$ are the maximal compact subgroups of $S$ and $A$ respectively.
The quotient $A / A_0$ is a free abelian group whose rank is the rank of $\Abf$. We fix once and for all a section $A / A_0 \to A$ of the natural projection, and call its image $\Lambda$.

Most of the results of this section are consequences of standard facts from Bruhat--Tits theory \cite{bruhat_tits_1}, \cite{bruhat_tits_2}.
We refer the unfamiliar reader to \cite{corvallis_tits} for a survey of the topic, and we also note that many of the properties we will use are contained in the first two sections of \cite{morris_level_zero_types}.
Let $\Bscr = \Bscr(\Gbf, L)$ be the (enlarged) Bruhat--Tits building of $\Gbf$ over $L$. Recall that this is the product of the Bruhat--Tits buildings of the adjoint group $\Gbf^\ad$ of $\Gbf$ (that is, the reduced building of $\Gbf$) and of the cocenter $\Gbf / \Gbf^\der$ of $\Gbf$, where $\Gbf^\der$ is the derived subgroup of $\Gbf$.
Let $\Ascr = \Ascr(\Gbf, \Sbf, L)$ be the apartment of $\Bscr$ corresponding to $\Sbf$; this is an affine space under the vector space $X_*(\Sbf) \otimes_\Z \R$ with an action of $\Nfrak / \Zfrak_0$.
Let us also write $\Ascr_M$ for the apartment corresponding to $\Sbf$ of the Bruhat--Tits building of $\Mbf$.
There is a natural map $\Ascr \to \Ascr_M$ which is an isomorphism of affine spaces (cf. \cite[7.6.4]{bruhat_tits_1}, \cite[4.2.18]{bruhat_tits_2}).
Let $\Fscr_M$ be a facet of $\Ascr_M$ and $\JM$ be its pointwise stabiliser in $M$.
The preimage of $\Fscr_M$ in $\Ascr$ is a union of facets.
Fix a maximal facet $\Fscr$ in this preimage, a chamber $\Cscr$ whose closure $\bar \Cscr$ contains $\Fscr$ and a special point $x_0 \in \bar \Cscr$.
Let $K$ be the stabiliser in $G$ of $x_0$, and write $J$ for the pointwise stabiliser of $\Fscr$. In particular, $\JM = J \cap M$ (by \cite[Lemma 1.13]{morris_level_zero_types}), and $J \cap \Zfrak = K \cap \Zfrak = \JM \cap \Zfrak = \Zfrak_0$.
We also let $I$ be the pointwise stabiliser in $G$ of $\Cscr$ and $\NJM$ the normaliser of $\JM$ in $M$.

There exists a maximal torus $\Tbf$ in $\boldsymbol{\Zfrak}$ defined over $L$ containing $\Sbf$ which splits over a finite Galois extension $L'/L$. For example, if $\Gbf$ is quasi-split we may take $\Tbf = \boldsymbol{\Zfrak}$. We normalise the valuation and absolute value of $L'$ so that it extends those of $L$, and we also denote it by $v$ and $| \blank |$. Let $\Phi'$ be the root system of the pair $(\Gbf, \Tbf)$ and define $\Phi'_M$ in the same way as $\Phi_M$.
We have restriction maps $X^*_{L'}(\Tbf) \to X^*_L(\Sbf)$ and $\Phi' \to \Phi \cup \{ 0 \}$.
Given $\beta \in \Phi'$, we will also write $\Ubf_\beta$ for the corresponding root subgroup defined over $L'$. We may embed $\Bscr$ into the points in the Bruhat--Tits building of $\Gbf$ over $L'$ which are fixed under the natural action of $\Gal(L'/L)$, and similarly for the apartment $\Ascr$ of the former and the apartment $\Ascr'$ of the latter corresponding to $\Tbf$. These embeddings are bijections if $L'/L$ is tamely ramified (cf. \cite{prasad_yu_buildings}).

\subsection{Root subgroups and Iwahori decompositions}

Recall that we may associate to an affine function $\psi$ on $\Ascr$ whose vector part is a root $\alpha \in \Phi$ a subgroup $U_{\psi}$ of $U_\alpha$ as follows.
Given $u \in U_\alpha$, then $U_{- \alpha} u U_{-\alpha} \cap \Nfrak$ consists of a single element $m(u)$, whose action on $\Ascr$ is a reflection around a hyperplane that is perpendicular to $\alpha$.
If $\Psi(u, \alpha)$ denotes the affine function with vector part $\alpha$ that vanishes on this hyperplane, then $U_\psi$ is the subgroup of $U_\alpha$ consisting of the identity and those $u \in U_\alpha$ such that $\Psi(u, \alpha) \geq \psi$.
Recall also that $\psi$ is called an affine root if $U_\psi$ is not contained in $U_{\psi + \varepsilon} \cdot U_{2 \alpha}$ (where $U_{2 \alpha} = 1$ if $2 \alpha \not\in \Phi$) for any $\varepsilon > 0$.

If $\alpha \in \Phi$ and $r \in \R$, we will write $\psi_{\alpha, r}$ for the affine function on $\Ascr$ with vector part $\alpha$ and $\psi_{\alpha, r}(x_0) = r$, and $U_{\alpha, r} = U_{\psi_{\alpha, r}}$.
This is a decreasing family of neighbourhoods of the identity in $U_\alpha$.
The group $\Nfrak$ acts on $\Ascr$, and for each $n \in \Nfrak$ and $\psi$ as above we have $n U_\psi n^{-1} = U_{\psi \circ n^{-1}}$.
An element $t \in \Zfrak$ acts on the apartment $\Ascr(\Gbf^\ad, \Sbf, L)$ by translation by $-\langle \alpha, \ord(t) \rangle$ in the direction of $\alpha \in \Phi^{+, \ndiv}$, so in particular we have $t U_{\alpha, r} t^{-1} = U_{\alpha, r + \langle \alpha, \ord(t) \rangle}$. If $\langle \alpha, \ord(t) \rangle > 0$, then $t U_{\alpha, r} t^{-1}$ is strictly contained in $U_{\alpha, r}$.
More generally, we can write any element $n$ of $\Nfrak$ as a product $n = wt$, where $w \in \Nfrak$ fixes the special point $x_0$ and $t \in \Zfrak$ (see \Cref{lemma: extended Weyl groups and stabilisers} below), and then $n U_{\alpha, r} n^{-1} = U_{n\alpha, r + \langle \alpha, \ord(t) \rangle}$.

Recall that we say that an open subgroup $H \subseteq G$ admits an Iwahori decomposition with respect to $\Pbf$ if the multiplication map
\begin{align}\label{eqn: definition Iwahori decomposition}
    (H \cap \bar N) \times (H \cap M) \times (H \cap N) \to H
\end{align}
is an isomorphism of $p$-adic locally analytic manifolds (note that $H$ is automatically a $p$-adic locally analytic manifold).
As the multiplication map $\bar \Nbf \times \Mbf \times \Nbf \to \Gbf$ is an isomorphism of algebraic varieties onto an open subset of $\Gbf$, the map \Cref{eqn: definition Iwahori decomposition} is an isomorphism of $p$-adic locally analytic manifolds if and only if it is surjective.
The subgroups that we will encounter admitting Iwahori decompositions will admit Iwahori decompositions of a special kind (which is rather common when the subgroups arise from Bruhat--Tits theory), which to avoid repetition we will call \emph{rooted}.
Recall from \cite[Definition 6.4.3]{bruhat_tits_1} that a function $f \colon \Phi \to \R$ is called convex if for any finite family $(\alpha_i)_i$ of elements of $\Phi$ such that $\sum_i \alpha_i \in \Phi$ we have $f(\sum_i \alpha_i) \leq \sum_i f(\alpha_i)$. For example, if $\Omega$ is a bounded subset of $\Ascr$, then the function defined by $f(\alpha) = - \inf \psi_{\alpha, 0}(\Omega)$ is convex. We say that an open subgroup $H$ admits a \emph{rooted Iwahori decomposition} if it admits an Iwahori decomposition and, moreover, there exists a convex function $f \colon \Phi \to \R$ such that
\begin{align}\label{eqn: rooted Iwahori decomposition}
    \prod_{\alpha \in \Phi^{\ndiv,+} \setminus \Phi_M} U_{f, \alpha} & \to H \cap N, &
    \prod_{\alpha \in \Phi^{\ndiv,-} \setminus \Phi_M} U_{f, \alpha} & \to H \cap \bar N
\end{align}
are isomorphisms of $p$-adic locally analytic manifolds for some, or equivalently any (by \cite[Proposition 6.4.9 (ii)]{bruhat_tits_1}) ordering of the terms in the products, where \linebreak $U_{f, \alpha} := U_{\alpha, f(\alpha)} \cdot U_{2\alpha, f(2 \alpha)}$ is a subgroup of $U_\alpha$ for all $\alpha \in \Phi^\ndiv \setminus \Phi_M$ (where $U_{2\alpha, f(2 \alpha)} = \{ 1 \}$ if $2\alpha \not\in \Phi$). As above, the condition that the maps in \Cref{eqn: rooted Iwahori decomposition} be isomorphisms of $p$-adic locally analytic manifolds is equivalent to them being surjective.

\begin{lemma}\label{lemma: Iwahori decomposition for larger parabolic}
Assume that $H$ admits a rooted Iwahori decomposition with respect to a standard parabolic $\Pbf' \subseteq \Pbf$. Then, $H$ admits a rooted Iwahori decomposition with respect to $\Pbf$.
\end{lemma}
\begin{proof}
Let $\Mbf'$ (resp. $\Nbf'$, resp. $\bar \Nbf'$) be the standard Levi subgroup of $\Pbf'$ (resp. the unipotent radical of $\Pbf'$, resp. the unipotent radical of the opposite parabolic $\bar \Pbf'$). Let $\Phi_{M'} \subseteq \Phi_M$ be the root system of $\Mbf'$, and choose a convex function $f$ realising the rooted Iwahori decomposition. We first show that the multiplication map
$$
    \prod_{\alpha \in \Phi^{\ndiv,+} \setminus \Phi_M} U_{f, \alpha} \to H \cap N
$$
is surjective, and hence an isomorphism of $p$-adic locally analytic manifolds, for any ordering of the factors in the product such where the first roots are $\Phi^{\ndiv,+} \setminus \Phi_{M}$.
Any element $x$ of $H \cap N \subseteq H \cap N'$ can be written as a product $\prod_{\alpha \in \Phi^{\ndiv,+} \setminus \Phi_{M'}} x_\alpha = (\prod_{\alpha \in \Phi^{\ndiv,+} \setminus \Phi_{M}} x_\alpha) \cdot (\prod_{\alpha \in \Phi_M^{\ndiv,+} \setminus \Phi_{M'}} x_\alpha)$ with $x_\alpha \in U_{f, \alpha}$ for all $\alpha$.
Both $x$ and the first term in the last expression lie in $N$, while the second term lies in $M$. As $M$ and $N$ have trivial intersection, this means that the second term is 1, so $x$ must lie in $\prod_{\alpha \in \Phi^{\ndiv,+} \setminus \Phi_{M}} x_\alpha$ as claimed. The same argument shows the corresponding statement for $H \cap \bar N$.

If $\Sigma$ is a subset of $\Phi^\ndiv$, write $U_{f,\Sigma} = \prod_{\alpha \in \Sigma} U_{f, \alpha}$ to simplify notation.
To finish the proof, we must show that the map
$$
    (H \cap \bar N) \times (H \cap M) \times (H \cap N) \to H
$$
is surjective. This follows from the Iwahori decomposition of $H$ with respect to $\Pbf'$ by writing
\begin{align*}
    & (H \cap \bar N') \times (H \cap M') \times (H \cap N') \\
    & \simeq  
    U_{f,\Phi^{\ndiv,-} \setminus \Phi_M} \times U_{f,\Phi_M^{\ndiv,-} \setminus \Phi_{M'}} \times (H \cap M') \times U_{f,\Phi_M^{\ndiv,+} \setminus \Phi_{M'}} \times U_{f,\Phi^{\ndiv,+} \setminus \Phi_M} \\
    & \subseteq
    (H \cap \bar N) \times (H \cap M) \times (H \cap N). \qedhere
\end{align*}
\end{proof}

\begin{lemma}\label{equation equivalences for A+}
Assume that $H$ admits a rooted Iwahori decomposition and let $t \in \Zfrak$. The following are equivalent:
\begin{enumerate}
    \item $\langle \alpha, \ord(t) \rangle \leq 0 \text{ for all } \alpha \in \Phi^+ \setminus \Phi_M$,
    \item $t (H \cap \bar N) t^{-1} \subseteq H \cap \bar N$,
    \item $t^{-1} (H \cap N) t \subseteq H \cap N$.
\end{enumerate}
Similarly, the following are equivalent:
\begin{enumerate}
    \item $\langle \alpha, \ord(t) \rangle < 0 \text{ for all } \alpha \in \Phi^+ \setminus \Phi_M$,
    \item $\bigcap_{n \geq 0} t^{n} (H \cap \bar N) t^{-n} = \{1\}$,
    \item $\bigcap_{n \geq 0} t^{-n} (H \cap N) t^{n} = \{1\}$.
\end{enumerate}
\end{lemma}
\begin{proof}
If $t \in \Zfrak$ and $\alpha \in \Phi$, then $t U_{\alpha, r} t^{-1} = U_{\alpha, r + \langle \alpha, \ord(t) \rangle}$, and if $\langle \alpha, \ord(t) \rangle \neq 0$ then these are not equal to $U_{\alpha, r}$. Thus, if $t \in \Zfrak$ and $\alpha \in \Phi^{\ndiv, +}$, for any convex function $f$ one has $t U_{f, -\alpha} t^{-1} \subseteq U_{f, -\alpha}$ if and only if $\langle \alpha, \ord(t) \rangle \leq 0$, if and only if $t^{-1} U_{f, \alpha} t \subseteq U_{f, \alpha}$. Using that $H$ admits a rooted Iwahori decomposition, we see that in the first case above, \itemnumber{2} and \itemnumber{3} are equivalent to \itemnumber{1}.
The second case follows similarly.
\end{proof}

In the same way as in \cite[Remark 3.23]{remey_bruhat_tits_buildings_and_analytic_geometry}, if $f \colon \Phi \to \Q$ is a convex function, we may define affinoid subspaces $\Ubf_{f, \alpha}$ of (the rigid-analytification of) $\Ubf_\alpha$ such that each $U_{f, \alpha}$ can be naturally identified the set of $L$-points of $\Ubf_{f, \alpha}$.
Let us recall this construction.
We have fixed maximal torus $\Tbf$ defined over $L$ containing $\Sbf$ which splits over a finite extension $L'$ of $L$.
There exists an isomorphism of algebraic groups
\begin{align}\label{eqn: product of U_b}
    L' \times_L \Ubf_\alpha 
    \simeq
    \prod_{\substack{\beta \in \Phi' \\ \beta|_{\Sbf} = \alpha}} \Ubf_\beta \times \prod_{\substack{\beta \in \Phi' \\ \beta|_{\Sbf} = 2\alpha}} \Ubf_\beta
\end{align}
giving an identification
\begin{align}\label{eqn: U_fa is product of U_fb}
    U_{f, \alpha} = \left( \prod_{\substack{\beta \in \Phi' \\ \beta|_{\Sbf} = \alpha}} U_{\beta, f(\alpha)} \times \prod_{\substack{\beta \in \Phi' \\ \beta|_{\Sbf} = 2\alpha}} U_{\beta, f(2\alpha)} \right)^{\Gal(L'/L)}.
\end{align}
Here, the indexing on the groups $U_\beta$ is given by the image $x'_0$ of $x_0$ in $\Ascr'$ and the valuation $v$ of $L'$ extending that of $L$. The subgroups $U_{\beta, r}$ for $\beta \in \Phi'$ are naturally the set of $L'$-points of rigid $L'$-analytic spaces $\Ubf_{\beta, r}$:
a choice of special vertex $y_0$ in $\Ascr'$ determines isomorphisms from the additive group over $L'$ to $\Ubf_\beta$, and we take $\Ubf_{\beta, r}$ to be the preimage under these isomorphisms of the closed ball of radius $q^{-r - \beta(y_0 - x'_0)}$ centered at the origin.
This definition is independent of the choice of $y_0$.
The product
\begin{align}\label{eqn: product of U_fb}
    \prod_{\substack{\beta \in \Phi' \\ \beta|_{\Sbf} = \alpha}} \Ubf_{\beta, f(\alpha)} \times \prod_{\substack{\beta \in \Phi' \\ \beta|_{\Sbf} = 2\alpha}} \Ubf_{\beta, f(2\alpha)}.
\end{align}
is stable under the action of $\Gal(L'/L)$ on $L' \times_L \Ubf_\alpha$ (it's enough to check this on points over finite extensions of $L$ by the argument used in the proof of \Cref{lemma: rigid analytic parahorics} below) and the affinoid $\Ubf_{f, \alpha}$ is defined as its quotient by this action.

\begin{lemma}\label{lemma: inclusion of root subgroups is compact}
Let $\alpha \in \Phi^\ndiv$ and $f, g \colon \Phi \to \R$ be convex functions with $f(\alpha) < g(\alpha)$ and $f(2\alpha) < g(2 \alpha)$ if $2 \alpha \in \Phi$. Then, the map of $L$-affinoid algebras $\O(\Ubf_{f, \alpha}) \to \O(\Ubf_{g, \alpha})$ induced by the inclusion $\Ubf_{g, \alpha} \into \Ubf_{f, \alpha}$ is compact as a map of $L$-Banach spaces.
\end{lemma}
\begin{proof}
Let $R_f$ be the $L'$-algebra of rigid $L'$-analytic functions on the affinoid space \Cref{eqn: product of U_fb} for some fixed $\alpha$. The inclusion $\Ubf_{g, \alpha} \into \Ubf_{f, \alpha}$ then corresponds to the bottom row of the diagram of $L$-algebras
$$
    \begin{tikzcd}
        R_f \ar[r] & R_{g} \\
        (R_f)^{\Gal(L'/L)} \ar[u, hook] \ar[r] & (R_{g})^{\Gal(L'/L)} \ar[u, hook].
    \end{tikzcd}
$$
The map in the top row is compact, since it corresponds to a product of inclusions of closed balls with different radii. In this case, the rightmost vertical arrow admits a section as a map of $L$-Banach spaces by \cite[Proposition 10.5]{nonarchimedean_functional_analysis}, so the bottom arrow factors through the top arrow, and therefore will also be compact.
\end{proof}

\subsection{The groups \texorpdfstring{$K$}{K}, \texorpdfstring{$J$}{J} and \texorpdfstring{$\NJM$}{M}}

\begin{lemma}\label{lemma: J admits rooted Iwahori decomposition}
The subgroup $J$ admits a rooted Iwahori decomposition with respect to $\Pbf$.
\end{lemma}
\begin{proof}
Let $\Pbf'$ be the parabolic subgroup of $\Gbf$ associated to $\Fscr$ by \cite[Theorem 2.1]{morris_level_zero_types}, whose definition we now recall.
We can associate to the facet $\Fscr$ the subroot system $\Phi_{\Fscr}$ of $\Phi$ consisting of the vector parts of affine roots of $G$ that vanish at $\Fscr$.
Consider the reductive subgroup of $\Gbf$ containing $\Sbf$ whose root system is the closure ${}^c \Phi_\Fscr$ of $\Phi_\Fscr$ in $\Phi$. If $\Abf'$ is its connected center, then $\Pbf'$ is the standard parabolic subgroup of $\Gbf$ whose standard Levi subgroup is the centraliser $\Mbf'$ of $\Abf'$ in $\Gbf$.
The condition that the facet $\Fscr$ is maximal in the preimage of $\Fscr_M$ in $\Ascr$ ensures that $\Phi_\Fscr$, and hence also ${}^c \Phi_\Fscr$, is contained in $\Phi_M$. In particular, $\Pbf' \subseteq \Pbf$. By \cite[Theorem 2.1]{morris_level_zero_types} and its proof, $J$ admits a rooted Iwahori decomposition with respect to $\Pbf'$, and hence also with respect to $\Pbf$ by \Cref{lemma: Iwahori decomposition for larger parabolic}.
\end{proof}

\begin{remark}\label{remark: J cap N contained in K cap N}
    In fact, \cite[Theorem 2.1]{morris_level_zero_types} shows that $J \cap \bar N$ and $J \cap N$ are contained in the pro-$p$ radical of the parahoric subgroup of $J$ corresponding to $\Fscr$ (see \cite[3.6 Definition and 3.7 Proposition]{henniart_vigneras_satake}), which in turn is contained in (the Iwahori subgroup of) $I$.
    In particular, $J \cap \bar N \subseteq K \cap \bar N$ and $J \cap N \subseteq K \cap N$.
    However, it is not necessarily true that $J \cap M \subseteq K \cap M$.
\end{remark}

The group $K$ will not in general admit an Iwahori decomposition, but the following result still holds.

\begin{lemma}\label{lemma: semi Iwahori decomposition for K}
The multiplication map $(K \cap M) \times (K \cap N) \to K \cap P$ is an isomorphism of $p$-adic locally analytic manifolds.
\end{lemma}
\begin{proof}
This is proven in \cite[Theorem 6.5]{henniart_vigneras_satake}.
\end{proof}

\begin{remark}\label{remark: normaliser of J_M has free quotient}
The group $\JM$ is the intersection of its normaliser $\NJM$ and the subgroup ${}^0 M$ of $M$ consisting of elements $m$ such that $v(\chi(m)) = 0$ for all $\chi \in X^*(\Mbf)$, since these are the groups fixing pointwise the image of $\Fscr_M$ in the buildings of $\Mbf^\ad$ and $\Mbf / \Mbf^\der$ respectively.
As $X^*(\Mbf)$ can be naturally identified with a finite index subgroup of $X^*(\Abf)$, we deduce from the commutative diagram
$$
\begin{tikzcd}
    A / A_0 \ar[r, hook] \ar[d, "\sim"]
    & \NJM / \JM \ar[r, hook]
    & M / {}^\circ M \ar[dd, hook] \\
    X_*(\Abf) \ar[d, "\sim"] & & \\
    \Hom(X^*(\Abf), \Z) \ar[rr]
    & & \Hom(X^*(\Mbf), \Z)
\end{tikzcd}
$$
that $\NJM / \JM$ is a free abelian group of rank equal to the rank of the torus $\Abf$ and that $A \JM$ is a subgroup of $\NJM$ of finite index. Note that this inclusion may be strict, for example this can happen when $\Mbf$ is $\GSp_4$.
Inclusion $\Nfrak \cap \NJM \subseteq \NJM$ induces an isomorphism \linebreak
$(\Nfrak \cap \NJM) / (\Nfrak \cap \JM ) \simto \NJM / \JM$ (this follows, for example, from \cite[7.4.15]{bruhat_tits_1}).
\end{remark}

Consider the submonoids of $M$ defined by
\begin{align*}
    M^{-} & := \{ m \in M: m (J \cap \bar N) m^{-1} \subseteq J \cap \bar N \text{ and } m^{-1} (J \cap N) m \subseteq J \cap N \}
\end{align*}
and $M^+ = (M^-)^{-1}$.
For any subgroup $H$ of $M$, we will also write $H^-$ for $H \cap M^-$ and $H^+ = H \cap M^+$.
\Cref{equation equivalences for A+} and \Cref{lemma: J admits rooted Iwahori decomposition} imply that
\begin{align}\label{eqn: description of Zfrak-}
    \Zfrak^{-} & = \{ t \in \Zfrak: \langle \alpha, \ord(t) \rangle \leq 0 \text{ for all } \alpha \in \Phi^+ \setminus \Phi_M \}.
\end{align}
Similarly,
\begin{align}\label{eqn: description of A-}
    A^- & = \{ a \in A: v(\alpha(a)) \leq 0 \text{ for all } \alpha \in \Delta \setminus \Delta_M \}
\end{align}
since $v(\alpha(a)) = \langle \alpha, \ord(a) \rangle$ for all $\alpha$ and $\alpha(z) = 1$ for all $\alpha \in \Phi_M$.
Define also
\begin{align*}
    A^{--} & = \{ a \in A: v(\alpha(a)) < 0 \text{ for all } \alpha \in \Delta \setminus \Delta_M \}.
\end{align*}
This subsemigroup of $A$ will usually not contain the identity element, but it is always non-empty.
Note that if $\mu \in \NJM^-$, then
$$
    \mu^{-1} (J \cap P) \mu = \JM \mu^{-1} (J \cap N) \mu \subseteq \JM (J \cap N) = J \cap P.
$$
The following two lemmas are straight-forward generalisations of results of Schneider--Stuhler, \cite[Lemma 4.10 and Proposition 4.7]{schneider_stuhler}.

\begin{lemma}\label{lemma: products of double cosets}
\begin{enumerate}
    \fixitem If $m \in M^-$ and $\mu \in \NJM^-$, then $J m J \mu J = J m \mu J$.
    \item Let $\mu, \tilde \mu \in \NJM^-$. If $x, \tilde x$ run through coset representatives of $(J \cap \bar N) / \mu (J \cap \bar N) \mu^{-1}$ and $(J \cap \bar N) / \tilde \mu (J \cap \bar N) \tilde \mu^{-1}$ respectively, then $J \mu \tilde \mu J = \coprod_{x, \tilde x} x \mu \tilde x \tilde \mu J$.
\end{enumerate}
\end{lemma}
\begin{proof}
Statement \itemnumber{1} follows from the fact that
\begin{align*}
    m J \mu J & = m (J \cap \bar N) \mu \cdot \mu^{-1} (J \cap P) \mu \cdot J \\
    & = m (J \cap \bar N) m^{-1} \cdot m \mu J \subseteq J m \mu J.
\end{align*}
For \itemnumber{2}, note that $J \mu J = (J \cap \bar N) \mu J = \coprod_x x \mu J$ and similarly for $J \tilde \mu J$, so by part \itemnumber{1} it is clear that $J \mu \tilde \mu J = \bigcup_{x, \tilde x} x \mu \tilde x \tilde \mu J$, so it's enough to check that the cosets in the right-hand side are pairwise disjoint.
For a fixed $x$, it is clear that the cosets $x \mu \tilde x \tilde \mu J$ are pairwise disjoint for the various $\tilde x$. 
Thus, it's enough to show that if $x \mu J \tilde \mu J \cap \mu J \tilde \mu J \neq \emptyset$ then $x \in \mu (J \cap \bar N) \mu^{-1}$.
As $\mu J \tilde \mu J = \mu (J \cap \bar N) \tilde \mu J$, there exist $\bar{n}_0, \bar{n}_1 \in \mu (J \cap \bar N) \mu^{-1}$ such that $x \bar{n}_0 \mu \tilde \mu J = \bar n_1 \mu \tilde \mu J$.
Since $x \bar n_0$ and $\bar n_1$ are in $J \cap \bar N$, we deduce from the Iwahori decomposition that $x \bar n_0 \cdot \mu \tilde \mu ( J \cap \bar N) \tilde \mu^{-1} \mu^{-1} = \bar{n}_1 \cdot \mu \tilde \mu (J \cap \bar N) \tilde \mu^{-1} \mu^{-1}$. In particular, $x \in \mu (J \cap \bar N) \mu^{-1}$ as required.
\end{proof}

\begin{remark}\label{remark: for sufficiently small t multiplication by JzJ works well}
As $J \cap \bar N \subseteq K \cap \bar N$, the argument in the proof of \itemnumber{1} shows that for $m \in M^-$ and $\mu \in \NJM^-$,
$$
    K m J \cdot J \mu J \subseteq K \cdot m (J \cap \bar N) m^{-1} \cdot m \mu \cdot \mu^{-1} (J \cap P) \mu \cdot J \subseteq K m \mu J.
$$
\end{remark}

In particular, it follows from \Cref{lemma: products of double cosets} \itemnumber{2} that the subsets $J \NJM^- J$ and $J A^- J$  of $G$ (or more generally, $J \HM^- J$ for any finite index subgroup $\HM$ of $\NJM$) are submonoids of $G$.

\begin{lemma}\label{lemma disjointness of JP cosets}
Let $g, \tilde g \in J$, $\mu \in \NJM^-$ such that $g \mu J P \cap \tilde g \mu J P \neq \emptyset$. Then, $g \mu J = \tilde g \mu J$.
\end{lemma}
\begin{proof}
We may assume that $\tilde g = 1$. Write $g = \bar n p$ with $\bar n \in J \cap \bar N$ and $p \in J \cap P$. Since $\mu^{-1} (J \cap P) \mu \subseteq J \cap P$, we have
$$
    g \mu J = \bar n p \mu J = \bar n \mu (\mu^{-1} p \mu) J = \bar n \mu J.
$$
Hence, $\bar n \mu (J \cap \bar N) \mu^{-1} P \cap \mu (J \cap \bar N) \mu^{-1} P \neq \emptyset$. As $\mu (J \cap \bar N) \mu^{-1} \subseteq J \cap \bar N$, it follows from the injectivity of multiplication $\bar N \times P \to G$ that $\bar n \mu (J \cap \bar N) \mu^{-1} \cap \mu (J \cap \bar N) \mu^{-1} \neq \emptyset$. In other words, $\bar n \in \mu (J \cap \bar N) \mu^{-1}$, so $g \mu J = \bar n \mu J = \mu J$.
\end{proof}

\begin{lemma}\label{lemma: JM generated by submonoids}
$\NJM$ is generated as a monoid by $\NJM^-$ and $A^+$.
In particular, $\NJM$ is generated as a group by $\NJM^-$.
\end{lemma}
\begin{proof}
Let $\mu$ be an element of $\NJM$. Conjugation by $\mu$ preserves both $N$ and $\bar N$. The group $\mu (J \cap \bar N) \mu^{-1}$ is contained in the subgroup generated by $U_{\alpha, r_\alpha}$ with $\alpha \in \Phi^- \setminus \Phi_M$ for some $r_\alpha \in \R$, and similarly for $\mu^{-1} (J \cap N) \mu$. We may choose $s$ sufficiently large so that for all $\alpha$, $\Ubf_{\alpha, r_\alpha + s}$ is contained in $\Ubf_{\alpha, f(\alpha)}$, where $f$ is a convex function giving the rooted Iwahori decomposition of $J$.
Choose $a \in A^-$ satisfying $v(\alpha(a)) \leq -s$ for all $\alpha \in \Phi^+ \setminus \Phi_M$. Then, we have $a \mu (J \cap \bar N) \mu^{-1} a^{-1} \subseteq J \cap \bar N$ and $\mu^{-1} a^{-1} (J \cap N) a \mu \subseteq J \cap N$. In other words, $a \mu \in \NJM^-$.
\end{proof}

\subsection{Congruence subgroups of \texorpdfstring{$J$}{J}}
\label{subsection: congruence subgroups}

Next, we will introduce certain families of subgroups of $J$ similar to those considered in \cite[\S I.2]{schneider_stuhler_sheaves}. Let $s \geq 0$. Consider the subset $T'_s$ of $\Tbf(L')$ consisting of elements $t$ such that $v(\chi(t) - 1) \geq s$ for all $\chi \in X^*_{L'}(\Tbf)$.
Choose an ordering of $\Phi'$ compatible with that of $\Phi$ (so that the positive roots of $\Phi'$ restrict to elements of $\Phi^+ \cup \{0\}$) and let $\Bbf'$ be the associated Borel subgroup of $\Gbf_{L'}$ containing $\Tbf_{L'}$ over $L'$, $\Ubf'$ its unipotent radical and $\bar \Ubf'$ the unipotent radical of the opposite Borel.
Let $N'_s$ (resp. $\bar N'_s$, resp. $U'_s$, resp $\bar U'_s$) be the subgroup generated by $\Ubf_\psi$ where $\psi$ runs over all affine functions of $\Ascr'$ such that $\inf \psi|_{\Ascr}(\Fscr) \geq s$ and the image in $X^*(\Sbf) \otimes_\Z \R$ of the vector part of $\psi$ (resp. the image in $X^*(\Sbf) \otimes_\Z \R$ of the vector part of $\psi$, resp. the vector part of $\psi$, resp. the vector part of $\psi$) lies in $\Phi^+ \setminus \Phi_M$ (resp. $\Phi^- \setminus \Phi_M$, resp. $(\Phi')^+$, resp. $(\Phi')^-$).
Write also $U'_{M,s} = U'_s \cap \Mbf(L')$ and $\bar U'_{M,s} = \bar U'_s \cap \Mbf(L')$.
The image of $\Fscr_M$ (resp. $\Fscr$) in $\Ascr'_M$ (resp. $\Ascr'$) is contained in a minimal union of facets, let $M'_0$ (resp. $J'_0$) be the pointwise stabiliser of this union.
When $s > 0$, we let $\JMs'$ (resp. $J'_s$) be the subgroup of $\Gbf(L')$ generated by by $T'_s$, $U'_{M,s}$ and $\bar U'_{M,s}$ (resp. $T'_s$, $U'_s$ and $\bar U'_s$).
When $s = 0$, we let $\JMs'$ (resp. $J'_s$) be the stabiliser in $\Mbf(L')$ (resp. $\Gbf(L')$) of the image of $\Fscr_M$ (resp. $\Fscr$) in $\Ascr'_M$ (resp. $\Ascr'$).
The groups $N'_s, \bar N'_s, \JMs'$ and $J'_s$ are all stable under the action of $\Gal(L'/L)$.
Define $N_s, \bar N_s, \JMs$ and $J_s$ as the subgroups of $\Gal(L'/L)$-fixed points of the respective subgroups of $\Gbf(L')$.
Note that when $s = 0$ the notation for $M_s$ agrees with our previous use of $M_0$.
Finally, let $J_{0,s}$ be the subgroup of $J$ generated by $N_s$, $\JM$ (as defined in the previous sections) and $J \cap \bar N$ and let $J_{1,s}$ be the subgroup of $J$ generated by $N_s$, $\JMs$ and $J \cap \bar N$.

\begin{lemma}\label{lemma: rooted decomposition for N_s}
Let $f \colon \Phi \to \R$ be the convex function defined by $f(\alpha) = - \inf \psi_{\alpha, 0}(\Fscr)$. Then, the multiplication maps
\begin{align*}
    \prod_{\alpha \in \Phi^{\ndiv,+} \setminus \Phi_M} U_{f + s, \alpha} & \to N_s, &
    \prod_{\alpha \in \Phi^{\ndiv,-} \setminus \Phi_M} U_{f + s, \alpha} & \to \bar N_s
\end{align*}
are isomorphisms of locally $L$-analytic manifolds for any ordering of the terms in the products.
\end{lemma}
\begin{proof}
We will only prove the statement for $N_s$ since the other case is analogous. As usual, it is enough to show that the map above is surjective.
Let $f' \colon \Phi' \to \R$ be the function defined by $f'(\beta) = f(\beta|_{\Sbf})$ is $\beta|_{\Sbf} \neq 0$ and $f'(\beta) = 0$ otherwise.
It is a convex function of $\Phi'$.
Thus, for an affine root $\psi$ on $\Ascr'$ with vector part $\beta$ we have $U_\psi \subseteq U_{\beta, f'(\beta) + s}$ if and only if $\inf \psi|_\Ascr ( \Fscr ) \geq s$, so $N'_s$ is generated by $U_{\beta, f'(\beta) + s}$ as $\beta$ runs through $\Phi'^{+} \setminus \Phi'_M$.
Hence, by \cite[6.1.6]{bruhat_tits_1} multiplication $\prod_{\beta \in \Phi'^{+} \setminus \Phi'_M} U_{f'+s, \beta} \to N'_s$ is a bijection. Taking $\Gal(L'/L)$-invariants, the result follows from \Cref{eqn: U_fa is product of U_fb}.
\end{proof}

In particular, $N_0 = J \cap N$ and $\bar N_0 = J \cap \bar N$.
Moreover, $J_0 = J_{0,0} = J_{1,0} = J$.
Similar decompositions to the ones in \Cref{lemma: rooted decomposition for N_s} can be given for $M_s$ when $s > 0$.

\begin{lemma}\label{lemma: J_s is normal in J}
Let $s \geq 0$.
\begin{enumerate}
    \item $\JMs$ is a normal subgroup of $\NJM$.
    \item $J_s$ is a normal subgroup of $J$.
    \item $J_{1,s}$ is a normal subgroup of $J_{0,s}$.
    \item If $\mu \in \NJM^-$, then $\mu \bar N_s \mu^{-1} \subseteq \bar N_s$ and $\mu^{-1} N_s \mu \subseteq N_s$.
    \item $(\JMs)_s$ is a basis of neighbourhoods of the identity in $\JM$ by open subgroups.
\end{enumerate}
\end{lemma}

\begin{proof}
Statement \itemnumber{5} follows from the corresponding fact for $\JMs'$.
We know the other statements for $s = 0$, so we will focus on the case of positive $s$.
First, we note that the construction of $M_s$ essentially coincides with that of the groups $U_F^{(e)}$ in \cite[\S I.2]{schneider_stuhler_sheaves} (except that instead of $H_{e+}$ one should use $H_e$ in the definition) and the arguments leading to the comments after \cite[Proposition I.2.6]{schneider_stuhler_sheaves} apply also in this case to show that $\mathfrak M$ normalises $M_s$, which proves \itemnumber{1}.
For the remaining statements, we make the following observations for $s, \tilde s \geq 0$, $\psi|_\Ascr(\Fscr) \geq s$, $\tilde \psi|_\Ascr(\Fscr) \geq \tilde s$, $u \in U_\psi, \tilde u \in U_{\tilde \psi}$ and $t \in T'_{\tilde s}$:
\begin{enumerate}[label=(\alph*)]
    \item\label{item: congruence a} $u \tilde u u^{-1}$ lies in the subgroup generated by $U_{\tilde \psi}$ and $J'_{s + \tilde s}$,
    \item\label{item: congruence c} if the vector part of $\psi$ and $\tilde \psi$ have images in $\Phi_M \cup \Phi^- \cup \{0\}$ and $\Phi^- \setminus \Phi_M$ respectively, then $u \tilde u u^{-1} \in \bar N'_0$,
    \item\label{item: congruence d} $u t u^{-1}$ is contained in the subgroup generated by $U_{\psi + \tilde s}$ and $T'_{\tilde s}$,
\end{enumerate}
Indeed, \Cref{item: congruence a}
follows from \cite[1.4.2]{corvallis_tits}, \Cref{item: congruence c} follows from the same reference and the fact that $\bar \Nbf$ is normalised by $\bar \Pbf$, and \Cref{item: congruence d}
follows from \cite[Example 6.4.16 (b)]{bruhat_tits_1}.
Facts \Cref{item: congruence a} and \Cref{item: congruence d} together imply that conjugation by an element of $U'_0$ or $\bar U'_0$ maps the groups $U'_s, \bar U'_s$ and $T'_s$ into $J'_s$.
It follows from this that $\bar N_0$ and $N_0$ normalise $J_s$ and, together with \Cref{item: congruence a} and \Cref{item: congruence c}, that $\bar N_0$ and $N_s$ normalise $J_{1,s}$.

According to \cite[7.4.4]{bruhat_tits_1}, $M_0$ is generated by $\Nfrak \cap M_0$ and the subgroups $U_\psi$ where $\psi$ runs over all affine roots of $\Ascr_M$ such that $\psi(\Fscr_M) \geq 0$.
Again, it follows from \Cref{item: congruence a}, \Cref{item: congruence c} and \Cref{item: congruence d} as above that such $U_\psi$ normalise $J_s$ and $J_{1,s}$.
Therefore, taking into account \itemnumber{1}, statements \itemnumber{2} and \itemnumber{3} will follow if we show that $\Nfrak \cap M_0$ normalises both $\bar N_s$ and $N_s$.
In fact, by the same arguments, \itemnumber{4} will also follow if we show that conjugation by $\Nfrak \cap \NJM^-$ (resp. $\Nfrak \cap \NJM^+$) preserves $\bar N_s$ (resp. $N_s$).
This is indeed true: it is a consequence of the explicit description of the action of $\Nfrak$ on root subgroups and the corresponding statement for $s = 0$.
\end{proof}

\begin{proposition}\label{lemma: Iwahori decomposition for deep level}
Let $s \geq 0$.
The groups $J_s, J_{1,s}$ and $J_{0,s}$ admit rooted Iwahori decompositions
\begin{align*}
    \bar N_s \times \JMs \times N_s & \simto J_{s}, \\
    \bar N_0 \times \JMs \times N_s & \simto J_{1,s}, \\
    \bar N_0 \times \JM \times N_s & \simto J_{0,s}.
\end{align*}
\end{proposition}
\begin{proof}
The case $s = 0$ was proven in \Cref{lemma: J admits rooted Iwahori decomposition}, so assume $s > 0$. 
To obtain the rooted Iwahori decomposition for $J_{1, s}$ (resp. $J_{s}$) we can apply \cite[6.4.9 and 6.4.48]{bruhat_tits_1} to obtain a similar decomposition for a minimal parabolic for $J_{1,s}$ (resp. for $J'_s$, and then take $\Gal(L'/L)$-invariants to obtain one for $J_s$), and apply \Cref{lemma: Iwahori decomposition for larger parabolic} and \Cref{lemma: rooted decomposition for N_s}.
The decomposition for $J_{0,s}$ is proved in the same way as \cite[Theorem 2.1 (ii)]{morris_level_zero_types}, as we now show. As usual, it's enough to show that multiplication $\bar N_0 \times \JM \times N_s \to J_{0,s}$ is surjective. Let $U$ be the pro-$p$ radical of the parahoric subgroup of $J$ corresponding to $\Fscr$.
Recall from the comment below \Cref{lemma: J admits rooted Iwahori decomposition} that $\bar N_0 = U \cap \bar N$ and $N_0 = U \cap N$.
For each $x \in J_{0,s}$, there exists $m_0 \in \JM$ such that $y := m_0^{-1} x \in U$.
If we show that $y = \bar n m n$ with $\bar n \in U \cap \bar N, m \in \JM$ and $n \in N_s$, then $x = (m_0 \bar n m_0^{-1}) (m_0 m) n$ will be in the image of the multiplication map above.
In other words, it's enough to show that $U \cap J_{0,s}$ admits an Iwahori decomposition. Let $f \colon \Phi \to \R$ be the function defined by
$$
    f(\alpha) = \begin{cases}
        r_\alpha & \text{ if } \alpha \in \Phi^{-} \cup \Phi_M, \\
        r_\alpha + s & \text{ if } \alpha \in \Phi^{+} \setminus \Phi_M,
    \end{cases}
$$
where $r_\alpha = -\inf \psi_{\alpha, 0}(\Fscr)$,
and define $f^*$ as in \cite[6.4.23]{bruhat_tits_1}. Then, the Iwahori decomposition for $U \cap J_{0,s}$ with respect to a minimal parabolic follows from \Cref{lemma: Iwahori decomposition for larger parabolic} and \cite[6.4.9 and 6.4.48]{bruhat_tits_1} applied to $f^*$, and the decomposition with respect to $P$ follows by \Cref{lemma: Iwahori decomposition for larger parabolic} again.
\end{proof}

We have the following version of \Cref{lemma: products of double cosets} for $J_{0,s}$.

\begin{lemma}\label{lemma: products of double cosets for J0s}
Let $\mu, \tilde \mu \in \NJM^-$. Then $J_{0,s} \mu \tilde \mu J_{0,s} = \coprod_{x, \tilde x} x \mu \tilde x \tilde \mu J_{0,s}$, where $x, \tilde x$ run through coset representatives of $\bar N_0 / \mu \bar N_0 \mu^{-1}$ and $\bar N_0 / \tilde \mu \bar N_0 \tilde \mu^{-1}$ respectively.
\end{lemma}
\begin{proof}
As in the proof of \Cref{lemma: products of double cosets} we have $J_{0,s} \mu \tilde \mu J_{0,s} = \cup_{x, \tilde x} x \mu \tilde x \tilde \mu J_{0,s}$, and the cosets in the union are disjoint because the (bigger) corresponding $J$-cosets are.
\end{proof}

In particular, $J_{0,s} \NJM^- J_{0,s}$ is a submonoid of $G$. From \Cref{lemma: Iwahori decomposition for deep level}, it is easy to check that the inclusion $\JM \to J_{0,s}$ induces an isomorphism $\JM / \JMs \simto J_{0,s} / J_{1,s}$.
We will need the following refined version of this fact.

\begin{corollary}\label{corollary: projection from J_0 to J_M monoid}
The inclusion $\NJM^- \to J_{0,s} \NJM^- J_{0,s}$ induces a monoid isomorphism 
$$
    \NJM^- / \JMs \simto J_{0,s} \NJM^- J_{0,s} / J_{1,s}.
$$
\end{corollary}
\begin{proof}
It follows from \Cref{lemma: J_s is normal in J} \itemnumber{4} and \Cref{lemma: Iwahori decomposition for deep level} that the multiplication map $\bar N_0 \times \NJM^- \times N_s \to J_{0,s} \NJM^- J_{0,s}$ is bijective.
In order to prove the theorem, we must show 
that if $\bar n, \bar n', \bar n'' \in \bar N_0, \mu, \mu', \mu'' \in \NJM^-$ and $n, n', n'' \in N_s$ satisfy $\bar n \mu n \bar n' \mu' n' = \bar n'' \mu'' n''$, then
\linebreak $\mu'' \JMs = \mu \mu' \JMs$.
This can be seen from \Cref{lemma: J_s is normal in J} \itemnumber{1} and \itemnumber{4} since
$n \bar n' \in J_{1,s}$.
\end{proof}

Using \Cref{lemma: rooted decomposition for N_s} and the paragraph above \Cref{lemma: inclusion of root subgroups is compact}, we may view $\bar N_s$ as the set of $L$-points of the rigid $L$-analytic space
$$
    \bar \Nbf_s = \prod_{\alpha \in \Phi^{\ndiv,-} \setminus \Phi_M} \Ubf_{f + s, \alpha},
$$
and similarly for $N_s$.
In fact, the Iwahori decompositions from \Cref{lemma: Iwahori decomposition for deep level} allow us to canonically define rigid $L$-analytic spaces whose sets of $L$-points are $J_s$ and $\JMs$ when $s > 0$.
In order to do this, define $\bar \Ubf'_{s}, \Ubf'_{s}, \bar \Nbf'_s$ and $\Nbf'_s$ in the same way as $\bar \Nbf_s$ and $\Nbf_s$, and let $\Tbf'_s$ be the open affinoid subspace of (the rigid $L'$-analytification of) $\Tbf'$ defined by the same inequalities as $T'_s$.
Define $\Jbf'_s = \bar \Ubf'_{s} \times \Tbf'_s \times \Ubf'_{s}$, so that we have a natural open immersion $\Jbf'_s \into \Gbf_{L'}$, and $\Jbf'_s(L') = J'_s$ by \cite[6.4.9 and 6.4.48]{bruhat_tits_1}.
Then $\Jbf'_s$ is a $\Gal(L'/L)$-invariant subspace of $\Gbf_{L'}$, so we may define $\Jbf_s$ as its quotient by this action, and define $\JMsbf$ similarly.

\begin{lemma}\label{lemma: JM- preserves N_s}
Let $\mu \in \NJM^-$. For any finite extension $L''$ of $L$, $\mu \bar \Nbf_s(L'') \mu^{-1} \subseteq \bar \Nbf_s(L'')$.
\end{lemma}
\begin{proof}
By enlarging $L'', L'$ and $L$ if necessary, we may assume that $L'' = L' = L$.
We have already proved this case in \Cref{lemma: J_s is normal in J}.
\end{proof}

\begin{lemma}\label{lemma: rigid analytic parahorics}
\begin{enumerate}
    \fixitem Conjugation by an element of $J$ induces a rigid $L$-analytic endomorphism of $\Jbf_s$.
    \item There is a rigid $L$-analytic isomorphism
    $$
        \bar \Nbf_s \times \JMsbf \times \Nbf_s \simto \Jbf_s.
    $$
    \item Let $\mu \in \NJM^-$. Then, conjugation by $\mu$ induces a morphism of rigid $L$-analytic spaces $\bar \Nbf_s \to \bar \Nbf_s$.
\end{enumerate}
\end{lemma}
\begin{proof}
It's enough to prove the corresponding statements for the rigid analytic spaces over $L'$.
Statement \itemnumber{2} follows from the definitions and the argument in \Cref{lemma: Iwahori decomposition for larger parabolic}.
For \itemnumber{3}, given $\mu \in \NJM^-$ we have an inclusion $\mu \bar \Nbf'_s \mu^{-1} \subseteq \bar \Nbf'_s$ \emph{on rigid-analytic points} by \Cref{lemma: JM- preserves N_s}.
There exists some affinoid open $\Ucal$ of $\bar \Nbf'$ through which the inclusion and conjugation by $\mu$ maps $\bar \Nbf'_s \to \bar \Nbf'$ factor (since $\bar \Nbf'$ is isomorphic to affine space), and the previous sentence then implies that conjugation by $\mu$ on $\bar \Nbf'_s$ factors through $\bar \Nbf'_s$, as $\bar \Nbf'_s$ is an affinoid subspace of $\Ucal$ by \cite[Corollary 8.1.2/4]{BGR_nonarchimedean_analysis}.
The same argument shows \itemnumber{1} by taking into account that the inclusion and conjugation by elements of $J$, $\Jbf_s \into \Gbf$, also factor through some affinoid open (because we can embed $\Gbf$ into affine space).
\end{proof}

\subsection{Bruhat--Tits decomposition}

Let $\tilde W = \Nfrak / \Zfrak_0$. If $H$ is a subgroup of $G$ or of $M$ containing $\Zfrak_0$, we will write $\tilde W^H$ for the group $(\Nfrak \cap H) / \Zfrak_0$.

\begin{lemma}\label{lemma: extended Weyl groups and stabilisers}
\begin{enumerate}
    \fixitem The natural homomorphism $\tilde W^K \to W$ is an isomorphism.
    \item The group $\tilde W$ is isomorphic to the semidirect product $\tilde W^{K} \ltimes (\Zfrak / \Zfrak_0)$.
    \item We have $\Nfrak \cap J = \Nfrak \cap \JM$.
    In particular, the inclusion $\JM \subseteq J$ induces an isomorphism $\tilde W^{\JM} \simto \tilde W^J$.
\end{enumerate}
\end{lemma}
\begin{proof}
The map in \itemnumber{1} is injective, and the fact that $x_0$ is special implies that it is also surjective.
Part \itemnumber{2} follows from part \itemnumber{1} together with the exact sequence
$$
    1 \to \Zfrak / \Zfrak_0 \to \tilde W \to W \to 1
$$
For \itemnumber{3}, note that the groups $\Nfrak \cap J$ and $\Nfrak \cap \JM$ are the subgroups of $\Nfrak$ fixing the points of $\Fscr$ and $\Fscr_M$ respectively. As $\Nfrak$ acts on $\Ascr$ by affine transformations and both facets generate the same affine subspace, both subgroups are the same.
\end{proof}

Consider the following right action of $\tilde W$ on $\Zfrak_0 \backslash \Zfrak$. If $\mu \in \Nfrak$ and $t \in \Zfrak$, write $\mu = n \tilde t$ with $\tilde t \in \Zfrak$ and $n \in \Nfrak \cap K$ according to the decomposition in \Cref{lemma: extended Weyl groups and stabilisers} \itemnumber{2}, and define $(\Zfrak_0 t) \cdot \mu := \Zfrak_0 \tilde t \mu^{-1} t \mu = \Zfrak_0 n^{-1} t \mu$.
The following proposition generalises \cite[Lemma 2.20]{herzig_mod_p}.

\begin{proposition}\label{proposition: important decomposition}
\begin{enumerate}
    \fixitem Inclusion $\Zfrak \into G$ induces a bijection
    $$
         \Zfrak_0 \backslash \Zfrak / \tilde W^{\JM} \simto K \backslash G / J,
    $$
    where the action of $\tilde W^{\JM}$ on $\Zfrak_0 \backslash \Zfrak$ is induced by that of $\tilde W$.
    In particular, $G = K \Zfrak J$.
    
    \item For all $t \in \Zfrak$, there exist $t_1, ..., t_r \in \Zfrak$ such that $t J \mu \subseteq \bigcup_i K t_i \mu J$ for all $\mu \in \NJM^-$.
    
    \item Let $\Zfrak^\sim$ be a subsemigroup of $\Zfrak$ defined by $\langle \alpha, \ord(t) \rangle \leq c_\alpha$ for $\alpha \in \Phi^+ \setminus \Phi_M$ and some $c_\alpha \in \R_{\leq 0}$. Then, for all $g \in G$, there exists $a \in A^-$ such that $g J a J \subseteq K \Zfrak^\sim J$.
\end{enumerate}
\end{proposition}
\begin{proof}
The Bruhat--Tits decomposition \cite[7.4.15]{bruhat_tits_1} for $\Fscr$ and $x_0$ implies that inclusion $\Nfrak \into G$ induces a bijection $\tilde W^{K} \backslash \tilde W / \tilde W^{J} \simeq K \backslash G / J$. Now, $\tilde W^{K} \backslash \tilde W \simeq \Zfrak_0 \backslash \Zfrak$ by \Cref{lemma: extended Weyl groups and stabilisers} \itemnumber{2} and it is straightforward to check that the right action of $\tilde W$ by multiplication corresponds to the action on $\Zfrak_0 \backslash \Zfrak$ defined above.
Thus, \itemnumber{1} follows from \Cref{lemma: extended Weyl groups and stabilisers} \itemnumber{3}.

Recall that $I$ is the pointwise stabiliser of $\Cscr$ in $G$.
As explained in \cite[Proof of Lemma 2.20]{herzig_mod_p}, one can deduce from \cite[Th\'eor\`eme 6.5]{bruhat_tits_1} that $(G, I, \Nfrak)$ is a generalised $(B, N)$-pair, and it follows from the axioms of generalised $(B, N)$-pairs that if $\tilde n \in \Nfrak$, there exist $n_1, ..., n_r \in \Nfrak$ such that for all $n \in \Nfrak$, $n I \tilde n \subseteq \bigcup_{i=1}^r I n n_i I$.
Applying this to $\tilde n = t^{-1}$ shows that there exist $n_1, ..., n_r \in \Nfrak$ such that
$$
    t J \mu = t \bar N_0 \mu \mu^{-1} (J \cap P) \mu \subseteq t I \mu J \subseteq \bigcup_j I n_j \mu J \subseteq \bigcup_j K n_j \mu J.
$$
By \Cref{lemma: extended Weyl groups and stabilisers} \itemnumber{1} and \itemnumber{2}, there exist $t_1, ..., t_r \in \Zfrak$ such that $K n_j = K t_j$ for all $j$, so \itemnumber{2} follows.

Finally, in \itemnumber{3} the statement for $g$ depends only on the double coset $K g J$, so by \itemnumber{1} we may assume without loss of generality that $g = t \in \Zfrak$. Let $t_1, ..., t_r$ be as in \itemnumber{2}.
If $v_p(\alpha(a))$ is small enough for $\alpha \in \Phi^+ \setminus \Phi_M$, then $\langle \alpha, \ord(t_j a) \rangle \leq c_\alpha$ for all $\alpha \in \Phi^+ \setminus \Phi_M$, so $t_j a \in \Zfrak^\sim$ and $t J a \subseteq \bigcup_j K t_j a J$.
\end{proof}

\begin{corollary}\label{corollary: KtzJ are disjoint for different z}
If $\mu, \tilde \mu \in \NJM$ and $t \in \Zfrak$ satisfy $K t \mu J = K t \tilde \mu J$, then $\mu \JM = \tilde \mu \JM$.
\end{corollary}
\begin{proof}
We may assume without loss of generality that $\mu, \tilde \mu \in \Nfrak \cap \NJM$.
By the proof of \Cref{proposition: important decomposition} \itemnumber{1} and \Cref{lemma: extended Weyl groups and stabilisers} \itemnumber{3}, we have
$$
    (\Nfrak \cap K) t \mu (\Nfrak \cap \JM) = (\Nfrak \cap K) t \tilde \mu (\Nfrak \cap \JM).
$$
If $t \mu = k t \tilde \mu m$ with $k \in \Nfrak \cap K$ and $m \in \Nfrak \cap \JM$, then $k = t \mu m^{-1} \tilde \mu^{-1} t^{-1} \in K \cap M$ and $\mu \tilde \mu^{-1} = t^{-1} k t \tilde \mu m \tilde \mu^{-1} \in \NJM \cap (t^{-1} (K \cap M) t \cdot \tilde \mu \JM \tilde \mu^{-1})$.
Thus, it suffices to show that the set in the right-hand side is contained in $\JM$.
The subgroup ${}^0 M$ of $M$ defined in \Cref{remark: normaliser of J_M has free quotient} contains all the compact subgroups of $M$, and hence also $t^{-1} (K \cap M) t \cdot \tilde \mu \JM \tilde \mu^{-1}$.
Therefore, the result follows from the fact that $\NJM \cap {}^0 M = \JM$.
\end{proof}

\subsection{Hecke algebras}
\label{subsection: hecke algebras}

Given a locally profinite group $H$, a compact open subgroup $H_0$ and a commutative ring $R$, we will write $\Hcal(H, H_0)_R$ for the convolution $R$-algebra of $H_0$-biinvariant compactly supported functions $f \colon H \to R$. Explicitly, multiplication is given by
$$
    (f * \tilde f)(h) = \sum_{x \in H / H_0} f(x) \tilde f(x^{-1} h).
$$
Given $h \in H$, write $[H_0 h H_0]$ for the characteristic function on the double coset $H_0 h H_0$.
These form an $R$-basis of $\Hcal(H, H_0)_R$. If $H^-$ is a submonoid of $H$ containing $H_0$, we will write $\Hcal(H^-, H_0)_R$ for the subalgebra of functions whose support is contained in $H^-$.
As usual, taking $H_0$-(co)invariants of $H^-$-modules gives modules for $\Hcal(H^-, H_0)_R$ in the following way:

\begin{lemma}\label{lemma: general action of Hecke algebra}
\begin{enumerate}
    \fixitem If $\Mcal$ is a right $R[H^-]$-module, then the $H_0$-coinvariants $\Mcal_{H_0}$ have a natural right action of $\Hcal(H^-, H_0)_R$ via
    $$
        [m] [H_0 \mu H_0] = \left[ \sum_{x \mu H_0 \in H_0 \mu H_0 / H_0} m x \mu \right],
    $$
    where $[\blank]$ denotes the canonical map $\Mcal \to \Mcal_{H_0}$.
    \item If $\Mcal$ is a left $R[H^-]$-module, then the $H_0$-invariants $\Mcal^{H_0}$ have a natural left action of $\Hcal(H^-, H_0)_R$ via 
    $$
        [H_0 \mu H_0] m = \sum_{x \mu H_0 \in H_0 \mu H_0 / H_0} x \mu m.
    $$
\end{enumerate}
\end{lemma}
\begin{proof}
This is well-known when $H^-$ is a group, and the same arguments show that this holds for submonoids.
\end{proof}

\begin{lemma}\label{lemma: product of characteristic functions in Hecke algebra}
Let $\mu, \tilde \mu \in \NJM^-$ and $s \geq 0$. Then, in $\Hcal(G, J_{0,s})_R$ one has
$$
    [J_{0,s} \mu J_{0,s}] * [J_{0,s} \tilde \mu J_{0,s}] = [J_{0,s} \mu \tilde \mu J_{0,s}].
$$
\end{lemma}
\begin{proof}
For $g \in G$, one computes
\begin{align*}
    ([J_{0,s} \mu J_{0,s}] * [J_{0,s} \tilde \mu J_{0,s}])(g)
    & = \sum_{x \in J_{0,s} \mu J_{0,s} / J_{0,s}} [J_{0,s} \tilde \mu J_{0,s}](x^{-1} g) \\
    & = \sum_{x' \in \bar N_0 / \mu \bar N_0 \mu^{-1}} [J_{0,s} \tilde \mu J_{0,s}]((x' \mu)^{-1} g),
\end{align*}
A necessary condition for this sum to be non-zero is that $g \in x' \mu J_{0,s} \tilde \mu J_{0,s}$ for some $x'$, and this is contained in $J_{0,s} \mu J_{0,s} \tilde \mu J_{0,s} = J_{0,s} \mu \tilde \mu J_{0,s}$.
Conversely, if $g \in J_{0,s} \mu \tilde \mu J_{0,s}$, then by \Cref{lemma: products of double cosets for J0s} we may write $g = x \mu \tilde x \tilde \mu y$, where $x, \tilde x \in \bar N_0$ and $y \in J_{0,s}$.
If $x' \in \bar N_0$, then the summand corresponding to $x'$ is non-zero if and only if $\mu^{-1} (x')^{-1} x \mu \tilde x \tilde \mu y \in J_{0,s} \tilde \mu J_{0,s}$, if and only if $\mu^{-1} (x')^{-1} x \mu$ lies in $\bar N_0$, if and only if $x$ and $x'$ belong to the same coset of $\mu \bar N_0 \mu^{-1}$ in $\bar N_0$.
Thus, all the terms in the sum are zero except one, which is 1.
\end{proof}

\begin{proposition}\label{proposition: Hecke algebra isomorphism}
Restriction to $M$ induces an isomorphism
$$
    \Hcal(J_{0,s} \NJM^- J_{0,s}, J_{0,s})_R \simto \Hcal(\NJM^-, \JM)_R.
$$
\end{proposition}
\begin{proof}
Let $\mu \in \NJM^-$.
The Iwahori decomposition of $J_{0,s}$ implies $J_{0,s} \mu J_{0,s} = \bar N_0 \mu \JM N_s$, so $J_{0,s} \mu J_{0,s} \cap M = \mu \JM$.
Thus, there is a bijection $\NJM^- / \JM \to J_{0,s} \backslash J_{0,s} \NJM^- J_{0,s} / J_{0,s}$ and $[J_{0,s} \mu J_{0,s}]|_M = [\JM \mu \JM]$.
This implies that restriction to $M$ gives an isomorphism of $R$-modules.
By \Cref{lemma: product of characteristic functions in Hecke algebra}, restriction is compatible with convolutions, and hence it is an isomorphism of $R$-algebras.
\end{proof}

\begin{corollary}\label{corollary: Hecke algebras for J and J0s are isomorphic}
Restriction induces an $R$-algebra isomorphism
$$
    \Hcal(J \NJM^- J, J)_R \to \Hcal(J_{0,s} \NJM^- J_{0,s}, J_{0,s})_R.
$$
\end{corollary}
\begin{proof}
The composition $\Hcal(J \NJM^- J, J)_R \to \Hcal(J_{0,s} \NJM^- J_{0,s}, J_{0,s})_R  \to \Hcal(\NJM^-, \JMs)_R$ is an isomorphism (by \Cref{proposition: Hecke algebra isomorphism} applied to $s = 0$), and so is the second map.
\end{proof}

\begin{corollary}\label{corollary: hecke algebra isomorphic to monoid ring}
The algebra $\Hcal(J \NJM^- J, J)_R$ is isomorphic to the monoid ring $R[\NJM^- / \JM]$, in particular it is commutative.
In fact, it is isomorphic to a subring of $R[X_1^{\pm 1}, ..., X_{\rank(\Abf)}^{\pm 1}]$.
\end{corollary}
\begin{proof}
It's enough to show this for $\Hcal(\NJM^-, \JM)_R$, which follows from \Cref{remark: normaliser of J_M has free quotient}.
\end{proof}

\begin{remark}\label{remark: for Hecke algebras can replace tildeJ_M by HM}
All of these results remain true if one replaces $\NJM$ with a subgroup $\HM$ containing $A \JM$.
\end{remark}

\subsection{A lemma on abelianisations}

If $H$ is a closed subgroup of $M$, let us write $\overline{[H, H]}$ for the closure its commutator subgroup.

\begin{lemma}\label{lemma: abelianisation has finite kernel and cokernel}
Let $H$ be an open subgroup of $M$ containing $A H_0$ as a subgroup of finite index for some compact open subgroup $H_0$ of $M$. Then, the map $H / \overline{[H, H]} \to M / \overline{[M, M]}$ has finite kernel and cokernel.
\end{lemma}
\begin{proof}
Let $\Mbf^\der$ be the derived subgroup of $\Mbf$ and $\varphi \colon \Mbf^\text{sc} \to \Mbf^\der$ be the simply connected cover of $\Mbf^\der$.
Then, the commutator map $\Mbf \times \Mbf \to \Mbf$ factors through $\varphi$.
The simply connected group $\Mbf^\text{sc}$ is isomorphic to the product of its simple factors $\Mbf_i$.
We claim that $\overline{[\Mbf_i(L), \Mbf_i(L)]}$ has finite index in $\Mbf_i(L)$.
If $\Mbf_i$ is $L$-isotropic, then it follows from the Kneser--Tits conjecture \cite[Theorem B]{Platonov_1969} that each group $\Mbf_i(L)$ is perfect, and the index is 1.
If $\Mbf_i$ is $L$-anisotropic, then $\Mbf_i(L)$ is compact, so $\Mbf_i(L) / \overline{[\Mbf_i(L), \Mbf_i(L)]}$ is an abelian, compact $p$-adic locally analytic group.
Its Lie algebra is an abelian quotient of the Lie algebra of $\Mbf_i(L)$, which is semisimple, so it must be trivial.
Thus, the quotient $\Mbf_i(L) / \bar{[\Mbf_i(L), \Mbf_i(L)]}$ is a finite union of points, as required.
This implies that the inclusion $\overline{[\Mbf^\text{sc}(L), \Mbf^\text{sc}(L)]} \subseteq \Mbf^\text{sc}(L)$ has finite index, and since
$$
    \varphi(\overline{[\Mbf^\text{sc}(L), \Mbf^\text{sc}(L)]}) \subseteq \overline{[M, M]} \subseteq \varphi(\Mbf^\text{sc}(L))
$$
the second inclusion has finite index as well.
As the kernel of $\varphi$ is finite, its Galois cohomology in degree 1 is finite as well, so we conclude that $\bar{[M, M]}$ has finite index in $\Mbf^\der(L)$.
Hence, the map $M / \overline{[M, M]} \to M / \Mbf^\der(L)$ has finite kernel and cokernel, and it suffices to prove the statement of the lemma with the source replaced by the target.

Next, note that we may replace $H$ by $A H_0$.
Indeed, if the composition
$$
    A H_0 / \overline{[H_0, H_0]} \to H / \overline{[H, H]} \to M / \Mbf^\der(L)
$$
has finite kernel and cokernel, then the first map has finite kernel and the second has finite cokernel. But $A H_0$ has finite index in $H$, so the first map also has finite cokernel, and hence the second has finite kernel.
In fact, the same argument shows that we may replace $H_0$ with any open compact subgroup of it.
In particular, we are free to replace $H_0$ with a subgroup of the form $H_1 H_2$ where $H_1$ is contained in the center of $M$, $H_2$ is contained in $\Mbf^\der(L)$, and $H_1$ and $H_2$ have trivial intersection.

Let us first show that the cokernel is finite. It's enough to check that the subgroup of $M$ generated by $\Mbf^\der(L)$, $A$ and $H_1$ has finite index in $M$. The homomorphism \linebreak $M / A \Mbf^\der(L) \to (\Mbf / \Abf \Mbf^\der)(L)$ has finite kernel and cokernel (since the intersection of $\Mbf^\der$ and the center of $\Mbf$ is finite, hence has finite Galois cohomology), so it's enough to check that the image of $H_1$ in the target has finite index. But the target is compact and the image of $H_1$ in it is an open subgroup, so the result follows.
Let us now show that the kernel $(A H_1 H_2 \cap \Mbf^\der(L))/\overline{[H_2, H_2]}$ is finite. Since
$$
    A H_1 H_2 \cap \Mbf^\der(L) = (A H_1 \cap \Mbf^\der(L)) H_2
$$
and $A H_1 \cap \Mbf^\der(L)$ is finite, we are reduced to showing that $\overline{[H_2, H_2]}$ has finite index in $H_2$. This follows from the same Lie algebra argument as in the first paragraph.
\end{proof}

\subsection{A lemma on roots}
\label{subsection: lemma on roots}

Write ${}^M W$ for the set of minimal length coset representatives of cosets in $W_M \backslash W$.
Alternatively, an element $w$ of $W$ belongs to ${}^M W$ if and only if 
$w^{-1} \Phi^+_M \subseteq \Phi^+$.
Write $\ell$ for the length function on $W$.

\begin{lemma}\label{lemma: on roots}
Let $s \in S^-$ and $\lambda \in X^*(\Sbf)$ be a dominant weight. The minimum of the quantities $v_L((w \lambda)(s))$ for $w \in {}^M W, w \neq 1,$ is achieved at $w = s_\alpha$ for some $\alpha \in \Delta \setminus \Delta_M$.
\end{lemma}
\begin{proof}
Let $w \in {}^M W$ be any non-trivial element.
We claim that there exists $\alpha \in \Delta \setminus \Delta_M$ such that $\ell(s_\alpha w) < \ell(w)$.
Indeed, this is equivalent to $w^{-1} \alpha \in \Phi^-$, so if this is false we would have $w^{-1} (\Delta \setminus \Delta_M) \subseteq \Phi^+$, and hence $w^{-1} \Phi^+ \subseteq \Phi^+$, so $\ell(w^{-1}) = w^{-1} \Phi^+ \cap \Phi^- = 0$, which contradicts our assumption that $w$ is non-trivial.
So let $\alpha \in \Delta \setminus \Delta_M$ be such a root.
Clearly, $s_\alpha \in {}^M W$, so it's enough to show that $v_L( (s_\alpha \lambda)(s)) \leq v_L((w \lambda)(s))$.
As $s \in S^-$, by \Cref{eqn: description of Zfrak-} it's enough to check that $s_\alpha \lambda \geq w \lambda$, which follows from the fact that $s_\alpha \leq w$ in the Bruhat order.
\end{proof}

\subsection{Direct products of groups}
\label{subsection: products of groups}

When we apply the results we have proven so far in \Cref{section: locally analytic}, we will do so not for a reductive group over a finite extension of $\Q_p$, but rather a product of reductive groups over finite extensions of $\Q_\ell$ for different primes $\ell$. We will now introduce this setting.

Let $\ell_1, ..., \ell_n$ be prime numbers (each of which can appear more than once). For each $i$, let $L_i$ be a finite extension of $\Q_{\ell_i}$ and $\Gbf_i$ a reductive group over $L_i$. We make the same definitions as in \Cref{subsection: local setting} and write them with subindex $i$. Thus, for example we have subgroups $\Abf_i \subseteq \Sbf_i \subseteq \boldsymbol{\Zfrak}_i \subseteq \Mbf_i \subseteq \Pbf_i \subseteq \Gbf_i$,
and the pair $(\Gbf_i, \Sbf_i)$ has root system $\Phi_i$ with simple roots $\Delta_i$ and Weyl group $W_i$. We set $G_i = \Gbf_i(L_i)$ and $G = \prod_{i=1}^n G_i$, and similarly for other groups. We will mostly view these simply as locally profinite groups, or sometimes as products of locally profinite groups and $p$-adic locally analytic groups (if some of the $\ell_i$ are $p$). We will also set $\Phi = \coprod_{i=1}^n \Phi_i$, etc.

All of the results we have proven thus far remain valid for this setting after making the appropriate modifications to the definitions if necessary. For example, we say that a subgroup of the form $H = \prod_{i=1}^n H_i$ with $H_i \subseteq G_i$ admits an Iwahori decomposition if the product $(H \cap \bar N) \times (H \cap M) \times (H \cap N) \to H$ is a homeomorphism (or equivalently, surjective), and that it is rooted if there exists a convex function $f \colon \Phi \to \R$ such that the multiplication map $\prod_{\alpha \in \Phi^{\ndiv,+} \setminus \Phi_M} U_{f, \alpha} \to H \cap N$ is a homeomorphism (or equivalently, surjective) for any ordering of its factors and similarly for $\bar N$. Here, given $\alpha \in \Phi$ belonging to $\Phi_i$, the group $U_{f, \alpha}$ is the group $U_{f|_{\Phi_i}, \alpha} \times \prod_{j \neq i} \{ 1 \}$. Then, it is clear that $H$ admits a (rooted) Iwahori decomposition if and only if each of the $H_i$ does.

Because all of these results generalise immediately to this setting, we will not repeat them and instead quote directly the original results.
Let us also mention that if $H_{i,0}$ is a compact open subgroup of $G_i$ and $H_i^-$ a submonoid containing $H_{i,0}$, then there is an isomorphism $\Hcal(\prod_{i=1}^n H_i^-, \prod_{i=1}^n H_{i,0})_R \simeq \bigotimes_{i=1}^n \Hcal(H_i^-, H_{i,0})_R$.

\section{A vanishing criterion for \texorpdfstring{$\Tor$}{Tor}}
\label{section: vanishing for Tor}

The key technical result in \cite{vanishing} (as well as in this article) is the vanishing of certain Tor groups with respect to the monoid algebra of $\Lambda^-$.
To obtain them, Kohlhaase--Schraen use that, in their setting, this monoid algebra is isomorphic to a polynomial algebra, and they carry out an explicit computation.
In this section, we will give a more general (but still specifically tailored to our needs) criterion for the vanishing of the relevant Tor groups, which we obtain by reducing to the case of polynomial algebras.
Let us forget the notation from previous sections.

Let $R$ be a commutative ring and $\Lambda$ a free abelian group of finite rank, whose group law we write using multiplicative notation, and let $\Lambda' \subseteq \Lambda^- \subseteq \Lambda$ be submonoids satisfying $\Lambda' (\Lambda')^{-1} = \Lambda$. If $\chi \colon \Lambda \to R^\times$ is a group homomorphism, we will write $R(\chi)$ for the $R[\Lambda]$-module whose underlying $R$-module is $R$ and where $\Lambda$ acts via $\chi$.

\begin{lemma}\label{lemma: vanishing Tor 1}
Let $\Mcal$ be an $R[\Lambda^-]$-module and $q \in \Z$.
\begin{enumerate}
    \item The natural $R[\Lambda^-]$-module map $\Tor^{R[\Lambda^-]}_q(R(\chi), \Mcal) \to \Tor^{R[\Lambda]}_q(R(\chi), R[\Lambda] \otimes_{R[\Lambda^-]} \Mcal)$ is an isomorphism.
    
    \item Inclusion $\Lambda' \subseteq \Lambda^-$ induces an isomorphism $\Tor^{R[\Lambda']}_q(R(\chi), \Mcal) \simto \Tor^{R[\Lambda^-]}_q(R(\chi), \Mcal)$.
    
    \item Assume that $\Mcal$ contains an $R[\Lambda^-]$-submodule $\Mcal^-$ such that for all $m \in \Mcal$ there exists $\lambda \in \Lambda^-$ such that $\lambda m \in \Mcal^-$. Then, inclusion $\Mcal^- \subseteq \Mcal$ induces an isomorphism
    $\Tor^{R[\Lambda^-]}_q(R(\chi), \Mcal^-) \simto \Tor^{R[\Lambda^-]}_q(R(\chi), \Mcal)$
\end{enumerate}
\end{lemma}
\begin{proof}
Note that $R[\Lambda]$ is the localisation of $R[\Lambda^-]$ at $S = \Lambda^-$. Hence, $S^{-1} \Tor^{R[\Lambda^-]}_q(R(\chi), \Mcal)$ is isomorphic to $\Tor^{R[\Lambda]}_q(S^{-1} R(\chi), S^{-1} \Mcal)$. But $R(\chi)$ is an $R[\Lambda]$-module, and hence so is 
$\Tor^{R[\Lambda^-]}_q(R(\chi), \Mcal)$, and thus $\Tor^{R[\Lambda^-]}_q(R(\chi), \Mcal) \simeq \Tor^{R[\Lambda]}_q(R(\chi), R[\Lambda] \otimes_{R[\Lambda^-]} \Mcal)$, which proves part \itemnumber{1}.

By \itemnumber{1} applied to both $\Lambda^-$ and $\Lambda$, statement \itemnumber{2} will follow if we show that the natural map $R[\Lambda] \otimes_{R[\Lambda']} \Mcal \to R[\Lambda] \otimes_{R[\Lambda^-]} \Mcal$ is an isomorphism. It is clearly surjective, and its kernel is generated by elements of the form $\lambda \otimes m - 1 \otimes \lambda m$ with $\lambda \in \Lambda^-$, $m \in \Mcal$. Writing $\lambda = \lambda' (\mu')^{-1}$ with $\lambda', \mu' \in \Lambda'$ we see that
$$
    \mu' (\lambda \otimes m) = \lambda' \otimes m = 1 \otimes \lambda' m = \mu' (1 \otimes \lambda m) \text{  in } R[\Lambda] \otimes_{R[\Lambda']} \Mcal.
$$
As $\mu'$ is invertible in $R[\Lambda]$, this means that $\lambda \otimes m - 1 \otimes \lambda m = 0$, so the above homomorphism is also injective.

For statement \itemnumber{3} it's enough to show by \itemnumber{1} that the natural map
$$
    R[\Lambda] \otimes_{R[\Lambda^-]} \Mcal^- \to R[\Lambda] \otimes_{R[\Lambda^-]} \Mcal
$$
is an isomorphism. As localisation is exact, it's enough to show that $R[\Lambda] \otimes_{R[\Lambda^-]} (\Mcal / \Mcal^-)$ vanishes. This follows from our assumption on $\Mcal^-$, since any element of the form $\mu \otimes m$ can also be written as $\mu \lambda^{-1} \otimes \lambda m$ for any $\lambda \in \Lambda^-$.
\end{proof}

\begin{lemma}\label{lemma: vanishing Tor 2}
Assume that $\Mcal$ admits a decomposition $\Mcal = \Mcal^+ \oplus \bigoplus_{\lambda \in \Lambda^-} \Mcal_\lambda$ as $R$-modules satisfying
\begin{enumerate}[label=(\alph*)]
    \item\label{item: lemma Tor 2 a} $\lambda \Mcal_\mu \subseteq \Mcal_{\lambda \mu}$ for $\lambda, \mu \in \Lambda^-$,
    \item\label{item: lemma Tor 2 b} for all $m \in \Mcal$ there exists $\lambda \in \Lambda^-$ such that $\lambda m \in \bigoplus_{\lambda \in \Lambda^-} \Mcal_\lambda$.
\end{enumerate}
Then, the same conditions hold for $\Lambda'$ instead of $\Lambda^-$, and there is an isomorphism
$$
    \Tor^{R[\Lambda^-]}_q(R(\chi), \Mcal)
    \simeq 
    \Tor^{R[\Lambda']}_q(R(\chi), \bigoplus_{\lambda \in \Lambda'} \Mcal_\lambda).
$$
\end{lemma}
\begin{proof}
We first show that \Cref{item: lemma Tor 2 a} and \Cref{item: lemma Tor 2 b} remain valid if we replace $\Lambda^-$ by $\Lambda'$ everywhere. This is immediate for \Cref{item: lemma Tor 2 a}. Given $m \in \Mcal$, by \Cref{item: lemma Tor 2 b} there exist $\lambda, \lambda_1, ..., \lambda_n \in \Lambda^-$ and $m_i \in \Mcal_{\lambda_i}$ for $i=1,..., n$ such that $\lambda m = m_1 + \cdots + m_n$. We may write $\lambda = \lambda' (\mu')^{-1}$ with $\lambda', \mu' \in \Lambda'$ and $\lambda_i = \lambda'_i (\mu'_i)^{-1}$ similarly. Write $\mu^{(i)} = \mu'_1 \cdots \mu'_{i-1} \mu'_{i+1} \cdots \mu'_n$. Then,
$$
    (\lambda' \mu'_1 \cdots \mu'_n) m
    =
    \mu' (\mu'_1 \cdots \mu'_n) \lambda m
    = \sum_{i=1}^n \mu' \mu^{(i)} \cdot \mu'_i m_i
    \in \sum_{i=1}^n \Mcal_{\mu' \mu^{(i)} \lambda'_i} \subseteq \bigoplus_{\mu \in \Lambda'} \Mcal_{\mu}.
$$
This proves the first statement. The second statement follows from applying \Cref{lemma: vanishing Tor 1} \itemnumber{3} to $\Lambda'$ and $\bigoplus_{\lambda' \in \Lambda'} \Mcal_\lambda \subseteq \Mcal$.
\end{proof}

\begin{lemma}\label{lemma: vanishing Tor 3}
Let $\beta_1, ..., \beta_r \colon \Lambda \to \Z$ be linearly independent $\Z$-linear maps, and consider $\Lambda^- := \{ \lambda \in \Lambda: \beta_i(\lambda) \leq 0 $ for all $i \}$. Then:
\begin{enumerate}
    \item $\Lambda^-$ contains a free abelian monoid $\Lambda'$ of rank $d$ such that $\Lambda' (\Lambda')^{-1} = \Lambda$.
    \item If $\Mcal$ is an $R[\Lambda^-]$-module as in \Cref{lemma: vanishing Tor 2}, then $\Tor_q^{R[\Lambda^-]} (R(\chi), \Mcal) = 0$ for $q > 0$.
\end{enumerate}
\end{lemma}
\begin{proof}
Part \itemnumber{1} is a consequence of \cite[Theorem 11.1.9]{cox_et_al_toric_varieties}.
Indeed, let $\sigma$ be the cone of $\Lambda \otimes_\Z \R$ defined by the same inequalities as $\Lambda^-$.
Let $\sigma'$ be any cone of maximal dimension in the refinement $\Sigma'$ obtained by applying this theorem to the fan $\Sigma$ consisting of $\sigma$ and its faces.
Then, $\sigma'$ is a smooth cone, which means that $\Lambda' := \sigma' \cap \Lambda$ is a free abelian monoid as described, thus proving \itemnumber{1}.
By \Cref{lemma: vanishing Tor 2}, part \itemnumber{2} is equivalent to $\Tor_q^{R[\Lambda']} (R(\chi), \bigoplus_{\lambda \in \Lambda'} \Mcal_\lambda) = 0$ for $q > 0$, which follows from \Cref{lemma: vanishing Tor 4} below and \cite[X.9.6, Proposition 5]{bourbaki_algebra}.
\end{proof}

\begin{lemma}\label{lemma: vanishing Tor 4}
Let $R$ be a ring, $R[X_1, ..., X_d]$ a polynomial ring in $d$ variables with its usual $\Z_{\geq 0}^d$ grading. Let $\Mcal = \bigoplus_{m \in \Z_{\geq 0}^d} \Mcal_m$ be a graded $R[X_1, ..., X_d]$-module. Let $a_1, ..., a_d \in R^\times$, and set $Y_i := X_i - a_i$. Then, the sequence $(Y_1, ..., Y_d)$ of endomorphisms of $\Mcal$ is regular, i.e. if $1 \leq j \leq d$, then the endomorphism of $\Mcal / \sum_{i=1}^{j-1} \Im(Y_i)$ induced by $Y_j$ is injective.
\end{lemma}
\begin{proof}
This statement is a generalisation of \cite[Theorem 2.6]{vanishing}, and the same proof applies to this level of generality.
\end{proof}

\begin{lemma}\label{lemma: vanishing Tor 5}
Let $\boldLambda$ be a free abelian group of finite rank containing $\Lambda$ as a finite index subgroup, $\boldLambda^-$ a submonoid of $\boldLambda$ such that $\boldLambda^- (\boldLambda^-)^{-1} = \boldLambda$ and $\boldLambda = \boldLambda^- \Lambda$.
Let
\linebreak $\Lambda^- = \Lambda \cap \boldLambda^-$.
Assume that the index $[ \boldLambda : \Lambda ]$ is invertible in $R$.
Then, for any $R[\boldLambda^-]$-module $\Mcal$, character $\chi \colon \boldLambda \to R^\times$ and $q \in \Z$, the natural $R[\boldLambda]$-module homomorphism $\Tor^{R[\Lambda^-]}_q(R(\chi), \Mcal) \to \Tor^{R[\boldLambda^-]}_q(R(\chi), \Mcal)$ identifies the target with a direct summand of the source.
\end{lemma}
\begin{proof}
We begin by observing that for any $R[\boldLambda^-]$-module $M$, the natural map
$$
    R[\Lambda] \otimes_{R[\Lambda^-]} M \to R[\boldLambda] \otimes_{R[\boldLambda^-]} M
$$
is an isomorphism.
Indeed, as $R[\Lambda] \otimes_{R[\Lambda^-]} M = R[\Lambda] \otimes_{R[\Lambda^-]} R[\boldLambda^-] \otimes_{R[\boldLambda^-]} M$, it's enough to prove that the natural map $R[\Lambda] \otimes_{R[\Lambda^-]} R[\boldLambda^-] \to R[\boldLambda]$ is an isomorphism. This is a localisation of the $R[\Lambda^-]$-module inclusion $R[\boldLambda^-] \into R[\boldLambda]$, and hence is injective, and surjectivity follows from the condition $\boldLambda = \boldLambda^- \Lambda$.
Apply this to $\Mcal$ to see that, if $\Ncal = R[\boldLambda] \otimes_{R[\boldLambda^-]} \Mcal$, then by \Cref{lemma: vanishing Tor 1} \itemnumber{1} it's enough to show that the homomorphism
$$
    \Tor^{R[\Lambda]}_q(R(\chi), \Ncal) \to \Tor^{R[\boldLambda]}_q(R(\chi), \Ncal)
$$
identifies the target with an $R[\boldLambda]$-module direct summand of the source.
Let $P_\bullet$ be a resolution of $R(\chi)$ by free $R[\boldLambda]$-modules (in particular by flat $R[\Lambda]$-modules).
It's enough to show that the natural map $P_\bullet \otimes_{R[\Lambda]} \Ncal \to P_\bullet \otimes_{R[\boldLambda]} \Ncal$ identifies the target with a direct summand of the source.
This holds since $P_\bullet \otimes_{R[\Lambda]} \Ncal$ is isomorphic to $(P_\bullet \otimes_{R[\Lambda]} R[\boldLambda]) \otimes_{R[\boldLambda]} \Ncal$ and the map $p \otimes \lambda \mapsto \lambda p \colon P_\bullet \otimes_{R[\Lambda]} R[\boldLambda] \to P_\bullet$ admits the section $\frac1{[\boldLambda : \Lambda]} \sum_{\lambda \in \boldLambda / \Lambda} \lambda p \otimes \lambda^{-1} \mapsfrom p$.
\end{proof}

Putting everything together, we obtain the main result of this section.

\begin{proposition}\label{proposition: vanishing Tor 6}
Let $R,\boldLambda, \boldLambda^-, \Lambda, \Lambda^-, \Mcal$ and $\chi$ be as in \Cref{lemma: vanishing Tor 5}. Assume that the submonoid $\Lambda^-$ of $\Lambda$ is of the form described in \Cref{lemma: vanishing Tor 3} \itemnumber{1} and that $\Mcal$ satisfies
\linebreak the conditions of \Cref{lemma: vanishing Tor 2} for $\boldLambda^-$. More specifically, $\Mcal$ admits a decomposition
\linebreak $\Mcal = \Mcal^+ \oplus \bigoplus_{\lambda \in \boldLambda^-} \Mcal_\lambda$ as $R$-modules satisfying
\begin{enumerate}[label=(\alph*)]
    \item\label{item: lemma Tor 5 a} $\lambda \Mcal_\mu \subseteq \Mcal_{\lambda \mu}$ for $\lambda, \mu \in \boldLambda^-$,
    \item\label{item: lemma Tor 5 b} for all $m \in \Mcal$ there exists $\lambda \in \boldLambda^-$ such that $\lambda m \in \bigoplus_{\lambda \in \boldLambda^-} \Mcal_\lambda$.
\end{enumerate}
Then, $\Tor_q^{R[\boldLambda^-]} (R(\chi), \Mcal) = 0$ for $q > 0$.
\end{proposition}
\begin{proof}
By \Cref{lemma: vanishing Tor 1} and \Cref{lemma: vanishing Tor 5} it's enough to show that $\Tor^{R[\Lambda]}_q(R(\chi), R[\Lambda] \otimes_{R[\Lambda^-]} \Mcal)$ vanishes for $q > 0$.
Recall from the proof of \Cref{lemma: vanishing Tor 2} that there is a natural isomorphism
$$
    R[\Lambda] \otimes_{R[\Lambda^-]} \bigoplus_{\lambda \in \boldLambda^-} \Mcal_\lambda \simto R[\Lambda] \otimes_{R[\Lambda^-]} \Mcal.
$$
Write $\Ncal$ for the source of the isomorphism. Then, $\Ncal$ admits a direct sum decomposition $\Ncal = \bigoplus_{\lambda \in \boldLambda} \Ncal_\lambda$, where $\Ncal_\lambda$ is the $R$-submodule generated by $\mu \otimes m$ with $m \in \Mcal_{\mu^{-1} \lambda}$ and $\mu \in \Lambda$ such that $\mu^{-1} \lambda \in \Lambda^-$.
For every coset $L$ in $\boldLambda / \Lambda$, the $R[\Lambda]$-submodule $\Ncal_L = \bigoplus_{\lambda \in L} \Ncal_\lambda$ satisfies the assumptions of \Cref{lemma: vanishing Tor 2}, and hence $\Tor^{R[\Lambda]}_q(R(\chi), \Ncal_L) = 0$ for $q > 0$ by \Cref{lemma: vanishing Tor 3}. In particular, $\Tor^{R[\Lambda]}_q(R(\chi), \Ncal) = 0$ for $q > 0$.
\end{proof}

\section{The locally analytic representations}
\label{section: locally analytic}

In this section we will introduce the locally analytic representations whose homology we will be studying in the rest of the article and study some of their properties.
For example, we will prove \Cref{theorem: introduction 4}.

\subsection{Locally \texorpdfstring{$\Sigma$}{Sigma}-analytic and compact induction}
\label{subsection: induction 1}

Our ultimate goal is to study the $S$-arithmetic (co)homology of $p$-adic representations of products of $p$-adic and $\ell$-adic groups for primes $\ell \neq p$.
We want the corresponding representations of $p$-adic groups to be locally analytic, and the corresponding representations of $\ell$-adic groups to be smooth.
For this reason, we introduce the following framework of locally $\Sigma$-analytic representations, which will allow us to work simultaneously with both types of representations.

Recall that a Hausdorff topological space $X$ is called strictly paracompact if every open cover has a refinement by pairwise disjoint open subsets.
A locally $\Sigma$-analytic manifold is a pair $X = (X_\Sigma, X^\Sigma)$ where $X_\Sigma$ is a (strictly) paracompact $p$-adic locally analytic manifold and $X^\Sigma$ is a Hausdorff topological space that can be written as a disjoint union of compact, totally disconnected subspaces.
We will usually write $X = X_\Sigma \times X^\Sigma$.
A morphism of locally $\Sigma$-analytic manifolds $f \colon X \to Y$ consists of a pair of a morphism $f_\Sigma \colon X_\Sigma \to Y_\Sigma$ of locally analytic manifolds and a continuous map $f^\Sigma \colon X^\Sigma \to Y^\Sigma$.
We may similarly define a locally $\Sigma$-analytic group as a pair of a $p$-adic locally analytic group and a locally profinite group with the property that their quotients by some (or equivalently, any) compact open subgroups are countable (these spaces are then automatically strictly paracompact).
We will start by defining locally $\Sigma$-analytic functions on a locally $\Sigma$-analytic manifold $X$, which stated roughly are those that are locally analytic in $X_\Sigma$ and smooth in $X^\Sigma$.
If $Y$ is an open subset of $X_\Sigma$ equipped with an isomorphism $\mathbb Y(\Q_p) \simto Y$, where $\mathbb Y$ is a rigid $\Q_p$-analytic closed ball, and $U$ is an open compact subset of $X$, we will say that a function from $Y \times U$ to a ($p$-adic) Banach space $W$ is $\mathbb Y \times U$-analytic if it factors through projection to $Y$ and is given by a rigid $\Q_p$-analytic function of $\mathbb Y$, i.e. by an element of $\O(\mathbb Y) \cotimes_{\Q_p} W$.
If $V$ is a Hausdorff locally convex $\Q_p$-vector space, we define the space of locally $\Sigma$-analytic functions $X \to V$ as the direct limit
\begin{align}\label{eqn: definining sigma analytic}
    \Ccal^{\Sigmala}(X, V) = \varinjlim_{ \{ (\mathbb X_i, U_j, W_i)_{(i,j) \in I \times J} \} } \prod_{(i,j) \in I \times J} \O(\mathbb X_i) \cotimes_{\Qp} W_i
\end{align}
running over all collections of triples $(\mathbb X_i, U_j, W_i)_{(i,j) \in I \times J}$, where $( \mathbb X_i )_{i \in I}$ is an analytic partition of $X_\Sigma$ in the sense of \cite[Section 2.1]{emerton_locally_analytic_vectors}, $(U_j)_{j \in J}$ is a partition of $X^\Sigma$ into pairwise disjoint open compact subspaces, and $W_i \into V$ is a BH-subspace of $V$ for each $i \in I$ (i.e. an injective continuous map from a Banach space into $V$).
We equip $\Ccal^{\Sigmala}(X, V)$ with the locally convex direct limit topology.
We identify an element $f$ of $\Ccal^{\Sigmala}(X, V)$ with a continuous function $X \to V$: if $f$ belongs to an element of the product in \Cref{eqn: definining sigma analytic}, say $f = (\sum_m f_{i,j,m} \otimes w_{i,j,m})_{i,j}$, and $x \in X$ belongs to $\mathbb X_i(\Q_p) \times U_j$, then $f(x) = \sum_m f_{i,j,m}(x) w_{i,j,m}$.
If $X^\Sigma$ is compact, then there is a topological isomorphism $\Ccal^{\Sigmala}(X, V) \simeq \Ccal^\sm(X^\Sigma, \Ccal^\la(X_\Sigma, V))$.

Now, let $G = G_\Sigma \times G^\Sigma$ be a locally $\Sigma$-analytic group.
If $V$ is endowed with a continuous action of $G$, we will say that this action is locally $\Sigma$-analytic if for all $v \in V$ the orbit map $g \mapsto g v$ is locally $\Sigma$-analytic.
We do not impose any conditions on the topology of $V$ such as being of compact type.

If $H \subseteq G$ is a closed (locally $\Sigma$-analytic) subgroup and $V$ is a locally $\Sigma$-analytic representation of $H$ on a locally convex $\Qp$-vector space, we write $\Ind_{H}^G (V)^\Sigmala$ for the locally analytic induction
$$
    \Ind_{H}^G (V)^\Sigmala = \{ f \in \Ccal^\Sigmala(G, V): f(gh) = h^{-1} f(g) \text{ for all } g \in G, h \in H \},
$$
and equip it with the topology it inherits as a closed subspace of $\Ccal^\Sigmala(G, V)$. The group $G$ acts on $\Ind_{H}^G (V)^\Sigmala$ by left translations: $(gf)(g') = f(g^{-1} g')$ for $g, g' \in G$ and 
$f \in \Ind_H^G(V)^\Sigmala$.
Given a subset $\Omega \subseteq G$, we will write $\Ind_H^G (V)^{\Sigmala}(\Omega)$ for the subspace of $\Ind_H^G (V)^{\Sigmala}$ consisting of functions $f$ whose support $\supp(f)$ is contained in $\Omega$.

We are also interested in the compact induction
\begin{align*}
    \cInd_H^G (V) = \{ & f \colon G \to V: f(gh) = h^{-1} f(g) \text{ for all } g \in G, h \in H, \\
    & \supp(f) / H \text{ is compact } \}
\end{align*}
with its action of $G$ by left translations.
Again, for $\Omega \subseteq G$, we will write $\cInd_H^G (V)(\Omega)$ for the subspace of $\cInd_H^G (V)$ consisting of functions whose support is contained in $\Omega$.
We equip $\cInd_H^G (V)$ with its locally convex direct limit topology obtained from the isomorphism
$$
    \cInd_H^G (V) \simeq \varinjlim_{\Omega} \cInd_H^G (V)(\Omega),
$$
where $\Omega$ runs through finite unions of left cosets of $H$ in $G$.
If $V$ has its finest locally convex topology (for example, if $V$ is finite-dimensional and Hausdorff), then the same is true of $\cInd_H^G (V)$.
Given $g \in G$ and $v \in V$, we write $[g,v]$ for the element of $\cInd_H^G (V)$ satisfying $\supp([g, v]) = g H$ and $[g,v](g) = v$.

The following lemma guarantees that our definition of $\Sigma$-analytic induction agrees with the example in the introduction.

\begin{lemma}\label{lemma: induction commutes with tensor products}
Assume that $G_{\Sigma} = G_{\Sigma, 1} \times G_{\Sigma,2}$ and $G^\Sigma = G^\Sigma_1 \times G^\Sigma_2$. For $i=1,2$, let $G_i = G_{\Sigma, i} \times G^\Sigma_i$ and let $H_{\Sigma,i}, H^\Sigma_i$ be closed subgroups of $G_{\Sigma,i}$, $H^\Sigma_i$, $H_i = H_{\Sigma, i} \times H^\Sigma_i$, and $V_i$ be a locally $\Sigma$-analytic representation of $H_i$ on a Hausdorff locally convex $E$-vector space of compact type, where $E$ is some finite extension of $\Q_p$.
\begin{enumerate}
    \item If $H_i$ is open in $G_i$ for $i=1,2$, then there is a topological isomorphism
    $$
        \cInd_{H_1}^{G_1} (V_1) \hat\boxtimes_E \cInd_{H_2}^{G_2} (V_2)
        \simto
        \cInd_{H_1 \times H_2}^{G_1 \times G_2} (V_1 \hat\boxtimes_E V_2).
    $$
    \item If $G_i / H_i$ is compact for $i=1,2$, then there is a topological isomorphism
    $$
        \Ind_{H_1}^{G_1} (V_1)^\Sigmala \hat{\boxtimes}_E \Ind_{H_2}^{G_2} (V_2)^\Sigmala
        \simto
        \Ind_{H_1 \times H_2}^{G_1 \times G_2} (V_1 \hat\boxtimes_E V_2)^\Sigmala.
    $$
\end{enumerate}
\end{lemma}
\begin{proof}
Let us start with the case of compact induction.
The morphism we seek sends $f_1 \otimes f_2$ to the map $(g_1, g_2) \mapsto f_1(g_1) \otimes f_2(g_2)$.
We must check that this is indeed a topological isomorphism.
As $H_i$ is open, a (right) $H$-equivariant function $G_i \to V_i$ has compact support if and only if its support is a finite union of left cosets of $H_i$.
Write each $V_i$ as a countable direct limit $V_i = \varinjlim_n V_{i,n}$ of Banach spaces with compact injective transition maps and choose a countable increasing system $(\Omega_{i,n})_n$ of finite sets of left $H_i$-cosets whose union is $G_i$.
Then,
$$
   \cInd_{H_i}^{G_i} (V_i) \simeq \varinjlim_\Omega \cInd_{H_i}^{G_i} (V_i)(\Omega) \simeq \varinjlim_{n} \cInd_{H_i}^{G_i} (V_{i,n})(\Omega_{i,n})
$$
and the transition maps are injective and compact, and we have a similar expression for the compact induction to $G_1 \times G_2$.
Thus, \cite[Proposition 1.1.32 (i)]{emerton_locally_analytic_vectors} implies that
\begin{align*}
    \cInd_{H_1}^{G_1} (V_1) \cotimes_E \cInd_{H_2}^{G_2} (V_2)
    & \simeq
    \varinjlim_{n} \cInd_{H_1}^{G_1} (V_{1,n})(\Omega_{1,n}) \cotimes_E \cInd_{H_2}^{G_2} (V_{2,n})(\Omega_{2,n})
    \\ & \simeq
    \varinjlim_{n} \cInd_{H_1 \times H_2}^{G_1 \times G_2} (V_{1,n} \cotimes_E V_{2,n})(\Omega_{1,n} \times \Omega_{2,n})
    \\ & \simeq
    \cInd_{H_1 \times H_2}^{G_1 \times G_2} (V_1 \cotimes_E V_2),
\end{align*}
as required (for the second isomorphism note that, for example, $\cInd_{H_1}^{G_1} (V_{1,n})(\Omega_{1,n})$ is isomorphic to a direct sum of $|\Omega_{1,n} / H_1|$ copies of $V_{1,n}$).
This proves \itemnumber{1}.

Next, let us deal with \itemnumber{2}. The map above is defined by the same formula as in the case of compact induction.
Choose for $i=1,2$ a section $s_i \colon G_i / H_i \to G_i$ of locally $\Sigma$-analytic manifolds, i.e. sections $G_{\Sigma,i} / H_{\Sigma,i} \to G_{\Sigma, i}$ of locally analytic manifolds and $G^\Sigma_i / H^\Sigma_i \to G^\Sigma_i$ of topological spaces.
These sections induce a commutative diagram
$$
    \begin{tikzcd}
        \Ind_{H_1}^{G_1} (V_1)^\Sigmala \hat{\boxtimes}_E \Ind_{H_2}^{G_2} (V_2)^\Sigmala
        \ar[r] \ar[d, "\sim"] &
        \Ind_{H_1 \times H_2}^{G_1 \times G_2} (V_1 \hat\boxtimes_E V_2)^\Sigmala \ar[d, "\sim"]
        \\
        \Ccal^\Sigmala(G_1 / H_1, V_1) \cotimes_E \Ccal^\Sigmala(G_2 / H_2, V_2)
        \ar[r] &
        \Ccal^\Sigmala((G_1 \times G_2) / (H_1 \times H_2), V_1 \cotimes_E V_2),
    \end{tikzcd}
$$
and it's enough to show that the bottom map is a topological vector space isomorphism.
Choose for each $i$ a cofinal system $((\mathbb X_{i,n,k}, U_{i,n,j})_{k,j})_n$ of pairs of analytic partitions of $G_{\Sigma,i} / H_{\Sigma,i}$, where $(\mathbb X_{i,n+1,k})_k$ is a relatively compact refinement of $(\mathbb X_{i,n,k})_k$, and disjoint partitions of $G^\Sigma_i / H^\Sigma_i$ into open compact subspaces.
Then, \cite[Proposition 1.1.32 (i)]{emerton_locally_analytic_vectors} shows that
$$
    \Ccal^\Sigmala(G_i / H_i, V_i)
    \simeq
    \varinjlim_n \prod_{k,j} \O(\mathbb X_{i,n,k}) \cotimes_{\Qp} V_{i,n},
$$
and the transition maps are injective and compact.
In particular, applying \cite[Proposition 1.1.32 (i)]{emerton_locally_analytic_vectors} again shows that
\begin{align}\label{eqn: locally analytic functions of product}
    \Ccal^\Sigmala(G_1 / H_1, V_1) \cotimes_E \Ccal^\Sigmala(G_2 / H_2, V_2)
    \simeq
    \varinjlim_n \prod_{k,j} \O(\mathbb X_{1,n,k} \times \mathbb X_{2,n,k}) \cotimes_{\Qp} (V_{1,n} \cotimes_E V_{2,n}).
\end{align}
As the quotients $G_i / H_i$ are compact, $((\mathbb X_{1,n,k} \times \mathbb X_{2,n,k})_k)_n$ is a cofinal system of analytic partitions of $(G_{\Sigma,1} \times G_{\Sigma,2}) / (H_{\Sigma,1} \times H_{\Sigma,2})$ and $((U_{1,n,j} \times U_{2,n,j})_j)_n$ is a cofinal system of partitions of $(G^\Sigma_1 \times G^\Sigma_2) / (H^\Sigma_1 \times H^\Sigma_2)$ into open compact subspaces. In particular, the right-hand side of \Cref{eqn: locally analytic functions of product} is isomorphic to $\Ccal^\Sigmala((G_1 \times G_2) / (H_1 \times H_2), V_1 \cotimes_E V_2)$, which completes the proof.
\end{proof}

\begin{lemma}\label{lemma: monoid hecke action on compact induction}
Assume that $H_0$ is an open subgroup of $G$.
If the action of $H_0$ on a vector space $V$ extends to an action of a submonoid $H^+$ of $G$ containing $H_0$, then there is a natural right action of the Hecke algebra $\Hcal(H^-, H_0)$ on $\cInd_{H_0}^G(V)$, where $H^- = (H^+)^{-1}$, such that for $\mu \in H^-$,
$$
    [g, v] [H_0 \mu H_0] = \sum_{x \mu \in H_0 \mu H_0 / H_0} [gx\mu, (x\mu)^{-1} v].
$$
This action commutes with the action of $G$.
\end{lemma}
\begin{proof}
We may view the action of $H^+$ on $V$ as a right action of $H^-$.
Consider the space $E[G] \otimes_E V$; it is a $(G, H^-)$-bimodule with $H^-$ acting via $h \cdot (g \otimes v) = g h \otimes h^{-1} v$ and $\cInd_{H_0}^G(V)$ can be identified with its right $H_0$-coinvariants (an element $g \otimes v$ of the latter corresponding to $[g, v]$).
The action of the Hecke algebra is that described in \Cref{lemma: general action of Hecke algebra}, and $G$-equivariance is immediate from this formula.
\end{proof}

\subsection{Parabolic and parahoric induction}
\label{subsection: induction 2}

We now retake the notation from \Cref{section: structure}, more specifically \Cref{subsection: products of groups}.
Thus, we let $\ell_1, ..., \ell_n$ be prime numbers, for each $i$ $L_i$ is a finite extension of $\Q_{\ell_i}$ and $\Gbf_i$ is a reductive group over $L_i$, etc.
Fix a subset $\Sigma$ of $\{ 1, ..., n \}$ such that for all $i \in \Sigma$, $\ell_i = p$, and set $G_\Sigma = \prod_{i \in \Sigma} G_i$ and $G^\Sigma = \prod_{i \not\in \Sigma} G_i$, and similarly for the subgroups of $G$ defined in \Cref{subsection: products of groups}.
Hence, $G = G_\Sigma \times G^\Sigma$ is a locally $\Sigma$-analytic manifold and similarly for its subgroups $P$, $M$, etc.
Analogously, we also have similar decompositions $\Phi = \Phi_\Sigma \coprod \Phi^\Sigma, \Delta = \Delta_\Sigma \coprod \Delta^\Sigma$, etc.
Recall from \Cref{subsection: products of groups} that we will usually reference results from \Cref{section: structure} that were stated (and proven) for each of the $G_i$ but also hold for $G$ when we want to use them in the latter case.
Let $\HM$ be a subgroup of $\NJM$ containing $A \JM$. Many of our results on $\NJM$ and $\NJM^-$ from \Cref{section: structure} (for example, \Cref{lemma: products of double cosets} and \Cref{lemma disjointness of JP cosets}) also hold for $\HM$ and $\HM^-$ for obvious reasons, and we will also use them without comment in this situation.

Let $R$ be a Noetherian Banach $\Qp$-algebra and $\sigma$ be a locally $\Sigma$-analytic representation of $\HM$ on some finitely generated $R$-module, which we also write as $\sigma$, equipped with its unique Banach $R$-module topology. We will write $V = \cInd_{\HM}^M (\sigma)$ for the compact induction of $\sigma$ to $M$.
As explained in the previous section, we equip $V$ with the direct limit locally convex topology induced from the expression
\begin{equation} \label{eqn: V is an LB space}
    V \simeq \varinjlim_{\Omega} \cInd_{\HM}^M (\sigma)(\Omega),
\end{equation}
where $\Omega$ runs over finite unions of cosets of $\HM$ in $M$, and the terms in the limit, which are isomorphic as $R$-modules to direct sums of finitely many copies of $\sigma$, are given their unique Banach $R$-module topologies.
We extend the action of $M$ on $V$ to a locally $\Sigma$-analytic action of $P$ by letting $N$ act trivially, and we extend $\sigma$ to $J \cap P$ similarly.
Restriction of functions induces a topological isomorphism
    $\Ind_{J \cap P}^J (\sigma)^\Sigmala \stackrel\sim\to \Ccal^{\Sigmala}(\bar N_0, \sigma)$
as in the proof of \Cref{lemma: induction commutes with tensor products}.
Likewise, restriction of functions induces topological isomorphisms
    $$
        \Ind_{P}^G(V)^\Sigmala(J P) \stackrel\sim\to \Ind_{J \cap P}^J (V)^\Sigmala \stackrel\sim\to \Ccal^{\Sigmala}(\bar N_0, V),
    $$
the first of which is $J$-equivariant.

Let $s > 0$. Recall from \Cref{lemma: Iwahori decomposition for deep level} that $J_s$ admits a rooted Iwahori decomposition, $J_s = \bar N_s \JMs N_s$.
For each $i \in \Sigma$, in \Cref{subsection: congruence subgroups} we constructed an affinoid $L_i$-space $\bar \Nbf_{i,s}$ whose set of $L_i$-points is $\bar N_{i,s}$.
In particular, we may view $\bar N_{\Sigma, s}$ as the set of $\Q_p$-points of the $\Q_p$-affinoid $\bar \Nbf_{\Sigma, s} := \prod_{i \in \Sigma} \Res_{L_i / \Q_p} \bar \Nbf_{i,s}$.
Given $s$ and $s'$, we will say that a locally analytic function $\bar N_0 \to \sigma$ is $(s, s')$-$\Sigma$-analytic if its restriction to each left $\bar N_{\Sigma, s} \times \bar N^\Sigma_{s'}$-coset is $\bar \Nbf_{\Sigma, s} \times \bar N^\Sigma_{s'}$-analytic.
Let us also say that a function $f$ in $\Ind_{J \cap P}^J (\sigma)^\la$ is $(s, s')$-$\Sigma$-analytic if its restriction to $\bar N_0$ is.
We will write $\Acal_\sigma^{(s, s')-\Sigmaan}$ for the subspace of $\Ind_{J \cap P}^J (\sigma)^\Sigmala$ consisting of such functions.
Restriction to $\bar N_0$ induces an isomorphism from $\Acal_\sigma^{(s,s')-\Sigmaan}$ to the space of $(s, s')$-$\Sigma$-analytic functions $\bar N_0 \to \sigma$, which we use to give the former the structure of a Banach $R$-module.
If $\sigma$ is free over $R$, then $\Acal_\sigma^{(s, s')-\Sigmaan}$ is orthonormalisable as it is isomorphic to the completed base change to $R$ of a Banach $\Q_p$-module.
When $s = s'$, we will speak of $s$-$\Sigma$-analytic functions instead and write $\Acal_\sigma^{s-\Sigmaan} = \Acal_\sigma^{(s, s)-\Sigmaan}$.
Note that $\Ind_{J \cap P}^J (\sigma)^\Sigmala$ is isomorphic to the colimit $\varinjlim_s \Acal_\sigma^{s-\Sigmaan}$ since any rigid analytic closed ball in $\bar \Nbf$ containing the identity is contained in some $\bar \Nbf_s$ and vice versa (but the $\bar \Nbf_s$ will not necessarily be closed balls if the $\Gbf_i$ aren't split).

\begin{lemma}
If $s, s'$ are large enough, then $\Acal_\sigma^{(s, s')-\Sigmaan}$ is stable under the action of $J$.
\end{lemma}
\begin{proof}
As $\sigma$ is finitely generated, we may choose $s$ and $s'$ large enough so that the orbits of $\JMs[i]$ in $\sigma$ are analytic for $i \in \Sigma$ (i.e. they come from rigid-analytic morphisms of $\JMsbf[i]$, as defined in \Cref{subsection: congruence subgroups}) and the orbits of
$\JMsub[i,s']$
are constant for $i \not\in \Sigma$.
Let $y \in J$ and $f \in \Acal_\sigma^{(s, s')-\Sigmaan}$.
We will show that $(y^{-1}f)|_{\bar N_0}$ is $(s, s')$-$\Sigma$-analytic.
By writing $f$ as a sum of functions with support in different cosets of $\bar N_{\Sigma, s} \times \bar N^\Sigma_{s'}$, it's enough to show that $(y^{-1}f)|_{\bar N_{\Sigma, s} \times \bar N^\Sigma_{s'}}$ is $\bar \Nbf_{\Sigma, s} \times \bar N^\Sigma_{s'}$-analytic.
Write $y = \bar n_0 m_0 n_0$ according to the Iwahori decomposition of $J$ (\Cref{lemma: J admits rooted Iwahori decomposition}), set $p_0 = m_0 n_0$, and let $\bar n \colon J_{\Sigma,s} \times J_{s'}^\Sigma \to \bar N_{\Sigma, s} \times \bar N^\Sigma_{s'}$ and $m \colon J_{\Sigma,s} \times J_{s'}^\Sigma \to \JMs[\Sigma] \times \JMsub[s']^\Sigma$ be the projections given by the Iwahori decomposition of $J_s$ (\Cref{lemma: Iwahori decomposition for deep level}).
Then, for all $x \in \bar N_{\Sigma, s} \times \bar N^\Sigma_{s'}$ one has
\begin{align}\label{eqn: formula equivariance of As}
    (y^{-1} f)(x) = m( p_0 x p_0^{-1} )^{-1} m_0^{-1} f(\bar n_0 \cdot \bar n(p_0 x p_0^{-1})).
\end{align}
By \Cref{lemma: rigid analytic parahorics} \itemnumber{1} and \itemnumber{2}, the map $x \mapsto \bar n(p_0 x p_0^{-1}) \colon \bar N_{\Sigma,s} \to \bar N_{\Sigma,s}$ and the map 
\linebreak $x \mapsto m (p_0 x p_0^{-1}) \colon$ $\bar N_{\Sigma,s} \to \JMs[\Sigma]$ are induced from rigid-analytic morphisms. We have similar continuous maps for $\bar N^\Sigma_{s'}$. By our choice of $s$ and $s'$, this implies that \Cref{eqn: formula equivariance of As} is $\bar \Nbf_{\Sigma, s} \times \bar N^\Sigma_{s'}$-analytic as a function of $x \in \bar N_{\Sigma, s} \times \bar N^\Sigma_{s'}$.
\end{proof}

We will write $s_\sigma$ for the smallest $s$ such that $\Acal_\sigma^{s-\Sigmaan}$ is stable under the action of $J$, and when writing $\Acal_\sigma^{s-\Sigmaan}$ we will always implicitly assume $s \geq s_\sigma$.
Similarly, when we consider $\Acal_\sigma^{(s, s')-\Sigmaan}$ we will always assume that $s$ and $s'$ are large enough that this is stable under the action of $J$.

\begin{lemma}\label{lemma: inclusion of As is compact}
If $\tilde s \geq s$ and $\tilde s' \geq s'$, then $\Acal_\sigma^{(s, s')-\Sigmaan} \subseteq \Acal_\sigma^{(\tilde s, \tilde s')-\Sigmaan}$. If $\tilde s > s$, then this inclusion is a compact map of Banach $R$-modules.
\end{lemma}
\begin{proof}
The first statement is immediate. Note that there is an isomorphism of $R$-Banach modules
\begin{align}\label{eqn: isomorphism for Acal}
    \left( \Ccal^{\bar \Nbf_s -\an} (\bar N_{\Sigma, 0}, \Q_p) \otimes_\Qp \Q_p^{|(\bar N_0^\Sigma) / \bar N^\Sigma_{s'}|} \right) \cotimes_\Qp \sigma  \simto \Acal_\sigma^{(s, s')-\Sigmaan},
\end{align}
where $\Ccal^{\bar \Nbf_s -\an} (\bar N_{\Sigma, 0}, \Q_p)$ denotes the space of functions $\bar N_{\Sigma, 0} \to \Q_p$ whose restriction to each $\bar N_s$ coset is $\bar \Nbf_s$-analytic.
The second statement follows from this isomorphism by \Cref{lemma: inclusion of root subgroups is compact} and \cite[Corollary 2.9]{buzzard_eigenvarieties}.
\end{proof}

The $\JM$-equivariant injection $\sigma \into V$ given by $v \mapsto [1,v]$ induces $J$-equivariant maps $\Ind_{J \cap P}^J (\sigma)^\Sigmala \to \Ind_{J \cap P}^J (V)^\Sigmala$. Composing with $\Ind_{J \cap P}^J (V)^\Sigmala \stackrel\sim\to \Ind_{P}^G(V)^\Sigmala( J P)$ we obtain a $J$-equivariant injection
$\Ind_{J \cap P}^J (\sigma)^\Sigmala \to \Ind_{P}^G(V)^\Sigmala( J P)$.

\begin{lemma}\label{lemma: get action of monoid}
The image of this map is stable under the action of $\HM^+ = (\HM^-)^{-1}$.
\end{lemma}
\begin{proof}
A function $f \in \Ind_P^G (V)^\Sigmala(JP)$ is in the image of this map if and only if there exists some $F \in \Ind_{J \cap P}^J (\sigma)^\Sigmala$ such that $f(y) = [1, F(y)]$ for all $y \in J$, or equivalently, for all $y \in \bar N_0$. Consider some such $f$ and $F$, and let $\mu \in \HM^+$ and $y \in \bar N_0$.
Then, $(\mu f)(y) = \mu f(\mu^{-1} y \mu) = [1, \mu F(\mu^{-1} y \mu)]$, so $\mu f$ lies in the image.
\end{proof}

It follows from \Cref{lemma: products of double cosets} that $J \HM^+ J$ is a submonoid of $G$, and \Cref{lemma: get action of monoid} implies that we may extend the action of $J$ on $\Ind_{J \cap P}^J (\sigma)^\Sigmala$ to an action of $J \HM^+ J$ making the injection $\Ind_{J \cap P}^J (\sigma)^\Sigmala \to \Ind_{P}^G(V)^\Sigmala( J P)$ a $J \HM^+ J$-equivariant map.
Explicitly, if $y$ is an element of $J$ with Iwahori decomposition $y = \bar n m n$ and $F \in \Ind_{J \cap P}^J (\sigma)^\Sigmala$, then the action of $\HM^+$ is given by $(\mu F)(y) = m^{-1} \mu F(\mu^{-1} \bar n \mu)$.
If $\mu \in \HM^+$, then $\mu$ preserves $\Acal_\sigma^{s-\Sigmaan}$ by \Cref{lemma: rigid analytic parahorics} \itemnumber{3}, and for a fixed norm on $\sigma$, the norm of $\mu$ as an operator on $\Acal_{\sigma}^{s-\Sigmaan}$ (with its norm induced by \Cref{eqn: isomorphism for Acal} and that of $\sigma$) is bounded above by the norm of $\mu$ on $\sigma$.
If, moreover, $\mu^{-1} \bar \Nbf_{\Sigma,s} \bar N^\Sigma_{s'} \mu \subseteq \bar \Nbf_{\Sigma,\tilde s} \bar N^\Sigma_{s'}$ for some $\tilde s < s$ (for example, if $\mu = a^{-1}$ with $a \in A^{--}_\Sigma$), then the action of $\mu$ on $\Acal_\sigma^{(\tilde s,s')-\Sigmaan}$ factors through the inclusion $\Acal_\sigma^{(s, s')-\an} \subseteq \Acal_\sigma^{(\tilde s, s')-\Sigmaan}$, and therefore induces an $R$-compact endomorphism of $\Acal_\sigma^{(\tilde s, s')-\Sigmaan}$ by \Cref{lemma: inclusion of As is compact}.
Moreover, applying \Cref{lemma: monoid hecke action on compact induction} and the results of \Cref{subsection: hecke algebras}, we obtain the following.

\begin{lemma}\label{lemma: U operators}
The action of $\HM^+$ on $\Ind_{J \cap P}^J (\sigma)^\Sigmala$ induces a (left) action of the Hecke algebra $\Hcal(J \HM^- J, J)_R$ on $\cInd_J^G(\Ind_{J \cap P}^J (\sigma)^\Sigmala)$ commuting with the action of $G$, where the characteristic functions $U_\mu = [J \mu J]$ for $\mu \in \HM^-$ act as
$$
    U_\mu([g, F]) = \sum_{x \mu \in J \mu J / J} [gx\mu, (x\mu)^{-1} F].
$$
Moreover, for each $s$ and $\mu \in \HM^-$, the operator $U_\mu$ preserves the subspace 
$\cInd_J^G(\Acal_\sigma^{s-\Sigmaan})$.
\end{lemma}

\subsection{A quasi-isomorphism}
\label{subsection: quasi isomorphism}
Write $\Hcal^-_R = \Hcal(J \HM^- J, J)_R$.
The $J$-equivariant embedding $\Ind_{J \cap P}^J (\sigma)^\Sigmala \to \Ind_{P}^G(V)^\Sigmala$ induces a $G$-equivariant $R$-linear map
$$
    \varphi \colon \cInd_J^G ( \Ind_{J \cap P}^J (\sigma)^\Sigmala ) \to \Ind_{P}^G (V)^\Sigmala.
$$
Explicitly, if $F \in \Ind_{J \cap P}^J (\sigma)^\Sigmala$, then $\varphi([1, F])(y) = [1,F(y)]$ for $y \in J$.
The source is a module for $\Hcal^-_R$, and we may regard $\Ind_{P}^G (V)^\Sigmala$ as a $\Hcal^-_R$-module by letting $U_\mu = [J \mu J]$ act trivially. Write $R(1)$ for the $\Hcal^-_R$-module whose underlying $R$-module is $R$ and where $U_\mu$ acts trivially. The goal of this section is to show that $\varphi$ induces an isomorphism in the derived category $D( \Hcal^-_R[G] )$ of $\Hcal^-_R[G]$-modules
$$
    R(1) \otimesL_{\Hcal^-_R} \cInd_J^G ( \Ind_{J \cap P}^J (\sigma)^\Sigmala ) \simto \Ind_{P}^G (V)^\Sigmala.
$$
To be more precise, we should be writing $R[G](1) \otimesL_{\Hcal^-_R[G]} \cInd_J^G ( \Ind_{J \cap P}^J (\sigma)^\Sigmala )$ for the source instead.
However, the complexes are isomorphic, so we will use the shorter notation.
As explained in the introduction, we follow the strategy of \cite[\S 2]{vanishing}.

\begin{proposition}\label{proposition: phi is surjective}
The map $\varphi$ is surjective.
\end{proposition}
\begin{proof}
We can write $G$ as a union of finitely many translates of $J P$, $G = \bigcup_{i=1}^m g_i J P$. It is easy to see that we may write any element of $\Ind_{P}^G(V)^\Sigmala$ as sum of a functions in $\Ind_{P}^G(V)^\Sigmala$ whose support intersects only one of the sets $g_i J P$.
In particular,
$$
    \Ind_{P}^G(V)^\Sigmala = \sum_{i=1}^m g_i \cdot \Ind_{P}^G(V)^\Sigmala(J P),
$$
so it's enough to show that the image of $\varphi$ contains $\Ind_{P}^G(V)^\Sigmala(J P)$. Let $f$ be an element of $\Ind_{P}^G(V)^\Sigmala(J P)$.
Since $f$ is locally $\Sigma$-analytic, $f|_{\bar N_0}$ factors through some BH-space $U \into V$.
It follows from the expression \Cref{eqn: V is an LB space} of $V$ as an LB-space and \cite[Corollary 8.9]{nonarchimedean_functional_analysis} that any BH-space $U \into V$ factors through $\cInd_{\HM}^M (\sigma)(\Omega)$ for some finite union $\Omega$ of left cosets of $\HM$ in $M$.
Hence, $f(\bar N_0)$ is contained in $\cInd_{\HM}^M (\sigma)(\Omega)$, where $\Omega = \bigcup_{i=1}^s m_i \HM$ for some integer $s \geq 0$ and $m_1, ..., m_s \in M$. For any $\bar n \in \bar N_0$ we have $f(\bar n) = \sum_{i=1}^s [m_i, F_i(\bar n)]$, where we have set $F_i(\bar n) := f(\bar n)(m_i)$. Thus, we may assume that $s=1$ and $f|_{\bar N_0} = [m, F]$ for some $F \in \Ccal^\Sigmala(\bar N_0, \sigma)$.
Choose $a \in A^-$ such that $a m^{-1} \bar N_0 m a^{-1} \subseteq \bar N_0$, so that by changing $F$ by $\bar n \mapsto a F(\bar n)$ we may assume that $m^{-1} \bar N_0 m \subseteq \bar N_0$.
As $(m^{-1} f)(m^{-1} \bar n m) = [1, F(\bar n)]$, we see by replacing $f$ with $m^{-1} f$ and $F$ with $y \mapsto F(m y m^{-1})$ for $y \in m^{-1} \bar N_0 m$ that it suffices to consider the case when $f|_{\bar N_0} = [1, F]$. But in this case, $f = \varphi([1,F])$.
\end{proof}

\begin{proposition}\label{proposition: kernel of phi}
The kernel of $\varphi$ is $\sum_{\mu \in \HM^-} \Im(U_\mu - 1)$. In fact, for $s \geq s_\sigma$, the kernel of the restriction of $\varphi$ to $\cInd_J^G( \Acal_\sigma^{s-\Sigmaan} )$ is $\sum_{\mu \in \HM^-} (U_\mu - 1)(\cInd_J^G( \Acal_\sigma^{s-\Sigmaan} ))$.
\end{proposition}
\begin{proof}
The first statement follows from the second, so this is the one we will prove.
Let $\mu \in \HM^-$. By \Cref{lemma disjointness of JP cosets}, we can write $JP = J \mu J P = \coprod_x x \mu JP$
where $x$ runs over coset representatives in $\bar N_0 / \mu \bar N_0 \mu^{-1}$,
and $\varphi(U_\mu [1, F]) = \sum_x \varphi([x \mu, (x \mu)^{-1} F])$, where $\varphi([x \mu, (x\mu)^{-1} F])$ has support in $x \mu J P$.
Hence, for any $y \in JP$, there exists a unique coset representative $x$ as above such that $y \in x \mu JP$, and we can write $y = x \mu \bar n p$ with $\bar n \in \bar N_0$ and $p \in P$. Then,
\begin{align*}
    \varphi(U_\mu [1, F])(y) & = [x\mu, (x\mu)^{-1}  F](x\mu \bar n p) \\
     & = p^{-1} [1, ((x\mu)^{-1} F)(\bar n)] \\
     &  = p^{-1} [1, \mu^{-1} F(x \mu \bar n \mu^{-1})] \\
     & = p^{-1} \mu^{-1} [1, F(x \mu \bar n \mu^{-1})] \\
     & = \varphi([1,F])(y)
\end{align*}
\newcommand{\thekernel}{\mathcal{K}}
since we can write $y = (x \mu \bar n \mu^{-1}) \cdot (\mu p)$ as a product of an element of $\bar N_0$ with an element of $P$.
Thus, $(U_\mu - 1) [1, F] \in \ker \varphi$, and since the elements of the form $[1, F]$ generate $\cInd_J^G (\Acal_\sigma^{s-\Sigmaan})$, we see that the subspace $\thekernel := \sum_{\mu \in \HM^-} (U_\mu - 1)(\Acal_\sigma^{s-\Sigmaan})$ of $\Acal_\sigma^{s-\Sigmaan}$ is contained the kernel of $\varphi$.

We now show that the kernel of $\varphi$ is contained in $\thekernel$. For $g \in G$, $F \in \Acal_\sigma^{s-\Sigmaan}$ and $\mu \in \HM^-$, we have
\begin{align}\label{eqn: can change element of kernel by element of V}
\begin{split}
    [g, F] =~ & (1 - U_\mu)[g, F] + U_\mu[g, F] \\
            & \in \thekernel + \cInd_J^G( \Acal_\sigma^{s-\an} )(g J \mu J).
\end{split}
\end{align}
Recall that $\bar N_0 \subseteq K \cap N$, so in particular for for $t \in \Zfrak^-$, we have
\begin{enumerate}[label=(\Alph*)]
    \item\label{assumption 1 for Zfraksim} $t \bar N_0 t^{-1} \subseteq K \cap \bar N$.
\end{enumerate}
Choose for each $\alpha \in \Phi^+ \setminus \Phi_M$ a constant $c_\alpha \leq 0$ that is small enough such that, defining $\Zfrak^\sim$ as in \Cref{proposition: important decomposition} \itemnumber{3}, for all $t \in \Zfrak^\sim$ we have
\begin{enumerate}[label=(\Alph*)]
    \setcounter{enumi}{1}
    \item\label{assumption 2 for Zfraksim} $t^{-1} (K \cap N) t \subseteq N_0$.
\end{enumerate}
Then, \Cref{eqn: can change element of kernel by element of V} and \Cref{proposition: important decomposition} \itemnumber{3} show for all $g \in G, F \in \Acal_\sigma^{s-\Sigmaan}$,
$$
    [g, F] \in \thekernel + \cInd_J^G(\Acal_\sigma^{s-\Sigmaan})(K \Zfrak^\sim J).
$$

Now, let $f \in \cInd_J^G(\Acal_\sigma^{s-\Sigmaan})(K \Zfrak^\sim J)$ be an element of the kernel of $\varphi$. We may write $f = \sum_{i=1}^r \sum_{j=1}^{r_i} [k_{i,j} t_i, F_{i,j}]$, where $k_{i,j} \in K$ and the $t_i \in \Zfrak^\sim$ belong to different double cosets in $K \backslash G / J$. By \Cref{proposition: important decomposition} \itemnumber{1}, \Cref{remark: for sufficiently small t multiplication by JzJ works well} and \Cref{eqn: can change element of kernel by element of V} we can further assume that if $i \neq i'$ then $K t_i \mu J \neq K t_{i'} J$ for any $\mu \in \HM^-$. In fact, by \Cref{lemma: JM generated by submonoids}, we may assume that the same is true for $\mu \in \HM$. We may also assume without loss of generality that if $k_{i,j} t_i J = k_{i,j'} t_i J$, then $j = j'$. We will show that any such $f$ must be zero.

Let $k \in K$ be any element, and write $\mathcal I_k$ for the set of pairs of indices $(i, j)$ such that $kP \subseteq k_{i,j} t_i JP$. For any $(i,j) \in \mathcal I_k$, we may write $(k_{i,j} t_i)^{-1} k = \bar n_{i,j} m_{i,j} n_{i,j}$ for some 
\linebreak $\bar n_{i,j} \in \bar N_0, m_{i,j} \in M, n_{i,j} \in N$. In particular, $k_{i,j}^{-1} k = t_i \bar n_{i,j} t_i^{-1} \cdot t_i m_{i,j} \cdot n_{i,j}$, so by \Cref{assumption 1 for Zfraksim}, $t_i m_{i,j} n_{i,j} \in K \cap P$. By \Cref{lemma: semi Iwahori decomposition for K} and the Levi decomposition of $P$, this means that $t_i m_{i,j} \in K \cap M$ and $n_{i,j} \in K \cap N$. In particular, $m_{i,j} (K \cap N) m_{i,j}^{-1} = t_i^{-1} (K \cap N) t_i \subseteq N_0$ by \Cref{assumption 2 for Zfraksim}.

The support of $\varphi([k_{i,j} t_i, F_{i,j}])$ is $k_{i,j} t_i JP$, and
$$
    \varphi([k_{i,j} t_i, F_{i,j}])(k) = \varphi([1, F_{i,j}])((k_{i,j} t_i)^{-1}k) = \varphi([1, F_{i,j}])(\bar n_{i,j} m_{i,j} n_{i,j}) = [m_{i,j}^{-1}, F_{i,j}(\bar n_{i,j})].
$$
In particular,
\begin{align}\label{eqn: summands of phi are 0}
    0 = \varphi(f)(k) = \sum_{(i,j) \in \mathcal I_k} [m_{i,j}^{-1}, F_{i,j}(\bar n_{i,j})].
\end{align}
We claim that the supports of the summands in the right-hand side are pairwise disjoint.
Indeed, assume that $m_{i,j}^{-1} \HM \cap m_{i' j'}^{-1} \HM \neq \emptyset$. Then, $\mu := m_{i,j} m_{i', j'}^{-1} \in \HM$ and
$$
    K t_i \mu J = K m_{i,j}^{-1} \mu J = K m_{i',j'}^{-1} J = K t_{i'} J,
$$
which by assumption implies $i = i'$. Moreover, \Cref{corollary: KtzJ are disjoint for different z} implies $\mu \in \JM$. Now, as $k_{i,j}^{-1} = t_i \bar n_{i,j} m_{i,j} n_{i,j} k^{-1}$, we have
\begin{align*}
    t_i^{-1} k_{i,j}^{-1} k_{i', j'} t_i
    & = \bar n_{i,j}  \cdot \mu \cdot m_{i,j'} n_{i,j} n_{i, j'}^{-1} m_{i,j'}^{-1} \cdot \bar n_{i,j'}^{-1} \\
    & \in \bar N_0 \cdot \JM \cdot m_{i,j'} (K \cap N) m_{i,j'}^{-1} \cdot \bar N_0 \\
    & \subseteq \bar N_0 \cdot \JM \cdot N_0 \cdot \bar N_0 \subseteq J,
\end{align*}
or in other words, $k_{i,j'} t_i J = k_{i,j} t_i J$. By assumption, this implies $j = j'$, thus proving the claim.

In particular, we may deduce from \Cref{eqn: summands of phi are 0} that $F_{i,j}(\bar n_{i,j}) = 0$ for all $(i, j) \in \mathcal I_k$. Now, let us fix any pair $(i,j)$. To show that $F_{i,j} = 0$, it's enough to show that $F_{i,j}(\bar n) = 0$ for all $\bar n \in \bar N_0$. Fix such a $\bar n$, and take $k = k_{i,j} t_i \bar n t_i^{-1}$, which lies in $K$ by \Cref{assumption 1 for Zfraksim}. Then, $(i,j) \in \mathcal I_k$, $\bar n_{i,j} = \bar n, m_{i,j} = t_i^{-1}$ and $n_{i,j} = 1$. In particular, $F_{i,j}(\bar n) = 0$, as required.
\end{proof}

\begin{corollary}\label{corollary: restriction of phi is surjective}
The restriction of $\varphi$ to $\cInd_J^G (\Acal_\sigma^{s-\Sigmaan})$ is surjective.
\end{corollary}
\begin{proof}
Let $s_\sigma \leq s' \leq s$.
Choose $\mu \in \HM^-$ such that the operator $U_\mu$ on $\cInd_J^G (\Acal_\sigma^{s-\Sigmaan})$ factors through
$\cInd_J^G (\Acal_\sigma^{s'-\Sigmaan})$. There is a commutative diagram
$$
    \begin{tikzcd}[column sep = large]
        \frac{\cInd_J^G (\Acal_\sigma^{s'-\Sigmaan}) }{ \sum_{\tilde \mu \in \HM^-} \Im(U_{\tilde \mu} - 1)} \ar[d, "U_\mu"] \ar[r,"\text{inclusion}"] &
        \frac{\cInd_J^G (\Acal_\sigma^{s-\Sigmaan}) }{ \sum_{\tilde \mu \in \HM^-} \Im(U_{\tilde \mu} - 1) } \ar[d, "U_\mu"] \ar[dl] \\
        \frac{\cInd_J^G (\Acal_\sigma^{s'-\Sigmaan}) }{ \sum_{\tilde \mu \in \HM^-} \Im(U_{\tilde \mu} - 1)}  \ar[r,"\text{inclusion}"] &
        \frac{\cInd_J^G (\Acal_\sigma^{s-\Sigmaan}) }{ \sum_{\tilde \mu \in \HM^-} \Im(U_{\tilde \mu} - 1) }
    \end{tikzcd}
$$
where the vertical arrows are isomorphisms. Thus, the horizontal arrows are also isomorphisms. The result then follows by passing to the limit and \Cref{proposition: phi is surjective}.
\end{proof}

\begin{proposition}\label{proposition: higher tor vanishes}
Let $\Acal_\sigma^\Sigma$ denote either $\Ind_{J \cap P}^J (\sigma)^\Sigmala$ or $\Acal_\sigma^{s-\Sigmaan}$ for some $s \geq s_\sigma$. Then, 
$\Tor^{\Hcal^-_R}_q(R(1), \cInd_J^G (\Acal_\sigma^\Sigma)) = 0$ for all $q > 0$.
\end{proposition}
\begin{proof}
We will prove this by applying our results from \Cref{section: vanishing for Tor}, specifically \Cref{proposition: vanishing Tor 6}. Let $\boldLambda = \HM / \JM$, $\boldLambda^- = \HM^- / \JM$ and $\Lambda \subseteq A$ inducing an isomorphism $\Lambda \simto A / A_0$ and $\Lambda^- = \Lambda \cap A^-$.
Recall from \Cref{remark: normaliser of J_M has free quotient} that $\boldLambda$ is a free abelian group of finite rank and $\Lambda$ can be seen as a finite index subgroup of it.
\Cref{lemma: JM generated by submonoids} implies that $\boldLambda = \boldLambda^- (\boldLambda^-)^{-1}$ and $\boldLambda = \boldLambda^- \Lambda$.
By \Cref{corollary: hecke algebra isomorphic to monoid ring} (cf. \Cref{remark: for Hecke algebras can replace tildeJ_M by HM}), there is an $R$-algebra isomorphism 
$R[\boldLambda^-] \simto \Hcal^-_R$, which we use to regard $\cInd_J^G (\Acal_\sigma)$ as an $R[\boldLambda^-]$-module.
The monoid $\Lambda^-$ satisfies the condition of \Cref{lemma: vanishing Tor 3} \itemnumber{1} by \Cref{eqn: description of A-}.
Thus, to apply \Cref{proposition: vanishing Tor 6} we need to check that $\Mcal := \cInd_J^G (\Acal_\sigma^\Sigma)$ admits a decomposition
$$
    \cInd_J^G (\Acal_\sigma^\Sigma) = \Mcal^+ \oplus \bigoplus_{\mu \in \boldLambda^-} \Mcal_\mu
$$
such that 
\begin{enumerate}[label=(\alph*)]
    \item\label{item: lemma Tor 2 a reprise} $U_\mu \Mcal_{\tilde \mu \JM} \subseteq \Mcal_{\mu \tilde \mu \JM}$ for $\mu, \tilde \mu \in \HM^-$,
    \item\label{item: lemma Tor 2 b reprise} for all $f \in \cInd_J^G (\Acal_\sigma^\Sigma)$, there exists $\mu \in \HM^-$ such that $U_\mu f \in \bigoplus_{\tilde \mu \in \Lambda^-} \Mcal_{\tilde \mu}$.
\end{enumerate}

Let $\Zfrak^\sim$ be as in the proof of \Cref{proposition: kernel of phi}. Choose a set $\Omega$ of representatives in $\Zfrak^\sim$ for each orbit of the (right) action of $\Nfrak \cap \HM^-$ on $\tilde W^K \backslash \tilde W / \tilde W^{\JM} \simeq \Zfrak_0 \backslash \Zfrak / \tilde W^{\JM}$ (this is possible by \Cref{proposition: important decomposition} \itemnumber{3}). Set, for $\mu \in \Nfrak \cap \HM^-$,
$$
    \Mcal_{\mu \JM} := \bigoplus_{t \in \Omega} \cInd_J^G (\Acal_\sigma^\Sigma)(K t \mu J),
    ~~~~~~~~~~
    \text{  and  }
    ~~~~~~~~~~
    \Mcal^+ := \bigoplus_{\mu \in \boldLambda \setminus \boldLambda^-} \Mcal_{\mu}.
$$
It then follows from \Cref{proposition: important decomposition} \itemnumber{1} and \Cref{corollary: KtzJ are disjoint for different z} that there is a decomposition $\cInd_J^G(\Acal_\sigma^\Sigma) = \Mcal^{+} \oplus \bigoplus_{\mu \in \boldLambda^-} \Mcal_\mu$.
Condition \Cref{item: lemma Tor 2 a reprise} above follows from the fact that for $k \in K,$ $t \in \Zfrak^\sim, \mu, \tilde \mu \in \HM^-$,
$$
    \supp( U_\mu [k t \tilde \mu, F] )
    \subseteq k t \tilde \mu J \mu J
    \subseteq K t \mu \tilde \mu J,
$$
where the second inclusion follows from \Cref{remark: for sufficiently small t multiplication by JzJ works well}.

Let us move on to condition \Cref{item: lemma Tor 2 b reprise}.
Let $\tilde t \in \Zfrak$, $k \in K$.
We will show that there exists $a \in A^-$ such that $\supp(U_{a} [k \tilde t, F]) \subseteq \coprod_{\mu \in \boldLambda^-, t \in \Omega} K t \mu J$ for any $F$.
By \Cref{proposition: important decomposition} \itemnumber{2}, there exist $\tilde t_1, ..., \tilde t_r \in \Zfrak$ such that $\tilde t J a \subseteq \bigcup_i K \tilde t_i a J$ for all $a \in A^-$.
Choose $\mu_i \in \Nfrak \cap \HM$ and $t_i \in \Omega$ such that the action of $\mu_i$ on $\Zfrak_0 \backslash \Zfrak$ sends $\Zfrak_0 t_i$ to $\Zfrak_0 \tilde t_i$.
Writing $\mu_i = n_i t'_i$ with $n_i \in \Nfrak \cap K$ and $t'_i \in \Zfrak$, this means that, multiplying $\tilde t_i$ by an element of $\Zfrak_0$ on the left if necessary, we have $\tilde t_i = n_i^{-1} t_i \mu_i$.
If we take $a \in A^-$ with $v_p(\alpha(a))$ small enough for all $\alpha \in \Phi^+ \setminus \Phi_M$ so that $\mu_i a \in \HM^-$, then $U_{a} [k \tilde t, F]$ has support in
$$
    K \tilde t J a J \subseteq \bigcup_i K \tilde t_i a J = \bigcup_i K n_i^{-1} t_i \mu_i a J = \bigcup_i K t_i (\mu_i a) J \subseteq \coprod_{\mu \in \boldLambda^-, t \in \Omega} K t \mu J,
$$
as required.
\end{proof}

\newcommand{\repnofG}{U}

We are almost ready to state and prove the main theorem of the section.
Before that, we need to introduce a bit more notation.
Let $G_0$ be an abstract group and let $\repnofG$ be a representation of $G_0 \times G$ on a finite flat $R$-module.
We set $\Acal_{\sigma, \repnofG}^{s-\Sigmaan} := \repnofG \otimes_R \Acal_\sigma^{s-\Sigmaan}$ and view it as a $G_0 \times G$-submodule (with $G$ acting diagonally) of $\repnofG \otimes_R \Ind_{J \cap P}^J (\sigma)^\Sigmala$.
There is a $G_0 \times G$-equivariant isomorphism
\begin{align}\label{eqn: isomorphism of Lambda modules algebraic smooth}
\begin{split}
    \repnofG \otimes_R \cInd_J^G (\Acal_{\sigma}^{s-\Sigmaan})
    & \simto \cInd_J^G (\Acal_{\sigma, \repnofG}^{s-\Sigmaan}) \\
    u \otimes [g, F]
    & \mapsto [g, g^{-1} u \otimes F].
\end{split}
\end{align}
If we regard $\repnofG$ as an $\Hcal^-_R$-module by letting the operators $U_\mu$ act as 1 and let $\Hcal^-_R$ act diagonally on the left-hand side, then \Cref{eqn: isomorphism of Lambda modules algebraic smooth} is also an $\Hcal^-_R$-module isomorphism.
The corresponding statements for $\repnofG \otimes_R \Ind_{J \cap P}^J (\sigma)^{\Sigmala}$ also hold.
The map $\varphi$ induces $\Hcal^-_R [G_0 \times G]$-linear maps
$$
\cInd_J^G(\Acal_{\sigma, \repnofG}^{s-\Sigmaan}) \to \repnofG \otimes_R \cInd_J^G (\Ind_{J \cap P}^J (\sigma)^{\Sigmala}) \to \repnofG \otimes_R \Ind_P^G(V)^{\Sigmala}
$$

\begin{theorem}\label{theorem: quasi-isomorphism locally algebraic}
The map $\varphi$ induces isomorphisms in the derived category $D(\Hcal^-_R [G_0 \times G])$ of $\Hcal^-_R [G_0 \times G]$-modules
\begin{align*}
    R(1) \otimesL_{\Hcal^-_R} \cInd_J^G (\Acal_{\sigma, U}^{s-\Sigmaan}) 
    & \simto R(1) \otimesL_{\Hcal^-_R} (U \otimes_R \cInd_J^G ( \Ind_{J \cap P}^J (\sigma)^{\Sigmala} ) ) \\
    & \simto U \otimes_R \Ind_P^G(V)^{\Sigmala}.
\end{align*}
\end{theorem}
\begin{proof}
When $U$ is the trivial representation, the theorem follows by putting together \Cref{proposition: phi is surjective}, \Cref{proposition: kernel of phi}, \Cref{corollary: restriction of phi is surjective} and \Cref{proposition: higher tor vanishes}.
For general $U$, the result follows from the fact that the functor $U \otimes_R \blank$ on $\Hcal^-_R$-modules is exact, preserves projectives (as $U$ is a projective $R$-module), and commutes with the functor $R(1) \otimes_{\Hcal^-_R} \blank$.
\end{proof}

Next, assume that $\sigma = \sigma_\Sigma \boxtimes_R \sigma^\Sigma$ where $\sigma_\Sigma$ is a locally analytic representation of $G_\Sigma$ and $\sigma^\Sigma$ is a smooth representation of $G^\Sigma$.
We may choose $s \geq s_\sigma$ such that $\JMs^\Sigma$ acts trivially on $\sigma^\Sigma$.
We can extend the action of $\JM^\Sigma$ on $\sigma^\Sigma$ to an action of $J_{0,s}^\Sigma$ via the isomorphism described above \Cref{corollary: projection from J_0 to J_M monoid}, and even to an action of the monoid $J_{0,s}^\Sigma \HM^{\Sigma,+} J_{0,s}^\Sigma$.
In particular, we may consider the $G_0 \times J_{0,s} \HM^+ J_{0,s}$-module $\Acal_{\sigma_\Sigma, \repnofG}^{s-\Sigmaan} \boxtimes_R \sigma^\Sigma$.

\begin{lemma}\label{lemma: isomorphism over J0s in locally algebraic case}
The map
\begin{align*}
    \xi \colon \Acal_{\sigma_\Sigma, \repnofG}^{s-\Sigmaan} \otimes_R \sigma^\Sigma & \to \Acal_{\sigma, \repnofG}^{s-\Sigmaan} \\
    F \otimes v & \mapsto (\bar n m n \mapsto F(\bar n m n) \otimes m^{-1} v)
\end{align*}
for $\bar n \in \bar N_0, m \in \JM, n \in N_0$,
is an isomorphism of $G_0 \times J_{0,s} \HM^+ J_{0,s}$-modules.
\end{lemma}
\begin{proof}
The map $\xi$ is well-defined, since the image of an element $f \otimes v$ as above is an element of $\Ind_{J \cap P}^J(\sigma)^{\Sigmala}$ whose restriction to $\bar N_0$ is $\bar \Nbf_{\Sigma,s} \times \bar N^\Sigma_s$-analytic.
Note that, under restriction to $\bar N_0$, $\xi$ corresponds to the natural isomorphism of $R$-modules between $U \otimes_R \Ccal^{\bar \Nbf_{\Sigma,s} \times \bar N^\Sigma_s - \an}(\bar N_0, \sigma)$ and $U \otimes_R \Ccal^{\bar \Nbf_{\Sigma,s} \times \bar N^\Sigma_s - \an}(\bar N_0, \sigma_\Sigma) \otimes_R \sigma^\Sigma$, and is thus an $R$-module isomorphism.
It is also clearly $G_0$-equivariant.

Let us check that $\xi$ is $J_{0,s}$-equivariant. Let $\bar n \in \bar N_0$ and $y \in J_{0,s}$, and let $m_0$ be the projection of $y$ to $\JM$ under the Iwahori decomposition of $J_{0,s}$. It follows from the comment above \Cref{corollary: projection from J_0 to J_M monoid} that we can write $y^{-1} \bar n = \bar n_1 m_0^{-1} m_1 n_1$ for some $\bar n_1 \in \bar N_0$, $m_1 \in \JMs, n_1 \in N_s$, and we have
\begin{align*}
    (y \xi (F \otimes v))(\bar n) & = \xi(F \otimes v)(y^{-1} \bar n) = F(y^{-1} \bar n) \otimes m_0 v, \\
    \xi (y(F \otimes v))(\bar n) & = \xi (yF \otimes yv)(\bar n) = (yF)(\bar n) \otimes yv = F(y^{-1} \bar n) \otimes m_0 v.
\end{align*}
Similarly, if $\mu \in \NJM^+$, then
\begin{align*}
    (\mu \xi (F \otimes v))(\bar n) & = \mu \cdot \xi(F \otimes v)(\mu^{-1} \bar n \mu) = \mu  F(\mu^{-1} \bar n \mu) \otimes \mu v, \\
    \xi (\mu(F \otimes v))(\bar n) & = \xi (\mu F \otimes \mu v)(\bar n) = (\mu F)(\bar n) \otimes \mu v = \mu F(\mu^{-1} \bar n \mu) \otimes \mu v. \qedhere
\end{align*}
\end{proof}

\subsection{Untwisting}
\label{subsection: untwisting}

In this section we assume that $\HM \to \HM / \JM$ admits a multiplicative section \linebreak $\HM / \JM \simto \boldLambda \subseteq \HM$ such that $\boldLambda$ acts on $\sigma$ via a character $\chi \colon \boldLambda \to R^\times$.
Note that this assumption is always satisfied if $\HM = A \JM$ (with $\boldLambda = \Lambda$) and $\sigma$ admits a central character.
Then, it follows from the Bruhat--Tits decomposition \cite[7.4.15]{bruhat_tits_1} that the map $\boldLambda \to J \backslash J \HM J / J$ is a bijection. Let $\eta$ be its inverse.
We will write $\Acal_{\sigma, U}^\Sigma$ to denote either $\Acal_{\sigma, U}^{s-\Sigmaan}$ or $U \otimes_R \Ind_{J \cap P}^J ( \sigma )^{\Sigmala}$.

In this situation, we may modify the action of $J \HM^+ J$ on $\Acal_{\sigma, U}^\Sigma$ by letting $\mu \in J \HM^+ J$ act via $(\mu * F)(y) := \chi(\eta(\mu))^{-1} (\mu F)(y)$. In other words,
$$
    (\mu * F)(\bar n) = F(\mu^{-1} \bar n \eta(\mu) ).
$$
for $\bar n \in \bar N_0$.
For clarity, we will write $\Acal_{\sigma, U, *}^\Sigma$ when we want to consider $\Acal_{\sigma, U}^\Sigma$ as a $J \HM^+ J$-module with this action.
We will call this action of $J \HM^+ J$ the \emph{$*$-action} or \emph{untwisted action} to differentiate it from the action induced by \Cref{lemma: get action of monoid}.
It may appear more natural to call this action ``twisted" as opposed to ``untwisted", so let us explain our choice of terminology. Assume for simplicity that $\HM = A \JM \simeq \Lambda \times \JM$, and that $\sigma = \chi \otimes_R \sigma_0$ for some character $\chi$ of $\Lambda$ and representation $\sigma_0$ of $\JM$ on a finitely generated $R$-module. The $J$-module structure on $\Acal_{\sigma, U}^\Sigma$ depends only on $\sigma_0$ and not $\chi$, but the action of $\HM^+$ depends on both. However, this is not true for the $*$-action, which like the action of $J$ depends only on $\sigma_0$. However, it of course depends on the choice of $\Lambda$.

As before, this action of $J \HM^+ J$ induces an action of $\Hcal^-_R$ on $\cInd_J^G( \Acal_{\sigma, U, *}^\Sigma )$. Let $R(\chi)$ denote the $\Hcal^-_R$-module whose underlying $R$-module is $R$ and such that $U_\mu$ acts as $\chi(\mu)$ for $\mu \in \boldLambda^-$. The functor $\Mcal \mapsto R(\chi) \otimes_R \Mcal$ on $\Hcal^-_R$-modules is exact, and
$$
    \cInd_J^G( \Acal_{\sigma, U, *}^\Sigma ) \simeq R(\chi) \otimes_R \cInd_J^G( \Acal_{\sigma, U}^\Sigma ).
$$
In this setting, we may restate \Cref{theorem: quasi-isomorphism locally algebraic} as saying that there is an isomorphism in $D(\Hcal^-_R[G_0 \times G])$
$$
    R(\chi) \otimesL_{\Hcal^-_R} \cInd_J^G ( \Acal_{\sigma, U, *}^\Sigma )
    \stackrel\sim\to
    R(\chi) \otimes_R \left( \repnofG \otimes_R \Ind_P^G(V)^{\Sigmala} \right).
$$

\section{\texorpdfstring{$S$}{S}-arithmetic (co)homology}
\label{section: homology}

\subsection{Generalities on \texorpdfstring{$S$}{S}-arithmetic (co)homology}
\label{subsection: arithmetic homology}

Let $F$ be a number field and $\Gbf$ be a reductive group over $F$, and let $S$ a finite set of non-archimedean primes of $F$. For any place $v$ of $F$, we let $F_v$ denote the completion of $F$ at $v$, and we write $F_\infty = F \otimes_\Q \R$. We will write $\A_F$ for the ring of adeles of $F$, $\A_F^\infty$ for the ring of finite adeles and $\A_F^{S, \infty}$ for the ring of finite adeles away from $S$. 

Let $K^S$ be a compact open subgroup of $\Gbf(\A_F^{S, \infty})$. There exists a finite set of (non-archimedean) primes $S_\ram(K^S)$ containing $S$ such that we may write $K^S$ as a product
$$
    K^S = \prod_{v \not\in S_\ram(K^S)} K_v \times K_{\ram}
$$
where $K_\ram \subseteq \prod_{v \in S_\ram(K^S) \setminus S} \Gbf(F_v)$ is a compact open subgroup and $K_v$ is a hyperspecial subgroup of $\Gbf(F_v)$ for each $v \not\in S_\ram(K^S)$. We will assume that $K^S$ is neat, i.e. that $\Gbf(F) \cap g K^S g^{-1}$ is neat for all $g \in \Gbf(\A_F^{S, \infty})$, where the intersection takes place in $\Gbf(\A_F^{S, \infty})$. Write $\Kcal^S = G_\infty \times G_S \times K^S$.

\subsubsection{Geometric definition}
\label{subsubsection: definition 1 of arithmetic cohomology}

Fix a maximal compact subgroup $K_\infty$ of $G_\infty := \Gbf(F_\infty)$ and let $\Bscr_\infty := \Gbf(F_\infty) / K_\infty$. For each $v \in S$, write $\Bscr_v$ for the Bruhat--Tits building of $\Gbf_{F_v}$ over $F_v$, and $G_S = \prod_{v \in S} \Gbf(F_v)$.
Consider the topological spaces $\Bscr_S = \prod_{v \in S} \Bscr_v$ and $X_S = \Bscr_\infty \times \Bscr_S \times \Gbf(\A_F)$, where we give the $\Gbf(\A_F)$ factor the discrete topology.
We may equip this space with an action of the group $\Gbf(F) \times \Kcal^S$ as follows: $\Gbf(F)$ acts diagonally on all factors of the product via the embeddings $\Gbf(F) \into \Gbf(F_v)$, and $\Kcal^S$ acts trivially on $\Bscr_\infty \times \Bscr_S$ and by right translations on $\Gbf(\A_F)$.

Let $M$ be a (left) $\Z[\Gbf(F) \times \Kcal^S]$-module. Then $M$ defines a local system on $\Gbf(F) \backslash X_S / \Kcal^S$ consisting of locally constant sections of the map
$$
    \Gbf(F) \backslash (X_S \times M) / \Kcal^S \to \Gbf(F) \backslash X_S / \Kcal^S.
$$
Then, the $S$-arithmetic cohomology $H^*(K^S, M)$ of $M$ (of level $K^S$) is defined as the cohomology of this local system.
More explicitly, this is also the cohomology of the complex
$$
    C_\ad^\bullet (K^S, M) := \Hom_{\Z[\Gbf(F) \times \Kcal^S]} (C_\bullet^\sing(X_S), M),
$$
where here $C^\sing_\bullet(Y)$ denotes the complex of singular chains on a topological space $Y$.
Similarly, we may define the $S$-arithmetic homology $H_*(K^S, M)$ as the homology of the complex 
$$
    C^\ad_\bullet (K^S, M) := C_\bullet^\sing(X_S) \otimes_{\Z[\Gbf(F) \times \Kcal^S]} M.
$$

Let $\Abf$ denote the maximal $\Q$-split torus in the center of $\Res_{F/\Q} \Gbf$ and $\Abf(\R)^\circ \subseteq \Abf(\R)$ the connected component of the identity (for the Euclidean topology).
It acts on $X_S$ by acting on the $\Bscr_\infty$ factor.
Then, the spaces $X_S$ and $X_S / \Abf(\R)^\circ \times \Abf(\R)^\circ$ are $\Gbf(F) \times \Kcal^S$-equivariantly homotopically equivalent, where the latter space is given the diagonal action where the second factor is given the trivial action.
In particular, when taking (co)homology as above, one can work with either of these spaces.

Usually, $M$ will be a $\Kcal^S$-module, which we will consider as a $\Gbf(F) \times \Kcal^S$-module by letting the first factor act trivially.
If $M$ is a $\Gbf(\A_F)$-module, then we may view it as a $\Gbf(F) \times \Kcal^S$-module in two ways: by letting the first factor act by restriction of the $\Gbf(\A_F)$-action and the second factor act trivially, or by letting the first factor act trivially and the second act by restriction.
Let us denote $M_1$ and $M_2$ for the corresponding $\Gbf(F) \times \Kcal^S$-modules respectively.
Then, there is an isomorphism
\begin{align*}
    \Hom_{\Z[\Gbf(F) \times \Kcal^S]} (C_\bullet^\sing(X_S), M_1) & \simto \Hom_{\Z[\Gbf(F) \times \Kcal^S]} (C_\bullet^\sing(X_S), M_2)
    \\
    \phi & \mapsto
    ( (\sigma, g) \mapsto g^{-1} \phi( (\sigma, g ) )
\end{align*}
for $\sigma \in C^\sing_\bullet(\Bscr_\infty \times \Bscr_S), g \in \Gbf(\A_F)$.
In fact, the two corresponding local systems are isomorphic.
In particular, we see that $H^*(K^S, M)$ depends only on the action of $\Gbf(F)$, up to isomorphism.
This justifies, for example, switching the action on the algebraic part of the coefficients from $\infty$ to $p$ in the introduction since the two actions of $\Gbf(F)$ through $\Gbf(F \otimes_\Q \Q_p)$ and $\Gbf(F \otimes_\Q \R)$ agree.
Similarly, the same phenomenon happens for homology.

\subsubsection{Definition using \texorpdfstring{$S$}{S}-arithmetic groups}
\label{subsubsection: definition 3 of arithmetic cohomology}

The double coset set $\Gbf(F) \backslash \Gbf(\A_F) / \Kcal^S$ is finite.
Let $g_{S,1}, ..., g_{S,n}$ be a system of representatives for these cosets, and consider the $S$-arithmetic groups $\Gamma_{S,i} := \Gbf(F) \cap g_{S,i} K^S g_{S,i}^{-1}$, where the intersection is taken inside $\Gbf(\A_F)$.
Let $M$ be a $\Kcal^S$-module, regarded as a $\Gbf(F) \times \Kcal^S$-module by letting $\Gbf(F)$ act trivially  as above.
For each $i$, write $M(g_{S,i})$ for the $\Gamma_{S,i}$-module whose underlying abelian group is $M$ and where the action of $\Gamma_{S,i}$ is induced by the map $\Gamma_{S,i} \to \Kcal^S : \gamma \mapsto g_{S,i}^{-1} \gamma g_{S,i}$.
Then there are isomorphisms
\begin{align*}
      H_*(K^S, M) & \simeq \bigoplus_i H_*(\Gamma_{S,i}, M(g_{S,i})),
    & H^*(K^S, M) & \simeq \bigoplus_i H^*(\Gamma_{S,i}, M(g_{S,i})),
\end{align*}
where the right hand sides denote group homology and cohomology respectively.
Let us be more precise.
First, the $S$-arithmetic groups $\Gamma_{S,i}$ act freely on the contractible space $\Bscr_\infty \times \Bscr_S$, so we may compute the group (co)homology of $\Gamma_{S,i}$ as the (co)homology of $\Gamma_{S,i} \backslash (\Bscr_\infty \times \Bscr_S)$ with local coefficients given by $M(g_{S,i})$.
Moreover, there are isomorphisms $\coprod_{i=1}^n \Gamma_{S,i} \backslash (\Bscr_\infty \times \Bscr_S) \simto \Gbf(F) \backslash X_S / \Kcal^S$ given by $\Gamma_{S,i} (b_\infty, b_S) \mapsto \Gbf(F) (b_\infty, b_S, g_{S,i}) \Kcal^S$.
The corresponding local system on the target is the one described in \Cref{subsubsection: definition 1 of arithmetic cohomology}.
Similarly, if $M$ is instead a $\Gbf(F)$-module, then we have isomorphisms $H_*(K^S, M) \simeq \bigoplus_i H_*(\Gamma_{S,i}, M)$ and $H^*(K^S, M) \simeq \bigoplus_i H^*(\Gamma_{S,i}, M)$.

\subsubsection{Algebraic definition}
\label{subsubsection: definition 5 of arithmetic cohomology}

Let $M$ be a $\Z[\Gbf(F) \times \Kcal^S]$-module again.
Then, there are isomorphisms
\begin{align*}
    H_*(K^S, M) & \simeq \Tor^{\Z[\Gbf(F) \times \Kcal^S]}_* (\Z[\Gbf(\A_F)], M), \\
    H^*(K^S, M) & \simeq \Ext_{\Z[\Gbf(F) \times \Kcal^S]}^* (\Z[\Gbf(\A_F)], M).
\end{align*}
Here, $\Z[\Gbf(\A_F)]$ is given an action of $\Gbf(F) \times \Kcal^S$ by letting the first factor act on the left and the second on the right.
Comparing with \Cref{subsubsection: definition 1 of arithmetic cohomology}, it's enough to show that $C_\bullet^\sing(X_S)$ is a free $\Z[\Gbf(F) \times \Kcal^S]$-module resolution of $\Z[\Gbf(\A_F)]$.
It is clear that the homology is $\Z[\Gbf(\A_F)]$ since $\Bscr_\infty$ and $\Bscr_S$ are contractible, so it's enough to show that $G(F) \times \Kcal^S$ acts freely on $X_S$.
If $(\gamma, k)$ stabilises a point $(b_\infty, b_S, g) \in \Bscr_\infty \times \Bscr_S \times \Gbf(\A_F)$, then $\gamma = g k^{-1} g^{-1} \in \Gbf(F) \cap g \Kcal^S g^{-1}$ and $\gamma$ lies in the compact subgroup of $G_\infty \times G_S$ stabilising $(b_\infty, b_S)$.
Thus, $\gamma$ lies in the discrete subgroup $\Gbf(F)$ and in a compact subgroup of $\Gbf(\A_F)$, so it must have finite order.
But the assumption that $K^S$ is neat then implies that $\gamma$ (and hence also $k$) is the identity element, which proves our claim.

In fact, it follows from the work of Borel--Serre \cite{borel_serre_S_arithmetic} that $\Z[\Gbf(\A_F)]$ admits a resolution by \emph{finite free} $\Z[\Gbf(F) \times \Kcal^S]$-modules.
Indeed, let us write $\bar \Bscr_\infty$ for the Borel--Serre compactification of $\Bscr_\infty$ and $\bar X_S = \bar \Bscr_\infty \times \Bscr_S \times \Gbf(\A_F)$.
By \cite[Remarque after Proposition 6.10]{borel_serre_S_arithmetic}\footnote{To be precise, Borel--Serre work with $\Bscr_\infty / \Abf(\R)^\circ$ instead of $\Bscr_\infty$ and similarly for the compactifications. However, the homotopy equivalence mentioned in \Cref{subsubsection: definition 1 of arithmetic cohomology} allows us to translate their results to our setting.} and the isomorphism $\coprod_{i=1}^n \Gamma_{S,i} \backslash (\bar \Bscr_\infty \times \Bscr_S) \simto \Gbf(F) \backslash \bar X_S / \Kcal^S$ analogous to the one in \Cref{subsubsection: definition 3 of arithmetic cohomology}, there exists on $\bar X_S$ a $\Gbf(F) \times \Kcal^S$-equivariant triangulation with trivial stabilisers and finitely many orbits.
This triangulation gives rise to a finite free complex $Q_{S,\bullet}$ of $\Z[\Gbf(F) \times \Kcal^S]$-modules and a $\Gbf(F) \times \Kcal^S$-equivariant homotopy equivalence $Q_{S,\bullet} \to C^\sing_\bullet(\bar X_S)$, and the latter is $\Gbf(F) \times \Kcal^S$-equivariantly homotopically equivalent to $C^\sing_\bullet(X_S)$.
We make a choice of such triangulation and homotopy equivalences and write $C_\bullet(K^S, M) = Q_{S,\bullet} \otimes_{\Z[\Gbf(F) \times \Kcal^S]} M$ and $C^\bullet(K^S, M) = \Hom_{\Z[\Gbf(F) \times \Kcal^S]}(Q_{S,\bullet},  M)$, which are homotopically equivalent to $C^\ad_\bullet(K^S, M)$ and $C_\ad^\bullet(K^S, M)$ respectively.
Following \cite{hansen_universal_eigenvarieties}, we call these \emph{Borel--Serre complexes}.
Each of the $\Z$-modules $C_q(K^S, M)$ and $C^q(K^S, M)$ is isomorphic to a finite direct sums of copies of $M$, so if $M$ is a topological module over some ring $R$, we may equip $C_q(K^S, M)$ and $C^q(K^S, M)$ with their direct sum topologies, and $H_q(K^S, M)$ and $H^q(K^S, M)$ with their subquotient topologies.
Note, however, that these complexes, and hence any topology we may put on them and the induced one on (co)homology, depend on the choices of triangulation.

Let us highlight that all the isomorphisms between the different definitions of $S$-arithmetic (co)homology we have provided extend to isomorphisms in the derived category of abelian groups (or of $R$-modules for a commutative ring $R$, if the action of $\Z[\Gbf(F) \times \Kcal^S]$ on $M$ extends to an action of $R[\Gbf(F) \times \Kcal^S]$).
We will usually simply write $C_\bullet(K^S, M)$ and $C^\bullet(K^S, M)$ to denote the corresponding objects in the derived category (having fixed homotopy equivalences as above).

\subsubsection{Hecke operators}

Next, we discuss the action of Hecke operators on $S$-arithmetic (co)homology.
We will do this following the ($S$-arithmetic analogue of the) description in \cite{hansen_universal_eigenvarieties} using the explicit complexes $C^\ad_\bullet ( K^S, M )$.
However, we will shortly see in \Cref{subsection: shapiros lemma} that they admit a more natural description.
Let $\Delta \subseteq \Gbf(\A_F^{S,\infty})$ a submonoid containing $K^S$.
Let $R$ be a commutative ring and $\Hcal(\Gbf(\A_F^{S, \infty}), K^S)_R$ be the convolution algebra of $K^S$-biinvariant locally constant functions $\Gbf(\A_F^{S, \infty}) \to R$ as in \Cref{subsection: hecke algebras}.
Assume now that $M$ is an $R[G_\infty \times G_S \times \Delta^{-1}]$-module. Then, the complex $C^\ad_\bullet ( K^S, M )$ can be equipped with a natural right action of the subalgebra $\Hcal(\Delta, K^S)_R$ of $\Hcal(\Gbf(\A_F^{S, \infty}), K^S)_R$ generated by the indicator functions $[K^S \delta K^S]$ with $\delta \in \Delta$.
If $\delta_1, ..., \delta_s$ is a system of representatives in the coset space $K^S \delta K^S / K^S$, which is finite, then this action is given by
$$
    (\sigma \otimes m)[K^S \delta K^S] = \sum_{j} \delta_j^{-1} \sigma \otimes \delta_j^{-1} m
$$
for $\sigma \in C^\sing_\bullet(X_S), m \in M$. Similarly, if $M$ is a $R[G_\infty \times G_S \times \Delta]$-module, then $C_\ad^\bullet(K^S, M)$ may be equipped with an action of $\Hcal(\Delta, K^S)_R$, where for $\delta$ and $\delta_1, ..., \delta_s$ as above,
$$
    ([K^S \delta K^S] \phi)(\sigma) = \sum_j \delta_j \phi( \delta_j^{-1} \sigma ) )
$$
for any $\phi \in C_\ad^\bullet(K^S, M)$ and $\sigma \in C^\sing_\bullet(X_S)$. 
Note that these actions come from particular cases of \Cref{lemma: monoid hecke action on compact induction}.
For any $\Kcal^S$-modules $M$ and $N$, there is a natural isomorphism
\begin{align}\label{eqn: hecke operators equivariant for dual}
    \Hom_R(C^\ad_\bullet(K^S, M), N) \simeq C_\ad^\bullet(K^S, \Hom_R(M,N))
\end{align}
which is equivariant for the action of $\Hcal(\Delta, K^S)_R$ for all $v$.
Using the homotopy equivalences we have fixed at the end of \Cref{subsubsection: definition 5 of arithmetic cohomology}, we can define actions of the elements of $\Hcal(\Delta, K^S)_R$ on $C_\bullet ( K^S, M )$ and $C^\bullet(K^S, M)$ in each respective case. Note, however, that the resulting actions of the different operators in $\Hcal(\Delta, K^S)_R$ on these complexes only commute up to homotopy. In particular, this induces a well-defined action of $\Hcal(\Delta, K^S)_R$ on (co)homology (or at the level of derived categories).
If the action of $\Kcal^S$ on $M$ extends to an action of $\Gbf(\A_F)$, then we can always take $\Delta = G_v K^S$ for any $v \not\in S_\ram(K^S)$, and we obtain actions of $\Hcal(G_v K^S, K^S)_R \simeq \Hcal(G_v, K_v)_R$ on homology and cohomology. We set
$$
    \Tbb^S(K^S)_R = \bigotimes'_{v \not\in S_\ram(K^S)} \Hcal(G_v, K_v)_{R}
$$
and $\Tbb^S(K^S) = \Tbb^S(K^S)_\Qp$.

\subsection{Shapiro's lemma}
\label{subsection: shapiros lemma}

Next, we discuss a (Hecke-equivariant) version of Shapiro's lemma for $S$-arithmetic homology, as well as a result on Hecke operators.
Both of these results admit analogues for cohomology, but we will work only with homology.
We continue to use the notation from the previous section, and let $T \subseteq S$ be a subset and $K^T \subseteq \Gbf(\A_F^{T, \infty})$ an open compact subgroup of the form $K^S \times K_{S \setminus T}$ with $K_{S \setminus T} \subseteq G_{S \setminus T}$ an open compact subgroup, and $\Delta_{S \setminus T} \subseteq G_{S \setminus T}$ a submonoid containing $K_{S \setminus T}$.
Set $\Kcal^T := G_\infty \times G_T \times K^T$.

By \cite[3.3]{bruhat_tits_1}, there is a point $\alpha$ in $\Bscr_{S \setminus T}$ that is fixed by the action of $K_{S \setminus T}$.
The projection map
$$
    X_S = \Bscr_\infty \times \Bscr_S \times \Gbf(\A_F) = \Bscr_\infty \times \Bscr_T \times \Bscr_{S \setminus T} \times \Gbf(\A_F) \stackrel{f}{\to} \Bscr_\infty \times \Bscr_T \times \Gbf(\A_F) = X_T
$$
is a $(\Gbf(F) \times \Kcal^T)$-equivariant homotopy equivalence $X_S \simeq X_T$, whose inverse is given, up to $(\Gbf(F) \times \Kcal^T)$-equivariant homotopy, by the map
\begin{align*}
    \Bscr_\infty \times \Bscr_T \times \Gbf(\A_F) & \stackrel{g}{\to} \Bscr_\infty \times \Bscr_T \times \Bscr_{S \setminus T} \times \Gbf(\A_F) \\
    (b_\infty, b_T, h) & \mapsto (b_\infty, b_T, h_{S \setminus T} \alpha, h),
\end{align*}
where $h_{S \setminus T}$ denotes the image of $h \in \Gbf(\A_F)$ in $G_{S \setminus T}$.
Indeed, we have $f \circ g = \id_{X_T}$, and
\begin{align*}
    \Bscr_\infty \times \Bscr_T \times \Bscr_{S \setminus T} \times \Gbf(\A_F) \times [0, 1] & \to \Bscr_\infty \times \Bscr_T \times \Bscr_{S \setminus T} \times \Gbf(\A_F) \\
    (b_\infty, b_T, \beta, h, t) & \mapsto (b_\infty, b_T, (1-t)\beta + t h_{S \setminus T} \alpha, h)
\end{align*}
is a $(\Gbf(F) \times \Kcal^T)$-equivariant homotopy between $\id_{X_S}$ and $g \circ f$.
Here, we have used the notation from \cite[Section 2.5]{bruhat_tits_1}.
Namely, there is an appartment in the building $\Bscr_{S \setminus T}$ containing both $\beta$ and $\alpha$, and $(1-t)\beta + t h_{S \setminus T} \alpha$ runs through  the line segment between these points for $t \in [0,1]$.

By \Cref{lemma: general action of Hecke algebra}, the association $M \mapsto H_0(K^T, M) = R[\Gbf(\A_F)] \otimes_{R[\Gbf(F) \times \Kcal^T]} M$ defines a right exact functor from the category of $R[\Gbf(F) \times G_\infty \times G_T \times (\Delta_{S \setminus T})^{-1} \times K^S]$ to the category of right $\Hcal(\Delta_{S \setminus T}, K_{S \setminus T})_R$-modules. 
The following proposition shows that we may see $T$-arithmetic homology, equipped with its Hecke action, as the left derived functors of degree 0 homology.

\begin{proposition}
\label{proposition: hecke action on derived category}
The association $M_\bullet \mapsto R[\Gbf(\A_F)] \otimesL_{R[\Gbf(F) \times \Kcal^T]} M_\bullet$ defines a functor from the derived category of $R[\Gbf(F) \times G_\infty \times G_T \times (\Delta_{S \setminus T})^{-1} \times K^S]$-modules to the derived category of right $\Tbb^S(K^S) \otimes \Hcal(\Delta_{S \setminus T}, K_{S \setminus T})_R$-modules such that the degree $n$ homology of the complex $R[\Gbf(\A_F)] \otimesL_{R[\Gbf(F) \times \Kcal^T]} M$ is $H_n(K^T, M)$ with its natural right Hecke action.
\end{proposition}
\begin{proof}
Consider $C^\sing_\bullet(X_S) \otimes_\Z R$, it is a free $R[\Gbf(F) \times \Kcal^S]$-module resolution of $R[\Gbf(\A_F)]$, and hence also a free resolution of it as a module over $R[\Gbf(F) \times \Kcal^T]$ and over \linebreak $R[\Gbf(F) \times G_\infty \times G_T \times (\Delta_{S \setminus T})^{-1} \times K^S]$.
The image of $M_\bullet$ under functor above is thus computed by $\tot(C^\sing_\bullet(X_S) \otimes_{Z[\Gbf(F) \times \Kcal^T]} M_\bullet)$, where $\tot$ denotes the total complex.
Now, for any such $M_\bullet$, the homotopy equivalences of complexes of $\Z[\Gbf(F) \times \Kcal^T]$-modules
$$
    C_\bullet^{\sing}(X_S)
    \stackrel[f_*]{g_*}{\leftrightarrows}
    C_\bullet^{\sing}(X_T),
$$
induced by the maps $f, g$ above induces homotopy equivalences of $\Z[\Gbf(F) \times \Kcal^T]$-modules
$$
    \tot(C_\bullet^{\sing}(X_S) \otimes_{\Z[\Gbf(F) \times \Kcal^T]} M_\bullet)
    \stackrel[f_*]{g_*}{\leftrightarrows}
    \tot(C_\bullet^{\sing}(X_T) \otimes_{\Z[\Gbf(F) \times \Kcal^T]} M_\bullet).
$$
In order to prove the proposition, it's enough to show that $f_*$ is equivariant for the action of $\Tbb^S(K^S)$ and the right action of $\Hcal(\Delta_{S \setminus T}, K_{S \setminus T})_R$ on either side.
This follows immediately from the fact that $f_* \colon C_\bullet^{\sing}(X_S) \to C_\bullet^{\sing}(X_T)$ is in fact $\Gbf(F) \times \Gbf(\A_F)$-equivariant (as $f$ is).
\end{proof}

Next, we have the following version of Shapiro's lemma.

\begin{proposition}\label{proposition: shapiros lemma}
Assume for simplicity that $\Hcal(\Delta_{S \setminus T}, K_{S \setminus T})_R$ is commutative.
Let $M$ be a $R[\Gbf(F) \times G_\infty \times G_T \times (\Delta_{S \setminus T})^{-1} \times \Gbf(\A_F^{S, \infty})]$-module, which we view as a module over $R[\Gbf(F) \times G_\infty \times G_T \times (\Delta_{S \setminus T})^{-1} \times K^S]$ by restriction. Consider the compact induction $\cInd_{\Gbf(F) \times \Kcal^T}^{\Gbf(F) \times \Kcal^S} (M) \simeq \cInd_{K_{S \setminus T}}^{G_{S \setminus T}} (M)$ with its action of $\Hcal(\Delta_{S \setminus T}, K_{S \setminus T})_R$ as in \Cref{lemma: monoid hecke action on compact induction}.
There is a functorial isomorphism in $D(\Tbb^S(K^S)_R \otimes \Hcal(\Delta_{S \setminus T}, K_{S \setminus T})_R)$
\begin{align}\label{eqn: shapiros lemma 0}
    C_\bullet(K^S, \cInd_{\Gbf(F) \times \Kcal^T}^{\Gbf(F) \times \Kcal^S} (M) )
    \simto
    C_\bullet(K^S, M).
\end{align}
\end{proposition}
\begin{proof}
By \Cref{proposition: hecke action on derived category}, it's enough to prove the statement with \Cref{eqn: shapiros lemma 0} replaced by
\begin{align}\label{eqn: shapiros lemma}
    R[\Gbf(\A_F)] \otimesL_{R[\Gbf(F) \times \Kcal^S]}  \cInd_{\Gbf(F) \times \Kcal^T}^{\Gbf(F) \times \Kcal^S} (M)
    \simeq
    R[\Gbf(\A_F)] \otimesL_{R[\Gbf(F) \times \Kcal^T]} M
\end{align}
There is an isomorphism $\cInd_{\Gbf(F) \times \Kcal^T}^{\Gbf(F) \times \Kcal^S} (M) \stackrel\sim\to R[\Gbf(F) \times \Kcal^S] \otimes_{R[\Gbf(F) \times \Kcal^T]} M$ sending $[g, m]$ to $g \otimes m$. Moreover, the image of the target in $D(\Hcal(\Delta_{S \setminus T}, K_{S \setminus T})_R)$ is isomorphic to 
$R[\Gbf(F) \times \Kcal^S] \otimesL_{R[\Gbf(F) \times \Kcal^T]} M$ since $R[\Gbf(F) \times \Kcal^S]$ is free over $R[\Gbf(F) \times \Kcal^T]$.
Thus, we see that the isomorphism \Cref{eqn: shapiros lemma} holds in $D(R)$, so it's enough to show that for any module $Q$ over $R[\Gbf(F) \times G_\infty \times G_T \times (\Delta_{S \setminus T})^{-1} \times \Gbf(\A_F^{S, \infty})]$, the isomorphism
$$
    \phi \colon Q \otimes_{R[\Gbf(F) \times \Kcal^S]}  \cInd_{\Gbf(F) \times \Kcal^T}^{\Gbf(F) \times \Kcal^S} (M)
    \simto
    Q \otimes_{R[\Gbf(F) \times \Kcal^T]} M
$$
is equivariant for the actions of $\Hcal(\Delta_{S \setminus T}, K_{S \setminus T})_R$ and $\Tbb^S(K^S)_R$. The case of $\Tbb^S(K^S)_R$ is straightforward, so let us show this only for $\Hcal(\Delta_{S \setminus T}, K_{S \setminus T})_R$.
If $\delta \in \Delta_{S \setminus T}$ and we write $K_{S \setminus T} \delta K_{S \setminus \Delta} = \coprod_j \delta_j K_{S \setminus T}$, we have
\begin{align*}
    q \otimes [g, m] & \xmapsto{[K_{S \setminus T} \delta K_{S \setminus T}]} q \otimes ([g, m] [K_{S \setminus T} \delta K_{S \setminus T}]) = q \otimes \sum_j [g \delta_j, \delta_j^{-1} m] \\
    & \stackrel{\phi}{\mapsto} \sum_j \delta_j^{-1} g^{-1} q \otimes \delta_j^{-1} m,
\end{align*}
whereas
\begin{align*}
    q \otimes [g, m] & \stackrel{\phi}\mapsto g^{-1} q \otimes m \\
    & \xmapsto{[K_{S \setminus T} \delta K_{S \setminus T}]} \sum_j \delta_j^{-1} g^{-1} q \otimes \delta_j^{-1} m = \sum_j \delta_j^{-1} g^{-1} q \otimes \delta_j^{-1} m. \qedhere
\end{align*}
\end{proof}

\subsection{The global setting and local data}
\label{subsection: global setting}

Let us now introduce the global setting that we will place ourselves in the rest of the article.
Let $F$ be a number field and $\Gbf$ a reductive group over $F$.
Write $\Scalp$ for the set of primes of $F$ above $p$, and let $R$ be an affinoid $\Qp$-algebra.

Let $S$ be a finite set of non-archimedean primes of $F$.
For each $v \in S$ we let $\Gbf_v = \Gbf_{F_v}$ and place ourselves in the setting of \Cref{subsection: products of groups} and \Cref{section: locally analytic} -- we will use the notation therein.
Thus, for example, we have fixed for all $v$ a maximal $F_v$-split torus $\Sbf_v$ contained in a maximal torus $\Tbf_v$ defined over $F_v$, an ordering of the relative root system $\Phi_v$ of $(\Gbf_v, \Sbf_v)$, and a parabolic subgroup $\Pbf_v$ with Levi subgroup $\Mbf_v$ containing $\Tbf_v$, etc.
We set $G_v = \Gbf_v(F_v)$ and $G = G_S = \prod_{v \in S} G_v$ and similarly for the other subgroups of $G_v$, and analogously $\Phi = \coprod_{v \in S} \Phi_v$, etc.
We will never use $S$ to denote the points of the maximal split torus, so there should be no confusion with our set of primes $S$.

If $\Sigma$ is a subset of $\Scalp$, then we will write $\Gbf_\Sigma = \prod_{v \in \Sigma} \Res_{F_v/\Q_p} \Gbf_v$, $G_\Sigma = \Gbf_\Sigma(\Q_p)$ and $\Gbf_{\Sigma,E} = E \times_\Qp \Gbf_\Sigma$ for any finite extension $E / \Q_p$, and similarly for $\Mbf_\Sigma$, etc. We will also fix for $v \in \Scalp \setminus S$ a maximal torus $\Tbf_v \subseteq \Gbf_v$ defined over $F_v$. For any $v \in \Scalp$, we let $L'_v$ be a finite Galois extension of $F_v$ over which $\Tbf_v$ splits. 
If $E$ is a finite extension of $\Q_p$ in $\bar \Q_p$ containing the images of all the $\Q_p$-algebra embeddings of the $L'_v$ into $\bar \Q_p$, then $\Gbf_{\Scalp,E} = \prod_{v \in \Scalp} \prod_{\tau \colon F_v \into E} E \times_{F_v, \tau} \Gbf_v$ is split with maximal split torus $\Tbf_{\Scalp,E}$.
Fix for each $v \in \Scalp$ and embedding $\tau \colon F_v \into E$, fix a Borel subgroup $\Bbf_{v, \tau}$ of $E \times_{F_v, \tau} \Gbf_v$ containing $E \times_{F_v, \tau} \Tbf_v$ and contained in $E \times_{F_v, \tau} \Pbf_{v}$ whenever $v \in S$, and write $\Bbf_{\Scalp,E} = \prod_{v, \tau} \Bbf_{v, \tau}$.

Let us now fix $\Sigma \subseteq S_p := S \cap \Scalp$.
We will use the results from \Cref{subsection: quasi isomorphism} applied to the decomposition $G = G_\Sigma \times G^\Sigma$ with
\begin{align*}
    G_\Sigma & = \prod_{v \in \Sigma} G_v = \Gbf_\Sigma(\Q_p),
    &
    G^\Sigma & = \prod_{v \in S \setminus \Sigma} G_v.
\end{align*}
Starting in \Cref{subsection: homology of ind}, we will let $\sigma$ be a representation of a subgroup $\HM$ of $\NJM$ containing $A \JM$ on a finitely generated $R$-module as in \Cref{subsection: induction 2}.
We will assume it is of the form $\sigma = \bigboxtimes_{v \in \Sigma} \sigma_v$, where $\sigma_v$ is a locally analytic (resp. smooth) representation of $G_v$ if $v \in \Sigma$ (resp. $v \in S \setminus \Sigma$), although many of our results hold without making this assumption.
Set also $\sigma_\Sigma := \bigboxtimes_{v \in \Sigma} \sigma_v$ and $\sigma^\Sigma := \bigboxtimes_{v \in S \setminus \Sigma} \sigma_v$, and $V := \cInd_{\HM}^M (\sigma)$, which by \Cref{lemma: induction commutes with tensor products} is isomorphic to $V_\Sigma \hat\boxtimes_E V^\Sigma$ for $V_\Sigma := \cInd_{\HM[\Sigma]}^{M_\Sigma} (\sigma_\Sigma)$ and $V^\Sigma := \cInd_{\HM^\Sigma}^{M^\Sigma} (\sigma^\Sigma)$ (and similarly for other subsets of $S$).
Let $\repnofG$ be a representation of $G_{\Scalp}$ on a finite flat $R$-module.
We will view it as a representation of $G_{\Scalp} \times G$ by letting $G$ act trivially.
In \Cref{subsection: overconvergent homology} our notation will be different and we will describe it therein.
Instead of $\Hcal(J \HM^- J, J)_R$ we will write $\Hcal_{S, R}^-$, or $\Hcal_{S,\HM, R}^-$ when we want to make the subgroup $\HM$ explicit.
Define for $s$ sufficiently large the $J \HM^+ J$-modules $\Acal_{\sigma, U}^{s-\Sigmaan}$ as in \Cref{subsection: induction 2} and write $\Dcal_{\sigma, U}^{s-\Sigmaan}$ for their continuous $R$-duals.
Define also $s_\sigma$ as in \Cref{subsection: induction 2}.
To simplify notation, we will write $\Acal_{\sigma, U}^{\Sigmala}$ and $\Dcal_{\sigma, U}^{\Sigmala}$ for $U \otimes_R \Ind_{J \cap P}^J (\sigma)^{\Sigmala}$ and its continuous $R$-linear dual, and $\SigmaInd_P^G (U, V)$ for $U \otimes_R \Ind_P^G (V)^{\Sigmala}$.
Sometimes we don't want to distinguish between $\Acal_{\sigma, U}^{s-\Sigmaan}$ for $s \geq s_\sigma$ and $\Acal_{\sigma, U}^{\Sigmala}$, so we will write $\Acal_{\sigma, U}^\Sigma$ for either of these, and similarly for $\Dcal_{\sigma, U}^\Sigma$.
We will view these as representations of $G_{S \cup \Scalp}$ via the embedding $G_{S \cup \Scalp} \into G_{\Scalp} \times G$.

Fix an isomorphism $\iota \colon \bar \Q_p \simto \C$. Let $E$ be a finite extension of $\Q_p$ contained in $\bar \Q_p$, and assume that $R$ is an affinoid $E$-algebra.
Our main case of interest is the case when $U$ is an (irreducible) algebraic representation of $G_{\Scalp \setminus \Sigma}$ (see \Cref{subsection: automorphic representations} for what we mean by this) and $\sigma$ is the twist by a locally analytic character of $M$ or of $\HM$ of the tensor product of an (irreducible) algebraic representation of $M_\Sigma$ and a smooth representation $\sigma_0$ of $M$ over $E$ such that $\C \otimes_{E, \iota} \cInd_{\HM}^M (\sigma_0)$ is a supercuspidal representation of $M$.
As we mentioned in the introduction, it follows from \cite{bushnell_kutzko}, \cite{stevens_supercuspidals_classical}, \cite{secherre_stevens_supercuspidals}, \cite{fintzen_types} that in many cases, all supercuspidal representations of $M$ can be obtained in this manner.
For example, from the latter we can deduce the following result.

\begin{theorem}\label{theorem: recollections supercuspidals}
Assume that for all $v \in S$, $\Mbf_v$ splits over a tamely ramified extension of $F_v$ and $v$ does not divide the order of the Weyl group of $\bar F_v \times_{F_v} \Mbf_v$. Let $\pi$ be a supercuspidal representation of $M$. Then, there exists a facet $\Fscr_M$ in the Bruhat--Tits building of $\Mbf$, a subgroup $\HM$ of $\NJM$ containing $A \JM$ (where $\NJM$ and $\JM$ are defined in terms of the facet $\Fscr_M$ as in \Cref{section: structure}), a finite extension $E / \Q_p$ in $\bar \Q_p$ and an irreducible smooth representation $\sigma_0$ over $E$ admitting a central character such that $\pi \simeq \C \otimes_{E, \iota} \cInd_{\HM}^M (\sigma_0)$.
\end{theorem}
\begin{proof}
This follows from the main result of \cite{fintzen_types}.
The representation $\pi$ decomposes as a tensor product $\pi = \bigotimes_{v \in S} \pi_v$ of supercuspidal representations of $M_v$ for $v \in S$.
According to \cite[Theorem 8.1]{fintzen_types}, there exists for each $v$ a facet $\Fscr_{M_v}$ of the Bruhat--Tits building of $\Mbf_v$ giving rise to subgroups $\JM, \NJM$, and a representation $\sigma_{v,\C}$ of a finite index subgroup $\HM[v]$ of $\NJM_v$ containing $A_v$ such that $\pi \simeq \cInd_{\HM[v]}^{M_v} (\sigma_{v,\C})$.
By compactly inducing $\sigma_{v,\C}$ if necessary, we may assume that $\HM[v]$ contains $A_v \JM[v]$.
Extend scalars along $\iota^{-1}$ to obtain a representation $\sigma_{v,\bar \Q_p}$ over $\bar \Q_p$.
Since $\HM[v]$ is topologically finitely generated for all $v \in S$, there exists an extension $E$ of $\Q_p$ in $\bar \Q_p$ and a representations $\sigma_{v}$ of $\HM[v]$ over $E$ such that $\sigma_{v,\bar \Q_p} = \bar \Q_p \otimes_{E_v} \sigma_{v}$ for all $v$.
Finally, take $\sigma_0 = \bigotimes_{v \in S} \sigma_v$ and $\HM = \prod_{v \in S} \HM[v]$.
\end{proof}

We will consider $\spadesuit$-arithmetic homology and cohomology for two different $\spadesuit$:
\begin{itemize}
    \item The arithmetic case, $\spadesuit = \emptyset$. We will usually omit the subindex $S$ from the notation in the previous section. Fix a compact open subgroup $K^S$ of $\Gbf(\A_F^{S, \infty})$. Our compact open subgroups will usually be $J \times K^S$ we will consider Hecke operators with respect to the monoids $\Delta_S = J \HM^- J \times K^S$ and $\Delta_v = G_v K^S$ for $v \not\in S_\ram(K^S)$. We fix a Borel--Serre complex and homotopy equivalences as in \Cref{subsubsection: definition 5 of arithmetic cohomology}.
    
    \item The $S$-arithmetic case, $\spadesuit = S$. We have already fixed a compact open subgroup $K^S$ of $\Gbf(\A_F^{S, \infty})$. Let us also fix Borel--Serre complexes and homotopy equivalences as in \Cref{subsubsection: definition 5 of arithmetic cohomology}. When defining Hecke operators we take $\Delta_v = G_v K^S$ for any $v \not\in S_\ram(K^S)$.
\end{itemize}
Finally, we will also write $\Tbb(K^S) := \Tbb^S(K^S) \otimes_\Qp \Hcal_{S, \Qp}^-$.

\subsection{Overconvergent homology}
\label{subsection: overconvergent homology}
Next, we will recall the definition and state some of the properties of (parabolic) overconvergent homology (see \cite{ash_steves_p-adic_deformations}, \cite{hansen_universal_eigenvarieties}, \cite{urban_eigenvarieties}, \cite{parabolic_overconvergent}) adapted to our general context, where we don't assume that $\Q_p \times_\Q \Res_{F/\Q} \Gbf$ is quasi-split.

We will assume in this section that $\sigma$ is a locally analytic representation of $\JM$ (instead of the bigger group $\HM$) on a finite free $R$-module, of the form $\sigma = \sigma_\Sigma \boxtimes \sigma^\Sigma$ with $\sigma_\Sigma$ a locally analytic representation of $\JM[\Sigma]$ and $\sigma^\Sigma$ a smooth representation of $\JM^\Sigma$. We may define $R$-modules $\Acal_{\sigma, U}^{s-\Sigmaan}, \Acal_{\sigma, U}^{\Sigmala}, \Acal_{\sigma, U}^{\Sigma}, \Dcal_{\sigma, U}^{s-\Sigmaan}, \Dcal_{\sigma, U}^{\Sigmala}, \Dcal_{\sigma, U}^{\Sigma}$ as in \Cref{subsection: global setting} equipped with an action of $J$ for $s \geq s_\sigma$ for some $s_\sigma \geq 0$. However, since $\sigma$ is not equipped with an action of $A \JM$, this action does not extend to an action of $J A^+ J$ (or $J A^- J$ for the duals) as in \Cref{lemma: get action of monoid} (and thus to any monoid of the form $J \HM^+ J$ where $\HM$ contains $A \JM$).
Instead, we note that we are in a similar situation to that of \Cref{subsection: untwisting}: the inclusion $\Lambda \into J A J$ induces a bijection between $\Lambda$ and $J \backslash J A J / J$. If we denote its inverse by $\eta$, then the action of $J$ on $\Acal_{\sigma, U}^{s-\Sigmaan}$ can be extended to an action of $J A^+ J$ by the formula
$$
    (a * F)(\bar n) = F(a^{-1} \bar n \eta(a) ).
$$
As remarked in \Cref{subsection: untwisting}, if the action of $\sigma$ does extend to an action of $A \JM$ and $\Lambda$ acts via a character $\chi$, then this agrees with the $*$-action we have previously defined.
For this reason, we will denote these spaces by $\Acal_{\sigma, U,*}^{s-\Sigmaan}, \Dcal_{\sigma, U,*}^{s-\Sigmaan}$, etc.

The $R$-module $\Acal_{\sigma, U,*}^{s-\Sigmaan}$ is an orthonormalisable Banach $R$-module, and for any $a \in \Lambda^{-}$ (resp. $a \in \Lambda_\Sigma^{--}$) the action of $a^{-1}$ gives a norm-decreasing (resp. $R$-compact) homomorphism $\Acal_{\sigma, U, *}^{s-\Sigmaan} \to \Acal_{\sigma, U, *}^{s-\Sigmaan}$ (by the arguments in the paragraph above \Cref{lemma: U operators}).
As we described in \Cref{subsection: arithmetic homology}, the action of the monoid $J A^+ J$ induces an action of $\Hcal_{S, R}^-$ on $C_\bullet(K^S J, \Acal_{\sigma, U, *}^{s-\Sigmaan})$. We denote the action of $[J a J]$ by $U_a$ as usual. This complex of $R$-modules is a finite direct sum of copies of $\Acal_{\sigma, U, *}^{s-\Sigmaan}$, and we see from the corresponding statement for the $*$-action that the operator $U_a$ on this complex is norm-decreasing (resp. $R$-compact) when this space is given its direct sum Banach topology.

We can use this to obtain properties of slope decompositions of $C_\bullet(K^S J, \Acal_{\sigma, U, *}^{s-\Sigmaan})$ with respect to $U_a$. Let us recall what this means.
If $Q(X) \in R[X]$ is a polynomial of degree $d$, we write $Q^*(X) = X^d Q(1/X)$.
Given a point $x$ of $\Sp(R)$, write $Q_x(X)$ for the specialisation of $Q(X)$ at $x$.
Let $h \in \Q_{\geq 0}$.
We say that $Q(X)$ has slope-$\leq h$ if for every point $x$ of $\Sp(R)$, every edge in the Newton polygon of $Q_x(X)$ has slope $\leq h$, and that $Q(X)$ is multiplicative if its leading coefficient is a multiplicative unit (for a fixed norm of $R$), i.e. it is a unit $x \in R^\times$ satisfying $|x^{-1}| = |x|^{-1}$. If $W$ is an $R$-module equipped with an $R$-linear operator $U$, then we say that a vector $w \in W$ has slope $\leq h$ if there exists some polynomial $Q(X) \in R[X]$ of slope $\leq h$ such that $Q^*(U) w = 0$. Finally, we say that a $U$-stable decomposition $W = W_{\leq h} \oplus W_{>h}$ is a slope-$\leq h$ decomposition of $W$ (with respect to $U$) if $W_{\leq h}$ is finitely generated and consists of elements of slope $\leq h$ and $Q^*(U) \colon W_{>h} \to W_{>h}$ is an $E$-linear isomorphism for every multiplicative polynomial $Q(X)$ of slope $\leq h$. If such a decomposition exists, then it is unique. We refer to \cite[Section 2.3]{hansen_universal_eigenvarieties} for a summary of the basic properties of slope decompositions. Let us fix some $a \in A^-$ such that $a^{-1} * \Acal_\sigma^{s-\an} \subseteq \Acal_\sigma^{(s-1)-\an}$; all slope decompositions will be taken with respect to the operator $U_a$.

\begin{proposition}\label{proposition: properties overconvergent homology}
Assume that $C_\bullet(K^p J, \Acal_{\sigma, U, *}^{s-\Sigmaan})$ admits a slope-$\leq h$ decomposition
$$
    C_\bullet(K^p J, \Acal_{\sigma, U, *}^{s-\Sigmaan}) = C_\bullet(K^p J, \Acal_{\sigma, U, *}^{s-\Sigmaan})_{\leq h} \oplus C_\bullet(K^p J, \Acal_{\sigma, U, *}^{s-\Sigmaan})_{>h}.
$$
Then:
\begin{enumerate}
    \item For any $s' \geq s_\sigma$, the complexes $C_\bullet(K^p J, \Acal_{\sigma, U, *}^{s'-\Sigmaan})$ and $C_\bullet(K^p J, \Acal_{\sigma, U, *}^{\Sigmala})$ admit slope-$\leq h$ decompositions with
    \begin{align*}
        C_\bullet(K^p J, \Acal_{\sigma, U, *}^{s-\Sigmaan})_{\leq h}
        & \simto
        C_\bullet(K^p J, \Acal_{\sigma, U, *}^{s'-\Sigmaan})_{\leq h} \\
        & \simto
        C_\bullet(K^p J, \Acal_{\sigma, U, *}^{\Sigmala})_{\leq h}.
    \end{align*}

    \item $H_*(K^p J, \Acal_{\sigma, U, *}^{\Sigma})$ admits a slope-$\leq h$ decomposition whose slope-$\leq h$ part is the homology of $C_\bullet(K^p J, \Acal_{\sigma, U, *}^{\Sigma})_{\leq h}$.
    
    \item For any Banach $R$-algebra $R'$,
    $$
        C_\bullet(K^p J, \Acal_{\sigma \otimes_R R', U \otimes_R R', *}^{\Sigma})_{\leq h} \simeq C_\bullet(K^p J, \Acal_{\sigma, U, *}^{\Sigma})_{\leq h} \otimes_R R'.
    $$
    
    \item For any Banach $R$-algebra $R'$, there is a first-quadrant homological spectral sequence
    $$
        E^2_{i,j} = \Tor_i^R (H_j(K^p J, \Acal_{\sigma, U, *}^{\Sigma})_{\leq h}, R') \implies H_{i+j} (K^p J, \Acal_{\sigma \otimes_R R', U \otimes_R R', *}^{\Sigma})_{\leq h}
    $$
    In particular, if $R'$ is flat over $R$, then 
    $$
        H_*(K^p J, \Acal_{\sigma \otimes_R R', U \otimes_R R', *}^{\Sigma})_{\leq h} \simeq H_*(K^p J, \Acal_{\sigma, U, *}^{\Sigma})_{\leq h} \otimes_R R'.
    $$
    
    \item\label{item: properties overconvergent homology power series} For a flat Banach $R$-algebra $R'$, let
    $$
        F_{R'}^s(X) = \det(1 - X U_a | C_\bullet(K^p J, \Acal_{\sigma \otimes_R R', U \otimes_R R', *}^{s-\Sigmaan})) \in R' \{\!\{ X \}\!\}
    $$
    be the characteristic power series of the operator $U_a$ on $C_\bullet(K^p J, \Acal_{\sigma \otimes_R R', U \otimes_R R', *}^{s-\Sigmaan})$. Then, $F_R^s(X)$ is independent of $s \geq s_\sigma$ and $F_{R'}^s(X)$ is the image of $F_R^s(X)$ in $R' \{\!\{ X \}\!\}$.
\end{enumerate}
\end{proposition}
\begin{proof}
Everything follows from the same arguments as in \cite[Sections 3.1 and 3.3]{hansen_universal_eigenvarieties}.
\end{proof}

\subsection{\texorpdfstring{$S$}{S}-arithmetic homology of parabolic induction}
\label{subsection: homology of ind}

Or next goal is to relate the $S$-arithmetic homology of $\SigmaInd_P^G (U, V)$ to the arithmetic homology of $\Acal_{\sigma, U}^{\Sigma}$ via the quasi-isomorphism from \Cref{theorem: quasi-isomorphism locally algebraic}.
Recall that we view $\SigmaInd_P^G (U, V)$ as a module over $\Hcal_{S, R}^-$ by letting $U_\mu$ act as the identity for all $\mu$.
We will do the same for the Borel--Serre complex $C_\bullet(K^S, \SigmaInd_P^G (U, V))$.

\begin{proposition}\label{proposition: quasi-isomorphism S-arithmetic}
There is an isomorphism in $D(\Tbb(K^S)_R)$\footnote{Similar to our remark at the beginning of \Cref{subsection: quasi isomorphism}, it would be more precise to write $\Tbb^S(K^S)_R(1) \otimes^{\mathbb{L}}_{\Tbb(K^S)_R} \blank$ instead of $R(1) \otimes^{\mathbb{L}}_{\Hcal_{S, R}^-} \blank$.}
$$
	C_\bullet(K^S, \SigmaInd_P^G (U, V))
	\simeq
	R(1) \otimes^{\mathbb{L}}_{\Hcal_{S, R}^-} C_\bullet(K^S, \cInd_J^G(\Acal_{\sigma, U}^{\Sigma})).
$$
\end{proposition}
\begin{proof}
Let $C_\bullet$ be a projective $R[\Gbf(F) \times \Gbf(\A_F)]$-module resolution of $R[\Gbf(\A_F)]$ and $P_\bullet$ be a projective $\Hcal_{S, R}^-$-module resolution of $R(1)$.
Consider the map $\varphi$ from \Cref{subsection: quasi isomorphism},
$$
    P_\bullet \otimes_{\Hcal_{S, R}^-} \cInd_J^G(\Acal_{\sigma, U}^{\Sigma}) \simto \SigmaInd_P^G (U, V),
$$
which is a quasi-isomorphism and commutes with the action of $G$.
In particular, we see that there are quasi-isomorphisms
\begin{align*}
    & R(1) \otimes^{\mathbb{L}}_{\Hcal_{S, R}^-} C_\bullet(K^S, \cInd_J^G(\Acal_{\sigma, U}^{\Sigma})) \\
    & \simeq \tot \left( P_\bullet \otimes_{\Hcal_{S, R}^-} ( \cInd_J^G(\Acal_{\sigma, U}^{\Sigma})  \otimes_{R[\Gbf(F) \times \Kcal^S]} C_\bullet ) \right) \\ %
    & \simeq \SigmaInd_P^G (U, V) \otimes_{R[\Gbf(F) \times \Kcal^S]} C_\bullet \\ %
    & \simeq C_\bullet(K^S, \SigmaInd_P^G (U, V)) %
\end{align*}
The composition of these quasi-isomorphisms is $\Tbb^S(K^S)$-equivariant, since it is the homomorphism induced by the $G$-equivariant map $\varphi$.
\end{proof}

Combining this with Shapiro's lemma, \Cref{proposition: shapiros lemma}, we obtain the following.

\begin{corollary}\label{corollary: spectral sequence for affinoids}
There is an isomorphism in $D(\Tbb(K^S)_R)$
$$
	C_\bullet(K^S, \SigmaInd_P^G (U, V)) \simeq R(1) \otimes^{\mathbb{L}}_{\Hcal_{S, R}^-} C_\bullet(K^S J, \Acal_{\sigma, U}^{\Sigma}).
$$
In particular, there is a $\Tbb^S(K^S)$-equivariant first quadrant (homological) spectral sequence of $R$-modules
$$
    E^2_{i, j} = \Tor^{\Hcal_{S, R}^-}_i( R(1), H_j(K^S J, \Acal_{\sigma, U}^{\Sigma})) \implies H_{i+j}(K^S, \SigmaInd_P^G (U, V)).
$$
\end{corollary}

In many cases it is convenient to replace the space $H_*(K^S J, \Acal_{\sigma, U}^{\Sigma})$ with something smaller.
In this situation, the following lemma will be useful.
It can be applied, for example, to slope-$\leq h$ decompositions when $|\chi(\mu)^{-1}| \leq p^h$.

\begin{lemma}\label{lemma: eigenvectors decomposition}
Recall that $\Hcal^-_{S,R} \simeq R[\HM / \JM]$.
Let $\chi \colon \HM / \JM \to R^\times$ be a character and let $M$ be an $\Hcal_{S, R}^-$-module admitting a decomposition $M = M_1 \oplus M_2$ such that, for some $\mu \in \HM^-$ the operator $1 - \chi(\mu)^{-1} U_\mu$ acts invertibly on $M_2$. Then,
$$
	R(\chi) \otimes^{\mathbb{L}}_{\Hcal_{S, R}^-} M
	\simeq
	R(\chi) \otimes^{\mathbb{L}}_{\Hcal_{S, R}^-} M_1.
$$
\end{lemma}
\begin{proof}
It's enough to show that $R(1) \otimes^{\mathbb{L}}_{\Hcal_{S, R}^-} M_2 = 0$. This follows from the fact that $1 - \chi(\mu)^{-1} U_\mu$ is zero on $R(\chi)$ but is invertible on $M_2$.
\end{proof}

Applying this to the character $U_\mu \mapsto 1$, it follows from this that if $C_\bullet(K^S J, \Acal_{\sigma, U}^{\Sigma})$ admits a slope-$\leq 0$ decomposition with respect to some $U_\mu$, then $H_*(K^S, \SigmaInd_P^G (U, V))$ is finitely generated over $R$. We will see in \Cref{subsection: construction of eigenvarieties} that this is true even when there is no slope-$\leq 0$ decomposition. Let us remark that if $\mu \in \HM^-$ (resp. $a \in A_\Sigma^{--}$) then the operator $U_\mu$ (resp. $U_a$) on $C_\bullet (K^S J, \Acal_{\sigma, U}^{s-\Sigmaan})$ is continuous (resp. $R$-compact).

\begin{lemma}\label{lemma: untwisting homology}
Assume that we are in the situation from \Cref{subsection: untwisting}. There is an isomorphism in $D(\Tbb(K^S)_R)$
$$
    R(\chi) \otimes^{\mathbb{L}}_{\Hcal_{S, R}^-} C_\bullet(K^S J, \Acal_{\sigma, U, *}^{\Sigma})
	\simeq
	R(\chi) \otimes_R \left( R(1) \otimes^{\mathbb{L}}_{\Hcal_{S, R}^-} C_\bullet(K^S J, \Acal_{\sigma, U}^{\Sigma}) \right).
$$
\end{lemma}
\begin{proof}
This follows from the fact that the functor $R(\chi) \otimes_R \blank$ on $\Hcal_{S,R}^-$-modules is exact, preserves projectives, and $R(\chi) \otimes_R (M \otimes_{\Hcal_{S,R}^-} N) \simeq (R(\chi) \otimes_R M) \otimes_{\Hcal_{S,R}^-} (R(\chi) \otimes_R N)$ for any $\Hcal_{S,R}^-$-modules $M, N$.
\end{proof}

\subsection{Finiteness and vanishing over a point}
\label{subsection: field case}

In the case where $R$ is a finite field extension of $\Q_p$, we can say more about the $S$-arithmetic homology spaces from the previous section.
Throughout this section, we assume that $R$ is such an extension, and write $E$ instead of $R$ to denote this.
Write $\m_1$ for the ideal of $\Hcal_{S, E}^-$ generated by $U_\mu-1$ for all $\mu \in \HM^-$.
Given an $\Hcal_{S, E}^-$-module $\Mcal$ and $n \geq 0$, we write $\Mcal[\m_1^n]$ for the submodule of elements which are killed by the ideal $\m_1^n$ and $\Mcal[\m_1^\infty] = \bigcup_{n \geq 0} \Mcal[\m_1^n]$. This is the generalised eigenspace for the system of eigenvalues $U_\mu \mapsto 1$.
As in \Cref{subsection: overconvergent homology}, when $a \in A_\Sigma^{--}$ the operator $U_a$ is compact.

\begin{proposition}\label{proposition: can take generalised eigenspace}
There is an isomorphism in $D(\Tbb(K^S)_E)$
$$
    E(1) \otimesL_{\Hcal_{S, E}^-} C_\bullet(K^S J, \Acal_{\sigma, U}^\Sigma) \simeq E(1) \otimesL_{\Hcal_{S, E}^-} C_\bullet(K^S J, \Acal_{\sigma, U}^\Sigma)[\m_1^\infty].
$$
There is a $\Tbb^S(K^S)$-equivariant spectral sequence
$$
	E^2_{i, j} = \Tor^{\Hcal_{S, E}^-}_i( E(1), H_j(K^S J, \Acal_{\sigma, U}^\Sigma)[\m_1^\infty]) \implies H_{i+j}(K^S, \SigmaInd_P^G (U, V)).
$$
\end{proposition}
\begin{proof}
By \cite[Theorem 2.3.8]{urban_eigenvarieties} and the compactness of the operators $U_a$ for $a \in A_\Sigma^{--}$, the complex of Banach spaces $C_\bullet(K^S J, \Acal_{\sigma, U}^{s-\Sigmaan})$ admits a slope-$\leq 0$ decomposition with respect to some $U_a$. As in \Cref{proposition: properties overconvergent homology}, this implies that $C_\bullet(K^S J, \Acal_{\sigma, U}^{\Sigmala})$ also has a slope-$\leq 0$ decomposition. Applying \Cref{lemma: eigenvectors decomposition} to these decompositions, we see that
$$
    E(1) \otimesL_{\Hcal_{S, E}^-} C_\bullet(K^S J, \Acal_{\sigma, U}^\Sigma)
    \simeq
    E(1) \otimesL_{\Hcal_{S, E}^-} C_\bullet(K^S J, \Acal_{\sigma, U}^\Sigma)_{\leq 0}
$$
As $C_\bullet(K^S J, \Acal_{\sigma, U}^\Sigma)_{\leq 0}$ is finite-dimensional (in each degree), there is a direct sum decomposition
$$
    C_\bullet(K^S J, \Acal_{\sigma, U}^\Sigma)_{\leq 0} = C_\bullet(K^S J, \Acal_{\sigma, U}^\Sigma) [\m_1^\infty ] \oplus C_\bullet(K^S J, \Acal_{\sigma, U}^\Sigma)_{\leq 0}^{\Hcal_{S, E}^- \neq 1}
$$
where, for some $\mu \in \HM^{-}$, $U_{\mu} - 1$ acts invertibly on the second summand. Applying \Cref{lemma: eigenvectors decomposition} to this decomposition the result follows.
The existence of the spectral sequence then follows as usual.
\end{proof}

\begin{corollary}\label{corollary: homology is finite dimensional}
\begin{enumerate}
    \fixitem The vector space $H_*(K^S, \SigmaInd_P^G (U, V))$ is finite-dimensional.
    \item $H_*(K^S, \SigmaInd_P^G (U, V))$ vanishes unless $1$ is an eigenvalue of $U_\mu$ in $H_*(K^S J, \Acal_{\sigma, U}^\Sigma)$ for all $\mu \in \HM^-$.
    \item Assume that for some choice of norm on the Banach space $U \otimes_E \sigma$, some element $\mu$ of $\HM^+$ acts as an operator of norm $< 1$. Then $H_*(K^S, \SigmaInd_P^G (U, V)) = 0$.
\end{enumerate}
\end{corollary}
\begin{proof}
As $H_*(K^S J, \Acal_{\sigma, U}^\Sigma)[\m_1^\infty]$ is finite-dimensional, we see by using a resolution of $E(1)$ by finite free $\Hcal_{S,E}^-$-modules to compute the Tor groups in the $E^2$ page of the spectral sequence from \Cref{proposition: can take generalised eigenspace} that these are again finite-dimensional. This proves \itemnumber{1}.
If for some $\mu \in \HM^-$, $1$ is not an eigenvalue of $U_{\mu}$ in $H_*(K^S J, \Acal_{\sigma,U}^\Sigma)$, then the generalised eigenspace for 1 of $H_*(K^S J, \Acal_{\sigma, U}^\Sigma)$ is zero, so all the terms in the $E^2$ page of this spectral sequence vanish, and hence so does $H_*(K^S, \SigmaInd_P^G (U, V))$.
Finally, if $\mu$ is as in \itemnumber{3} and $x$ is its norm as an operator on $U \otimes_E \sigma$ for some choice of norm on this space, then the action $\Acal_{\sigma,U}^{s-\Sigmaan} \to \Acal_{\sigma,U}^{s-\Sigmaan}$ of $\mu$ has norm $\leq x$ (for the norm induced by that of $U \otimes_E \sigma$ via the obvious analogue of \Cref{eqn: isomorphism for Acal} including $U$), and one can deduce from this
that the operator $U_{\mu^{-1}}$ on $C_\bullet(K^S J, \Acal_{\sigma, U}^{s-\Sigmaan})$ also has norm $\leq x$. The same must then be true of its eigenvalues, so by \itemnumber{2}, if $S$-arithmetic homology is non-zero then $1 \leq x$.
\end{proof}

We will now shift our focus from homology to cohomology. A lot of the previous statements have analogues in this setting. Recall that we write $(\blank)'$ to denote the continuous duals of locally convex vector spaces over $E$.

\begin{proposition}\label{proposition: cohomology is dual of homology}
The locally convex topology on $H_*(K^S, \SigmaInd_P^G (U, V))$ induced by that of $C_\bullet(K^S, \SigmaInd_P^G (U, V))$ is the Hausdorff topology.
There is a $\Tbb^S(K^S)$-equivariant isomorphism
$$
    H_n(K^S, \SigmaInd_P^G (U, V))' \simeq H^n(K^S, \SigmaInd_P^G (U, V)')
$$
for all $n \in \Z$. Moreover, there is a first-quadrant $\Tbb^S(K^S)$-equivariant (cohomological) spectral sequence
$$
	E_2^{i, j} = \Ext_{\Hcal_{S,E}^-}^i( E(1), H^j(K^S J, \Dcal_{\sigma, U}^\Sigma)[\m_1^\infty]) \implies H^{i+j}(K^S, \SigmaInd_P^G (U, V)').
$$
\end{proposition}
\begin{proof}
For each $n$, $C_n(K^S, \SigmaInd_P^G (U, V))$ is naturally a locally convex vector space of compact type, as it is a direct sum of finitely many copies of $\SigmaInd_P^G (U, V)$, which is itself of compact type. Write $Z_n \subseteq C_n(K^S, \SigmaInd_P^G (U, V))$ for the subspace of $n$-cycles, and choose a finite-dimensional subspace $W_n$ of $Z_n$ splitting the projection map $Z_n \to H_n(K^S, \SigmaInd_P^G (U, V))$. The subspace $W_n$ is then Banach, as it is finite-dimensional and Hausdorff, and the $E$-linear map
$$
    C_{n+1}(K^S, \SigmaInd_P^G (U, V)) / Z_{n+1} \oplus W_n \to Z_n
$$
is a continuous and bijective map of compact type spaces, and hence an isomorphism by the open mapping theorem \cite[Theorem 1.1.17]{emerton_locally_analytic_vectors}.

In particular, $W_n$ is topologically isomorphic to $H_n(K^S, \SigmaInd_P^G (U, V))$, which proves the first statement.
\textit{A posteriori}, all the differentials in the complex
$C_\bullet(K^S, \SigmaInd_P^G (U, V))$ are strict,
so taking continuous duals is exact.
Since the continuous dual of this complex is precisely $C^\bullet(K^S, \SigmaInd_P^G (U, V)')$ by a continuous analogue of \Cref{eqn: hecke operators equivariant for dual}, taking cohomology gives the second statement.

By \Cref{proposition: can take generalised eigenspace}, there is a spectral sequence
$$
	E^2_{i, j} = \Tor^{\Hcal_{S,E}^-}_i( E(1), H_j(K^S J, \Acal_{\sigma, U}^\Sigma)[\m_1^\infty]) \implies H_{i+j}(K^S, \SigmaInd_P^G (V)).
$$
Dualising and using the same arguments as before to restrict to generalised eigenspaces, it's enough show that the dual of $\Tor^{\Hcal_{S,E}^-}_i( E(1), H_j(K^S J, \Acal_{\sigma, U}^\Sigma)[\m_1^\infty])$ is precisely \linebreak $\Ext_{\Hcal_{S,E}^-}^i( E(1), H^j(K^S J, \Dcal_{\sigma, U}^\Sigma)[\m_1^\infty])$.
This is straightforward and can be seen by computing the Tor and Ext groups with a finite resolution of $E(1)$ and noting that the same argument as in the previous paragraph shows that the dual of $H_j(K^S J, \Acal_{\sigma, U}^\Sigma)[\m_1^\infty]$ is isomorphic to $H^j(K^S J, \Dcal_{\sigma, U}^\Sigma)[\m_1^\infty]$ (and using again that passing to duals is exact).
Note that the generalised eigenspace given by the $\m_1^\infty$-torsion of $H_j(K^S J, \Acal_{\sigma, U}^\Sigma)$ is a direct summand of the latter, and hence its dual is the $\m_1^\infty$-torsion of the dual space $H^j(K^S J, \Dcal_{\sigma, U}^\Sigma)$.
\end{proof}

\begin{corollary}\label{corollary: cohomology is finite dimensional}
\begin{enumerate}
    \fixitem The vector space $H^*(K^S, \SigmaInd_P^G (U, V)')$ is finite-dimensional.
    \item $H^*(K^S, \SigmaInd_P^G (U, V)')$ vanishes unless $1$ is an eigenvalue of $U_\mu$ in $H^*(K^S J, \Dcal_{\sigma, U}^\Sigma)$ for all $\mu \in \HM^-$.
    \item Assume that for some choice of norm on the Banach space $U \otimes_E \sigma$, some element $\mu$ of $\HM^+$ acts as an operator of norm $< 1$. Then $H^*(K^S, \SigmaInd_P^G (U, V)') = 0$.
\end{enumerate}
\end{corollary}

Recall that in \Cref{subsection: induction 2} we defined subspaces $\Acal^{(s, s')-\Sigmaan}_\sigma$ of $\Ind_{J \cap P}^J (\sigma)^\Sigmala$ for $s \geq s_\sigma$ and $s' \geq s'_\sigma$. Set $\Acal^{(s, s')-\Sigmaan}_{\sigma, U} = U \otimes_R \Acal^{(s, s')-\Sigmaan}_\sigma$. Recall also from \Cref{corollary: Hecke algebras for J and J0s are isomorphic} that there is an $E$-algebra isomorphism $\Hcal_{S,E}^- \simeq \Hcal(J_{0,s} \HM^- J_{0,s}, J_{0,s})_E$ for any $s \geq 0$.

\begin{proposition}
\label{proposition: homology algebraic}
There is a natural $\Tbb^S(K^S) \otimes \Hcal_{S,E}^-$-equivariant isomorphism
$$
    H^*(K^S J, \Dcal_{\sigma, U}^\Sigma)[\m_1^\infty]
    \simeq
    H^*( K^S J_{0,s}, \Dcal_{\sigma_\Sigma, U}^{(s, 0)-\Sigmaan} \otimes_E (\sigma^\Sigma)' )[\m_1^\infty].
$$
Moreover,
$$
    H^*( K^S J_{0,s}, \Dcal_{\sigma_\Sigma, U}^{(s, 0)-\Sigmaan} \otimes_E (\sigma^\Sigma)' )
    \simeq 
    \left( H^*(K^S J_{1,s}, \Dcal_{\sigma_\Sigma, U}^{(s, 0)-\Sigmaan}) \otimes_E (\sigma^\Sigma)' \right)^{J_{0,s} / J_{1,s}}.
$$
\end{proposition}
\begin{proof}
Fix $a \in A_\Sigma^{--}$; all slope decompositions will be with respect to $U_a$.
We may assume that $\Dcal_{\sigma, U}^\Sigma = \Dcal_{\sigma, U}^{s-\Sigmaan}$.
The argument in the proof of \cite[Lemma 4.3.6]{urban_eigenvarieties} shows that
$$
    H^*(K^S J, \Dcal_{\sigma, U}^{s-\Sigmaan})_{\leq 0} \simeq H^*(K^S J_{0,s}, \Dcal_{\sigma, U}^{s-\Sigmaan})_{\leq 0}.
$$
This argument uses the fact that, if $\mu \in \HM^-$ satisfies $\mu^{-1} N_s \mu \subseteq N_{\tilde s}$ for some $\tilde s < s$ (for example, if $\mu = a$), then by \Cref{lemma: Iwahori decomposition for deep level} one has
$$
    J_{0, \tilde s} \cap \mu^{-1} J_{0, s} \mu = \bar N_0 \JM \mu^{-1} N_{s} \mu = J_{0, s} \cap \mu^{-1} J_{0, s} \mu.
$$
\Cref{lemma: isomorphism over J0s in locally algebraic case} asserts that if $s$ is large enough so that $\JMs$ acts trivially on $\sigma^\Sigma$ then there is an isomorphism of $J_{0,s} \HM^- J_{0,s}$-modules
$$
    \Acal_{\sigma, U}^{s-\Sigmaan}
    \simeq
    \Acal_{\sigma_\Sigma, U}^{s-\Sigmaan} \otimes_E \sigma^\Sigma.
$$
By \cite[Proposition 20.13]{nonarchimedean_functional_analysis}, this implies that $\Dcal_{\sigma, U}^{s-\Sigmaan} \simeq \Dcal_{\sigma_\Sigma, U}^{s-\Sigmaan} \otimes_E (\sigma^\Sigma)'$.
Using \cite[Lemma 2.3.13]{urban_eigenvarieties} we then see that
$$
    H^*(K^S J_{0,s}, \Dcal_{\sigma, U}^{s-\Sigmaan})_{\leq 0}
    \simeq
    H^*(K^S J_{0,s}, \Dcal_{\sigma_\Sigma, U}^{(s, s')-\Sigmaan} \otimes_E (\sigma^\Sigma)')_{\leq 0}
$$
for any $s' \leq s$.
In particular, both spaces are isomorphic to
$$
    H^j(K^S J_{0,s}, \Dcal_{\sigma_\Sigma, U}^{(s, 0)-\Sigmaan} \otimes_E (\sigma^\Sigma)')_{\leq 0},
$$ so taking generalised eigenspaces proves the first isomorphism.
Recall that $J_{0,s}$ acts on $\sigma^\Sigma$ via the surjective homomorphism $J_{0,s} \to \JM / \JMs$ described above \Cref{corollary: projection from J_0 to J_M monoid}, whose kernel is $J_{1,s}$. In particular, the restriction of $\sigma^\Sigma$ (and its dual) to $J_{1,s}$ is trivial, so
\begin{align*}
    & H^*( K^S J_{0,s}, \Dcal_{\sigma_\Sigma, U}^{(s, 0)-\Sigmaan} \otimes_E (\sigma^\Sigma)' ) \\
    & \simeq
    H^*(K^S J_{1,s}, \Dcal_{\sigma_\Sigma, U}^{(s, 0)-\Sigmaan} \otimes_E (\sigma^\Sigma)')^{J_{0,s} / J_{1,s}} \\
    & \simeq
    \left( H^*(K^S J_{1,s}, \Dcal_{\sigma_\Sigma, U}^{(s, 0)-\Sigmaan}) \otimes_E (\sigma^\Sigma)' \right)^{J_{0,s} / J_{1,s}}. \qedhere
\end{align*}
\end{proof}

\subsection{Relation to automorphic representations}
\label{subsection: automorphic representations}

We continue to assume that $R = E$ is a finite field extension of $\Q_p$ in $\bar \Q_p$.
We assume moreover that it is sufficiently large, in the sense that it contains all the embeddings of the splitting fields $L'_v$ of \Cref{subsection: global setting}.
Recall that we have fixed an isomorphism $\iota \colon \bar \Q_p \simto \C$.
We will use $\iota$ to view $\C$ as an extension of $E$.
Assume also that $\C \otimes_{E, \iota} \cInd_{\HM}^M(\sigma)$ is a supercuspidal representation of $M$ and that $U$ is an irreducible algebraic representation of $G_{\Scalp}$.
More precisely, assume we are in the following situation.
Recall that in \Cref{subsection: global setting} we defined $\Tbf_{\Scalp,E} \subseteq \Bbf_{\Scalp,E} \subseteq \Gbf_{\Scalp,E}$.
Choose a dominant weight $\lambda \in X^*(\Tbf_{\Scalp,E})$.
This is equivalent to choosing for each $v$ and embedding $\tau \colon F_v \into E$ a dominant weight $\lambda_{v,\tau} \in X^*(E \times_{F_v, \tau} \Tbf_v)$.
Let $V_{\Gbf_{\Scalp,E}}(\lambda)$ be the irreducible representation of $\Gbf_{\Scalp,E}$ over $E$ of highest weight $\lambda$ and $V_{G_{\Scalp}}(\lambda)$ its restriction to $G_{\Scalp} \subseteq \Gbf_{\Scalp, E}(E)$.
We assume that $U = V_{G_{\Scalp}}(\lambda)'$.
Finally, we will assume also that $\Sigma = \emptyset$ and $\sigma_\Sigma$ is isomorphic to $E$.

Given a vector space $\Mcal$ over $E$ (or $\Q_p$), we will write $\Mcal_{\C} = \C \otimes_{E, \iota} \Mcal$ for its extension of scalars to $\C$ via $\iota$.
The representation $\SigmaInd_P^G (U, V)_\C$ is thus isomorphic to the locally algebraic representation $V_{G_{\Scalp}}(\lambda)'_\C \otimes_\C \Ind_P^G( V_\C )^\sm$ of $G_{\Scalp} \times G$ over $\C$, and $\SigmaInd_P^G (U, V)'_\C$ is isomorphic to its abstract $\C$-dual (since, before base changing to $\C$, $\SigmaInd_P^G (U, V)$ was equipped with its finest locally convex topology).
The weight $\lambda$ gives rise to a weight $\lambda_\C$ of $\C \times_{E, \iota} \Tbf_{\Scalp, E}$, and the irreducible representation of $\C \times_{E, \iota} \Gbf_{\Scalp, E} \simeq \C \times_\Q \Res_{F/\Q} \Gbf$ of highest weight $\lambda_\C$ is isomorphic to the base change of $V_{\Gbf_{\Scalp, E}}(\lambda)$.
We will use this to view the latter as a representation of $G_\infty$.

Recall (cf. \cite{borel_wallach}, \cite{franke_l2}, \cite{li_schwermer_eisenstein_cohomology}) that for any open compact subgroup $K_S \subseteq G_S$, the arithmetic cohomology $H^*(K^S K_S, V_{G_{\Scalp}}(\lambda))_\C \simeq H^*(K^S K_S, V_{G_{\Scalp}}(\lambda)_\C)$ admits a canonical $\Tbb^S(K^S)_\C \otimes_\C \Hcal(G_S, K_S)_\C$-stable decomposition
\begin{align}\label{eqn: decomposition of cohomology into cuspidal and Eisenstein}
    H^*(K^S K_S, V_{G_{\Scalp}}(\lambda))_\C
    & \simeq H^*_\cusp(K^S K_S, V_{G_{\Scalp}}(\lambda))_\C \oplus H^*_\Eis(K^S K_S, V_{G_{\Scalp}}(\lambda))_\C,
\end{align}
into so-called \emph{cuspidal} and \emph{Eisenstein} summands, and moreover cuspidal cohomology admits a decomposition
\begin{align}\label{eqn: decomposition of cuspidal cohomology}
\begin{split}
    & H^*_{\cusp}(K^S K_S, V_{G_{\Scalp}}(\lambda))_\C \\
    & \simeq \bigoplus_\pi \C^{m(\pi)} \otimes_\C (\pi^\infty)^{K^S K_S} \otimes_\C H^*(\tilde \g_\infty, K_\infty; \pi_\infty \otimes_\C V_{\Gbf_{\Scalp,E}}(\lambda)_\C),
\end{split}
\end{align}
where the sum ranges over all cuspidal automorphic representations $\pi$ of $\Gbf(\A_F)$, $m(\pi)$ denotes the multiplicity of $\pi$ as a subrepresentation of $L^2_\cusp(\Gbf(F) \backslash \Gbf(\A_F))$, $\pi^\infty$ (resp. $\pi_\infty$) is the finite (resp. infinite) component of $\pi$ and $\tilde \g_\infty$ is the Lie $\R$-algebra of the intersection of all the kernels of the $\Q$-rational characters of $\Res_{F/\Q} \Gbf$.
If $\pi$ to contributes to $H^*(K^S K_S, V_{G_{\Scalp}}(\lambda))_\C$ (i.e. if the direct summand corresponding to $\pi$ is non-zero) we will say that $\pi$ is \emph{cohomological of level $K^S K_S$ and weight $\lambda_\C$}.
For such $\pi$, $\pi^\infty$ has non-zero $K^S K_S$-invariant vectors, and $\pi_\infty$ has the same central character $\omega_\lambda$ and infinitesimal character as $V_{\Gbf_{\Scalp, E}}(\lambda)'_\C$.

\Cref{proposition: homology algebraic} and the paragraph before it imply that for sufficiently large $s$ there is an isomorphism
$$
    H^*(K^S J, \Dcal_{\sigma, U}^\Sigma)[\m_1^\infty]
    \simeq
    \left(
    \left( H^*(K^S J_{1,s}, V_{G_{\Scalp}}(\lambda)) \otimes_E \sigma' \right)^{J_{0,s} / J_{1,s}}
    \right)[\m_1^\infty].
$$
The decomposition of $H^*(K^S J_{1,s}, V_{G_{\Scalp}}(\lambda))_\C$ into cuspidal and Eisenstein summands is compatible with the action of $J_{0,s} / J_{1,s}$.
Applying $\iota$, we may thus define a decomposition
$$
    H^*(K^S J, \Dcal_{\sigma, U}^\Sigma)_{\C}[\m_{1, \C}^\infty] \simeq H^*_\cusp(K^S J, \Dcal_{\sigma, U}^\Sigma)_{\C}[\m_{1, \C}^\infty] \oplus H^*_\Eis(K^S J, \Dcal_{\sigma, U}^\Sigma)_{\C}[\m_{1, \C}^\infty].
$$
and similarly we may write $H^*_\cusp(K^S J, \Dcal_{\sigma, U}^\Sigma)_{\C}[\m_{1, \C}^\infty]$ as a direct sum indexed by cuspidal automorphic representations.

\begin{proposition}
    When $K_S = J_{1,s}$, the spectral sequence in \Cref{proposition: cohomology is dual of homology} splits as a direct sum of two analogous spectral sequences involving the cuspidal and Eisenstein parts of arithmetic cohomology.
    The cuspidal direct summand spectral sequence degenerates at the $E_2$ page.
\end{proposition}
\begin{proof}
    In order to prove that the spectral sequence splits as such a direct sum, by its construction it is enough to observe that the decomposition \Cref{eqn: decomposition of cohomology into cuspidal and Eisenstein} is obtained by \linebreak taking cohomology of an analogous decomposition in the derived category of \linebreak $\Tbb^S(K^S)_\C \otimes_\C \Hcal(G_S, K_S)_\C$-modules,
    \begin{align*}
        C^\bullet(K^S K_S, V_{G_{\Scalp}}(\lambda))_\C
        & \simeq C^\bullet_\cusp(K^S K_S, V_{G_{\Scalp}}(\lambda))_\C \oplus C^\bullet_\Eis(K^S K_S, V_{G_{\Scalp}}(\lambda))_\C.
    \end{align*}
    This follows from the work of Franke \cite{franke_l2}, more concretely from \cite[Theorem in \S2.2 and Theorem 4]{franke_l2} and their proofs.

    Let us prove that the cuspidal part degenerates. The decomposition \Cref{eqn: decomposition of cuspidal cohomology} is also obtained from taking cohomology of an isomorphism in the derived category of \linebreak $\Tbb^S(K^S)_\C \otimes_\C \Hcal(G_S, K_S)_\C$-modules, namely
    \begin{align*}
       & C^\bullet_{\cusp}(K^S K_S, V_{G_{\Scalp}}(\lambda))_\C \\
    & \simeq \Hom_{(\tilde \g_\infty, K_\infty)} (\bigwedge^\bullet \left( \tilde \g_\infty / \mathfrak k \right), L^2_\cusp( \Gbf(F) \backslash \Gbf(\A_F), \omega_\lambda)^{K^S K_S} \otimes_\C V_{\Gbf_{\Scalp,E}}(\lambda)_\C), \\
    & \simeq \bigoplus_\pi \C^{m(\pi)} \otimes_\C (\pi^\infty)^{K^S K_S} \otimes_\C \Hom_{(\tilde \g_\infty, K_\infty)} (\bigwedge^\bullet \left( \tilde \g_\infty / \mathfrak k \right), \pi_\infty \otimes_\C V_{\Gbf_{\Scalp,E}}(\lambda)_\C).
    \end{align*}
    The latter complex is, as a complex in the derived category of $\Hcal(G_S, K_S)_\C$-modules, isomorphic to the direct sum of its shifted cohomologies.
    In particular, this induces a decomposition of the cuspidal part of the spectral sequence from \Cref{proposition: cohomology is dual of homology} as a direct sum of spectral sequences that collapse at the $E_2$ page.
\end{proof}

This proposition has a few important consequences for $S$-arithmetic cohomology. For example, we can use it to define a decomposition
$$
    H^*(K^S, \SigmaInd_P^G (U, V)')_{\C} \simeq H^*_\cusp(K^S, \SigmaInd_P^G (U, V)')_{\C} \oplus H^*_\Eis(K^S, \SigmaInd_P^G (U, V)')_{\C},
$$
each summand being the abutment of the corresponding spectral sequence.
Moreover, the proof of the proposition allows us to give a very explicit description of the cuspidal part.
Given a cuspidal automorphic representation $\pi$, we will write $\pi = \bigotimes'_v \pi_v$, $\pi_S = \bigotimes_{v \in S} \pi_v$ and $\pi^{S \infty} = \bigotimes'_{v \not\in S, v \nmid \infty} \pi_v$.
There is thus a spectral sequence converging to $H^*_\cusp(K^S, \SigmaInd_P^G (U, V)')_{\C}$ with terms
\begin{align*}
    E_2^{i,j} = \bigoplus_\pi & \C^{m(\pi)} \otimes_{\C} (\pi^{S\infty})^{K^S} \otimes_{\C} H^j(\tilde \g_\infty, K_\infty; \pi_\infty \otimes_\C V_{G_{\Scalp}}(\lambda)_\C) \\
    & \otimes_{\C} \Ext_{\Hcal_{S, \C}^-}^i \left( {\C}(1),  ( \pi_S \otimes_{\C} \sigma_\C')^{J_{0,s} / J_{1, s}}  \right),
\end{align*}
where the sum ranges over all cuspidal $\pi$ appearing in \Cref{eqn: decomposition of cuspidal cohomology} for $K_S = J_{1,s}$, which not only degenerates at the $E_2$ page, but is in fact a direct sum of spectral sequences that \emph{collapse} at the $E^2$ page.

In particular, we can decompose $H^*_\cusp(K^S, \SigmaInd_P^G (U, V)')_{\C}$ as a direct sum indexed by cuspidal automorphic representations $\pi$.
If $\pi$ contributes to this direct sum, then the (honest) eigenspace for 1 and the action of $\Hcal_{S, \C}$ in $( \pi_S \otimes_{\C} \sigma_\C')^{J_{0,s} / J_{1, s}}$ is non-zero.
This space is isomorphic to $\Hom_{J_{0,s}}(\sigma_\C, \pi_S) = \Hom_{J_{0,s}}(\sigma_\C, \pi_S^{J_{1,s}})$.
Because the action of $J_{1,s}$ on $\sigma$ is trivial and a set of representatives for $J_{1,s} \mu J_{1,s} / J_{1,s}$ will also be a set of representatives for $J_{0,s} \mu J_{0,s} / J_{0,s}$ for any $\mu \in \HM^-$, any eigenfunction for the operators $U_\mu$ with $\mu \in \HM^-$ in $\Hom_{J_{0,s}}(\sigma_\C, \pi_S)$ must take values in $[J_{1,s} a J_{1,s}] \pi_S^{J_{1,s}}$ for all $a \in A^-$ (since the image of $U_a f$ is contained in this subspace).
Recall from \cite[Proposition 4.1.4]{casselman_padic_book} that there is a natural isomorphism of vector spaces over $\C$
$$
    [J_{1,s} a J_{1,s}] \pi_S^{J_{1,s}}
    \simto
    (\pi_S)_{\bar N}^{\JMs}
$$
where $( \pi_S )_{\bar N}$ denotes the (non-normalised) Jacquet module of $\pi_S$ with respect to $\bar N$,
and that for $m \in M^-$, the Hecke operator $[J_{1,s} m J_{1,s}]$ on the left-hand side is identified with the action of $m$ twisted by the modulus character $\delta_P$ of $P$.
It follows that non-zero eigenfunctions in $\Hom_{J_{0,s}}(\sigma_\C, \pi_S)$ of eigenvalue 1 for the operators $U_\mu$ are in bijection with non-zero element of $\Hom_{\HM^-}( \sigma_\C, (\pi_S)_{\bar N} \otimes \delta_P)$.
Consequently,
\begin{align*}
    0 & \neq \Hom_{\HM^-}( \sigma_\C, (\pi_S)_{\bar N} \otimes \delta_P) \\
    & = \Hom_{\HM}( \sigma_\C, (\pi_S)_{\bar N} \otimes \delta_P) \\
    & = \Hom_{M}( V_\C \otimes \delta_P^{-1}, (\pi_S)_{\bar N}) \\
    & = \Hom_{M}( (\tilde{\pi_S})_{N}, \tilde V_\C \otimes \delta_P),
\end{align*}
where we use $\tilde{(\blank)}$ to denote smooth contragredients.
For the last equality we have used the fact that $\tilde {(\pi_S)_{\bar N}} \simeq (\tilde{\pi_S})_{N}$ by \cite[Corollary 4.2.5]{casselman_padic_book}.
Putting everything together, we have proved the following theorem, which is a more precise version of \Cref{theorem: introduction 1}.

\begin{theorem}\label{theorem: properties of classical cuspidals in cohomology}
There is a $\Tbb^S(K^S)_\C$-equivariant decomposition
$$
    H^*_\cusp(K^S, V_{G_{\Scalp}}(\lambda)_\C \otimes_\C (\Ind_P^G(V_\C)^\sm)')_{\C} \simeq \bigoplus_\pi \C^{m(\pi)} \otimes_{\C} (\pi^{S\infty})^{K^S} \otimes_{\C} H^*_\pi
$$
where $\pi$ ranges over (cohomological) cuspidal automorphic representations $\pi$ of $\Gbf(\A_F)$ of level $K^S$ away from $S$, weight $\lambda_\C$, and such that $\tilde V_\C \otimes \delta_P$ is a quotient of the Jacquet module with respect to $P$ of $\tilde{\pi_S}$ (or equivalently, $V_\C \otimes \delta_{\bar P}$ is a subrepresentation of the Jacquet module of $\pi_S$ with respect to $\bar P$), where for all $n$,
$$
    H^n_\pi
    \simeq
    \bigoplus_{i=1}^n
    H^{n-i}(\tilde \g_\infty, K_\infty; \pi_\infty \otimes_\C V_{G_{\Scalp}}(\lambda)_\C)
    \otimes_{\C}
    \Ext_{\Hcal_{S, \C}^-}^i \left( {\C}(1),  ( \pi_S \otimes_{\C} \sigma_\C')^{J_{0,s}}  \right).
$$
\end{theorem}

\subsection{Orlik--Strauch representations and classicality}
\label{subsection: orlik strauch homology}

We continue to assume that $E$ is as in last section, but impose no algebraicity or supercuspidality conditions on $U, \sigma$.
We will assume for simplicity that $\sigma$ is smooth as this case is sufficient for our purposes.
Let $W_{S_p, E}$ denote the Weyl group of $\Gbf_{S_p,E}$ with respect to $\Tbf_{S_p, E}$, and define $\Delta_{S_p, E}$ and $\Delta_{S_p,M, E}$ similarly. We make analogous definitions for subsets of $S_p$.
Recall also that the \emph{dot-action} of $W_{S_p, E}$ on the set of weights $X^*(\Tbf_{S_p, E})$ is defined by $w \cdot \lambda = w ( \lambda + \rho) - \rho$ where $\rho$ is the half-sum of positive roots of $\Gbf_{S_p, E}$. In particular, this action depends on our choice of Borel subgroup $\Bbf_{S_p, E}$.

Let $\g, \m$ and $\p$ denote the Lie $\Q_p$-algebras of $G_{S_p}, M_{S_p}$ and $P_{S_p}$ respectively and $U(\g), U(\m)$ and $U(\p)$ their universal enveloping algebra. Let $\g_E = E \otimes_\Qp \g$ be the Lie $E$-algebra of $\Gbf_{S_p,E}$, and similarly for $\m_E, \p_E, U(\g_E), U(\p_E)$ and $U(\m_E)$.
Consider the categories $\O^{\p_E}$ and $\O^{\p_E}_\alg$ of $U(\g_E)$-modules defined in \cite{orlik_strauch}, and write $\text{Rep}^\text{sm,adm}(M_{S_p})$ (resp. $\text{Rep}^{\la}(G_{S_p})$) for the category of smooth admissible representations of $M_{S_p}$ (resp. locally $S_p$-analytic representations of $G$) over $E$. Following the strategy of Orlik and Strauch \cite{orlik_strauch} we may define a functor
$$
    \Fcal_{P_{S_p}}^{G_{S_p}} \colon (\O^{\p_E}_\alg)^\text{op} \times \text{Rep}^\text{sm,adm}(M_{S_p}) \to \text{Rep}^{\la}(G_{S_p})
$$
which is exact in both arguments.
We refer to the appendix for its construction and the proof of these properties when the group is non-split.

\begin{example}\label{example: verma modules}
Let $\lambda \in X^*(\Tbf_{S_p,E})$ be a dominant weight for $\Mbf_{S_p,E}$ (with respect to $\Mbf_{S_p,E} \cap \Bbf_{S_p,E}$).
Let $V_{M_{S_p}}(\lambda)$ be the restriction to $M_{S_p} = \Mbf_{S_p}(\Q_p) \subseteq \Mbf_{S_p,E}(E)$ of the irreducible algebraic representation of $\Mbf_{S_p,E}$ of highest weight $\lambda$.
We may view $V_{M_{S_p}}(\lambda)$ as a representation of $\m_E$, and of $\p_E$ by inflation.
Let $\Mcal_{P_{S_p}}^{G_{S_p}}(\lambda) := U(\g_E) \otimes_{U(\p_E)} V_{M_{S_p}}(\lambda) \in \O^{\p_E}_\alg$ be the parabolic Verma module of highest weight $\lambda$.
Then,
$$
    \Fcal_{P_{S_p}}^{G_{S_p}} (\Mcal_{P_{S_p}}^{G_{S_p}}(\lambda), V_{S_p}) \simeq \Ind_{P_{S_p}}^{G_{S_p}}( V_{M_{S_p}}(\lambda)' \otimes_E V_{S_p} )^{\la},
$$
(recall that $V_{S_p}$ is defined in \Cref{subsection: global setting}) and if $\lambda$ is dominant for $\Gbf_{S_p,E}$, then
$$
    \Fcal_{P_{S_p}}^{G_{S_p}} (V_{G_{S_p}}(\lambda), V_{S_p}) \simeq V_{G_{S_p}}(\lambda)' \otimes_E \Ind_{P_{S_p}}^{G_{S_p}} ( V_{S_p} )^{\sm}.
$$
The natural projection $\Mcal_{P_{S_p}}^{G_{S_p}} (\lambda) \onto V_{G_{S_p}}(\lambda)$ induces the natural inclusion
$$
    V_{G_{S_p}}(\lambda)' \otimes_E \Ind_{P_{S_p}}^{G_{S_p}} ( V_{S_p} )^{\sm} \into \Ind_{P_{S_p}}^{G_{S_p}} ( V_{M_{S_p}}(\lambda)' \otimes_E V_{S_p} )^{\la}
$$
of locally algebraic induction into locally analytic induction. More generally, for any $\Sigma \subseteq S_p$ we may write $\lambda = (\lambda_\Sigma, \lambda^\Sigma) \in X^*(\Tbf_{\Sigma,E}) \times X^*(\Tbf_{S_p \setminus \Sigma,E})$
and
\begin{align*}
    \Fcal_{P_{S_p}}^{G_{S_p}} (\Mcal_{P_\Sigma \times G_{S_p \setminus \Sigma}}^{G_{S_p}}(\lambda), V_{S_p})
    &
    \simeq
    V_{G_{S_p \setminus \Sigma}}(\lambda^\Sigma)' \otimes_E \Ind_{P_{S_p}}^{G_{S_p}} ( V_{M_{\Sigma}}(\lambda_\Sigma)' \otimes_E V_{S_p} )^{\Sigma-\la} 
    \\
    &
    = \SigmaInd_{P_{S_p}}^{G_{S_p}} (V_{G_{S_p \setminus \Sigma}}(\lambda^\Sigma)', V_{M_{\Sigma}}(\lambda_\Sigma)' \otimes_E V_{S_p}).
\end{align*}
\end{example}

Let us remark that, for any locally convex space $W$ of compact type, the bifunctor $\Fcal_{P_{S_p}}^{G_{S_p}} ( \blank, \blank ) \cotimes_E W$ is also exact in both arguments and in $W$ (by \cite[Proposition 1.1.26]{emerton_locally_analytic_vectors}, reflexivity, and the exactness of passing to duals).
We will abuse notation and write
$$
    \Fcal_P^G (X, V) := \Fcal_{P_{S_p}}^{G_{S_p}} (X, V_{S_p}) \cotimes_E \Ind_{P_{S \setminus S_p}}^{G_{S \setminus S_p}} (V^{S_p})^\sm.
$$
This is still exact in both arguments.

\begin{proposition}\label{proposition: finite dimensionality for Orlik Strauch}
Let $X \in \O^{\p_E}_\alg$. Then, the $S$-arithmetic homology space
$$
    H_*(K^S, U \otimes_E \Fcal_P^G (X, V) )
$$
is finite-dimensional, and similarly for the cohomology of the dual representation.
\end{proposition}
\begin{proof}
The Verma modules $\Mcal_{P_{S_p}}^{G_{S_p}}(\lambda)$ form a basis of the Grothendieck group of the (Artinian) category $\O^{\p_E}_\alg$ as $\lambda$ runs through the set of dominant weights for $\Mbf_{S_p,E}$. Thus, the image of the representation $X$ in the Grothendieck group of $\O^{\p_E}_\alg$ may be written as a finite sum
$$
    [X] = \sum_{i=1}^r [\Mcal_{P_{S_p}}^{G_{S_p}}(\lambda_i)] - \sum_{i=r+1}^s [\Mcal_{P_{S_p}}^{G_{S_p}}(\lambda_i)].
$$
Therefore, the proposition follows by induction on $s$ by the exactness of the functor $U \otimes_E \Fcal_P^G(\blank, V)$ and passing to continuous duals, \Cref{corollary: homology is finite dimensional}, \Cref{corollary: cohomology is finite dimensional} and the long exact sequences in (co)homology associated to a short exact sequence.
\end{proof}

Our next goal is to prove \Cref{theorem: introduction 2} from the introduction.
For this, we want to relate the $S$-arithmetic homology of the various representations $\SigmaInd_P^G(U, V)$ as we vary $\Sigma$ (and also $U$ and $V$ accordingly).
Let $\tilde \Sigma \subseteq \Sigma$ be a subset.
In our application to \Cref{theorem: introduction 2} we will only need the case where $\tilde \Sigma$ is the empty set, so the reader is welcome to assume this.
Write $W_{\Sigma, M, E}$ (resp. $\Delta_{\Sigma, M, E}$) for the Weyl group (resp. root basis) of $\Mbf_{\Sigma, E}$ and similarly for other subsets of $S_p$ and Levi subgroups of parabolics.
When the Levi subgroup is the group $G$ itself, we will omit it from the notation.
Define \linebreak ${}^{M} W_{\Sigma \setminus \tilde \Sigma, E} := {}^{M_\Sigma} W_{\Sigma, M_{\tilde \Sigma} \times G_{\Sigma \setminus \tilde \Sigma}, E}$ as in \Cref{subsection: lemma on roots}, i.e. as the set of minimal length representatives of cosets in $W_{\Sigma, M_{\Sigma}, E} \backslash W_{\Sigma, M_{\tilde \Sigma} \times G_{\Sigma \setminus \tilde \Sigma}, E}$.
Note that this coset set is naturally in bijection with $W_{\Sigma \setminus \tilde \Sigma, M, E} \backslash W_{\Sigma \setminus \tilde \Sigma, E}$, which will hopefully justify the notation.
Similarly, $\Delta_{\Sigma, M_{\tilde \Sigma} \times G_{\Sigma \setminus \tilde \Sigma}, E} \setminus \Delta_{\Sigma, M_{\Sigma}, E} \simeq \Delta_{\Sigma \setminus \tilde \Sigma, E} \setminus \Delta_{\Sigma \setminus \tilde \Sigma, M, E}$.

\begin{definition}\label{definition: small slope}
Assume that $A_\Sigma$ acts on $U \otimes_E V$ via a character $\chi$.
Let $\lambda$ be a dominant regular weight of $\Mbf_{\tilde \Sigma} \times \Gbf_{\Sigma \setminus \tilde \Sigma}$.
We say that $A_\Sigma$ acts numerically non-critically (with respect to $\tilde \Sigma \subseteq \Sigma$ and $\Pbf$) on the representation $U \otimes_E V_{M_\Sigma}
(\lambda)' \otimes_E V$ if for any simple reflection $w$ corresponding to an element of $\Delta_{\Sigma \setminus \tilde \Sigma, E} \setminus \Delta_{\Sigma \setminus \tilde \Sigma, M, E}$, there exists $a \in A_\Sigma^-$ such that
\begin{align}\label{eqn: inequality small slope}
    v_p(\chi(a)) < v_p( (w \cdot \lambda)(a) ).
\end{align}

\end{definition}

\begin{remark}\label{remark: small slope}
With all notations as in \Cref{definition: small slope}, the inequality \Cref{eqn: inequality small slope} holds for all $w \in {}^{M} W_{\Sigma \setminus \tilde \Sigma, E} \setminus \{ 1 \}$ (for some $a$ depending of $w$) by the proof of \Cref{lemma: on roots}.
Indeed, given $w$, we may choose $\alpha \in \Delta_{\Sigma \setminus \tilde \Sigma, E} \setminus \Delta_{\Sigma \setminus \tilde \Sigma, M, E}$ as in that proof and $a \in A_\Sigma^-$ for which \Cref{eqn: inequality small slope} is valid for $s_\alpha$.
Then, it is valid also for $w$ (as the proof of \Cref{lemma: on roots} shows).
\end{remark}

\Cref{theorem: introduction 2} follows from applying the following more general theorem when $\tilde \Sigma$ is empty and $U$ is a finite-dimensional irreducible algebraic representation of $G_{\Scalp \setminus \Sigma}$.

\begin{theorem}\label{theorem: classicality}
Assume that $A_\Sigma$ acts on $U \otimes_E V$ via a character $\chi$.
Let $\lambda = (\lambda_{\tilde \Sigma}, \lambda_{\Sigma \setminus \tilde \Sigma})$ be a dominant regular weight of $\Mbf_{\tilde \Sigma} \times \Gbf_{\Sigma \setminus \tilde \Sigma}$ such that $A_\Sigma$ acts numerically non-critically on $U \otimes_E V_{M_\Sigma}
(\lambda)' \otimes_E V$.
Then the natural homomorphism from
$$
    H_*(K^S, \tilde \Sigma \text{-Ind}_P^G(U \otimes_E V_{G_{\Sigma \setminus \tilde \Sigma}}(\lambda_{\Sigma \setminus \tilde \Sigma})', V_{M_{\tilde \Sigma}}(\lambda_{\tilde \Sigma})' \otimes_E V) )
$$
to
$$
    H_*(K^S, \Sigma \text{-Ind}_P^G(U, V_{M_{\Sigma}}(\lambda_{\Sigma})' \otimes_E V) )
$$
is an isomorphism, and the same is true for the cohomology of the duals.
\end{theorem}
\begin{proof}
The statement for cohomology is analogous to the one for homology, so let us only prove the latter.
Note that the condition that $A_\Sigma$ acts on $U \otimes_E V$ via $\chi$ is equivalent to the condition that it acts on $U \otimes_E \sigma$ via $\chi$.
By \Cref{example: verma modules},
it's enough to show that the natural map
$$
    H_*(K^S, U \otimes_E \Fcal_P^G ( \Mcal_{P_{\tilde \Sigma} \times G_{S_p \setminus \tilde \Sigma}}^{G_{S_p}}(\lambda), V) )  \simto
    H_*(K^S, U \otimes_E \Fcal_P^G( \Mcal_{P_{\Sigma} \times G_{S_p \setminus \Sigma}}^{G_{S_p}}(\lambda), V) )
$$
is an isomorphism.

Consider the parabolic BGG resolution of the representation $V_{G_{\Sigma \setminus \tilde \Sigma}}(\lambda_{\Sigma \setminus \tilde \Sigma})$ of the Lie algebra of $\Gbf_{\Sigma \setminus \tilde \Sigma, E}$,
$$
    \bigoplus_{\substack{w \in {}^{M} W_{\Sigma \setminus \tilde \Sigma, E} \\ \ell(w) = \bullet}} \Mcal_{P_{\Sigma \setminus \tilde \Sigma}}^{G_{\Sigma \setminus \tilde \Sigma}} (w \cdot \lambda_{\Sigma \setminus \tilde \Sigma}) \to V_{M_{\Sigma \setminus \tilde \Sigma}}(\lambda_{\Sigma \setminus \tilde \Sigma}) \to 0.
$$
Tensoring with $\Mcal_{P_{\tilde \Sigma} \times G_{S_p \setminus \Sigma}}^{G_{S_p \setminus (\Sigma \setminus \tilde \Sigma)}} ( \lambda_{\tilde \Sigma} )$ (where we extend $\lambda_{\tilde \Sigma}$ to a weight of $\Mbf_{\tilde \Sigma} \times \Gbf_{S_p \setminus \Sigma}$ trivially), we obtain an exact sequence in $\O^{\p_E}_\alg$
\begin{align}\label{eqn: BGG resolution}
    \bigoplus_{\substack{w \in {}^{M} W_{\Sigma \setminus \tilde \Sigma, E} \\ \ell(w) = \bullet}} \Mcal_{P_{\Sigma} \times G_{S_p \setminus \Sigma}}^{G_{S_p}} (w \cdot \lambda) \to \Mcal_{P_{\tilde \Sigma} \times G_{S_p \setminus \tilde \Sigma}}^{G_{S_p}}(\lambda) \to 0.
\end{align}
Applying the exact functor $U \otimes_E \Fcal_P^G ( \blank, V)$ and taking $S$-arithmetic homology, we obtain a converging spectral sequence
\begin{align*}
    E^1_{i,j}
    & =
    \bigoplus_{\substack{w \in {}^{M} W_{\Sigma \setminus \tilde \Sigma, E} \\ \ell(w) = j}} H_i(K^S,
    U \otimes_E \Fcal_P^G( \Mcal_{P_{\Sigma} \times G_{S_p \setminus \Sigma}}^{G_{S_p}} (w \cdot \lambda) , V) )
    \\ &
    \implies
    H_{i+j}(K^S, U \otimes_E \Fcal_P^G( \Mcal_{P_{\tilde \Sigma} \times G_{S_p \setminus \tilde \Sigma}}^{G_{S_p}}(\lambda), V) ).
\end{align*}
Hence, it's enough to show that the representations
$$
    U \otimes_E \Fcal_P^G( \Mcal_{P_{\Sigma} \times G_{S_p \setminus \Sigma}}^{G_{S_p}} (w \cdot \lambda) , V)
    =
    \SigmaInd_P^G (U, V_{M_{\Sigma}}(w \cdot \lambda)' \otimes_E V)
$$
have vanishing $S$-arithmetic homology for $w \in {}^{M} W_{\Sigma \setminus \tilde \Sigma, E}$, $w \neq 1$.
Let $w \neq 1$ and let $a$ be as in \Cref{remark: small slope}.
By \Cref{corollary: homology is finite dimensional} \itemnumber{3}, it's enough to show that $a$ acts on $U \otimes_E V_{M_\Sigma}(w \cdot \lambda)' \otimes_E \sigma$ 
as an operator of norm strictly greater than one.
But $a$ acts by multiplication by 
$\chi(a) (w \cdot \lambda)(a)^{-1}$, and the numerical non-criticality condition (taking into account \Cref{remark: small slope}) states precisely that this has norm strictly greater than one.
\end{proof}

Let us remark that we could have proven this statement by working exclusively with overconvergent homology (where similar criterions are known to hold in many cases, cf. \cite[Section 4]{parabolic_overconvergent}) and using the spectral sequences relating it to $S$-arithmetic homology, but this would require us to work with analogues of the locally analytic BGG resolution (i.e. the exact sequence obtained by applying $U \otimes_E \Fcal_P^G( \blank, V)$ to \Cref{eqn: BGG resolution}) for representations of $J$ in an \textit{ad hoc} manner.

\subsection{The case of \texorpdfstring{$\GL_n$}{GLn}}
\label{subsection: GLn}

In this section, we will prove \Cref{theorem: introduction 1.5}.
We will expand on the results of the previous section and prove that, when $\Gbf_{F_v}$ is isomorphic to $\GL_{n / F_v}$ for all $v \in S$, all the representation in the image of the functor $\Fcal_P^G$ where the second argument has finite length have finite-dimensional $S$-arithmetic homology. We will deduce this result from the case of the smooth parabolic induction of a supercuspidal representation by using Bernstein--Zelevinsky theory (especially \cite{zelevinsky_GLn}).
As in the previous section, given 
$X \in \O^{\p_E}_\alg$, $\pi_{S_p} \in \text{Rep}^{\sm,\adm}(M_{S_p})$ and $\pi^{S_p} \in \text{Rep}^{\sm,\adm}(M^{S_p})$, we will abuse notation and write $\Fcal_P^G(X, \pi_{S_p} \boxtimes \pi^{S_p}) := \Fcal_{P_{S_p}}^{G_{S_p}}(X, \pi_{S_p}) \hat\boxtimes \Ind_{P_{S \setminus S_p}}^{G_{S \setminus S_p}} (\pi^{S_p})^\sm$.
We will now introduce some notation exclusive to this section.

Let $\ell$ be a prime, $L$ a finite extension of $\Q_\ell$.
If $\rho = \rho_1 \boxtimes \cdots \boxtimes \rho_r$ is a smooth representation of $\GL_{n_1} (L) \times \cdots \times \GL_{n_r}(L)$, viewed as the Levi subgroup of block-diagonal matrices in $\GL_n$ for $n = n_1 + \cdots + n_r$ of a standard parabolic $\mathbb P$ (containing the upper-triangular matrices) in the usual way, we will write $\rho_1 \times \cdots \times \rho_r$ to denote \emph{normalised} smooth parabolic induction of $\rho$ with respect to $P$, i.e.
$$
    i(\rho) = \Ind_{\mathbb P(L)}^{\GL_n(L)} ((\rho_1 \boxtimes \cdots \boxtimes \rho_r) \otimes \delta_{\mathbb P}^{1/2})^\sm
$$
where $\delta_{\mathbb P}$ is the modulus character of $\mathbb P$.
If $\rho$ is a supercuspidal representation of $\GL_n(L)$, we will write $\rho(s) = \rho | \det |^{s-1}$.
By \cite[Proposition 2.10]{zelevinsky_GLn}, the representation \linebreak
$\rho(1) \times \rho(2) \times \cdots \times \rho(s)$ of $\GL_{ns}(L)$ has a unique irreducible representation $Z_s(\rho)$.

\begin{lemma}\label{theorem: exact sequence Zelevinsky}
There is an exact sequence
\begin{align*}
    0 \to
    Z_s(\rho)
    \to
    Z_{s-1}(\rho) \times \rho(s)
    \to
    \rho(s) \times  Z_{s-1}(\rho)
    \to
    Z_{s}(\rho) \to 0.
\end{align*}
\end{lemma}
\begin{proof}
This follows from the results of \cite{zelevinsky_GLn}, as we now explain.
Let us recall some terminology and facts from \loccit~
A set of the form $\Delta = \{ \rho(1), ..., \rho(s) \}$ is called a \emph{segment}.
To such a segment $\Delta$ we can associate a graph $\Gamma$, which is a line segment with $s$ vertices corresponding to $\rho(1), ..., \rho(s)$ (in this order).
An orientation of $\Gamma$ is a choice of direction for each edge of $\Gamma$.
Any bijection $\lambda \colon \{1, ..., s\} \to \Delta$ (which may be seen as an ordering of $\Delta$) determines an ordering of $\Gamma$: the edge between $\rho(i)$ and $\rho(i+1)$ goes from $\rho(i)$ to $\rho(i+1)$ if $\lambda^{-1}(\rho(i)) < \lambda^{-1}(\rho(i+1))$.
Every orientation arises from such a $\lambda$.
Let 
$\pi(\lambda) = \lambda(1) \times \cdots \times \lambda(s)$.
The orientation of $\Gamma$ determined by $\lambda$ determines the isomorphism class of $\pi(\lambda)$ and conversely.
Then, \cite[Theorem 2.2 and 2.8]{zelevinsky_GLn} state that there is a bijection between the set of irreducible constituents of $\pi(\lambda)$ (which does not depend on $\lambda$) and the set orientations of $\Gamma$, which sends the orientation corresponding to $\lambda$ to the unique irreducible subrepresentation of $\pi(\lambda)$.
In fact, by \cite[Theorem 2.8]{zelevinsky_GLn} this extends to an isomorphism between the lattice of submodules of $\pi(\lambda)$ and the sublattice of the set of orientations of $\Gamma$ generated by the subsets $\O(\lambda, k)$ of orientations such that the edge between $\rho(k)$ and $\rho(k+1)$ is oriented in the same way as in the orientation corresponding to $\lambda$, for $k=1,...,s-1$.
More explicitly, for $k=1, ..., s-1$, there exists some $\lambda'$ giving the same orientation as $\lambda$ and such that for some $j$, either $\lambda'(j) = \rho(k)$ and $\lambda'(j+1) = \rho(k+1)$, or $\lambda'(j) = \rho(k+1)$ and $\lambda'(j+1) = \rho(k)$ (depending on the orientation of the edge between these two representations), and $\O(\lambda, k)$ corresponds to
$$
    \pi(\lambda, k) = \lambda'(1) \times \cdots \times \lambda'(j-1) \times Z_2(\rho) |\det|^{k-1} \times \lambda'(j+2) \times \cdots \times \lambda'(s) \subseteq \pi(\lambda') \simeq \pi(\lambda).
$$
In particular, all proper subrepresentations of $\pi(\lambda)$ are iterated sums and intersections of the subrepresentations $\pi(\lambda, k)$.
As an example, consider the bijection $\lambda_0 \colon \{1, ..., s\} \to \Delta$ given by $\lambda_0(i) := \rho(i)$. The corresponding orientation is
$$
    \bullet \to \bullet \cdots \bullet \to \bullet.
$$
Then, $Z_s(\rho) = \bigcap_{k=1}^{s-1} \pi(\lambda_0, k)$ and $Z_{s-1}(\rho) \times \rho(s) = \bigcap_{k=1}^{s-2} \pi(\lambda_0, k)$ -- these correspond to the sets of orientations
\begin{align*}
    \{ \bullet \to \bullet \cdots \bullet \to \bullet \},
    & &
    \{ \bullet \to \bullet \cdots \bullet \to \bullet, ~ ~ ~ \bullet \to \cdots \bullet \to \bullet \from \bullet \}
    &
\end{align*}
respectively.
In particular, the quotient $Q_s(\rho)$ of $Z_{s-1}(\rho) \times \rho(s)$ by $Z_s(\rho)$ is irreducible. In fact, $Q_s(\rho)$ is isomorphic to the unique irreducible submodule of $\pi(\mu)$, where $\mu$ is defined by $\mu(i) = \rho(i-1)$ if $i \geq 2$ and $\mu(1) = \rho(s)$, and has corresponding orientation
$$
    \bullet \to \bullet \cdots \bullet \to \bullet \from \bullet.
$$
Now, $Q_s(\rho) = \bigcap_{k=1}^{s-1} \pi(\mu, k)$, and moreover $\rho(s) \times Z_{s-1}(\rho) = \bigcap_{k=1}^{s-2} \pi(\mu, k)$, and as before the quotient of the latter by the former is isomorphic to $Z_s(\rho)$.
The claimed exact sequence then follows by composing the middle maps in
\begin{equation*}
    Z_s(\rho)
    \to
    Z_{s-1}(\rho) \times \rho(s)
    \to
    Q_s(\rho)
    \to
    \rho(s) \times  Z_{s-1}(\rho)
    \to
    Z_{s}(\rho).
    \qedhere
\end{equation*}
\end{proof}

Fix for all $v \in S$ an isomorphism $\Gbf_{F_v} \simto \GL_{n / F_v}$.
We will assume that $\Sbf_v = \Tbf_v$ correspond to the torus of diagonal matrices, and that $\Pbf_v$ contains the Borel subgroup of upper-triangular matrices.

\begin{theorem}
Let $\pi$ be an irreducible smooth representation of $M$ over $E$, $U$ be as usual and $X \in \O^{\p_E}_\alg$.
Then, the spaces $H_*(K^S, U \otimes_E \Fcal_P^G( X, \pi))$ and $H^*(K^S, U' \otimes_E \Fcal_P^G( X, \pi)')$ are finite-dimensional.
\end{theorem}
\begin{proof}
We have proven in \Cref{proposition: finite dimensionality for Orlik Strauch} that this holds when $\pi$ is the smooth parabolic induction of a supercuspidal representation. As in the proof of \Cref{proposition: finite dimensionality for Orlik Strauch}, it's enough to show this as $\pi$ runs through a basis of the Grothendieck group of finite length smooth (admissible) representations of $M$. According to \cite[Corollary 7.5]{zelevinsky_GLn}, one such basis consists of representations of the form
$$
    \bigboxtimes_{v \in S} \prod_{i=1}^{r_v} Z_{s_{v,i}} (\rho_{v,i} )
$$
where for all $v$ we have $\sum_{i=1}^{r_v} s_{v,i} = n$, and all $s_{v,i}$ are positive.
We proceed by induction  on $\mathfrak S := \sum_{v \in S} \sum_{i=1}^{r_v} (s_{v,i} - 1)$.
If $\mathfrak S = 0$, then we are dealing with a representation of the form $\bigboxtimes_{v \in S} \prod_{i=1}^{n} \rho_{v,i}$, and as observed above we have already proven that $S$-arithmetic (co)homology is finite-dimensional in this case. If $\mathfrak S \geq 1$, then there exist some $v_0$ and $i_0$ such that $s_{v_0,i_0} \geq 2$. In light of \Cref{theorem: exact sequence Zelevinsky}, there is an exact sequence
$$
\scalebox{0.9}{\parbox{0pt}{
\begin{align*}
    0
    & \to
    \bigboxtimes_{v \in S} \prod_{i=1}^{r_v} Z_{s_{v,i}} (\rho_{v,i} ) \\
    & \to
    \left( \bigboxtimes_{\substack{v \in S \\ v \neq v_0}} \prod_{i=1}^{r_v} Z_{s_{v,i}} (\rho_{v,i} ) \right) \boxtimes \left( \left( \prod_{i=1}^{i_0-1} Z_{s_{v_0,i}} (\rho_{v_0,i} ) \right) \times ( Z_{s_{v_0, i_0}-1}(\rho_{v_0, i_0}) \times \rho(s) ) \times \left( \prod_{i=i_0+1}^{r_{v_0}} Z_{s_{v_0,i}} (\rho_{v_0,i} ) \right) \right) \\
    & \to
    \left( \bigboxtimes_{\substack{v \in S \\ v \neq v_0}} \prod_{i=1}^{r_v} Z_{s_{v,i}} (\rho_{v,i} ) \right) \boxtimes \left( \left( \prod_{i=1}^{i_0-1} Z_{s_{v_0,i}} (\rho_{v_0,i} ) \right) \times ( \rho(s) \times Z_{s_{v_0, i_0}-1}(\rho_{v_0, i_0}) ) \times \left( \prod_{i=i_0+1}^{r_{v_0}} Z_{s_{v_0,i}} (\rho_{v_0,i} ) \right) \right) \\
    & \to
    \bigboxtimes_{v \in S} \prod_{i=1}^{r_v} Z_{s_{v,i}} (\rho_{v,i} )
    \to
    0.
\end{align*}
}}
$$
Applying $U \otimes_E \Fcal_P^G(X, \blank)$ we get yet another exact sequence, which to avoid running over the margins we will write simply as
$$
    0 \to \Pi_4 \to \Pi_3 \to \Pi_2 \to \Pi_1 \to 0,
$$
where $\Pi_1 = \Pi_4$ is the representation whose $S$-arithmetic homology we want to prove is finite-dimensional.
Moreover, $\Pi_2$ and $\Pi_3$ are obtained in the same way as $\Pi_1$, except the corresponding sums $\mathfrak S$ are smaller than the one for $\Pi_1$.
By our induction hypothesis, both representations have finite-dimensional $S$-arithmetic homology.
This short exact sequence induces a converging spectral sequence
$$
    E_{i,j}^1 = H_i(K^S, \Pi_j) \implies 0.
$$
Assume for the sake of contradiction that $H^*(K^S, \Pi_1)$ is infinite-dimensional. Let $i_0$ be the greatest index $i$ for which $H_i(K^S, \Pi_1)$ is infinite-dimensional. Thus, $H_i(K^S, \Pi_j)$ is finite-dimensional for all $j$ and all $i > i_0$. In particular, the entry $E_{i_0, 4}$ will be infinite-dimensional in all pages starting in the $E^1$ page, which contradicts the fact that the spectral sequence converges to 0. The same argument works for cohomology.
\end{proof}

\begin{corollary}
Every locally algebraic representation of $G$ of finite length has finite-dimensional $S$-arithmetic homology at level $K^S$, and similarly for the cohomology of its (abstract) dual.
\end{corollary}

\begin{remark}
This result is false without the finite length assumption, even for smooth admissible representations.
Indeed, this fails already for $F = \Q$, $S = \{p\}$ and $\Gbf = \GL_1$.
Choose a sequence $\psi_n$ of unitary Hecke characters of conductors $p^{r_n}$ with $r_n \rightarrow \infty$ as $n \rightarrow \infty$, and let $\psi_{n, p}$ be their components at $p$.
Let $K^p = \prod_{\ell \neq p} \Z_\ell^\times$.
Then, $\Pi := \bigoplus_n \psi_{n, p}$ is an admissible smooth representation of $\GL_1(\Q_p)$ (of infinite length) and
$$
    H^0(K^p, \Pi') = \prod_n H^0(K^p, \psi_{n, p}^{-1})
$$
is infinite-dimensional, since each term in the product is non-zero by \Cref{theorem: properties of classical cuspidals in cohomology}.
\end{remark}

\subsection{\texorpdfstring{$p$}{p}-arithmetic and completed cohomology}
\label{subsection: completed cohomology}
We will now prove a result relating the degree zero $p$-arithmetic (i.e. $S$-arithmetic when $S = \Scalp$) cohomology groups of duals of locally analytic representations to completed cohomology when $\Gbf$ is $F$-anisotropic.
We hope to investigate the general relation between $p$-arithmetic and completed cohomology in the future.
We assume that $S = \Scalp$ and write $K^p$ and $\Tbb^p(K^p)$ instead of $K^S$ and $\Tbb^S(K^S)$ respectively.
Recall the definition of (rational) completed cohomology,
$$
    \tilde H^i (K^p) := \left( \varprojlim_{s} \varinjlim_{K_p} H^i(K^p K_p, \O_E / p^s \O_E) \right) \otimes_{\O_E} E,
$$
where $K_p$ runs through the compact open subgroups of $G = \Gbf(F \otimes_\Q \Q_p)$.

\begin{proposition}\label{proposition: p-arithmetic and completed}
Assume that $\Gbf$ is $F$-anisotropic.
Let $V$ be a locally analytic representation of $G$ on a locally convex $E$-vector space of compact type.
There is a $\Tbb^p(K^p)$-equivariant isomorphism $H^0(K^p, V' ) \simeq \Hom_{G, \text{cts}}(V, \tilde H^0 (K^p)^\la)$.
\end{proposition}
\begin{proof}
Let $W$ be a locally convex space with a continuous action of $G$.
If $W$ is a locally convex space with a continuous action of $G$, we let $\tilde S(K^p, W)$ be the space of continuous functions $\Gbf(F) \backslash \Gbf(\A_{F}^\infty) / K^p \to W$.
This space is equipped with a natural action of $G$ (via 
$(g f)(x) := g f(x g)$) and of $\Tbb^p(K^p)$.
Unwinding the definition of $H^0(K^p, W)$ shows that it is isomorphic as a $\Tbb^p(K^p)$-module to $\tilde S(K^p, W)^G$.
On the other hand, if $W = V'$, the example at the end of \cite[\S17]{nonarchimedean_functional_analysis} and \cite[Corollary 18.8]{nonarchimedean_functional_analysis} show that there are topological isomorphisms $\tilde S(K^p, V') \simeq \tilde S(K^p, E) \cotimes_E V' \simeq \Hom_{\text{cts}}(V, \tilde S(K^p, E))$ (note that the space $\Gbf(F) \backslash \Gbf(\A_{F}^\infty) / K^p$ is compact by our anisotropy assumption).
Moreover, $\tilde S(K^p, E)$ is isomorphic to $\tilde H^0(K^p)$.
Thus, the result follows from the fact that, as $V$ is locally analytic, any $G$-equivariant continuous homomorphism $V \to \tilde S(K^p, E)$ factors through $\tilde S(K^p, E)^\la$.
\end{proof}

\section{Eigenvarieties}
\label{section: eigenvarieties}

In this section we will use the results from \Cref{section: homology} to construct eigenvarieties.
We will start by proving the spectral theorem that powers the construction, and then give a first definition of eigenvarieties.
Then, we will give a different construction using overconvergent homology and determine the relation between both constructions, and use it to give some of their properties.
Finally, in the case of definite unitary groups, we will discuss the relation between these eigenvarieties and the Bernstein eigenvarieties studied by Breuil--Ding \cite{breuil_ding_bernstein_eigenvarieties}.

\subsection{A spectral theorem}
\label{subsection spectral theorem}

If $R$ is a commutative Banach $\Q_p$-algebra and $M_\bullet$ is a complex of Banach $R$-modules, by a chain operator we mean an $R$-linear, graded map $\phi \colon M_\bullet \to M_\bullet$ that commutes with the differentials and is continuous in each degree.
We say such a map is simultaneously $R$-compact if there exist chain operators $\{ \psi_n \}_{n \geq 0}$ on $M_\bullet$ that are of finite $R$-rank at each degree and converge to $\phi$ in each degree as operators on a Banach space.
We will need the following spectral theorem generalising \cite[Proposition 2.2.6]{emerton_jacquet_one} for orthonormalisable modules.
Note that, unlike in \cite[Proposition 2.2.6]{emerton_jacquet_one}, we have to work with algebraic (as opposed to completed) tensor products.

\begin{theorem}\label{theorem: spectral theorem}
Let $R$ be a commutative Noetherian Banach $\Q_p$-algebra, $(M_\bullet, d)$ a complex of orthonormalisable $R$-Banach modules, $\phi_1, ..., \phi_n, \phi$ continuous chain operators on $M_\bullet$. Assume that $\phi$ is simultaneously $R$-compact on $M_\bullet$, $\phi_1, ..., \phi_n$ commute in $H_*(M_\bullet)$, and that the composition $\phi_1 \cdots \phi_n$ agrees with $\phi$ on $H_*(M_\bullet)$. We let $R[T_1, ..., T_n]$ act on $H_*(M_\bullet)$ by letting $T_i$ act as $\phi_i$. Let $R'$ be a flat commutative Noetherian Banach $R$-algebra and $t_1, ..., t_n \in (R')^\times$. Then, for all $q \in \Z$,
\begin{align*}
	& \Tor^{R' [ T_1, ..., T_n]}_q \left( \frac{ R' [ T_1, ..., T_n] }{ (T_1-t_1, ..., T_n-t_n) }, H_s(M_\bullet \cotimes_R R')  \right) \\
	& \simeq \Tor^{R [ T_1, ..., T_n]}_q \left( \frac{ R' [ T_1, ..., T_n] }{ (T_1-t_1, ..., T_n-t_n) }, H_s(M_\bullet)  \right).
\end{align*}
is a finitely generated $R'$-module.
\end{theorem}
\begin{proof}
We will prove the theorem by induction on $n$, starting with the case $n=1$. We may replace $\phi_1$ by $\phi$, and we drop the subindices for notational convenience and assume for simplicity that $s = 0$. Write $N$ for the kernel of the map $M_0 \to M_{-1}$ and $N'$ for the kernel of $M_0 \cotimes_R R' \to M_{-1} \cotimes_R R'$.
We will write $\phi$ instead of $\phi \cotimes 1 \colon M_\bullet \cotimes_R R' \to M_\bullet \cotimes_R R'$.
In this case, the Tor groups can be non-zero only in degrees 0 and 1, and
$$
    \Tor^{R' [ T]}_0 \left( \frac{ R' [T] }{ (T-t) }, H_0(M_\bullet \cotimes_R R')  \right) \simeq \frac{N'}{(\phi-t)N' + d(M_1 \cotimes_R R')},
$$
and $\Tor^{R' [ T]}_1 \left( \frac{ R' [T] }{ (T-t) }, H_0(M_\bullet \cotimes_R R') \right)$ is isomorphic to the subspace $(N'/d(M_1 \cotimes_R R'))^{\phi=t}$ of $N'/d(M_1 \cotimes_R R')$ consisting of classes of elements $n$ such that $\phi(n) = t n + d(m)$ for some $m \in M_1 \cotimes_R R'$.
On the other hand,
\begin{align*}
    \Tor^{R[T]}_0 \left( \frac{ R'[T] }{ (T-t) }, H_0(M_\bullet)  \right) 
    & \simeq
    \Tor^{R'[T]}_0 \left( \frac{ R'[T] }{ (T-t) }, H_0(M_\bullet \otimes_R R')  \right)
    \\ & \simeq
    \frac{N \otimes_R R'}{(\phi-t)(N \otimes_R R') + d(M_1 \otimes_R R')},
\end{align*}
and $\Tor^{R[T]}_1 \left( \frac{ R'[T] }{ (T-t) }, H_0(M_\bullet) \right)$ is isomorphic to $(N \otimes_R R' /d(M_1 \otimes_R R'))^{\phi=t}$.

Choose $a \in R^\times$ such that $a \phi, at$ and $at^{-1}$ are power-bounded on $M_0 \cotimes_R R'$ and on $M_1 \cotimes_R R'$. Since $\phi$ is compact, we may find a chain operator $\psi$ on $M_\bullet$ of finite $R$-rank in each degree such that $a^{-1}(\phi - \psi)$ is topologically nilpotent. Since $M_0$ and $M_1$ are orthonormalisable, there is a commutative diagram whose rows compose to $\psi$
$$
\begin{tikzcd}
    M_1 \ar[r, "\psi'"] \ar[d, "d"] & R^r \ar[r, "\psi''"] & M_1 \ar[d, "d"] \\
    M_0 \ar[r, "\psi'"] \ar[d, "d"] & R^r \ar[r, "\psi''"] \ar[d, "d_0"] \ar[dl] & M_0 \ar[d, "d"] \\
    M_{-1} \ar[r, "\psi'"] & R^r \ar[r, "\psi''"] & M_{-1}
\end{tikzcd}
$$
for some $r$. Indeed, we can apply \cite[Lemma 2.2.1]{emerton_jacquet_one} to show the existence of each row with $\Im(\psi'') = \Im(\psi)$, and choosing preimages for a set of generators for the map $\psi'$ in the bottom row gives the diagonal map after possibly increasing $r$. Moreover, if we set $P := \ker d_0$, we have $\psi''(P) \subseteq N$ and hence a commutative diagram.
$$
\begin{tikzcd}
    M_1 \ar[r, "\psi'"] \ar[d, "d"] & R^r \ar[r, "\psi''"] \ar[d, "d_1"] \ar[dl] & M_1 \ar[d, "d"] \\
    N \ar[r, "\psi'"] \ar[d, hook] & P \ar[r, "\psi''"] \ar[d, hook] & N \ar[d, hook] \\
    M_0 \ar[r, "\psi'"] \ar[d, "d"] & R^r \ar[r, "\psi''"] \ar[d, "d_0"] & M_0 \ar[d, "d"] \\
    M_{-1} \ar[r, "\psi'"] & R^r \ar[r, "\psi''"] & M_{-1}
\end{tikzcd}
$$
where the diagonal map exists by the same argument as before (after possibly increasing $r$).
Moreover, there is also a commutative diagram
$$
\begin{tikzcd}
    M_1 \cotimes_R R' \ar[r, "\psi'"] \ar[d, "d"] & (R')^r \ar[r, "\psi''"] \ar[d, "d_1"] & M_1 \cotimes_R R' \ar[d, "d"] \\
    N' \ar[r, "\psi'"] \ar[d, hook] & P \otimes_R R' \ar[r, "\psi''"] \ar[d, hook] & N' \ar[d, hook] \\
    M_0 \cotimes_R R' \ar[r, "\psi'"] \ar[d, "d"] & (R')^r \ar[r, "\psi''"] \ar[d, "d_0"] & M_0 \cotimes_R R' \ar[d, "d"] \\
    M_{-1} \cotimes_R R' \ar[r, "\psi'"] & (R')^r \ar[r, "\psi''"] & M_{-1} \cotimes_R R'.
\end{tikzcd}
$$

Consider the chain operator
\begin{align*}
    \theta \colon M_\bullet \cotimes_R R' & \to M_\bullet \cotimes_R R' \\
    m & \mapsto \sum_{i \geq 0} t^{-i-1} (\phi - \psi)^i (m).
\end{align*}
Note that $t^{-i-1} (\phi - \psi)^i = t^{-1} (at^{-1})^i (a^{-1} (\phi - \psi))^i$ is a topologically nilpotent operator, hence the sum on the right converges and defines a continuous map. The map $\theta$ commutes with $\phi - \psi$ and we also have $\theta \circ (t - \phi + \psi) = (t - \phi + \psi) \circ \theta = \id$. 
From this, we see that, for all $m$ in $M_\bullet \cotimes_R R'$,
\begin{align*}
    (\psi \circ \theta)(m) & = m + (\phi - t) \theta(m), \\
    (\theta \circ \psi)(m) & = m +  \theta(\phi - t)(m). \\
\end{align*}
From the first equality, we see that the composition
\begin{align*}
    N' / d(M_1 \cotimes_R R')
    &
    \stackrel\theta\to
    N'/d(M_1 \cotimes_R R')
    \stackrel{\psi'}\to
    \frac{P \otimes_R R'}{d_1((R')^r)} \stackrel{\psi''}\to
    N' / d(M_1 \cotimes_R R')
    \\ &
    \onto
    \frac{N'}{d(M_1 \cotimes_R R') + (\phi-t)(N')}
\end{align*}
is just projection onto the quotient.
In particular, the map
$$
	\frac{P \otimes_R R'}{d_1((R')^r)} \stackrel{\psi''}\to \frac{N'}{d(M_1 \cotimes_R R') + (\phi-t)(N')}
$$
is surjective, so the target is finitely generated. %
Similarly, from the second equality we see that the composition
\begin{align*}
    \left( \frac{N'}{d(M_1 \cotimes_R R')} \right)^{\phi = t} 
    &
    \into
    N'/d(M_1 \cotimes_R R')
    \stackrel{\psi'}\to
    \frac{P \otimes_R R'}{d_1((R')^r)}
    \stackrel{\psi''}\to
    N' / d(M_1 \cotimes_R R')
    \\ &
    \stackrel\theta\to N' / d(M_1 \cotimes_R R')
\end{align*}
is simply the inclusion, which shows that the map $\left( \frac{N'}{d(M_1 \cotimes_R R')} \right)^{\phi = t}  \stackrel{\psi'} \to P \otimes_R R' / d_1((R')^r)$ is injective, so the source is finitely generated.

Note that we have a commutative diagram
$$
\begin{tikzcd}
\left( \frac{N \otimes_R R'}{d(M_1 \otimes_R R')} \right)^{\phi = t} \ar[rr, equals] \ar[d, hook] \ar[ddd, bend right, shift right=2.5ex] & & \left( \frac{N \otimes_R R'}{d(M_1 \otimes_R R')} \right)^{\phi = t} \ar[d, hook]  \ar[ddd, bend left, shift left=1.5ex] \\
\frac{N \otimes_R R'}{d(M_1 \otimes_R R')} \ar[r, "\psi'"] \ar[d] & \frac{P \otimes_R R'}{ d_1((R')^r) } \ar[r, "\theta \psi''"] \ar[d, equals] & \frac{N \otimes_R R'}{d(M_1 \otimes_R R')} \ar[d] \\
\frac{N'}{d(M_1 \cotimes_R R')} \ar[r, "\psi'"] & \frac{P \otimes_R R' }{ d_1((R')^r) } \ar[r, "\theta \psi''"] & \frac{N'}{d(M_1 \cotimes_R R')} \\
\left( \frac{N'}{d(M_1 \cotimes_R R')} \right)^{\phi = t} \ar[rr, equals] \ar[u, hook] & & \left( \frac{N'}{d(M_1 \cotimes_R R')} \right)^{\phi = t}. \ar[u, hook]
\end{tikzcd}
$$
In particular, the upper right vertical injection factors through $\left( \frac{N'}{d(M_1 \cotimes_R R')} \right)^{\phi = t}$, so the bent vertical arrows are injective.

Moreover, the inclusion $\left( \frac{N'}{d(M_1 \cotimes_R R')} \right)^{\phi = t} \into \frac{N'}{d(M_1 \cotimes_R R')}$ factors through the 
middle-top right $\frac{N \otimes_R R'}{d(M_1 \otimes_R R')}$.
Let $y$ be the image by $\theta \psi$ in $N \otimes_R R'$ of some lift to $N'$ of an element of $\left( \frac{N'}{d(M_1 \cotimes_R R')} \right)^{\phi = t}$.
Thus, $y$ is a lift in $N \otimes_R R'$ of the same element.
We can see from the top square of the diagram that $\theta \psi (y) = y + d(m)$ for some $m \in M_1 \otimes_R R'$.
But also $\theta \psi(y) = y + \theta (\phi - t)(y)$, so
$$
    (\phi - t)(y) = \theta^{-1} d(m) = (t - \phi + \psi) d(m) \in d(M_1 \otimes_R R').
$$
Thus, $y \in \left( \frac{N \otimes_R R'}{d(M_1 \otimes_R R')} \right)^{\phi = t}$. Hence, the bent vertical maps are also surjective, so
$$
	\Tor^{R' [ T ]}_1 \left( \frac{ R' [ T ] }{ (T - t) }, H_0(M_\bullet \cotimes_R R')  \right)
	\simeq
	\Tor^{R [ T ]}_1 \left( \frac{ R [ T ] }{ (T - t) }, H_0(M_\bullet)  \right).
$$

The argument for the zeroth Tor space is very similar. In this case, we have a commutative diagram
$$
\begin{tikzcd}
\frac{N \otimes_R R'}{d(M_1 \otimes_R R') + (\phi  - t )(N \otimes_R R')} \ar[rr, equals] \ar[ddd, bend right, shift right=2.5ex] & & \frac{N \otimes_R R'}{d(M_1 \otimes_R R') + (\phi  - t )(N \otimes_R R')}  \ar[ddd, bend left, shift left=1.5ex] \\
\frac{N \otimes_R R'}{d(M_1 \otimes_R R')} \ar[r, "\psi' \theta "] \ar[d] \ar[u, twoheadrightarrow] & \frac{P \otimes_R R'}{ d_1((R')^r) } \ar[r, "\psi'' "] \ar[d, equals] & \frac{N \otimes_R R'}{d(M_1 \otimes_R R')} \ar[d] \ar[u, twoheadrightarrow] \\
\frac{N'}{d(M_1 \cotimes_R R')} \ar[r, "\psi' \theta "] \ar[d, twoheadrightarrow] & \frac{P \otimes_R R' }{ d_1((R')^r) } \ar[r, "\psi'' "] & \frac{N'}{d(M_1 \cotimes_R R')} \ar[d, twoheadrightarrow] \\
\frac{N'}{d(M_1 \cotimes_R R') + (\phi  - t )(N')} \ar[rr, equals] & & \frac{N'}{d(M_1 \cotimes_R R') + (\phi  - t )(N')}.
\end{tikzcd}
$$
The lower left vertical surjection factors through the bent vertical maps, so these must be surjective. Let $y \in N \otimes_R R'$ be a lift of an element in the kernel of the bent map, i.e. $y = d(m) + (\phi - t)(n')$ for some $m \in M_1 \cotimes_R R', n' \in N'$. Using that
$$
    \psi \theta (\phi - t) = \psi \theta \psi - \psi = (\phi - t) \theta \psi,
$$
we see that
$$
\psi \theta y = \psi \theta d(m) + \psi \theta (\phi - t) n' = d \psi \theta m + (\phi - t) \theta \psi n'.
$$
Since the image of $\psi$ is contained in $M_\bullet \otimes_R R'$, we see that the right hand side lies in $d(M_1 \otimes_R R') + (\phi-t)(N \otimes_R R')$. Since $y - \psi \theta y \in d(M_1 \otimes_R R') + (\phi - t)(N \otimes_R R')$ by the top square of the diagram above, this shows that
$$
    y \in d(M_1 \otimes_R R') + (\phi - t)(N \otimes_R R'), 
$$
so the bent maps are also injective. This completes the case $n = 1$.

Assume now that $n > 1$. The first step is to reduce to the case where $\phi_1$ is compact.
Write $\Lambda = \Z^n$, $\Lambda^- = \Z_{\geq 0}^n$.
Let $e_1, ..., e_n$ denote the canonical basis of $\Z^n$ and let \linebreak $v = e_1 + \cdots + e_n$.
Write $\Lambda'$ for the (free abelian) submonoid of $\Lambda^-$ generated by $v, e_2, ..., e_n$.
We can identify $R[\Lambda^-]$ with $R[T_1, ..., T_n]$ by identifying $T_k$ with $e_k$,
and then $v$ is identified with $Z := T_1 \cdots T_n$.
Write $z := t_1 \cdots t_n$. Note that $Z$ acts on $H_*(M\bullet)$ as $\phi$, which is compact operator on $M_\bullet$.
\Cref{lemma: vanishing Tor 1} \itemnumber{2} implies that
$$
    \Tor^{R'[T_1, ..., T_n]}_q \left( \frac{ R'[T_1, ..., T_n] }{ (T_1-t_1, ..., T_n - t_n) }, H_s(M_\bullet \cotimes_R R') \right)
$$
is isomorphic to
$$
    \Tor^{R'[Z, T_2, ..., T_n]}_q \left( \frac{ R'[Z, T_2, ..., T_n] }{ (Z - z, T_2 - t_2, ..., T_n - t_n) }, H_s(M_\bullet \cotimes_R R') \right),
$$
and similarly if we replace $H_s(M_\bullet \cotimes_R R')$ by $H_s(M_\bullet \otimes_R R')$.
Thus, we may assume without loss of generality that $\phi_1$ is compact.
Write for brevity $R'_k = R' [T_1, ..., T_k]$ and $\m_k$ for the ideal $(T_1 - t_1, ..., T_k - t_k)$. By \cite[X.9.5, Corollaire 2]{bourbaki_algebra}, if $\mathcal M$ denotes either $H_s(M_\bullet \cotimes_R R')$ or $H_s(M_\bullet \otimes_R R')$, then there is an exact sequence
\begin{align}\label{eqn: exact sequence for tors}
\begin{split}
    0 & \to R'_{k+1} / (T_{k+1} - t_{k+1}) \otimes_{R'_{k+1}} \Tor^{R'_k}_q ( R'_k / \m_k, \Mcal ) \\
    & \to \Tor^{R'_{k+1}}_q ( R'_{k+1} / \m_{k+1}, \Mcal ) \\
    & \to \Hom_{R'_{k+1}}(R'_{k+1} / (T_{k+1} - t_{k+1}),  \Tor^{R'_k}_{q-1} ( R'_k / \m_k, \Mcal )) \to 0
\end{split}
\end{align}
so it follows by induction on $k$ that the middle module is finitely generated, and the 5-lemma implies the isomorphism in the statement of the theorem.
\end{proof}

\begin{remark}
In the setting of the statement of \Cref{theorem: spectral theorem}, if $t_1, ..., t_n$ belong to the image of $R$ in $R'$, writing also $t_1, ..., t_n$ for some of their preimages in $R$ it follows from the theorem that
\begin{align*}
	& \Tor^{R' [ T_1, ..., T_n]}_q \left( \frac{ R' [ T_1, ..., T_n] }{ (T_1-t_1, ..., T_n-t_n) }, H_s(M_\bullet \cotimes_R R')  \right) \\
	& \simeq \Tor^{R [ T_1, ..., T_n]}_q \left( \frac{ R [ T_1, ..., T_n] }{ (T_1-t_1, ..., T_n-t_n) }, H_s(M_\bullet)  \right) \otimes_R R'.
\end{align*}
\end{remark}

\begin{remark}\label{remark: comparing spectral theorems}
Let us show how \Cref{theorem: spectral theorem} \itemnumber{1} implies \cite[Proposition 2.2.6]{emerton_jacquet_one} for orthonormalisable modules. Let $M, K, A, \phi, x$ and $t$ be as in \loccit, and assume that $M$ is orthonormalisable.
Applying \Cref{theorem: spectral theorem} to $M$ (seen as a complex concentrated in degree zero) with $R = A$ and $R' = A \langle xt, xt^{-1} \rangle$, we see that $\frac{R'[T]}{(T-t)} \otimes_{R[T]} M$ is a finitely generated $R'$-module.
In fact, the proof of \Cref{theorem: spectral theorem} realises this \emph{topologically} as a quotient of a finite $R'$-module, and hence its topology is the Banach $R'$-module topology.
In particular, $\frac{R'[T]}{(T-t)} \otimes_{R[T]} M$ is isomorphic to the completed tensor product $\frac{R'[T]}{(T-t)} \cotimes_{R[T]} M$. Finally, this is isomorphic to $K \langle xt, x t^{-1} \rangle \cotimes_{K[t]} M \simeq K \langle xt, x t^{-1} \rangle \cotimes_{K \langle x t \rangle} M$.
\end{remark}

In this context, we will also need an analogue of \cite[Proposition 3.2.28]{emerton_jacquet_one}.
The statement is a bit technical, but informally it will say that if we can write $R' = R'' \{\!\{ t_{i_1}, ..., t_{i_r} \}\!\}$, then the (higher) part of the higher Tor groups above corresponding to the operators $\phi_{i_1}, ..., \phi_{i_r}$ vanishes, and the corresponding part of the zeroth Tor groups is especially nice.
Before we state the precise result, it will be useful to consider the following lemma.

\begin{lemma}\label{lemma: flatness of product with Gm}
Let $R$ be an affinoid algebra and let $m \geq 1$.
Given $h \in \Z_{\geq 0}$, let us write 
$R^{(h)} = R \langle p^h y_1, p^h y_1^{-1}, ..., p^h y_m, p^h y_m^{-1} \rangle$ and $R^{(\infty)} := R \{\!\{ y_1, y_1^{-1}, ..., y_n, y_n^{-1} \}\!\} = \varprojlim_h R^{(h)}$. Then, $R^{(h)}$ and $R^{(\infty)}$ are flat over $R[y_1, ..., y_n]$, $R^{(h')}$ is flat over $R^{(h)}$ for $h' \leq h$, and $R^{(\infty)}$ is a Frechet--Stein algebra.
\end{lemma}
\begin{proof}
That $R^{(h')}$ is flat over $R^{(h)}$ follows from the fact that $\Sp(R^{(h')})$ is an affinoid subdomain of $\Sp(R^{(h)})$.
This implies that $R^{(\infty)} = \varprojlim_h R^{(h)}$ is a Frechet--Stein algebra since the transition maps have dense images.
That $R^{(h)}$ and $R^{(\infty)}$ are flat over $R$ follows from the argument in \cite[Corollary 4.4]{derived_eigenvarieties} applied to the subring $R^\circ \subseteq R$ of power-bounded elements (instead of $\O_K \subseteq K$).
\end{proof}

\begin{proposition}\label{prop: comparison of finite slopes}
Let $R, M_\bullet, \phi_1, ..., \phi_n$ be as in \Cref{theorem: spectral theorem} and $t_1, ..., t_n \in R^\times$, and assume that $R$ is an affinoid algebra. Let $\psi_1, ..., \psi_m, \psi$ be chain operators on $M_\bullet$ with $m \geq 1$, $\psi$ simultaneously compact and which commute in homology such that $\psi$ agrees in homology with $\psi_1 \cdots \psi_m$.
View $H_*(M_\bullet)$ as an $R[T_1, ..., T_n, Y_1, ..., Y_m,Z]$-module by letting $T_i$ (resp. $Y_i$, resp. $Z$) act by $\phi_i$ (resp. $\psi_i$, resp. $\psi$).
Define $R^{(h)}$ and $R^{(\infty)}$ as in \Cref{lemma: flatness of product with Gm}.
We view these as $R[Y_1, ..., Y_n, Z]$-modules by letting $Y_i$ act as $y_i$ and $Z$ as $y_1 \cdots y_m$.
Write $R^{(h)}(t_*)$ for the $R^{(h)}[T_1, ..., T_n, Y_1, ..., Y_m,Z]$-module whose underlying module is $R^{(h)}$, and where $T_i$ acts by multiplication by $t_i$, and define an $R[T_1, ..., T_n]$-module $R(t_*)$ similarly.
Then, there is an isomorphism of $R^{(\infty)}[T_1, ..., T_n, Y_1, ..., Y_m, Z]$-modules
\begin{align}\label{eqn: isomorphisms finite slope part}
\begin{split}
    &
    \varprojlim_h
    \Tor^{R^{(h)}[T_1, ..., T_n, Y_1, ..., Y_m]}_i ( R^{(h)}(t_*), H_s(M_\bullet \cotimes_R R^{(h)}))
    \\ & 
    \simeq
    \varprojlim_h
    R^{(h)} \otimes_{R[Y_1, ..., Y_m]}
    \Tor^{R[T_1, ..., T_n]}_i ( R(t_*), H_s(M_\bullet))
    \\ &
    \simeq
    \varprojlim_h
    R \langle p^h z, p^h z^{-1} \rangle \otimes_{R[Z]}
    \Tor^{R[T_1, ..., T_n]}_i ( R(t_*), H_s(M_\bullet))
\end{split}
\end{align}
(in particular, the last term admits a structure of $R^{(\infty)}[T_1, ..., T_n, Y_1, ..., Y_m, Z]$-module).
Moreover, these modules can be equipped with locally convex topologies that make them coadmissible modules over $R^{(\infty)}$ and over $R\{\!\{ z, z^{-1} \}\!\}$.
\end{proposition}
\begin{proof}
To simplify notation, write $S = R[T_1, ..., T_n]$ and $S^{(h)} = R^{(h)}[T_1, ..., T_n]$ for any 
$h \in \Z_{\geq 0} \cup \{ \infty \}$.
Let $h \in \Z_{\geq 0}$.
Let us start by showing the first isomorphism.
In fact, we will show that the terms in both limits are isomorphic.
By \Cref{theorem: spectral theorem}, the left-hand side of \Cref{eqn: isomorphisms finite slope part} is isomorphic to the inverse limit of $\Tor^{S[Y_1, ..., Y_m]}_i ( R^{(h)}(t_*), H_s(M_\bullet))$.

As $R^{(h)}(t_*) \simeq R(t_*) \otimes_R R^{(h)}$ as a module over $S[Y_1, ..., Y_m] \simeq S \otimes_R R[Y_1, ..., Y_m]$ and all of the terms in the tensor products are flat over $R$, there is a spectral sequence
$$
    E^2_{i,j} =
    \Tor^{R[Y_1, ..., Y_m]}_i \left(
    R^{(h)},
    \Tor^{S}_j ( R(t_*) , H_s(M_\bullet) )
    \right)
    \implies
    \Tor^{S[Y_1, ..., Y_m]}_{i+j} ( R^{(h)}(t_*), H_s(M_\bullet) ).
$$
As $R^{(h)}$ is flat over $R[Y_1, ..., Y_m]$, this spectral sequence collapses and we obtain isomorphisms
\begin{align}\label{eqn: isomorphisms finite slope part 2}
\begin{split}
    \Tor^{S[Y_1, ..., Y_m]}_{i} ( R^{(h)}(t_*), H_s(M_\bullet) )
    & \simeq
    R^{(h)} \otimes_{R[Y_1, ..., Y_m]} \Tor^{S}_i ( R(t_*) , H_s(M_\bullet) ).
    \\ & \simeq
    \Tor^{S}_i ( R(t_*) , R^{(h)} \otimes_{R[Y_1, ..., Y_m]} H_s(M_\bullet) ).
\end{split}
\end{align}
This proves the first isomorphism in \Cref{eqn: isomorphisms finite slope part}.
The proof of \Cref{theorem: spectral theorem} realises the $R^{(h)}$-module $R^{(h)} \otimes_{R[Y_1, ..., Y_m]} H_s(M_\bullet)$ topologically as the quotient of a finitely generated one, so it is a finitely generated Banach $R^{(h)}$-module isomorphic to the completed tensor product $R^{(h)} \cotimes_{R[Y_1, ..., Y_m]} H_s(M_\bullet)$.
In particular, if $h' < h$ then
$$
    R^{(h')} \cotimes_{R^{(h)}} \left( R^{(h)} \otimes_{R[Y_1, ..., Y_m]} H_s(M_\bullet) \right)
    \simeq
    R^{(h')} \otimes_{R[Y_1, ..., Y_m]} H_s(M_\bullet).
$$
Thus, $\Tor^{S[Y_1, ..., Y_m]}_{i} ( R^{(h)}(t_*), H_s(M_\bullet) )$ is also finitely generated over $R^{(h)}$ and
$$
    R^{(h')} \cotimes_{R^{(h)}} \Tor^{S}_i ( R(t_*) , R^{(h)} \otimes_{R[Y_1, ..., Y_m]} H_s(M_\bullet) )
    \simeq
    \Tor^{S}_i ( R(t_*) , R^{(h')} \otimes_{R[Y_1, ..., Y_m]} H_s(M_\bullet) ).
$$
Here, we have given the Tor groups their unique Banach $R^{(h)}$ and $R^{(h')}$-module topologies -- thus, the completed tensor product is the same as the non-completed tensor product.
Therefore, the limit over all $h$ of the terms in \Cref{eqn: isomorphisms finite slope part 2} is a coadmissible $R^{(\infty)}$-module.

It remains to describe the $S^{(\infty)}$-action on the last term of \Cref{eqn: isomorphisms finite slope part} and to prove the second isomorphism and coadmissibility over $R \{\!\{ z, z^{-1} \}\!\}$.
There is a natural homomorphism
\begin{align}\label{eqn: finite slope transition map}
    R \langle p^{mh} z, p^{mh} z^{-1} \rangle \otimes_{R[Z]} H_s(M_\bullet)
    \to
    R^{(h)} \otimes_{R[Y_1, ..., Y_m]} H_s(M_\bullet)
\end{align}
sending $z$ to $y_1 \cdots y_m$.
The left-hand side is a finitely generated Banach $R \langle p^{mh} z, p^{mh} z^{-1} \rangle$-module by \Cref{theorem: spectral theorem}, which is also a module over $R\langle p^{mh} z, p^{mh} z^{-1} \rangle [Y_1, Y_1^{-1}, ..., Y_m, Y_m^{-1}]$: $Y_i$ acts via $1 \otimes \psi_i$ and $Y_i^{-1}$ acts via $z^{-1} \otimes (\psi_1 \cdots \psi_{i-1} \psi_{i+1} \cdots \psi_m)$.
Considering the norms of these operators, we see that this action extends to an action of $R^{(h')}$ for some $h' \geq h$, and is compatible with that of $R \langle p^{mh} z, p^{mh} z^{-1} \rangle$, in the sense that $(Y_1 \cdots Y_m)^{\pm 1}$ acts as $z^{\pm1 }$. This allows us to construct an $R^{(h')}$-module map
$$
    R^{(h')} \otimes_{R[Y_1, ..., Y_m]} H_s(M_\bullet) \to R \langle p^{mh} z, p^{mh} z^{-1} \rangle \otimes_{R[Z]} H_s(M_\bullet)
$$
whose composition with \Cref{eqn: finite slope transition map} on either side gives the transition maps for the inverse systems $R \langle p^{mh} z, p^{mh} z^{-1} \rangle \otimes_{R[Z]} H_s(M_\bullet)$ and $R^{(h)}(y_*) \otimes_{R[Y_1, ..., Y_m]} H_s(M_\bullet)$.
Via \Cref{eqn: isomorphisms finite slope part 2}  and its analogue for $Z$ instead of $Y_1, ..., Y_m$, we obtain the second isomorphism in \Cref{eqn: isomorphisms finite slope part}, and coadmissibility over $R \{\!\{ z, z^{-1} \}\!\}$ also follows.
\end{proof}

Finally, let us state the version of the spectral theorem that we will actually use in practice.

\begin{corollary}\label{corollary: spectral theorem 2}
Let $R$ be a commutative Noetherian Banach algebra and $M_\bullet$ a complex of orthonormalisable $R$-Banach modules. Let $\boldLambda, \boldLambda^-, \Lambda, \Lambda^-$ and $\chi$ be as in \Cref{proposition: vanishing Tor 6}. Assume that for each $\lambda \in \boldLambda^-$ we are given a chain operator $\phi_\lambda$ on $M_\bullet$ and that $\phi_\mu$ is $R$-compact for some $\mu \in \boldLambda^-$. Assume that for any $\lambda_1, \lambda_2 \in \boldLambda^-$, the operators $\phi_{\lambda_1} \phi_{\lambda_2}$ and $\phi_{\lambda_1 \lambda_2}$ agree on $H_*(M_\bullet)$. Then:
\begin{enumerate}
    \item The space 
$$ \Tor^{R[\boldLambda^-]}_q \left( R(\chi), H_s(M_\bullet)  \right) $$
is a finitely generated $R$-module for all $q$ and $s$, where $\lambda$ acts on $H_s(M_\bullet)$ via $\phi_\lambda$.
    \item If $R \to R'$ is a flat continuous map of Noetherian Banach algebras, then
\begin{align*}
    \Tor^{R'[\boldLambda^-]}_q \left( R'(\chi), H_s(M_\bullet \cotimes_R R')  \right) = \Tor^{R[\boldLambda^-]}_q \left( R(\chi), H_s(M_\bullet)  \right) \otimes_R R'.
\end{align*}
\end{enumerate}
\end{corollary}
\begin{proof}
Note that if $\phi_\mu$ is $R$-compact, we may assume that for any given $\lambda \in \boldLambda^-$ the operator $\phi_{\mu \lambda}$ is also $R$-compact (by replacing it with the operator $\phi_\mu \phi_\lambda$ if necessary).
In particular, we may assume that $\mu \in \Lambda^-$.
Let us start by proving \itemnumber{1} and \itemnumber{2} for $\Lambda^-$ instead of $\boldLambda^-$
This follows from a similar argument as in the end of the proof of \Cref{theorem: spectral theorem}. By \Cref{lemma: vanishing Tor 3} \itemnumber{1} there exists a free abelian submonoid $\Lambda' \subseteq \Lambda$ such that $\Lambda = \Lambda' (\Lambda')^{-1}$. Writing $\mu = \lambda \tilde \lambda^{-1}$ with $\lambda, \tilde \lambda \in \Lambda'$ and replacing $\mu$ with $\lambda = \mu \tilde \lambda$ we may assume that $\mu \in \Lambda'$.

Let $e_1, ..., e_n$ be a basis of $\Lambda'$.
We may not apply \Cref{theorem: spectral theorem} yet, as we can't guarantee that $\phi_{e_1 \cdots e_n}$ is $R$-compact.
If $n = 1$, we can replace $\Lambda$ by the subgroup generated by $\mu$ so that this holds.
If $n \geq 2$, we may write $\mu = e_1^{a_1} \cdots e_n^{a_n}$ for some non-negative integers $a_1, ..., a_n$, and changing the order of the basis elements and multiplying $\mu$ by $e_1$ or $e_2$ and replacing $\mu$ by $\mu e_1$ or $\mu e_2$ as in the previous paragraph, we may assume that $a_1$ and $a_2$ are coprime.
In particular, there exist integers $c_1, c_2 \geq 0$ such that $a_1 c_2 - a_2 c_1 = 1$. Let $\lambda = e_1^{c_1} e_2^{c_2} \in \Lambda'$, and consider the (free abelian) submonoid of $\Lambda'$ generated by $\mu, \lambda, e_3, ..., e_n$. Then, $\Lambda = \Lambda'' (\Lambda'')^{-1}$. By \Cref{lemma: vanishing Tor 1} \itemnumber{2},
$$
    \Tor^{R[\Lambda^-]}_q ( R(\chi), H_s(M_\bullet)  ) = \Tor^{R[\Lambda'']}_q ( R(\chi), H_s(M_\bullet)  )
$$
and in \Cref{theorem: spectral theorem} we showed that the latter is finitely generated over $R$ and commutes with completed tensor products with flat $R$-modules.

Next, we will prove the result for $\boldLambda$. For this, we use \Cref{lemma: vanishing Tor 5}, which says that there is a natural map
$$
    \Tor^{R[\boldLambda^-]}_q ( R(\chi), H_s(M_\bullet)  ) 
    \to
    \Tor^{R[\Lambda^-]}_q ( R(\chi), H_s(M_\bullet)  )
$$
that identifies the source with a direct summand of the target. Since the target is finitely generated over $R$, so is the source, and if $R'$ is as in \itemnumber{2}, then there is a diagram
$$
\begin{tikzcd}
    \Tor^{R[\boldLambda^-]}_q ( R(\chi), H_s(M_\bullet)  ) \otimes_R R' \ar[r] \ar[d]
    &
    \Tor^{R[\Lambda^-]}_q ( R(\chi), H_s(M_\bullet)  ) \ar[d, "\sim"] \ar[l, bend right=15] \otimes_R R'
    \\
    \Tor^{R'[\boldLambda^-]}_q ( R'(\chi), H_s(M_\bullet \cotimes_R R') ) \ar[r]
    &
    \Tor^{R'[\Lambda^-]}_q ( R'(\chi), H_s(M_\bullet \cotimes_R R') ), \arrow[l, bend left=15]
\end{tikzcd}
$$
where the bent horizontal arrows are sections of the straight ones, and the vertical maps commute with the straight (resp. bent) horizontal arrows.
This implies that the left vertical arrow must also be an isomorphism.
\end{proof}

\subsection{Character spaces}
Let $E$ be a finite extension of $\Q_p$.
In this section we gather some facts about character spaces of locally analytic groups.
For example, they exist:

\begin{lemma}
Let $G$ be a topologically finitely generated $p$-adic locally analytic group. There exists a rigid $E$-analytic strictly quasi-Stein space $\hat G$ representing the functor that associates to a rigid $E$-analytic space $X$ to the set of continuous characters $G \to \O_X(X)^\times$.
In fact, there exist non-negative integers $d_1, d_2$ such that $\hat{G}$ is isomorphic to a union of finitely many copies of the product of $\Gm^{d_1}$ (where $\Gm$ is the rigid $E$-analytic multiplicative group) and open unit $d_2$-dimensional balls.
\end{lemma}
\begin{proof}
Let $\overline{[G,G]}$ be the closure of the commutator subgroup of $G$. Any continuous character $G \to \O_X(X)^\times$ is trivial on $\overline{[G,G]}$, so if $H := G / \overline{[G,G]}$, then we can identify the set of continuous characters $G \to \O_X(X)^\times$ with the set of continuous characters $H \to \O_X(X)^\times$. Since $H$ is a topologically finitely generated locally analytic abelian group, by \cite[Proposition 6.4.5]{emerton_locally_analytic_vectors} there exists a rigid analytic space $\hat{H}$ as claimed representing this functor, so we can set $\hat G := \hat H$.
\end{proof}

\begin{lemma}\label{lemma: finite group acting on character space}
Let $G$ be a topologically finitely generated $p$-adic locally analytic group, and $\mathcal T \subseteq \hat{G}(E)$ a finite group of characters. Then, $\hat{G}$ can be admissibly covered by affinoid subspaces that are stable under the action of $\mathcal T$ by multiplication.
\end{lemma}
\begin{proof}
Let $\hat G = \bigcup_{n \geq 0} X_n$ realise $\hat{G}$ as a strictly quasi-Stein space. For each $n$, define $Y_n = \bigcap_{\tau \in \mathcal T} X_n \tau$. Then $\{Y_n\}$ is an admissible covering satisfying our requirements.
\end{proof}

Let us now place ourselves in our familiar setting of reductive groups over $\Q_p$. More concretely, $\Gbf$ is a reductive group over $\Q_p$ with parabolic subgroup $\Pbf$, $\Mbf$ is a Levi subgroup of $\Gbf$, $\Abf$ is a maximal $\Q_p$-split torus in the center of $\Mbf$, etc. Write $\Zbf$ for the center of $\Mbf$.

\begin{lemma}\label{lemma map M to A J_M is finite flat}
Let $H_0 \subseteq M$ be a compact open subgroup and $H$ an open subgroup of $M$ containing $A H_0$ with finite index. Then, the morphism of rigid spaces $\hat{M} \to \hat{H}$ induced by restriction of characters is finite flat.
\end{lemma}
\begin{proof}
It follows from the constructions of $\hat{M}$ and $\hat{H}$ and a case by case analysis that it's enough to show that the homomorphism $H / \overline{[H,H]} \to M / \overline{[M,M]}$ has finite kernel and cokernel. We proved this in \Cref{lemma: abelianisation has finite kernel and cokernel}.
\end{proof}

\begin{proposition}\label{proposition: dimension of character space}
Let $H_0 \subseteq M$ be a compact open subgroup.
Then, $\dim(\hat{ H_0} ) = \dim(\Zbf)$.
\end{proposition}
\begin{proof}
Let $Z_0$ be the maximal compact subgroup of $\Zbf(\Q_p)$.
The homomorphisms
$$
    Z_0
    \simeq
    \Zbf(\Q_p) / \Lambda
    \rightarrow
    (M / \Mbf^\der(\Q_p)) / \Lambda
    \leftarrow
    (M / \bar{[M, M]}) / \Lambda
    \leftarrow
    \Lambda H_0 / \Lambda \bar{[H_0, H_0]}
    \simeq
    H_0 / \bar{[H_0, H_0]}
$$
all have finite kernel and cokernel by the proof of \Cref{lemma: abelianisation has finite kernel and cokernel}, so all the corresponding character spaces have the same dimension. Thus, the proposition follows from the fact that, as $Z_0$ is compact, we have 
$\dim(\hat{ Z_0 }) = \dim(Z_0) = \dim( \Zbf )$.
\end{proof}

Recall that a subset $S$ of a rigid $E$-analytic space $X$ is called Zariski-dense if the smallest Zariski closed subvariety of $X$ containing $S$ is $X$ itself, that $S$ accumulates at a point $x$ of $X$ if any affinoid open neighbourhood of $x$ contains an affinoid subdomain $U$ such that $S \cap U$ is Zariski-dense in $U$, and that $S$ is an accumulation subset of $X$ if it accumulates at every point of $S$.

Assume now that $\Gbf = \Res_{L/\Q_p} \tilde \Gbf$ for some finite extension $L / \Q_p$ such that $E$ contains the images of all the embeddings of $L$ into $\bar \Q_p$ and \emph{split} reductive group $\tilde \Gbf$ over $L$.
Then, $\Pbf, \Mbf$ and $\Tbf$ admit similar descriptions.
The group $\Gbf$ is quasi-split, so we may fix a Borel subgroup $\Bbf \subseteq \Pbf$.
Let $\Delta_E$ (resp. $\Delta_{M,E}$) be the set of simple positive roots of $\Gbf_E$ (resp. $\Mbf_E$) with respect to $\Bbf_E$ (resp. $\Mbf_E \cap \Bbf_E$).
Fix an algebraic character $\lambda$ of $\Tbf_E$ that is dominant with respect to $\Mbf_E \cap \Bbf_E$.
If $H$ is closed subgroup of $M$, by an algebraic (resp. dominant regular) weight of $\hat{H}$ we mean the restriction to $H$ of an algebraic character of $\Mbf_E$ (resp. an algebraic character $\psi$ of $\Mbf_E$ such that $\lambda \psi$ is dominant regular with respect to $\Bbf_E$).

\begin{proposition}
\label{proposition: small slope weights are dense}
With the assumptions above, fix $a \in A^{--}$ and let $H_0$ be a compact open subgroup of $M$ such that $H_0 \cap Z = Z_0$ is the maximal compact subgroup of $Z = \Zbf(\Q_p)$.
Fix $h > 0$.
Then, the set of dominant regular weights $\psi$ such that
\begin{align}\label{eqn: small slope weight condition}
    v_p((s_{\alpha} \cdot (\lambda \psi))(a)) - v_p((\lambda \psi)(a)) > h
\end{align}
for all $\alpha \in \Delta_{E} \setminus \Delta_{M, E}$ accumulates at any algebraic weight of $\hat{H_0}$ and are Zariski-dense in the connected components of $\hat{H_0}$ containing an algebraic weight.
\end{proposition}
\begin{proof}
The proof of \Cref{proposition: dimension of character space} shows that the restriction map $\hat{H_0} \to \hat{Z_0}$ is finite flat.
By \cite[Lemme 6.2.8]{chenevier_gln}, it suffices to prove the statement with $H_0$ replaced by $Z_0$.
Let $\tilde \Zbf$ be the center of $\tilde \Mbf$ and $\tilde \Abf$ its identity component, or equivalently its maximal subtorus.
Let $\tilde A := \tilde \Abf(L)$ and let $\tilde A_0$ be its maximal compact subgroup.
Then, $\tilde A_0$ has finite index in $Z_0$, so \cite[Lemme 6.2.8]{chenevier_gln} again shows that it's enough to prove the statement of the proposition for $\hat{ \tilde A_0 }$ instead of $\hat{ Z_0 }$ (or $\hat{H_0}$).
The proof in this setting is standard.
\end{proof}

\subsection{Eigenvarieties from \texorpdfstring{$S$}{S}-arithmetic homology}
\label{subsection: construction of eigenvarieties}

We are finally ready to tackle our construction of eigenvarieties.
We fix throughout a base field $E$ that is a finite extension of $\Q_p$.

We place ourselves in the setting of \Cref{subsection: global setting}, except we assume that $\sigma$ and $U$ are $E$-vector spaces instead of $R$-modules.
If $\Ucal$ is an open affinoid subspace of $\hat{M_\Sigma}$, $\Ucal = \Sp(R)$, the inclusion $\Ucal \subseteq \hat{M_\Sigma}$ induces a continuous character $\delta_\Ucal \colon M_\Sigma \to R^\times$. For any such affinoid, we let $\sigma_\Ucal = \sigma \otimes_E \delta_\Ucal$ and $V_\Ucal = \cInd_{\HM}^M (\sigma_\Ucal) \simeq V \cotimes_E \delta_\Ucal$. These give rise to representations
$$
    \SigmaInd_P^G(U, V_\Ucal) := U \otimes_E \Ind_P^G(V_\Ucal)^{\Sigmala} \simeq (U \otimes_E R) \otimes_R \Ind_P^G(V_\Ucal)^{\Sigmala},
$$
$\Acal_{\sigma_\Ucal, U}^{\Sigmala}$, $\Acal_{\sigma_\Ucal, U}^{s-\Sigmaan}$, etc. (this is the same notation as in \Cref{subsection: global setting}, except we write $U$ instead of $U \otimes_E R$).

\begin{proposition}
The association $\Ucal \mapsto H_*(K^S, \SigmaInd_P^G(U, V_\Ucal ))$ defines a coherent sheaf $\mathscr H$ on $\hat{M_\Sigma}$.
\end{proposition}
\begin{proof}
We need to check that if $\Ucal = \Sp(R)$, then $H_*(K^S, \SigmaInd_P^G(U, V_\Ucal ))$ is finitely generated over $R$ and that if $\Ucal' \subseteq \Ucal$ is an admissible open affinoid, $\Ucal = \Sp(R')$, then $H_*(K^S, \SigmaInd_P^G(U, V_{\Ucal'} )) \simeq H_*(K^S, \SigmaInd_P^G(U, V_\Ucal )) \otimes_R R'$.
In view of the spectral sequence in \Cref{corollary: spectral sequence for affinoids}, it's enough to prove the same statements for $\Tor_i^{\Hcal_{S,R}^-}( R(1), H_j(K^S J, \Acal_{\sigma_\Ucal, U}^{s-\Sigmaan} ))$ and its analogue over $R'$ for all $i, j$.
This follows from \Cref{corollary: spectral theorem 2} since $C_\bullet(K^S J, \Acal_{\sigma_\Ucal, U}^{s-\Sigmaan} )$ is a complex of orthonormalisable Banach $R$-modules 
(the term in each degree being isomorphic to a finite direct sum of copies of $\Acal_{\sigma_\Ucal, U}^{s-\Sigmaan}$),
the operators $U_\mu$ give a well-defined action of $\Hcal_{S, R}^-$ on homology (recall also \Cref{corollary: hecke algebra isomorphic to monoid ring}),
and if $a \in A^{--}_\Sigma$ then the operator $U_a$ is $R$-compact.
\end{proof}

The action of $\Tbb^S(K^S)$ on $H_*(K^S, \SigmaInd_P^G(U, V_\Ucal))$ induces a $\Q_p$-algebra homomorphism $\Tbb^S(K^S) \to \End_{\hat{M_\Sigma}}(\mathscr H)$, so we obtain a coherent sheaf $\Tscr$ of $\O_{\hat {M_\Sigma}}$-algebras on $\hat{M_\Sigma}$ assigning to $\Ucal = \Sp(R) \subseteq \hat{M_\Sigma}$ the $R$-subalgebra $\Tbb_{\Ucal}$ of $\End_R(H_*(K^S, \SigmaInd_P^G(U, V_\Ucal) ) )$ generated by the image of $\Tbb^S(K^S)$ under the above homomorphism.

We write $\Xscr_{V, U}^\Sigma(K^S)$ for the relative spectrum of $\Tscr$ over $\hat{M_\Sigma}$. As $V$ may be isomorphic to some of its twists by characters of $M_\Sigma$, it is possible for the same systems of eigenvalues for $\Tbb^S(K^S) \otimes \Hcal_{S,E}^-$ to be found on the fibers of several characters (note that we are including the Hecke operators at $p$). This can be solved by taking the quotient of $\Xscr_{V, U}^\Sigma(K^S)$ by the group of such characters.
The following lemma shows that we can actually do this in most cases of interest.

\begin{lemma}\label{lemma: can take quotients of X}
Assume that $V_\Sigma$ has a central character.
The group $\Tcal_M$ of continuous characters $\delta \colon M_\Sigma \to \bar \Q_p^\times$ such that $\bar \Q_p \otimes_E V \simeq V \otimes \delta$ is finite. If all such characters are defined over $E$ and $\End_{G}(V) = E$, this group acts compatibly on the spaces $\Xscr_{V, U}^\Sigma(K^S)$ and $\hat{M_\Sigma}$, which have (compatible) admissible affinoid coverings which are stable under this action.
\end{lemma}
\begin{proof}
Comparing the central characters of $V_\Sigma$ and $V_\Sigma \otimes \delta$ for a character $\delta$ as above we see that the restriction of $\delta$ to the center of $M_\Sigma$ is trivial. It is also trivial on $\overline{ [ M_\Sigma, M_\Sigma ] }$, which has finite index in the group of points of the derived group, $\Mbf_\Sigma^\der(\Q_p)$. Thus, $\delta$ factors through the (finite) quotient of $M_\Sigma$ by the subgroup generated by its center and $\overline{ [ M_\Sigma, M_\Sigma ] }$. In particular, there are only finitely many possibilities for $\delta$.

Assume that all such $\delta$ are defined over $E$. By \Cref{lemma: finite group acting on character space}, this group acts on $\hat{M_\Sigma}$, which admits an admissible covering $\{ \Ucal_i \}_i$ by open affinoids that are stable under this action. Assume that $\End_G(V) = E$. This implies that the isomorphisms $V \simeq V(\delta)$ for $\delta$ as above are uniquely defined up to a constant. The pullback $\delta^* \mathscr H$ sends $\mathcal U_i$ to $H_*(K^S, \SigmaInd_P^G(U, V_{\Ucal_i} \otimes \delta ))$, which is isomorphic to $H_*(K^S, \SigmaInd_P^G(U, V_{\Ucal_i} ))$ canonically-up-to-scalar. This determines a canonical isomorphism $(\delta^* \mathscr T)(\Ucal_i) \simeq \mathscr T (\Ucal_i)$ that is compatible with multiplication of characters and with restriction to (invariant) affinoid subspaces. This gives rise to an action on $\Xscr_{V, U}^\Sigma(K^S)$ as claimed in the statement.
\end{proof}

\begin{definition}
If all the assumptions in (both parts of) \Cref{lemma: can take quotients of X} are satisfied, we call the quotient $\Yscr_{V, U}^\Sigma(K^S)$ of $\Xscr_{V, U}^\Sigma(K^S)$ by the action of $\Tcal_M$ the eigenvariety of tame level $K^S$ associated to $(V, U)$ (with variation at places in $\Sigma$).
\end{definition}

Let us remark that this space does not depend on the choice of $\sigma$ giving rise to $V$ up to canonical isomorphism, and depends on $U$ only up to isomorphism, which is canonical if $U$ is an irreducible algebraic representation of $G_{\Scalp}$ (or more generally, if $U$ has no non-trivial automorphisms). Moreover, twisting $U$ or $V$ by a locally analytic character of $M^\Sigma$ gives canonically isomorphic eigenvarieties.

It is usually more convenient to work directly with $\Xscr_{V, U}^\Sigma(K^S)$ and pass to the quotient $\Yscr_{V, U}^\Sigma(K^S)$ afterwards if necessary, so we will not make use of $\Yscr_{V, U}^\Sigma(K^S)$ very often.
Note that the construction of $\Xscr_{V, U}^\Sigma(K^S)$ can be carried out verbatim if one works over $\hat{ \HM[\Sigma] }$ instead of $\hat{M_\Sigma}$, the only difference being that it does not make sense to say that $V_\Ucal := \cInd_{\HM}^M(\sigma_\Ucal) \simeq V \otimes_E \delta_\Ucal$ for $\Ucal = \Sp(R) \subseteq \hat {\HM[\Sigma]}$ an open affinoid with associated character $\delta_\Ucal$ (since $V \otimes_E \delta_\Ucal$ is not a representation of $M$).
Thus, this construction depends on $\sigma$ and $\HM$ and not just on $V$.
We will write $\mathscr H_{\HM}$ and $\Tbb_{\HM, \Ucal}$ for the analogues of $\mathscr H$ and $\Tbb_{\Ucal}$.
Write $\Xscr_{\HM, \sigma, U}^\Sigma(K^S)$ for the resulting variety over $\hat{\HM[\Sigma]}$.

\begin{proposition}\label{proposition: eigenvarieties fiber product}
The inclusion $\HM \into M$ induces an isomorphism $$
    \Xscr_{V, U}^\Sigma(K^S) \simto \Xscr_{\HM, \sigma, U}^\Sigma(K^S) \times_{\hat{\HM[\Sigma]}} \hat{M_\Sigma}
$$
\end{proposition}
\begin{proof}
By \Cref{lemma map M to A J_M is finite flat}, the natural morphism $\hat{M_\Sigma} \to \hat{\HM[\Sigma]}$ is finite flat. Let $\Ucal = \Sp(R)$ be an open affinoid subset of $\hat{\HM[\Sigma]}$. The preimage $\Vcal$ of $\Ucal$ in $\hat{M_\Sigma}$ is then an affinoid open subset of $\hat{M_\Sigma}$, $\Vcal = \Sp(B)$, where $B$ is a finite flat $R$-algebra. Write $\delta_\Ucal \colon \HM[\Sigma] \to R^\times$ and $\delta_\Vcal \colon M_\Sigma \to B^\times$ for the associated characters. The restriction $\delta_\Vcal|_{\HM}$ is simply the composition of $\delta_\Ucal$ and the morphism $R \to B$. In particular, it follows from \Cref{corollary: spectral theorem 2} \itemnumber{2} together with the spectral sequence from \Cref{corollary: spectral sequence for affinoids} that there is a $\Tbb^S(K^S)$-equivariant isomorphism
$$
    H_*(K^S, \SigmaInd_P^G(U, V_{\Vcal} )) \simeq H_*(K^S, \SigmaInd_P^G(U, V_\Ucal )) \otimes_R B
$$
This shows the coherent sheaf $\mathscr H$ on $\hat{M_\Sigma}$ is the pullback of its analogue on $\hat{\HM[\Sigma]}$. The same is then true for the coherent sheaves of $\O_{\hat{M_\Sigma}}$-algebras and $\O_{\hat{\HM[\Sigma]}}$-algebras obtained from the action of $\Tbb^S(K^S)$. Taking relative spectra, we obtain the result.
\end{proof}

\begin{proposition}\label{proposition: eigenvarieties fiber product 2}
\begin{enumerate}
    \fixitem If $\HM \subseteq \HM' \subseteq \NJM$, there is a natural isomorphism
    $$
        \Xscr_{\HM', \cInd_{\HM}^{\HM'} (\sigma), U}^\Sigma(K^S)
        \simto
        \Xscr_{\HM, \sigma, U}^\Sigma(K^S) \times_{\hat{\HM[\Sigma]}} \hat{\HM[\Sigma]'}.
    $$
    
    \item If $\HM \subseteq \HM' \subseteq \NJM$ and $\sigma$ extends to a representation of $\HM'$, then there is a natural closed immersion
    $$
        \Xscr_{\HM', \sigma, U}^\Sigma(K^S) \into \Xscr_{\HM, \sigma|_{\HM}, U}^\Sigma(K^S) \times_{\hat{\HM[\Sigma]}} \hat{\HM[\Sigma]'}
    $$
\end{enumerate}
\end{proposition}
\begin{proof}
Statement \itemnumber{1} follows from the transitivity of compact induction from $\HM$ to $\HM'$ and $\HM'$ to $M$ and the same arguments as in the proof of \Cref{proposition: eigenvarieties fiber product}.
Let us prove \itemnumber{2}. Let 
$\Ucal = \Sp(R) \subseteq \hat{\HM[\Sigma]}$ be an admissible open affinoid and $\Ucal' = \Sp(R')$ be its preimage in $\hat{\HM[\Sigma]'}$.
The result will follow if we show that for all $i$ and $j$, $\Tor_i^{\Hcal_{S, \HM', R'}^-}( R'(1), H_j(K^S J, \Acal_{\sigma_{\Ucal'}, U}^{s-\Sigmaan} ) )$ is a direct summand of $\Tor_i^{\Hcal_{S, \HM, R}^-}( R(1), H_j(K^S J, \Acal_{\sigma_\Ucal, U}^{s-\Sigmaan} )) \otimes_R R'$,
which by \Cref{corollary: spectral theorem 2} is isomorphic to $\Tor_i^{\Hcal_{S, \HM, R'}^-}( R'(1), H_j(K^S J, \Acal_{\sigma_{\Ucal'}, U}^{s-\Sigmaan} ))$. This follows from \Cref{lemma: vanishing Tor 5}.
\end{proof}

In preparation for the next section, let us introduce two other auxiliary variants of these eigenvarieties.
Consider instead of the coherent sheaf on $\hat{\HM[\Sigma]}$ corresponding to $S$-arithmetic cohomology as above the coherent sheaves $\mathscr H^2_{\HM}$ and $\mathscr H^\infty_{\HM}$ determined by
\begin{align*}
    \Ucal = \Sp(R) & \mapsto \mathscr H^2_{\HM}(\Ucal) :=  \bigoplus_{j} E^2_{0,j} = R(1) \otimes_{\Hcal_{S,\HM,R}^-} H_*(K^S J, \Acal_{\sigma_\Ucal, U}^{s-\Sigmaan} ) ), \\
    \Ucal = \Sp(R) & \mapsto \mathscr H^\infty_{\HM}(\Ucal) := \bigoplus_{j} E^\infty_{0,j},
\end{align*}
where $E^r_{i,j}$ denotes the spectral sequence from \Cref{corollary: spectral sequence for affinoids}.
We will write $\Tbb_{\HM, \Ucal}^2$ and $\Tbb_{\HM, \Ucal}^\infty$ for the $R$-algebra generated by the image of $\Tbb^S(K^S)$ in the algebra of $R$-module endomorphisms of the right hand sides and $\Xscr_{\HM}^{2}(K^S)$ and $\Xscr_{\HM}^{\infty}(K^S)$ for the varieties formed by gluing these $R$-algebras (since these varieties will only play an auxiliary role, we will suppress $\Sigma, \sigma$ and $U$ from the notation).
The edge maps in the spectral sequence from \Cref{corollary: spectral sequence for affinoids} determine maps
$$
     \mathscr H^2_{\HM}(\Ucal) \onto \mathscr H^\infty_{\HM}(\Ucal) \into \mathscr H_{\HM}(\Ucal)
$$
and hence surjections
\begin{align}\label{eqn: surjection of Hecke algebras}
    \Tbb_{\HM, \Ucal}^2 \onto \Tbb_{\HM, \Ucal}^\infty \twoheadleftarrow \Tbb_{\HM, \Ucal}.
\end{align}
In conclusion, there are closed immersions
$$
    \Xscr_{\HM}^{2}(K^S) \stackrel{\iota_2}\hookleftarrow \Xscr_{\HM}^{\infty}(K^S) \stackrel{\iota}\into \Xscr_{\HM, \sigma, U}^\Sigma(K^S).
$$

\begin{proposition}
\label{proposition: relation with thick eigenvariety}

\begin{enumerate}
    \fixitem $\iota_2$ and $\iota$ become isomorphisms after passing to nilreductions.
    
    \item Assume that $\HM = A \JM$ and that $\Lambda$ acts on $\sigma$ by a character. If $\Sigma = S$, then $\iota_2$ and $\iota$ are isomorphisms.
\end{enumerate}
\end{proposition}
\begin{proof}
For \itemnumber{1}, since affinoid algebras are Jacobson, it's enough to show that for all open affinoids $\Ucal = \Sp(R)$ of $\hat{\HM[\Sigma]}$ the $R$-algebra affinoid epimorphisms \Cref{eqn: surjection of Hecke algebras} induce bijections on maximal ideals. Injectivity is immediate, so we only need to show that the induced maps on maximal ideals are surjective.
For $\iota_2$, this is equivalent to showing that any system of Hecke eigenvalues appearing in the $i=0$ column of the $E^\infty_{i,j}$ page of the spectral sequence from \Cref{corollary: spectral sequence for affinoids} appears also in the same column for the $E^2$ page.
For $\iota$, it's enough to show that any system of eigenvalues contributing to the $E^\infty_{i,j}$ page appears in the column $i=0$.
Therefore, it suffices to show that any system of eigenvalues contributing to the $E^\infty$ page appears in a term of the $i=0$ column in the $E^2_{i,j}$ page whose localisation at the corresponding maximal ideal is stable.
Let $\m$ be a maximal ideal in $\Tbb(K^S) \otimes_\Qp R$ such that the localisation $(E^\infty)_\m$ is non-zero.
The intersection of $\m$ and $\Hcal_{S,R}^-$ is the ideal $\m_1$ generated by $U_\mu - 1$ with $\mu \in \HM^-$.
Consider the spectral sequence of $\Tbb(K^S) \otimes_\Qp R$-modules obtained by localising at $\m$, with $E^2$ page
$$
    E^2_{i,j} = \Tor^{\Hcal_{S, R}^-}_i( R(1), H_j(K^S J, \Acal_{\sigma_\Ucal, U}^\Sigma)_{\m} ).
$$
Let $j_0$ be the smallest degree for which $H_j(K^S J, \Acal_{\sigma_\Ucal, U}^\Sigma)_{\m}$ is non-zero (it exists since $(E^\infty)_\m$ is non-zero). Then the $E^2_{0,j_0}$-term in the spectral sequence above is also non-zero.
This term is stable, which proves \itemnumber{1}.
The proof of \itemnumber{2} essentially follows from \Cref{prop: comparison of finite slopes}, which implies that the $E^2_{i,j}$ page above can be non-zero only in the $i=0$ column.
We defer the details of this proof to the next section.
\end{proof}

\subsection{Eigenvarieties from overconvergent homology}
\label{subsection: eigenvarieties overconvergent homology}

Next, we will define analogues of some of the varieties from the previous section using overconvergent homology.
We will relate both constructions and study the latter to deduce some properties of the former. We change the setting slightly for now and assume that $\sigma$ is a representation of $\JM$ on a finite-dimensional $E$-vector space that decomposes as $\sigma = \sigma_\Sigma \boxtimes_E \sigma^\Sigma$ as usual. Consider the rigid $E$-analytic space $\hat{\JM[\Sigma]}$ parametrising continuous characters of $\JM[\Sigma]$. For any open affinoid $\Omega = \Sp(B)$ of $\hat{\JM[\Sigma]}$ with associated character $\delta_\Omega \colon \JM[\Sigma] \to B^\times$ consider the $B[\JM]$-module $\sigma_\Omega := \sigma \otimes_E \delta_\Omega$ and the Banach $B$-modules $\Acal^{s-\Sigmaan}_{\sigma_\Omega, U, *}$ defined in \Cref{subsection: overconvergent homology}. Fix $a \in \Lambda_\Sigma^{--}$.

The construction using overconvergent homology follows the same steps as in \cite[Section 4]{hansen_universal_eigenvarieties}, and we refer to this reference for the details.
Let $F_\Omega(X) = \det(1 - U_a X)$ be the characteristic power series of the operator $U_a$ on $C_\bullet(K^S J, \Acal^{s-\Sigmaan}_{\sigma_\Omega, U, *})$.
This is well-defined, since the complex $C_\bullet(K^S J, \Acal^{s-\Sigmaan}_{\sigma_\Omega, U, *})$ consists of orthonormalisable Banach $B$-modules and the operator $U_a$ is compact, and independent of $s$ by \Cref{proposition: properties overconvergent homology} \Cref{item: properties overconvergent homology power series}.
Moreover, by \Cref{proposition: properties overconvergent homology} \Cref{item: properties overconvergent homology power series} again, if $\Omega' \subseteq \Omega$ is an open affinoid, $\Omega' = \Sp(B')$, then $F_{\Omega'}(X) = F_\Omega(X) \otimes_B B'$.
Thus, we may glue the power series $F_\Omega$ to form a Fredholm series $F \in \O(\hat{\JM[\Sigma]}) \{\!\{ X \}\!\}$.
Let $\Zscr \subseteq \hat{\JM[\Sigma]} \times \Gm$ be the corresponding Fredholm hypersurface, i.e. the vanishing locus of $F$ in $\hat{\JM[\Sigma]} \times \Gm$, where $\Gm$ denotes the rigid $E$-analytic multiplicative group.
The open affinoids $\Zscr_{\Omega, h} = \Zscr \cap (\Omega \times \{ p^{-h} \leq |t| \leq 1 \})$ such that $C_\bullet(K^S J, \Acal^{s-\Sigmaan}_{\sigma_\Omega, U, *})$ has a slope-$\leq h$ decomposition with respect to $U_a$ form an admissible cover of $\Zscr$.
For such $\Omega = \Sp(B)$ and $h$, the slope decomposition induces a factorisation $F_\Omega(X) = Q(X) R(X)$, where $Q(X)$ is the characteristic polynomial $\det(1 - U_a X)$ on the slope-$\leq h$ part, which is multiplicative and has slope $\leq h$, and $\O(\Zscr_{\Omega, h}) \simeq B[X] / Q(X)$.
Let $\chi^\Sigma \colon \Lambda^\Sigma \to E^\times$ be a character and let $\m_{\chi^\Sigma,E}^{\Sigma}$ be the ideal of $\Hcal^-_{S \setminus \Sigma, A \JM, E}$ generated by $U_{\tilde a} - \chi^\Sigma(\tilde a)$ for $\tilde a \in (A^\Sigma)^-$.
The module $H_*(K^S J, \Acal^{s-\Sigmaan}_{\sigma_\Omega, U, *})_{\leq h}$ does not depend on the choice of $s$, so we will write $\Acal^\Sigma_{\sigma_\Omega, U, *} = \Acal^{s-\Sigmaan}_{\sigma_\Omega, U, *}$.
We can view $H_*(K^S J, \Acal^\Sigma_{\sigma_\Omega, U, *})_{\leq h}$ as an $\O(\Zscr_{\Omega, h})$-module by letting $X$ act via $U_a^{-1}$, and the assignment $\Zscr_{\Omega, h} \mapsto H_*(K^S J, \Acal^\Sigma_{\sigma_\Omega, U, *})_{\leq h} / \m_{\chi^\Sigma,E}^{\Sigma} H_*(K^S J, \Acal^\Sigma_{\sigma_\Omega, U, *})_{\leq h}$ defines a coherent sheaf $\mathscr H_*$ on $\Zscr$ with an action of the $\Q_p$-algebra $\Tbb^S(K^S)$.
Taking the relative spectrum of the coherent subsheaf of $\O(\Zscr_{\Omega, h})$-algebras of $\End(\mathscr H_*)$ induced by this action as in the previous section, we obtain a rigid $E$-analytic space $\Xscr_{\sigma, U, *}^\Sigma(K^S)$.

The main difference between this construction and that of the previous section (apart from working with the untwisted version of overconvergent homology) is the way in which we truncate slopes or pass to finite slope subspaces.
While on this construction this is achieved by taking slope decompositions, in the previous one (or rather, its untwisted analogue) it was done by taking (the derived functors of) $R(\chi) \otimes_{\Hcal^-_{S,E}} \blank$.
Working with slope decompositions is easier and gives more information (for example, it shows that eigenvarieties are locally finite over weight space $\hat{\HM[\Sigma]}$), but the gluing of the local pieces is more subtle.
In order to compare eigenvarieties, we will have to compare both of these procedures.
We will do this in two steps: constructing the analogues of the open affinoids $\Zscr_{\Omega, h}$ for the varieties in \Cref{subsection: construction of eigenvarieties}, and showing that these do indeed give an admissible cover of such varieties.

Recall that the product map $\Lambda_\Sigma \times \JM[\Sigma] \to A_\Sigma \JM[\Sigma]$ is an isomorphism. Extend the action of $\sigma$ to an action of $A \JM$ by letting $\Lambda_\Sigma$ act trivially and letting $\Lambda^\Sigma$ act via the character $\chi^\Sigma$.
We have $\hat{A_\Sigma \JM[\Sigma]} \simeq \hat{\JM[\Sigma]} \times \hat{\Lambda_\Sigma}$. Since $\Lambda_\Sigma$ is a free abelian group of rank $d = \rank(\Abf_\Sigma)$, the choice of a basis of $\Lambda_\Sigma$ determines an isomorphism $\hat{\Lambda_\Sigma} \simeq \Gm^{d}$.
Let $\Ucal = \Omega \times \Theta \subseteq \hat{A_\Sigma \JM[\Sigma]}$ be an admissible open affinoid, $\Ucal = \Sp(R)$, with $\Omega = \Sp(B) \subseteq \hat{\JM[\Sigma]}$ and $\Theta \subseteq \hat{\Lambda_\Sigma}$.
As remarked when defining the $*$-action in \Cref{subsection: overconvergent homology}, there is an isomorphism of $J A^+ J$-modules $\Acal^\Sigma_{\sigma_\Ucal, U, *} \simeq \Acal^\Sigma_{\sigma_\Omega, U, *} \cotimes_B R$.
Note that $\Acal^\Sigma_{\sigma_\Ucal, U, *}$ is defined in \Cref{subsection: untwisting}, whereas $\Acal^\Sigma_{\sigma_\Omega, U, *}$ is defined in \Cref{subsection: overconvergent homology}. For the former, we are indeed in the situation described in that section, since $\Lambda$ acts on $\sigma_\Ucal$ via $\chi \otimes \delta_\Theta$, where $\delta_\Theta$ is the inflation to $\Lambda$ of the character of $\Lambda_\Sigma$ corresponding to $\Theta$.

By \Cref{lemma: vanishing Tor 3} \itemnumber{1} there exists a free abelian submonoid $\Lambda_\Sigma' \subseteq \Lambda_\Sigma^-$ such that \linebreak $\Lambda_\Sigma = \Lambda_\Sigma' (\Lambda_\Sigma')^{-1}$, and by the arguments in the proof of \Cref{corollary: spectral theorem 2} we may assume that $\Lambda'$ admits a basis $a_1, ..., a_d$ such that $a_1 \in \Lambda_\Sigma^{--}$. Thus, we may assume that the element $a \in \Lambda_\Sigma^{--}$ fixed in the beginning of the section is $a_1$. From now on, whenever we identify $\hat{\Lambda_\Sigma}$ with $\Gm^d$, we will do so with respect to this basis. Write $T_1, ..., T_d$ for the coordinates of $\Gm^d$.

Given  $\Omega \subseteq \hat{\JM}, \Theta \subseteq \hat{\Lambda_\Sigma}$, let $R = \O(\Omega \times \Theta)$ and
\begin{align*}
    \Tbb_{\Omega, \Theta} & := R \cdot \Im \left( \Tbb^S(K^S) \to \End_{R} ( R(\chi^\Sigma \delta_{\Omega \times \Theta}) \otimes_{\Hcal^-_{S,A \JM, R}} H_*(K^S J, \Acal^\Sigma_{\sigma_{\Omega \times \Theta}, U, *}) ) \right)
\end{align*}
(a priori, this may depend on the degree of analyticity $s$ since we are not taking slope $\leq h$ parts, but we omit it from the notation nonetheless). In particular, by \Cref{lemma: untwisting homology} there is an isomorphism $\Tbb_{\Omega, \Theta} \simeq R(\chi^\Sigma \delta_{\Omega \times \Theta}) \otimes_R \Tbb^2_{A \JM, \Omega \times \Theta}$ of $\Tbb^S(K^S) \otimes_\Qp \Hcal_{S,R}^-$-algebras, and hence there is an isomorphism $\Tbb_{\Omega, \Theta} \simeq \Tbb^2_{A \JM, \Omega \times \Theta}$ of affinoid $R$-algebras.
Note that the character $\delta_{\Omega \times \Theta}$ sends $a_i \in \Lambda$ to $T_i \in \O(\Theta)$.
Thus, $\Xscr_{A \JM}^{2}(K^S)$ may be seen as being formed by gluing the $\Tbb_{\Omega, \Theta}$ for different $\Omega, \Theta$.
If $h \geq 0$ is such that $C_\bullet(K^S J, \Acal^\Sigma_{\sigma_\Omega, U, *})$ admits a slope-$\leq h$ decomposition, let
\begin{align*}
    \Tbb_{\Omega, h} & := \O(\Zscr_{\Omega, h}) \cdot \Im \left( \Tbb(K^S) \to \End_{\O(\Zscr_{\Omega, h})} \left( \frac{ H_*(K^S J, \Acal^\Sigma_{\sigma_\Omega, U, *})_{\leq h} }{ \m_{\chi^\Sigma, E}^\Sigma H_*(K^S J, \Acal^\Sigma_{\sigma_\Omega, U, *})_{\leq h} } \right) \right).
\end{align*}
Hence, the space $\Xscr_{\sigma|_{\JM}, U, *}^\Sigma(K^S)$ is forming by gluing $\Tbb_{\Omega, h}$ as we vary over $\Omega$ and $h$.

Given $\Omega$ and $h$ as above, for $i=1, ..., d$, the image of $U_{a_i}$ in
$$
    \mathcal E := \End_{\O(\Zscr_{\Omega, h})} \left( \frac{ H_*(K^S J, \Acal^\Sigma_{\sigma_\Omega, U, *})_{\leq h} }{ \m_{\chi^\Sigma, E}^\Sigma H_*(K^S J, \Acal^\Sigma_{\sigma_\Omega, U, *})_{\leq h} } \right)
$$ is a unit, so there exists $h_i$ such that $p^{-h_i} \leq | U_{a_i}^{-1} |_{\mathcal E}^{-1} \leq 1$ (recall that the operators $U_a$ are norm-decreasing).
For example, we may take $h_1 = h$.
In any case, we assume that $h_1 \leq h$.
Let $\hbf =(h_1, ..., h_d)$.
Define an affinoid subset of $\hat{\Lambda_\Sigma} \simeq \Gm^d$ by
$$
    \Theta(\Omega, \hbf) := \{ (T_1, ..., T_d) : p^{-h_i} \leq |T_i| \leq 1 \text{ for all } i \}.
$$
Note that by \cite[Proposition 2.3.4]{hansen_universal_eigenvarieties} $H_*(K^S J, \Acal^\Sigma_{\sigma_{\Omega \times \Theta(\Omega, \hbf)}, U, *})$ admits a slope-$\leq h$ decomposition, and by
\Cref{lemma: eigenvectors decomposition} applied to this decomposition (and the assumption that $h_1 \leq h$) we could take the slope-$\leq h$ part of homology in the definition of $\Tbb_{\Omega, \Theta(\Omega, \hbf)}$ without altering it. By definition of $\Theta(\Omega, \hbf)$, there is a morphism of affinoid algebras
$$
    \O(\Omega \times \Theta(\Omega, \hbf)) \to \Tbb_{\Omega, h}
$$
sending $T_i$ to the image of $U_{a_i}$, so we may regard the target as an $\O(\Omega \times \Theta(\Omega, \hbf))$-algebra.
The following lemma and its proof complete the first step outline above, and give a comparison of the two methods we have to truncate slopes.

\begin{lemma}\label{lemma: map of algebras in eigenvarieties}
There is an isomorphism of $\O(\Omega \times \Theta(\Omega, \hbf))$-algebras $\Tbb_{\Omega, \Theta(\Omega, \hbf)} \to \Tbb_{\Omega, h}$.
\end{lemma}
\begin{proof}
We will show that we may canonically identify $\frac{ H_*(K^S J, \Acal^\Sigma_{\sigma_\Omega, U, *})_{\leq h} }{ \m_{\chi^\Sigma, E}^\Sigma H_*(K^S J, \Acal^\Sigma_{\sigma_\Omega, U, *})_{\leq h} }$ with \linebreak
$ R(\delta_{\Omega \times \Theta(\Omega, \hbf)} \chi^\Sigma) \otimes_{\Hcal_{S, A \JM, R}^-} H_j(K^S J, \Acal^\Sigma_{\sigma_{\Omega \times \Theta(\Omega, \hbf)}, U, *}) $
as $\O(\Omega \times \Theta(\Omega, \hbf)) \otimes \Tbb(K^S)$-modules.
Let $R = \O(\Omega \times \Theta(\Omega, \hbf))$ and $B = \O(\Omega)$.
We first claim that the natural map
$$
    H_*(K^S J, \Acal^\Sigma_{\sigma_\Omega, U, *})_{\leq h} \simto R \otimes_{B[T_1, ..., T_d]} H_*(K^S J, \Acal^\Sigma_{\sigma_\Omega, U, *})_{\leq h}.
$$
is an isomorphism of $R$-modules (where we consider the target as an $R$-module by letting $R$ act on either of the terms in the tensor product, the action on the second being induced by $R \mapsto \Tbb_{\Omega, h}$) whose inverse is given by the $R$-algebra structure on $H_*(K^S J, \Acal^\Sigma_{\sigma_\Omega, U, *})_{\leq h}$. Indeed, the fact that $H_*(K^S J, \Acal^\Sigma_{\sigma_\Omega, U, *})_{\leq h}$ is finitely generated over $B[T_1, ..., T_d] \subseteq R$ (since it is over $B$) and $B[T_1, ..., T_d]$ is dense in $R$ implies that
\begin{align*}
    R \otimes_{B[T_1, ..., T_d]} H_*(K^S J, \Acal^\Sigma_{\sigma_\Omega, U, *})_{\leq h}
    & \simeq
    R \cotimes_{B[T_1, ..., T_d]} H_*(K^S J, \Acal^\Sigma_{\sigma_\Omega, U, *})_{\leq h} \\
    & \simeq
    R \cotimes_R H_*(K^S J, \Acal^\Sigma_{\sigma_\Omega, U, *})_{\leq h} \\
    & \simeq
    H_*(K^S J, \Acal^\Sigma_{\sigma_\Omega, U, *})_{\leq h}.
\end{align*}
By \Cref{lemma: vanishing Tor 1} \itemnumber{2}, we can rephrase this as saying that there is an $R \otimes \Tbb(K^S)$-module isomorphism
\begin{align}\label{eqn: Emerton slope part of slope decomposition}
    \frac{ H_*(K^S J, \Acal^\Sigma_{\sigma_\Omega, U, *})_{\leq h} }{ \m_{\chi^\Sigma, E}^\Sigma H_*(K^S J, \Acal^\Sigma_{\sigma_\Omega, U, *})_{\leq h} }
    \simeq
    R(\delta_{\Omega \times \Theta(\Omega, \hbf)} \chi^\Sigma) \otimes_{\Hcal_{S, A \JM, R}^-} H_*(K^S J, \Acal^\Sigma_{\sigma_{\Omega \times \Theta(\Omega, \hbf)}, U, *})_{\leq h}.
\end{align}
As $|\delta_{\Omega \times \Theta(\Omega, \hbf)}(a_1)^{-1}| = |T_1^{-1}|_{\Theta(\Omega, \hbf)} \leq p^h$,
\Cref{lemma: eigenvectors decomposition} (and the comment above it) shows that the right-hand side is isomorphic to $R(\delta_{\Omega \times \Theta} \chi^\Sigma) \otimes_{\Hcal_{S, A \JM, R}^-} H_*(K^S J, \Acal^\Sigma_{\sigma_{\Omega \times \Theta(\Omega, \hbf)}, U, *})$, thus proving our claim and the lemma.
\end{proof}

The second step in our comparison of eigenvarieties is a consequence of \Cref{prop: comparison of finite slopes}, as the following proposition shows.

\begin{proposition}\label{proposition: affinoid covering}
The morphism $\hat{ A_\Sigma \JM[\Sigma] } \to \hat{\JM[\Sigma]} \times \Gm$ induced by the morphism
$\hat{A_\Sigma} \to \hat{\langle a \rangle} \simeq \Gm$ induces a finite morphism $\Xscr^2_{A \JM}(K^S) \to \hat{\JM[\Sigma]} \times \Gm$ that factors through a closed subvariety of the target whose underlying analytic subset is a Zariski closed subset of $\Zscr$, and such that the preimage of $\Zscr_{\Omega, h}$ is $\Sp(\Tbb_{\Omega, \Theta(\Omega, \hbf)})$.
In particular, the affinoid open subspaces $\Sp(\Tbb_{\Omega, \Theta(\Omega, \hbf)})$ form an admissible cover of $\Xscr^2_{A \JM}(K^S)$.
\end{proposition}
\begin{proof}
Consider the pushforward of $\mathscr H^2_{A \JM}$ along the morphism $\hat{ A_\Sigma \JM[\Sigma] } \to \hat{\JM[\Sigma]} \times \Gm$. By \Cref{prop: comparison of finite slopes} and the way we have chosen $a_1, ..., a_d$ following the strategy of \Cref{corollary: spectral theorem 2}, it is a coherent sheaf that sends $\Omega \times \Sp( E \langle p^h T_1, p^h T_1^{-1} \rangle )$ to
$$
    E \langle p^h T_1, p^h T_1^{-1} \rangle \otimes_{E[T_1]} \O(\Omega)( \chi^\Sigma ) \otimes_{\Hcal^-_{S \setminus \Sigma, \O(\Omega)}} H_*(K^S J, \Acal^\Sigma_{\sigma_{\Omega}, U, *}).
$$
Thus, the morphism $\Xscr^2_{A \JM}(K^S) \to \hat{\JM[\Sigma]} \times \Gm$ is indeed finite and factors through the schematic support of this coherent sheaf.
This is contained in the analytic subset of $\hat{ \JM[\Sigma] } \times \Gm$ corresponding to the support of the coherent sheaf sending the affinoid $\Omega \times \Sp( E \langle p^h T_1, p^h T_1^{-1} \rangle )$ to
$ E \langle p^h T_1, p^h T_1^{-1} \rangle \otimes_{E[T_1]} H_*(K^S J, \Acal^\Sigma_{\sigma_{\Omega}, U, *}). $
Thus, it's enough to show that this is precisely the underlying analytic subset of $\Zscr$.
This follows from the same argument as in \cite[Lemme 3.9]{breuil_hellmann_schraen_1}, and the statement on preimages follows easily.
\end{proof}

\begin{proposition}\label{proposition: comparison eigenvarieties overconvergent}
The maps $\Tbb_{\Omega, \Theta(\Omega, \hbf)} \to \Tbb_{\Omega, h}$ induce an isomorphism
$$
    \Xscr_{\sigma|_{\JM}, U, *}^\Sigma(K^S) \simto \Xscr^2_{A \JM}(K^S).
$$
\end{proposition}
\begin{proof}
This is an immediate consequence of \Cref{lemma: map of algebras in eigenvarieties} and \Cref{proposition: affinoid covering}.
\end{proof}

We can extend these results to include the case where $\Lambda_\Sigma$ acts on $\sigma$ through a non-trivial character.
If $\chi_\Sigma$ is a character of $\Lambda_\Sigma$ with coefficients in $E$, we may consider $\chi_\Sigma$ as a character of $A_\Sigma \JM[\Sigma] = \Lambda_\Sigma \times \Lambda^\Sigma \JM$ by imposing $\chi(\Lambda^\Sigma \JM) = 1$. Multiplication by $\chi_\Sigma$ induces an automorphism $m_{\chi_\Sigma}$ of $\hat{\Lambda^\Sigma}$, and the description of $\Xscr^2_{A \JM, \sigma}(K^S)$ via the $*$-action shows that
\begin{align}\label{eqn: twisting chi on Lambda Sigma}
    \Xscr^2_{A \JM, \sigma \otimes \chi_\Sigma}(K^S)
    \simeq
    \Xscr^2_{A \JM, \sigma}(K^S)
    \times_{\hat{\Lambda_\Sigma}, m_{\chi_\Sigma}} \hat{\Lambda_\Sigma}.
\end{align}
We can now prove \Cref{proposition: relation with thick eigenvariety} \itemnumber{2}.

\vspace{0.5cm}
\begin{proof}[Proof of \Cref{proposition: relation with thick eigenvariety} \itemnumber{2}]
Assume that $\Sigma = S$.
By the previous paragraph we may assume that $\Lambda$ acts trivially on $\sigma$.
Let $\Omega \subseteq \hat{ \JM }, \Theta \subseteq \hat \Lambda$, and $R = \O(\Omega \times \Theta)$.
By \Cref{prop: comparison of finite slopes},
\begin{align*}
    & \Tor^{\Hcal_{S, A \JM, R}^-}_i(R(\delta_{\Omega \times \Theta}), H_j(K^p J, \Acal^\Sigma_{\sigma_{\Omega \times \Theta}, U, *}) ) \\
    & = 
    \begin{cases}
        R(\delta_{\Omega \times \Theta}) \otimes_{\Hcal_{S, A \JM, \O(\Omega)}^-} H_j(K^p J, \Acal^\Sigma_{\sigma_{\Omega}, U, *}) & \text{ if } i = 0, \\
        0 & \text{ if } i \neq 0.
    \end{cases}
\end{align*}
Thus, the spectral sequence in \Cref{corollary: spectral sequence for affinoids} collapses at the $E^2$ page.
As commented upon when proving part \itemnumber{1}, this implies the result.
\end{proof}

Putting everything together, we have proven the following.

\begin{theorem}
    Let $\Xscr_{A \JM, \sigma, U}^\Sigma(K^S)$ be the eigenvariety constructed in \Cref{subsection: construction of eigenvarieties} using $S$-arithmetic cohomology and $\Xscr_{\sigma|_{\JM}, U, *}^\Sigma(K^S)$ the eigenvariety constructed in this section using overconvergent cohomology.
    The underlying reduced varieties are isomorphic and, if $\Sigma = S$, then the varieties themselves are isomorphic.
\end{theorem}

\subsection{Some properties of eigenvarieties}
Throughout this section, we assume that $\Lambda$ acts on $\sigma$ by a character $\chi$.
Fix $a \in \Lambda_\Sigma^{--}$ as in \Cref{subsection: eigenvarieties overconvergent homology}.

\begin{corollary}
The variety $\Xscr_{V, U}^\Sigma(K^S)$ admits an admissible covering by open affinoids that are finite over their image in $\hat{\JM[\Sigma]}$. In particular, $\Xscr_{V, U}^\Sigma(K^S)$ and $\Yscr_{V, U}^\Sigma(K^S)$ (when defined) have dimension at most $\dim( \Zbf_{\Sigma} )$, where $\Zbf_{\Sigma}$ is the center of $\Mbf_\Sigma$.
\end{corollary}
\begin{proof}
The statement on admissible covers follows from \Cref{proposition: eigenvarieties fiber product},  \Cref{proposition: eigenvarieties fiber product 2}, \Cref{eqn: twisting chi on Lambda Sigma}, and the corresponding fact for $\Xscr_{\sigma|_{\JM}, U, *}^\Sigma(K^S)$, where such an admissible cover is given by the preimages of the $\Zscr_{\Omega, h}$.
Thus, $\Xscr_{V, U}^\Sigma(K^S)$ and $\Yscr_{V, U}^\Sigma(K^S)$ have dimension bounded by that of $\hat{\JM[\Sigma]}$, which is $\dim(\Zbf_\Sigma)$ by \Cref{proposition: dimension of character space}.
\end{proof}

\begin{proposition}
Let $\delta$ be a point of $\hat{ M_\Sigma }(\bar \Q_p)$. Then, the fiber of $\delta$ on $\Xscr_{V, U}^\Sigma(K^S)$ is natural bijection with systems of Hecke eigenvalues for the action of $\Tbb^S(K^S)$ on $S$-arithmetic homology $H_*(K^S, \SigmaInd(U, V \otimes \delta))_{\bar \Q_p}$ (or on $H^*(K^S, \SigmaInd(U, V \otimes \delta)')_{\bar \Q_p}$).
\end{proposition}
\begin{proof}
The point $\delta$ is defined over some finite extension $E'$ of $E$ in $\bar \Q_p$, and by base change we may assume that $E' = E$.
We will also write $\delta$ for the restriction of $\delta$ to $A_\Sigma \JM[\Sigma]$, i.e. for the image of $\delta$ under the map $\hat{M_\Sigma} \to \hat{A_\Sigma \JM[\Sigma]}$. Let $a \in A_\Sigma^{--}$.
Choose an affinoid open $\Omega \subseteq \hat{\JM[\Sigma]}$ containing the image of $\delta$ in $\hat{\JM[\Sigma]}$ and $h \geq v_p(\chi \delta(a))$ such that $C_\bullet( K^S J, \Acal_{\sigma_\Omega,U,*}^{s-\Sigmaan})$ admits a slope-$\leq h$ decomposition with respect to $U_a$.
If we let $\Ucal$ be the preimage in $\hat{ M_\Sigma }$ of $\Omega \times \Theta(\Omega, \hbf)$ for some $\hbf$ as in the previous section, then $C_\bullet( K^S J, \Acal_{\sigma_\Ucal,U,*}^{s-\Sigmaan})$ admits a slope decomposition.
Set $R = \O(\Ucal)$ and $\Theta = \Theta(\Omega, \hbf)$.
By \Cref{lemma: eigenvectors decomposition} (and the paragraph above it) and \Cref{lemma: untwisting homology}, there are $\Tbb^S(K^S)$-equivariant quasi-isomorphisms
\begin{align*}
    C_\bullet(K^S, \SigmaInd(U,V_{\Omega \times \Theta})) \otimesL_R E
    & \simeq 
    R(\chi \delta_{\Omega \times \Theta}) \otimesL_{\Hcal^-_{S,R}} C_\bullet( K^S J, \Acal_{\sigma_{\Omega \times \Theta},U,*}^{s-\Sigmaan})_{\leq h} \otimesL_R E
    \\
    & \simeq 
    R(\chi \delta_{\Omega \times \Theta}) \otimesL_{\Hcal^-_{S,R}} C_\bullet( K^S J, \Acal_{\sigma \otimes \delta,U,*}^{s-\Sigmaan})_{\leq h}
    \\
    & \simeq 
    E(\chi \delta) \otimesL_{\Hcal^-_{S,E}} C_\bullet( K^S J, \Acal_{\sigma \otimes \delta,U,*}^{s-\Sigmaan})_{\leq h} \\
    & \simeq 
    C_\bullet(K^S, \SigmaInd(U,V \otimes \delta)),
\end{align*}
where the map $R \to E$ corresponds to the point $\delta$ in $\Omega \times \Theta$.
Here, the second line follows from \Cref{proposition: properties overconvergent homology} \itemnumber{3} and the fact that $C_\bullet( K^S J, \Acal_{\sigma_{\Omega \times \Theta},U,*}^{s-\Sigmaan})_{\leq h}$ is a projective $R$-module,
and the third line follows from the fact that we can find a resolution $P_\bullet$ of $R(\chi \delta_{\Omega \times \Theta})$ by projective $\Hcal^-_{S, R}$-modules such that $P_\bullet \otimes_R E$ is a resolution of $E(\chi \delta)$ (this follows by twisting from the corresponding result for the trivial character, where it is true since such a resolution can be defined over $\Q_p$).
Thus, we obtain a spectral sequence relating the $S$-arithmetic homologies of $\SigmaInd(U,V_{\Omega \times \Theta})$ and of $\SigmaInd(U,V \otimes \delta)$, and the same analysis as in \cite[Theorem 4.3.3]{hansen_universal_eigenvarieties} gives the result.
\end{proof}

Let $\lambda = (\lambda_\Sigma, \lambda^\Sigma)$ be a dominant regular weight of $\Mbf_{\Sigma,E} \times \Gbf_{\Scalp \setminus \Sigma,E}$.
We also assume that $U = V_{G_{\Scalp \setminus \Sigma}}(\lambda^\Sigma)$ and $\sigma_\Sigma = \sigma_{\Sigma, \sm} \otimes_E V_{M_\Sigma}(\lambda_\Sigma)'$, where $\sigma_{\Sigma, \sm}$ is smooth. Let us set $V_\sm = \cInd_{\HM}^M (\sigma_{\Sigma, \sm} \boxtimes \sigma^\Sigma)$.

\begin{definition}
Let $x$ be a point of $\Xscr_{V, U}^\Sigma(K^S)$.
\begin{enumerate}
    \item We say that $x$ has locally algebraic (resp. dominant regular) weight if its image in $\hat{M}$ is a locally algebraic character $\delta = \delta_\sm \delta_\alg$ (resp. if moreover $\lambda_\Sigma \delta_\alg$ is a dominant regular of $\Gbf_\Sigma$).
    \item We say that $x$ is \emph{classical} if it has dominant regular weight $\delta = \delta_\sm \delta_\alg$ and the localisation of $H_*(K^S, V_{G_{\Scalp}}(\lambda \delta_\alg)' \otimes_E \Ind_P^G ( V_\sm \otimes \delta_\sm)^\sm )$ at the maximal ideal of $\Tbb^S(K^S)$ corresponding to $x$ is non-zero.
    \item We say that $x$ is \emph{non-critical} if it is classical of weight $\delta = \delta_\sm \delta_\alg$ and the natural homomorphism $$
        H_*(K^S, V_{G_{\Scalp}}(\lambda \delta_\alg)' \otimes_E \Ind_P^G ( V_\sm \otimes \delta_\sm)^\sm ) \to H_*(K^S, \SigmaInd_P^G (U, V \otimes \delta) )
    $$
    is an isomorphism after localisation at the maximal ideal above.
\end{enumerate}
\end{definition}

We can make analogous definitions for points of $\Xscr_{\sigma|_{\JM}, U, *}^\Sigma(K^S)$.
\Cref{theorem: classicality} gives a numerical criterion on a locally algebraic character $\delta$ for its fiber in $\Xscr_{V, U}^\Sigma(K^S)$ to consist of non-critical classical points.
We will call the points satisfying this criterion \emph{numerically non-critical}.

\begin{proposition}\label{proposition: classical points are dense}
Assume that $G_\infty$ is compact, $\Sigma = S$ and that for all $v \in \Sigma$, $\Res_{F_v/\Q_p} \Gbf_v$ is the restriction of scalars of a \emph{split} reductive group over a finite extension of $\Q_p$.
Then, the set of (numerically) non-critical classical points is a Zariski dense in $\Xscr_{V, U}^\Sigma(K^S)$ and $\Xscr_{\sigma|_{\JM}, U, *}^\Sigma(K^S)$ and accumulates at points of locally algebraic weight.
Moreover, $\Xscr_{V, U}^\Sigma(K^S)$ and $\Xscr_{\sigma|_{\JM}, U, *}^\Sigma(K^S)$ are equidimensional of dimension $\dim(\Zbf)$, where $\Zbf$ is the center of $\Mbf$.
\end{proposition}
\begin{proof}
This follows a standard argument.
Since $G_\infty$ is compact, the topological space $\Gbf(F) \backslash \Bscr_\infty \times \Gbf(\A_F) / \Kcal^S J$ is a disjoint union of points and the complexes $C_\bullet(K^S J, \Acal_{\sigma \otimes \delta, U, *}^{s-\an})$ are concentrated in degree 0 for any continuous character $\delta$ of $\JM$.
Let $\Omega \subseteq \hat{ \JM }$ be an irreducible admissible affinoid open subspace and $h \geq 0$ such that $C_0(K^S J, \Acal_{\sigma_\Omega, U, *}^{s-\an})$ admits a slope-$\leq h$ decomposition.
Then $H_0(K^S J, \Acal_{\sigma_\Omega, U, *}^{s-\an})_{\leq h}$ is a finite projective $\O(\Omega)$-module.
By \cite[Lemme 6.2.10]{chenevier_gln}, the map $\Sp(\Tbb_{\Omega, h}) \to \Omega$ is surjective after restricting it to any irreducible component.
Thus, all irreducible components of $\Xscr_{\sigma|_{\JM}, U, *}^\Sigma(K^S)$ have the same dimension as $\hat{ \JM }$, that is $\dim(\Zbf)$.
If $\Ucal$ is the preimage of $\Omega \times m_{\chi_\Sigma}^{-1} \Theta(\Omega, \hbf)$ on $\hat{M}$ (cf. \Cref{eqn: twisting chi on Lambda Sigma}), then by \Cref{proposition: eigenvarieties fiber product}, \Cref{proposition: eigenvarieties fiber product 2}, \Cref{lemma: map of algebras in eigenvarieties} and their proofs, the value at $\Ucal$ of coherent sheaf $\mathscr H$ on $\hat{M}$ is a direct summand of $\O(\Ucal) \otimes_{\O(\Omega \times \Theta(\Omega, \hbf))} H_0(K^S J, \Acal_{\sigma_\Omega, U, *}^{s-\an})_{\leq h}$ (with a twisted action of $\Tbb^S(K^S)$). In particular, the same analysis applies to $\Xscr_{V, U}^\Sigma(K^S)$ instead of $\Xscr_{\sigma|_{\JM}, U, *}^\Sigma(K^S)$.

For the Zariski density and accumulation statement, it's enough to show that every irreducible component has a point of locally algebraic weight, and any such point has arbitrarily small affinoid neighbourhoods where non-critical classical points are dense.
For simplicity, let us focus only on the case of $\Xscr_{\sigma|_{\JM}, U, *}^\Sigma(K^S)$.
The previous paragraph implies that the morphism $\Xscr_{\sigma|_{\JM}, U, *}^\Sigma(K^S) \to \Zscr$ sends irreducible components to irreducible components.
Since the image of an irreducible component of $\Zscr$ in $\hat{ \JM }$ is a Zariski-open subset, the same is true for the images of irreducible components of $\Xscr_{\sigma|_{\JM}, U, *}^\Sigma(K^S)$.
Fix a connected component $C$ of $\hat{\JM}$ and a smooth character $\phi$ lying on it.
By \Cref{proposition: small slope weights are dense}, the product of $\phi$ and dominant regular weights in $C$ are Zariski dense and accumulation.
In particular, their intersection with the image of an irreducible component of $\Xscr_{\sigma|_{\JM}, U, *}^\Sigma(K^S)$ lying above $C$ is nonempty.
Therefore, every such irreducible component contains a point with locally algebraic weight.
Now, let $x$ be such a point and $\delta = \delta_\sm \delta_\alg$ its image in $\hat{ \JM }$.
Fix $h \geq 0$ and $\Omega \subseteq \hat{ \JM }$ a connected affinoid open such that $x$ lies in $\Sp(\Tbb_{\Omega, h})$.
By \Cref{proposition: small slope weights are dense} again, the weights of the form $\delta_\sm \psi$ with $\psi \lambda_\Sigma$ dominant regular and $v_p((w \cdot \psi \lambda_\Sigma)(a)) - v_p((\psi \lambda_\Sigma)(a)) > h$ are Zariski dense in $\Omega$ and accumulate at $\delta$.
Their preimages in $\Sp(\Tbb_{\Omega, h})$ are thus also Zariski dense and accumulate at $x$ by finiteness of the morphism $\Sp(\Tbb_{\Omega, h}) \to \Omega$ and \cite[Lemme 6.2.8]{chenevier_gln}.
The condition that $v_p((w \cdot \psi \lambda_\Sigma)(a)) - v_p((\psi \lambda_\Sigma)(a)) > h$ implies that these points are numerically non-critical.
\end{proof}

\subsection{Comparison with the work of Breuil--Ding}

We will now compare our construction of eigenvarieties with that of \cite{breuil_ding_bernstein_eigenvarieties}.
We need to make additional assumptions to place ourselves in the setting of \loccit~ 
Assume that $F$ is a totally real number field and $\tilde F / F$ is a totally imaginary quadratic extension of $F$ over which all primes in $\Scalp$ split.
We will assume that $\Gbf$ is a unitary group over $F$ that becomes isomorphic to $\GL_n$ after base change to $\tilde F$ and that $G_\infty$ is compact.
For each $v \in \Scalp$, fix a prime $\tilde v$ in $\tilde F$ dividing it, as well as an isomorphism $\Gbf_{F_v} \simto \Gbf_{{\tilde F}_{\tilde v}} \simto \GL_{n / {\tilde F}_{\tilde v}}$.
We will use these isomorphisms to identify $\Gbf_{F_v}$ and $\GL_{n / {\tilde F}_{\tilde v}}$, and we will do so without comment.
Assume also that $S = \Sigma = \Scalp$, and that for each $v \in \Scalp$ the subgroups $\Tbf_v$ and $\Sbf_v$ correspond to the subgroups of diagonal matrices, and that the Borel subgroup $\Bbf_{p,E}$ is of the form $E \times_\Qp \prod_{v \in \Scalp} \Res_{F_v / \Q_p} \Bbf_v$, where $\Bbf_v$ corresponds to the subgroup of upper-triangular matrices.
The parabolics $\Pbf_v$ contain $\Bbf_v$, so $\Mbf_v$ corresponds to a subgroup of block-triangular matrices, with block sizes $n_{v, 1}, ..., n_{v, r_v}$, say.
Moreover, $\Abf_v$ is the full center of $\Mbf_v$, consisting of diagonal matrices with identical entries in each block.
Fix a uniformiser $\varpi_v$ for all $v \in \Scalp$, and let $\Lambda_v$ consist of diagonal matrices with entries
$$
    (\overbrace{\varpi_v^{e_1}, ..., \varpi_v^{e_1}}^{n_{v,1}}, ..., \overbrace{\varpi_v^{e_{r_v}}, ..., \varpi_v^{e_{r_v}}}^{n_{v,r_v}}).
$$

The assumption that $G_\infty$ is compact implies that arithmetic (co)homology occurs only in degree 0, and similarly for completed cohomology (cf. \Cref{subsection: completed cohomology}).
Moreover, the systems of Hecke eigenvalues in the $p$-arithmetic (co)homology groups considered in \Cref{section: homology} occur all in degree 0, as they are in bijection with systems of Hecke eigenvalues in the corresponding arithmetic (co)homology groups, which will already be detected by the zeroth Tor and Ext groups over the Hecke algebra at $p$.
Thus, we may (and will) always focus our discussion to (co)homology in degree 0, where we can apply \Cref{proposition: p-arithmetic and completed}.

Let us recall the construction of eigenvarieties in \cite{breuil_ding_bernstein_eigenvarieties}.
Assume that $E$ contains the $m$-th roots of unity in $\bar \Q_p$ for any $m \leq n$.
Fix a cuspidal Bernstein component $\Omega$ for $\Mbf$ (defined over $E$) and a dominant weight $\lambda$ for $\Mbf$. Let $\pi$ be a point of $\Omega$, corresponding to a supercuspidal representation of $\Mbf$ over $\bar \Q_p$ that we also denote by $\pi$, and write $\pi = \bar \Q_p \otimes_E \cInd_{A \JM}^M (\sigma)$ for some representation $\sigma$ of $A \JM$ for $\JM := \prod_{v \in \Scalp} \GL_n(\O_{F_v})$.
For any admissible locally analytic representation $W$ of $G$ over $E$, set
\begin{align*}
    J_P(W)_\lambda
    & :=
    \Hom_{\m_E^\der} ( V_M(\lambda), J_P(W)),
    \\
    B_{\Omega, \lambda} ( W )
    & :=
    \Hom_{\JM} (\sigma, J_P(W)_\lambda \cotimes_E \Ccal^\la(A_0, E ) ),
\end{align*}
where $\m^\der_E$ is the Lie $E$-algebra of $\Mbf^\der$ and $J_P(W)$ denotes Emerton's locally analytic Jacquet functor.
The representation $B_{\Omega, \lambda} ( W )$ is endowed with an action of $A = \Lambda \times A_0$ by letting $\Lambda$ act on $J_P(W)_\lambda$ and $A_0$ act on $\Ccal^\la(A_0, E )$, and these actions make $B_{\Omega, \lambda} ( W )$ an essentially admissible locally analytic representation of $A$.
In particular, it gives rise to a coherent sheaf $\mathcal H$ on $\hat A$ endowed with an action of $\Tbb^p(K^p)$.
When $W = \tilde H^0(K^p)$, we will write $\mathcal D_{\Omega, \lambda} (K^p)$ for its relative spectrum.
In fact, the action of $\Lambda$ factors through the action of the Bernstein spectrum $\mathcal Z_\Omega = \End_M(\cInd_{\JM}^M (\sigma))$ on the first argument of
$$
    \Hom_{M} (\cInd_{\JM}^M (\sigma), J_P(W)_\lambda \cotimes_E \Ccal^\la(A_0, E ) ) \simeq B_{\Omega, \lambda} ( W ).
$$
If $\Tcal_A^\un$ denotes the group of unramified characters $\delta$ of $A$ such that $\pi \otimes (\delta \circ \det) \simeq \pi$, then $\mathcal Z_\Omega \simeq E[\Lambda]^{\Tcal_A^\un}$, and so $(\Spec \mathcal Z_\Omega)^\an \simeq \hat{\Lambda} / \Tcal_A^\un$.
Hence, the coherent sheaf on $\hat A$ associated to $B_{\Omega, \lambda} ( W )$ in fact descends to one on $\hat{A} / \Tcal_A^\un$.
When $W = \tilde H^0(K^p)$, we call the resulting relative spectrum of the Hecke algebra $\mathcal E_{\Omega, \lambda} (K^p)$.

Note that the block-wise determinant map $\det_{M} \colon M \to A$ sending a block-diagonal matrices $(m_{v,i})_{v,i}$ to the diagonal matrix with entries
$$
    (\overbrace{\det(m_{v,1}), ..., \det(m_{v,1})}^{n_{v,1}}, ..., \overbrace{\det(m_{v,r_v}), ..., \det(m_{v,r_v})}^{n_{v,r_v}})_v
$$
induces isomorphisms $\hat A \simto \hat M, \hat{ A_0 } \simto \hat{ \JM }$ and $\hat \Lambda \simto (\hat M)^\un$, where $(\hat M)^\un \subseteq \hat M$ denotes the (closed) subspace of unramified characters.
Moreover, the subgroup $\Tcal_M^\un$ of $\Tcal_M$ (cf. \Cref{lemma: can take quotients of X}) consisting of unramified characters is isomorphic to $\Tcal_A^\un$ in the same way. In particular, we have morphisms $\hat A / \Tcal_A^\un \simto \hat M / \Tcal_M^\un \to \hat M / \Tcal_M$.
Write $q$ for the composition of these morphisms and the multiplication by $\delta_P^{-1}$ morphism $\hat M / \Tcal_M \to \hat M / \Tcal_M$.

Let $V_\sm = \cInd_{\JM}^M(\sigma)$ and $V = V_M(\lambda) \otimes_E V_\sm$, and let $U = 1$ be the trivial representation.
To make comparisons with \cite{breuil_ding_bernstein_eigenvarieties} easier, we will slightly modify our conventions.
Namely, when taking locally analytic inductions, we will induce from the opposite parabolic $\bar P$ instead of $P$.
As a consequence of this change, we will sometimes have to replace the simple positive roots $\Delta_E$ of $E \times_\Qp \Res_{F / \Q} \Gbf$ by their negations $\bar \Delta_E$, and $\Lambda^-$ by $\Lambda^+$.
We will also denote by $\bar \Xscr_{V}(K^p)$ and $\bar \Yscr_{V}(K^p)$ the analogues of $\Xscr_{V, 1}^{\Scalp}(K^p)$ and $\Yscr_{V, 1}^{\Scalp}(K^p)$ respectively.
Note that the algebraic part of $V$, $V_M(\lambda)$ is the dual of the irreducible algebraic representation of $M$ of highest weight $- \lambda$ \emph{with respect to the opposite Borel} $\bar \Bbf$.

Before stating and proving the main result of the section comparing the two constructions, let us sketch roughly how the comparison goes.
Recall from \cite[Proposition 3.2.2]{breuil_ding_bernstein_eigenvarieties} that a $\bar \Q_p$-point $x$ of $\mathcal E_{\Omega, \lambda} (K^p)$ corresponds to a triple $(\eta_x, \pi_x, \chi_x)$ where $\eta_x \colon \Tbb^S(K^S) \to \bar \Q_p$ is a system of Hecke eigenvalues, $\pi_x$ is a supercuspidal representation of $M$ in the Bernstein component $\Omega$ and $\chi_x \colon A_0 \to \bar \Q_p^\times$ is a continuous character.
On the other hand, $\bar \Q_p$-points $y$ of $\bar \Yscr_{V}(K^p)$ correspond to $(\nu_y, \delta_y)$ where $\nu_y \colon \Tbb^S(K^S) \to \bar \Q_p$ is a system of eigenvalues and $\delta_y$ is an equivalence class of characters of $M$ modulo those characters which twist $\pi = \bar \Q_p \otimes V_\sm$ into itself.
Given such a point $x$, we may write $\pi_x = \pi \otimes \psi$ for some unramified character $\psi$ of $M$, and the point $x$ will then correspond to the point $y$ with $\nu_y = \eta_x$ and $\delta_y = \psi (\chi_x \circ \det_M) \delta_P^{-1}$.
The only subtlety in making this precise is how one reconciles the fact that one construction relies Emerton's Jacquet functor whereas the other relies on parabolic induction, and these are not adjoint in general.
However, the adjunction theorem of Bergdall--Chojecki \cite[Theorem A]{bergdall_adjunction} can be applied at numerically non-critical points, and using their Zariski density one can extend the comparison to all points by an interpolation argument.

\begin{theorem}
There is an isomorphism of reduced varieties
$$
    \mathcal E_{\Omega, \lambda} (K^p)^\red \simeq \bar \Yscr_{V}(K^p)^\red \times_{\hat M / \Tcal_M, q} (\hat A / \Tcal_A^\un)
$$
(here, the superscript $\red$ is used to denote the underlying reduced varieties).
\end{theorem}
\begin{proof}
Let us write $\delta_P^{-1} \mathcal D_{\Omega, \lambda} (K^p)$ for the variety $\mathcal D_{\Omega, \lambda} (K^p) \times_{\hat M, \delta_P} \hat M$ over $\hat M$ (of course, this is isomorphic to $\mathcal D_{\Omega, \lambda} (K^p)$ but not as varieties over $\hat M$).
We will start by showing that the images of $\bar \Xscr_{V}(K^p)$ and $\delta_P^{-1} \mathcal D_{\Omega, \lambda} (K^p)$ in $\hat M$ have the same underlying analytic subset.
For this, we will use that numerically non-critical points are Zariski-dense in $\bar \Xscr_{V}(K^p)$.
Let us unwind the definition of numerical non-criticality, \Cref{definition: small slope}, in this case.
Write $\omega$ for the central character of $V_\sm$, so that the central character of $V \otimes \delta$ is $\chi_\delta := \omega \lambda \delta$ for any locally algebraic character $\delta = \delta_\sm \delta_\alg \colon M \to E^\times$.
Write $V_\sm = \bigboxtimes_{v,i} V_{\sm, v, i}$ according to the decomposition $M \simeq \prod_{v,i} \GL_{n_i} (F_i)$ and similarly for $V$, and write $\omega = (\omega_{v,i})_{v,i}$ (with each $\omega_{v,i}$ a character of $F_v^\times$) and $\delta = (\delta_{v,i})_{v,i}$ accordingly.
We will also write $\mu_{v, i} = (\mu_{v, \tau, i})_\tau$ for the weights of $\delta_{v,i}$.
As $\delta$ factors through the block-wise determinant $M \subseteq A$, we will sometimes consider each $\delta_{v,i}$ as a character on $F_v^\times$.
Write also $\lambda = (\lambda_{v,\tau,j})_{v,\tau,j}$, where $\tau$ runs through the $\Q_p$-algebra embeddings of $F_v$ into $E$ and $1 \leq j \leq n$, to mean $\lambda(\mathbf t) = (\prod_{j=1}^n t_{v, \tau, j}^{\lambda_{v,\tau, j}})_{v, \tau}$ for $\mathbf t = (t_{v,\tau,j})_{v, \tau, j} \in \Tbf(E) \simeq \prod_{v,\tau,j} E^\times$.
For such $\tau$ and $v$, set $s_{v,i} = n_{v, 1} + \cdots + n_{v, i}$ and $\alpha_{v,\tau,j}$ for the simple negative root of $\Mbf_v$ given by $(t_{v,\tau,k})_{v, \tau, k} \mapsto t_{v, \tau, j}^{-1} t_{v, \tau, j+1}$.
Then $\bar \Delta_E \setminus \bar \Delta_{M, E} = \bigcup_{v, \tau} \{ \alpha_{v, \tau, s_{v,1}}, ..., \alpha_{v, \tau, s_{v, r_v-1}} \}$.
The condition \Cref{eqn: inequality small slope} from \Cref{definition: small slope} applied to $V \otimes \delta$, the root $\alpha_{v,\tau,s_{v,k}}$, and the diagonal element $a_{v,k} \in \Lambda_v^+$ with entries
$$
    (\overbrace{\varpi_v, ..., \varpi_v}^{s_{v,k}}, 1, ..., 1)
$$
asks that
\begin{align*}
    \sum_{i=1}^k v_p ( \omega_{v,i} (\varpi_v) \delta_{\sm,v,i} ( \varpi_v) )
    <
    \frac1{[F_v : \Q_p]} \Bigg( \Bigg. &
    - \left( \sum_{\tilde \tau \colon F_v \into E} \left( \sum_{j=1}^{s_{v,k}} \lambda_{v, \tilde \tau, j} + \sum_{i=1}^{k} n_j \mu_{v, \tilde \tau, i} \right) \right) \\
    & + (\lambda_{v, \tau, s_{v,k}} + \mu_{v, \tau, k} - \lambda_{v, \tau, s_{v,k} + 1} - \mu_{v, \tau, k + 1} + 1)
    \Bigg. \Bigg)
\end{align*}
By \Cref{proposition: classical points are dense}, the analytic subset underlying the image of $\bar \Xscr_{V}(K^p)$ in $\hat M$ is the Zariski closure of the locally algebraic characters $\delta$ such that $- \lambda \delta_\alg$ is dominant with respect to $\bar \Bbf$ and that satisfy this equation.
But this exactly the numerical classicality condition from \cite[Proposition 3.2.9]{breuil_ding_bernstein_eigenvarieties} for the point corresponding to the character $\delta_P \delta$:
note that we are using different normalisations for the valuations and that,
writing $\delta_{?} = \delta_{?,\un} \delta_{?,\ram}$, where $? \in \{ \emptyset, \sm, \alg \}$, $\delta_{?,\un}$ is trivial on $A_0$ and $\delta_{?,\ram}$ is trivial on $\Lambda$, the representation $\pi_x$ in \loccit~ (associated to the image of $x$ in $\hat M$) is $V_\sm \otimes \delta_\un \simeq (V_\sm \otimes \delta_{\sm,\un}) \otimes \delta_{\alg, \un}$, and hence the restriction of its central character to $\Lambda$ is the same as that of $(V_\sm \otimes \delta_\sm) \otimes \delta_\alg$.
Using the density of numerically non-critical points in the image of $\bar \Xscr_{V}(K^p)$ in $\hat M$, and the analogous result \cite[Theorem 3.2.11]{breuil_ding_bernstein_eigenvarieties} for $\mathcal D_{\Omega, \lambda} (K^p)$, our claim follows.

In particular, the underlying analytic subsets of the images of $\bar \Xscr_{V}(K^p)$ and $\delta_P^{-1} \mathcal D_{\Omega, \lambda} (K^p)$ in $\hat{ \JM } \times \Gm$ under the composition $r$ of the maps $\hat M \to \hat{ \JM } \times \hat \Lambda \to \hat{ \JM } \times \Gm$, where the last map is given by restriction to the subgroup generated by some $a \in \Lambda^{--}$, also agree, and by \Cref{proposition: affinoid covering} is contained in the Fredholm hypersurface $\Zscr$ from \Cref{subsection: eigenvarieties overconvergent homology}.
To finish the proof of the theorem, we will use the interpolation theorem \cite[Theorem 5.1.6]{hansen_universal_eigenvarieties}.
In the notation of \cite{hansen_universal_eigenvarieties}, $\bar \Xscr_{V}(K^p)$ is the eigenvariety associated to the eigenvariety datum $(\hat{\JM}, \Zscr, r_* \bar{\mathscr H}, \Tbb(K^p), \psi_1)$, where $\bar{\mathscr H}$ is the analogue for $\bar P$ of $\mathscr H$ and $\psi_1$ is the natural morphism $\psi_1 \colon \Tbb(K^p) \to \End_\Zscr(r_* \mathscr H)$, and $\delta_P^{-1} \mathcal D_{\Omega, \lambda} (K^p)$ is the eigenvariety associated to the eigenvariety data $(\hat{\JM}, \Zscr, r_* (\delta_P)^* \Hcal, \Tbb(K^p),$ $\psi_2)$ for $\psi_2$ as $\psi_1$.
Thus, in order to apply \cite[Theorem 5.1.6]{hansen_universal_eigenvarieties} to obtain mutually inverse morphisms
$$
    \bar \Xscr_{V}(K^p)^\red \leftrightarrows \mathcal D_{\Omega, \lambda} (K^p)^\red \times_{\hat M, \delta_P} \hat M
$$
it's enough to show that for all numerically non-critical characters $\delta$, the fibers of $\bar{\mathscr H}$ and $(\delta_P)^* \Hcal$ at
$\delta$ are isomorphic as $\Tbb^S(K^S)$-modules.
Once this is shown, the theorem will follow by taking quotients.
According to \cite[Proposition 3.1.8]{breuil_ding_bernstein_eigenvarieties}, the fiber of $(\delta_P)^* \Hcal$ at $\delta$ is dual to $\Hom_{M}(V \otimes \delta_P \delta, J_P(\tilde H^0(K^p)^\la))$.
On the other hand, the fiber of $\bar{\mathscr H}$ on $\delta$ is isomorphic to $H_0(K^p, \Ind_{\bar P}^G (V \otimes \delta)^\la)$ (this can be deduced from the proof of \Cref{proposition: relation with thick eigenvariety} \itemnumber{2}),
which by \Cref{proposition: p-arithmetic and completed} is dual to $\Hom_G(\Ind_{\bar P}^G (V \otimes \delta)^\la, \tilde H^0(K^p)^\la)$.

Thus, it's enough to check that for numerically non-critical $\delta$, there is a Hecke equivariant isomorphism
$$
    \Hom_G(\Ind_{\bar P}^G (V \otimes \delta)^\la, \tilde H^0(K^p)^\la)
    \to
    \Hom_M(V \otimes \delta_P \delta, J_P(\tilde H^0(K^p)^\la)).
$$
This will follow from Bergdall--Chojecki's adjunction theorem \cite[Theorem A]{bergdall_adjunction} (or rather, the analogous result products of groups of the form considered in \cite{bergdall_adjunction}) once we check that its assumptions are satisfied.
We need to check that $\tilde H^0(K^p)^\la$ is f-$\p$-acyclic (cf. \cite[Definition 4.1]{bergdall_adjunction}), which follows from \cite[Example 4.2]{bergdall_adjunction}, and the pair $(V_M(\lambda  \delta_\alg), V_\sm \otimes \delta_P \delta_\sm)$ is non-critical with respect to $\tilde H^0(K^p)^\la$ (cf. \cite[Definition 4.4]{bergdall_adjunction}).
As explained in \cite[Remark 4.6]{bergdall_adjunction}, it's enough to show that $\omega \delta_P \delta \lambda$ is of non-critical slope in the sense of \cite[Definition 4.4.3]{emerton_jacquet_one}.
This is equivalent to $\delta$ being numerically non-critical.
\end{proof}

\appendix
\renewcommand{\thesection}{\Alph{section}}
\section{Orlik--Strauch functors}

In this appendix we will define and study the basic properties of the Orlik--Strauch functors $\Fcal_P^G$ for reductive groups over $\Q_p$ that are not necessarily the restriction of scalars of a split reductive group over a finite extension of $\Q_p$.
Let $\Gbf$ be a reductive group over $\Q_p$ and $\Pbf \subseteq \Gbf$ a $\Q_p$-parabolic, and let $\Nbf$ be its unipotent radical, $\Mbf$ a Levi subgroup and $\bar \Nbf$ the unipotent radical of the opposite parabolic.
Choose a maximal torus $\Tbf$ in $\Mbf$ defined over $\Q_p$ and a finite extension $E$ of $\Q_p$ in $\bar \Q_p$ containing a Galois extension of $\Q_p$ over which $\Tbf$ splits.
We will write $\Gbf_E = E \times_\Qp \Gbf$, $G = \Gbf(\Q_p)$, $G_E = \Gbf_E(E)$, $\g$ for the Lie $\Q_p$-algebra of $\Gbf$, $\g_E$ for the Lie $E$-algebra of $\Gbf_E$, and $U(\g), U(\g_E)$ for their universal enveloping algebras, and similarly for $\Pbf_E$, $P$, etc.
We will also fix a special point in the Bruhat--Tits building of $G$ that lies in the image of the building of $M$, and let $J$ be its stabiliser in $G$.
It is also a special point for $M$, and we will use some of the groups from \Cref{section: structure} for the facets $\Fscr_M \supseteq \Fscr$ for $M$ and $G$ in some fixed apartment for $M$ that contain it.
Note that $J$ is the pointwise stabiliser of $\Fscr$ in $G$ and $\JM = J \cap M$ is the pointwise stabiliser of $\Fscr_M$ in $M$.
More concretely, we will consider $\bar N_s$ and the analogue $\bar J_{0,s}$ of $J_{0,s}$ defined by swapping $N$ and $\bar N$ in the definition of the latter, and also the corresponding subgroups $\bar N_{E, s} := \bar N'_s, \bar J_{E,0,s} := \bar J'_{0,s}$ of $G_E$ (defined as in \Cref{subsection: congruence subgroups} with $L' = E$), as well as their rigid-analytic variants.
We will follow the arguments in \cite{orlik_strauch} very closely (as all of the arguments generalise), but will not discuss the irreducibility of representations in the image of the functor $\Fcal_P^G$.

Recall that the category $\O^{\p_E}_\alg$ is the full subcategory of the category of representations $X$ of $\g_E$ over $E$ such that
\begin{enumerate}
    \item $X$ is finitely generated over $U(\g_E)$,
    \item $X$ is a direct sum of finite-dimensional irreducible $U(\m_E)$-modules that are obtained from finite-dimensional irreducible representations of $M_E$,
    \item the action of $U(\p_E)$ is locally finite (i.e. the orbits of this action are finite-dimensional).
\end{enumerate}

Let $X \in \O^{\p_E}_\alg$ and let $V$ be an admissible smooth representation of $M$. Choose a finite-dimensional $\p_E$-subrepresentation $W$ that generates $X$ as a $U(\g_E)$-module and let $\dfrak$ be the kernel of the natural surjection $U(\g_E) \otimes_{U(\p_E)} W \to X$.
By \cite[Lemma 3.2]{orlik_strauch}, $W$ extends to an $E$-linear (algebraic) representation of $P_E$, and we may restrict it to a representation of $P \subseteq P_E$.
Let $D(G, E)$ (resp. $D(P, E)$) be the locally analytic distribution algebra of $G$ (resp. $P$) over $E$.
Recall that there is a natural injective algebra homomorphism $U(\g_E) \into D(G, E)$.
The pairing
\begin{align*}
    (D(G,E) \otimes_{D(P,E)} W) \times \Ind_P^G(W' \otimes_E V)^\la
    & \to
    \Ccal^\la(G, V)
    \\
    (\delta \otimes w, f)
    & \mapsto
    (\delta \otimes w) \cdot f \colon g \mapsto \delta(h \mapsto f(g h)(w))
\end{align*}
gives rise to a pairing $(U(\g_E) \otimes_{U(\p_E)} W) \times \Ind_P^G(W' \otimes_E V) \to \Ccal^\la(G, V)$.
Define
$$
    \Fcal_P^G(X, W, V) := \Ind_P^G(W' \otimes_E V)^\dfrak := \{ f \in \Ind_P^G(W' \otimes_E V)^\la \colon \delta \cdot f = 0 \text{ for all } \delta \in \dfrak \}.
$$

\begin{lemma}\label{lemma: appendix 1}
Assume that $V = E$ is the trivial representation.
Then, the dual of $\Fcal_P^G(X, W, E)$ is isomorphic to the following:
\begin{enumerate}
    \item $( D(G, E) \otimes_{D(P, E)} W ) / D(G, E) \dfrak$ as a $D(G, E)$-module, and $D(G, E) \dfrak$ is the submodule of elements $D(G, E) \otimes_{D(P, E)} W$ of that annihilate $\Ind_P^G(W')^\dfrak$ under the pairing $\cdot$.
    \item $D(G, E) \otimes_{D(\g_E, P, E)} X$ as a $D(G, E)$-module, where $D(\g_E, P, E) = U(\g_E) D(P, E)$ is the subalgebra of $D(G, E)$ generated by $U(\g_E)$ and $D(P, E)$.
    \item $\varprojlim_n \Ind_{\bar J_{0,n}}^J ( (U(\bar{\mathfrak n}_E)_n \otimes_E W) / \dfrak_n )$ as a $D(J, E)$-module, where $\Ind_{\bar J_{0,n}}^J$ denotes ordinary induction, $U(\bar{\mathfrak n}_E)_n$ is the completion of $U(\bar{\mathfrak n}_E)$ with respect to the norm $|\blank|_n$ described below
    and $\dfrak_n$ is the closure of $\dfrak \subseteq U(\g_E) \otimes_{U(\p_E)} W \simeq U(\bar{\mathfrak n}_E) \otimes_E W$ in $U(\bar{\mathfrak n}_E)_n \otimes_E W$.
\end{enumerate}
\end{lemma}
\begin{proof}
Statement \itemnumber{1} follows from the same arguments as \cite[Proposition 3.3 (ii)]{orlik_strauch} and similarly for \itemnumber{2} and \cite[\S 3.4]{orlik_strauch}.
Let us prove \itemnumber{3} in more detail.
The proof is essentially the same as that in \cite[\S 3.8]{orlik_strauch} except that the groups $U_P^{-,n}$ and $P^n$ in \loccit ~must be replaced by $\bar N_{n}$ and $\bar J_{0,n}$ (or their analogues over $E$).
The Iwasawa decomposition $G = J P$ shows that $\Ind_P^G(W')^\la \simeq \Ind_{J \cap P}^J (W')^\la$, and we may define a subspace $\Ind_{J \cap P}^J (W')^\dfrak$ of the latter as above, which is therefore isomorphic to $\Fcal_P^G(X, W, E)$ as a $D(J, E)$-module.
We see as in \cite[Lemma 3.9]{orlik_strauch} that
$$
    \Ind_{J \cap P}^J (W')^\la
    \simeq
    \varinjlim_n \Ind_{\bar J_{0,n}}^J ( \O( \bar \Nbf_{s} ) \otimes_\Qp W' )
    \simeq
    \varinjlim_n \Ind_{\bar J_{0,n}}^J ( \O( \bar \Nbf_{E,s} ) \otimes_E W' ).
$$
The algebra $U(\bar{\mathfrak n}_E)$ acts on $\O(\bar \Nbf_{E,n})$ by taking derivatives as usual, and hence also on $\O( \bar \Nbf_{E,s} ) \otimes_E W'$, so it makes sense to define
$( \O( \bar \Nbf_{E,s} ) \otimes_E W' )^\dfrak$.
There are isomorphisms
$$
    \Ind_{J \cap P}^J (W')^\dfrak
    \simeq
    \varinjlim_n \Ind_{\bar J_{0,n}}^J ( \O( \bar \Nbf_{s} ) \otimes_\Qp W' )^\dfrak
    \simeq
    \varinjlim_n \Ind_{\bar J_{0,n}}^J ( \O( \bar \Nbf_{E,s} ) \otimes_E W' )^\dfrak .
$$

Choose coordinates $X_1, ..., X_r$ of the root subgroups of $\bar \Nbf_E$ that identify $\bar \Nbf_{E,n}$ with the closed ball of radius $p^{-n}$ centered at the origin. This is possible, as if this holds for one $n$ then it holds for all.
Let $\bar n_1, ..., \bar n_r$ be the basis of $\bar{\mathfrak n}_E$ corresponding to this choice, so that $\bar n_i$ can be seen as taking derivatives in the direction of $X_i$, i.e. $\bar n_i X_j$ is 1 if $i = j$ and 0 otherwise.
Then, the norm $| \blank |_n$ on $U(\bar{\mathfrak n}_E)$ is defined by
$$
    \left| \sum_{(i_1, ..., i_r) \in \Z_{\geq 0}^r} a_{i_1, ..., i_r} \bar n_1^{i_1} \cdots \bar n_r^{i_r} \right|_n
    :=
    \sup_{(i_1, ..., i_r) \in \Z_{\geq 0}^r} | (i_1)! \cdots (i_r)! a_{i_1, ..., i_r} | p^{-n(i_1 + \cdots + i_r)}.
$$
Let $\O_b( \bar \Nbf_{E,n} )$ denote the ring of bounded functions on $\bar \Nbf_{E,n}$,
$$
    \O_b( \bar \Nbf_{E,n} ) = \left\{
    f = \sum_{(i_1, ..., i_r)} a_{i_1, ..., i_r} X_1^{i_1} \cdots X_r^{i_r} : |f| := \sup_{(i_1, ..., i_r)} | a_{i_1, ..., i_r} | p^{-n(i_1 + \cdots + i_r)} < \infty
    \right\},
$$
which is a Banach space with norm $f \mapsto |f|$.
The pairing $\O(\bar \Nbf_E) \times U(\bar{\mathfrak n}_E) \to E$ given by $(f, \mathfrak r) \mapsto (\mathfrak r \cdot f)(1)$ extends to a pairing $\O_b( \bar \Nbf_{E,n} ) \times U(\bar{\mathfrak n}_E)_n \to E$ that realises $\O_b( \bar \Nbf_{E,n} )$ as the continuous dual of $U(\bar{\mathfrak n}_E)_n$, and similarly $\O_b( \bar \Nbf_{E,n} ) \otimes_E W'$ is isomorphic to the dual of $U(\bar{\mathfrak n}_E)_n \otimes_E W$. 
Now, the proof is completed in the same way as in \cite[\S 3.4]{orlik_strauch} by observing that
\begin{align*}
    \varinjlim_n \Ind_{\bar J_{0,n}}^J ( \O( \bar \Nbf_{E,n} ) \otimes_E W' )^\dfrak
    & \simeq
    \varinjlim_n \Ind_{\bar J_{0,n}}^J ( \O_b( \bar \Nbf_{E,n} ) \otimes_E W' )^\dfrak
    \\ & \simeq
    \varinjlim_n \Ind_{\bar J_{0,n}}^J ( \O_b( \bar \Nbf_{E,n} ) \otimes_E W' )^{\dfrak_n}
    \\ & \simeq
    \left( \varprojlim_n \Ind_{\bar J_{0,n}}^J ( (U(\bar{\mathfrak n}_E)_n \otimes_E W) / \dfrak_n ) \right)'. \qedhere
\end{align*}
\end{proof}

\begin{theorem}
The locally analytic representation $\Fcal_P^G(X, W, V)$ does not depend on $W$ up to canonical isomorphism. Thus, we may write simply $\Fcal_P^G(X, V)$, and it satisfies the following properties:
\begin{enumerate}
    \item $\Fcal_P^G$ is a bi-functor that is contravariant in the first argument and covariant in the second.
    \item $\Fcal_P^G$ is exact on both arguments.
    \item $\Fcal_P^G(X, V)$ is an admissible locally analytic representation, and if $V$ is of finite length then it is strongly admissible.
    \item Assume that $\Pbf' \supseteq \Pbf$ is a parabolic subgroup of $\Gbf$ defined over $\Q_p$ with Levi subgroup $\Mbf' \supseteq \Mbf$ such that $X \in \O^{\p'_E}_\alg$.
    Then, $\Fcal_P^G(X, V) \simeq \Fcal_{P'}^G(X, \Ind_{M' \cap P}^{M'}(V)^\sm)$.
\end{enumerate}
\end{theorem}
\begin{proof}
All of these statements follow exactly as the corresponding analogues in \cite{orlik_strauch}:
that $\Fcal_P^G(X, W, V)$ does not depend on $W$ up to canonical isomorphism follows as in \cite[\S 4.6]{orlik_strauch} from \Cref{lemma: appendix 1} \itemnumber{2},
\itemnumber{1} follows as in \cite[Proposition 4.7]{orlik_strauch},
\itemnumber{2} and \itemnumber{4} follow as in \cite[Proposition 4.2 and Proposition 4.9]{orlik_strauch} (here one needs to use \Cref{lemma: appendix 1} \itemnumber{2} and \itemnumber{3})
and \itemnumber{3} follows as \cite[Proposition 4.8]{orlik_strauch}.
\end{proof}

\bibliographystyle{alpha}
\bibliography{references}

\end{document}